\documentclass[10pt]{article}
\usepackage[top=0.8in, bottom=1.2in, left=0.9in, right=0.9in]{geometry}

\usepackage[utf8]{inputenc}
\usepackage{amsthm,amsmath}
\usepackage{amssymb, mathrsfs, color}
\usepackage{enumerate}
\usepackage{verbatim}
\usepackage{tikz-cd}
\usepackage{enumitem}
\usepackage{indentfirst}
\usepackage{hyperref}
\hypersetup{
    colorlinks=true,
    linkcolor=blue,
    citecolor=brown,
    filecolor=magenta,
    urlcolor=blue,
}
\usepackage{titlesec}
\usepackage{fancyhdr}
\usepackage{mathtools}
\usepackage{CJKutf8}
\usepackage{enumitem}

\usetikzlibrary{decorations.pathmorphing}

\numberwithin{equation}{section}

\newtheorem{thm}{Theorem}[section]

\newtheorem{cor}[thm]{Corollary}
\newtheorem{lem}[thm]{Lemma}
\newtheorem*{lem*}{Lemma}
\newtheorem{prop}[thm]{Proposition}

\newtheorem{innercustomgeneric}{\customgenericname}
\providecommand{\customgenericname}{}
\newcommand{\newcustomtheorem}[2]{%
  \newenvironment{#1}[1]
  {%
   \renewcommand\customgenericname{#2}%
   \renewcommand\theinnercustomgeneric{##1}%
   \innercustomgeneric
  }
  {\endinnercustomgeneric}
}

\newcustomtheorem{thm*}{Theorem}
\newcustomtheorem{prop*}{Proposition}
\newcustomtheorem{cor*}{Corollary}

\theoremstyle{definition}
\newtheorem{defn}[thm]{Definition}
\newtheorem{conv}[thm]{Convention}
\newtheorem{note}[thm]{Notation}

\newtheorem{remark}[thm]{Remark}
\newtheorem{example}[thm]{Example}

\newcommand{\1}{\mathbf{1}}
\newcommand{\A}{\mathscr{A}}
\newcommand{\B}{\mathbb{B}}
\newcommand{\Bs}{\mathscr{B}}
\newcommand{\Bf}{\mathfrak{B}}
\newcommand{\C}{\mathbb{C}}
\newcommand{\Cc}{\mathcal{C}}
\newcommand{\Co}{\mathscr{C}}
\newcommand{\D}{\mathscr{D}}
\newcommand{\Do}{\mathcal{D}}
\newcommand{\E}{\mathcal{E}}
\newcommand{\Ex}{\mathbf{E}}
\newcommand{\F}{\mathcal{F}}

\newcommand{\Ga}{\mathbf{\Gamma}}

\newcommand{\Hb}{\mathbb{H}}

\newcommand{\Ic}{\mathcal{I}}
\newcommand{\Mf}{\mathfrak{M}}
\newcommand{\Nf}{\mathfrak{N}}

\newcommand{\Oh}{\mathcal{O}}
\newcommand{\Pv}{\mathbf{P}}
\newcommand{\Pc}{\mathcal{P}}
\newcommand{\Pf}{\mathfrak{P}}

\newcommand{\R}{\mathbb{R}}
\newcommand{\Se}{\mathcal{S}}
\newcommand{\Sp}{\mathbb{S}}
\newcommand{\Sc}{\mathscr{S}}

\newcommand{\Tc}{\mathcal{T}}
\newcommand{\Tf}{\mathfrak{T}}
\newcommand{\U}{\mathcal{U}}

\newcommand{\V}{\mathcal{V}}

\newcommand{\X}{\mathfrak{X}}
\newcommand{\Xs}{\mathscr{X}}
\newcommand{\Y}{\mathscr{Y}}
\newcommand{\Z}{\mathbb{Z}}
\newcommand{\Zf}{\mathfrak{Z}}
\newcommand{\Gr}{\mathbf{Gr}}

\newcommand{\re}{\operatorname{Re}}
\newcommand{\im}{\operatorname{Im}}

\newcommand{\Inv}{\mathrm{Inv}}

\newcommand{\id}{\operatorname{id}}

\newcommand{\essup}{\mathop{\operatorname{essup}}}
\newcommand{\eps}{\varepsilon}

\newcommand{\Span}{\operatorname{Span}}

\newcommand{\vol}{\operatorname{vol}}
\newcommand{\loc}{\mathrm{loc}}

\newcommand{\supp}{\operatorname{supp}}
\newcommand{\dist}{\operatorname{dist}}

\newcommand{\divg}{\operatorname{div}}
\newcommand{\rank}{\operatorname{rank}}

\newcommand{\Coorvec}[1]{\frac\partial{\partial#1}}

\providecommand{\keywords}[1]
{
  {\small
  \textbf{\textit{Keywords ---}} #1}
}

\begin{document}
\author{Liding Yao}
\title{Sharp H\"older Regularity for Nirenberg's Complex Frobenius Theorem}
\date{}

\maketitle

\begin{abstract}
Nirenberg's famous complex Frobenius theorem gives necessary and sufficient conditions on a locally integrable structure for when the manifold is locally diffeomorphic to $\mathbb R^r\times\mathbb C^m\times \mathbb R^{N-r-2m}$ through a coordinate chart $F$ in such a way that the structure is locally spanned by $F^*\frac\partial{\partial t^1},\dots,F^*\frac\partial{\partial t^r},F^*\frac\partial{\partial z^1},\dots,F^*\frac\partial{\partial z^m}$, where we have given $\mathbb R^r\times\mathbb C^m \times\mathbb R^{N-r-2m}$ coordinates $(t,z,s)$.  In this paper, we give the optimal H\"older-Zygmund regularity for the coordinate charts which achieve this realization.  Namely, if the structure has H\"older-Zygmund regularity of order $\alpha>1$, then the coordinate chart $F$ that maps to $\mathbb R^r\times\mathbb C^m \times\mathbb R^{N-r-2m}$ may be taken to have H\"older-Zygmund regularity of order $\alpha$, and this is sharp.  Furthermore, we can choose this  $F$ in such a way that the vector fields $F^*\frac\partial{\partial t^1},\dots,F^*\frac\partial{\partial t^r},F^*\frac\partial{\partial z^1},\dots,F^*\frac\partial{\partial z^m}$ on the original manifold have H\"older-Zygmund regularity of order $\alpha-\varepsilon$ for every $\varepsilon>0$, and we give an example to show that the regularity for $F^*\frac\partial{\partial z}$ is optimal.
\end{abstract}
{\small\tableofcontents}
\section{Introduction}

For $\alpha\in(0,\infty)$, let $\Co^\alpha$ be the class of H\"older-Zygmund functions\footnote{For non-integer exponents, the H\"older-Zygmund space agrees with the classical H\"older space. When $\alpha\in\Z_+$ we have $C^{\alpha-1,1}\subsetneq \Co^\alpha$. See Remark \ref{Rmk::Hold::RmkforLPHoldChar}.} of order $\alpha$ (See Definitions \ref{Defn::Intro::DefofHold}). Let $\Co^{\alpha-}=\bigcap_{\eps>0}\Co^{\alpha-\eps}$. And we denote $\Co^\infty=C^\infty$ as the class of smooth functions.

Given a smooth manifold $\Mf$, a \textbf{$\Co^\alpha$-complex Frobenius structure} on $\Mf$ is a $\Co^\alpha$-complex tangent subbundle\footnote{For (real or complex) linear spaces $V,W$, we use $V\le W$ as saying that $V$ is a (real or complex) subspace of $W$. And we use $\Se\le \C T\Mf$ as saying $\Se$ is a complex subbundle of $\C T\Mf$.} $\Se\le\C T\Mf$, such that $\Se+\bar\Se$ is also a complex tangential subbundle, and both $\Se$, $\Se+\bar\Se$ satisfy the \textbf{involutivity condition}, namely 
$$\forall X,Y\in C^1_\loc(\Mf;\Se)\Rightarrow[X,Y]\in C^0_\loc(\Mf;\Se),\quad\text{and } \forall X,Y\in C^1_\loc(\Mf;\Se+\bar\Se)\Rightarrow[X,Y]\in C^0_\loc(\Mf;\Se+\bar\Se).$$
See Definition \ref{Defn::ODE::CpxStr}.

The complex Frobenius theorem was first introduced by Louis Nirenberg. It is a generalization of the real Frobenius theorem and the Newlander-Nirenberg theorem.
\setcounter{thm}{-1}
\begin{prop}[L. Nirenberg \cite{Nirenberg}]\label{Prop::Intro::Nirenberg}
Let $r,m,N$ be nonnegative integers satisfying $r+2m\le N$. Let $\Mf$ be an $N$-dimensional smooth manifold and let $\Se$ be a smooth complex Frobenius structure on $\Mf$, such that
$\rank(\Se\cap\bar\Se)=r$ and $\rank\Se=r+m$.
Then for any $p\in \Mf$ there is a smooth coordinate system $F=(F',F'',F'''):U\subseteq \Mf\to\R^r_t\times\C^m_z\times\R^{N-r-2m}_s$ near $p$, such that $\Se|_{U}$ is spanned by the smooth vector fields $F^*\Coorvec{t^1},\dots,F^*\Coorvec{t^r},F^*\Coorvec{z^1},\dots,F^*\Coorvec{z^m}$.
\end{prop}
Here $(t,z,s)=(t^1,\dots,t^r,z^1,\dots,z^m,s^1,\dots,s^{N-r-2m})$ is the standard coordinate system for $\R^r\times\C^m\times\R^{N-r-2m}$. All the ranks and spans are in the sense of complex subspaces, not as real subspaces. 


\subsection{A first glimpse to the main results}
In this paper we find such a coordinate chart $F$ with optimal H\"older-Zygmund regularity, when the complex Frobenius structure $\Se$ has low smoothness. To obtain the best results, we compute the regularities $F'$, $F''$, $F'''$ separately, instead of merely giving a single regularity result for $F=(F',F'',F''')$.

We say a complex tangent subbundle $\Se\le\C T\Mf$ is $\Co^\alpha$, if locally it is spanned by $\Co^\alpha$-vector fields (Definitions \ref{Defn::ODE::CpxSubbd}).

\begin{thm}\label{Thm::ThmCoor1}
Let $\alpha\in(1,\infty)$, let $\Mf$ be an $N$-dimensional smooth manifold, and let $\Se$ be a $\Co^\alpha$-complex Frobenius structure such that $\rank(\Se\cap \bar\Se)=r$ and $\rank\Se=r+m$.

Then for any $p\in \Mf$ there is a $\Co^\alpha$-coordinate system $F=(F',F'',F'''):U\subseteq \Mf\to\R^r_t\times\C^m_z\times\R^{N-r-2m}_s$ near $p$, such that $\Se|_U$ is spanned by $F^*\Coorvec{t^1},\dots,F^*\Coorvec{t^r},F^*\Coorvec{z^1},\dots,F^*\Coorvec{z^m}$.

In fact $F'\in\Co^\infty$, $F''\in\Co^\alpha$, $F'''\in\Co^\alpha$. Moreover $F^*\Coorvec{t^1},\dots,F^*\Coorvec{t^r},F^*\Coorvec{z^1},\dots,F^*\Coorvec{z^m}$ are $\Co^{\alpha-}$-complex vector fields in $U$.
\end{thm}

In this theorem $F'\in\Co^\infty$, $F''\in\Co^\alpha$ and $F^*\Coorvec z\in\Co^{\alpha-}$ are all optimal, in the sense that given $\alpha>1$, we have examples of complex Frobenius structures $\Se$, such that if $F$ is a $C^1$-chart representing $\Se$, then $F''\notin\Co^{\alpha+\eps}$ for any $\eps>0$,  $F'''\notin\Co^{\alpha+\eps}$ for any $\eps>0$, and $F^*\Coorvec z\notin\Co^\alpha$, respectively. See Section \ref{Section::SharpGen}. As a corollary, the regularity statement of our coordinate chart $F\in\Co^\alpha$ is optimal.

\medskip
The subbundle $\Se+\bar\Se$ and $\Se\cap \bar\Se$ (if they have constant ranks, see Remark \ref{Rmk::ODE::RmkInvStr} \ref{Item::ODE::RmkInvStr::NotBund}) are always at least as regular as the subbundle $\Se$.  However, it is possible that $\Se+\bar\Se$ and $\Se\cap \bar\Se$ have better regularity than $\Se$. See Remark \ref{Rmk::ODE::RmkInvStr} \ref{Item::ODE::RmkInvStr::MoreReg}.
When $\Se\cap \bar\Se$ is more regular than $\Se$, we have an improvement of Theorem \ref{Thm::ThmCoor1}:

\begin{thm}\label{Thm::ThmCoor2}
Let $\alpha,\beta\in(1,\infty)$ satisfy $\alpha+1\le\beta$ and let $\Mf$ be an $N$-dimensional smooth manifold. Let $\Se$ be a $\Co^\alpha$-complex Frobenius structure, such that $\rank\Se=r+m$ and $\Se\cap\bar\Se$ is a $\Co^\beta$-complex tangential subbundle with complex rank $r$.

Then for any $p\in \Mf$ there is a $\Co^\alpha$-coordinate system $F:U\subseteq \Mf\to\R^r\times\C^m\times\R^{N-r-2m}$ near $p$, such that in addition to the consequences in Theorem \ref{Thm::ThmCoor1}, we have that $F^*\Coorvec{t^1},\dots,F^*\Coorvec{t^r}$ are $\Co^{\beta}$-vector fields on $U$.
\end{thm}

There is another characterization for representing the complex Frobenius structure using the inverse of coordinate charts, the so-called ``regular parameterization''. Let\footnote{We use $f^\Inv$ instead of $f^{-1}$ to refer the inverse map of $f$, if it is bijective.} $\Phi=F^\Inv:F(U)\subseteq\R^N\to \Mf$ be a such parameterization. For $F$ we work with the regularities of $F'$, $F''$ and $F'''$, while for $\Phi$ we work with that of $[t\mapsto\Phi(t,z_0,s_0)]$, $[z\mapsto\Phi(t_0,z,s_0)]$ and $[s\mapsto\Phi(t_0,z_0,s)]$, for any fixed $t_0,z_0,s_0$. 

For a general version of Theorems \ref{Thm::ThmCoor1} and \ref{Thm::ThmCoor2}, which includes the estimates for $\Phi$, see Theorems \ref{Thm::TrueThm1} and \ref{Thm::TrueThm2}.

\medskip
To see that $F^*\Coorvec z\in\Co^{\alpha-}$ is sharp; in Section \ref{Section::Sharpddz} we construct an example that meets the following:
\begin{thm}\label{Thm::Intro::Sharpddz}
Let $\alpha>1$. There is a $\Co^\alpha$-complex Frobenius structure $\Se$ on $\R^3$, such that $\rank\Se =1$, $\rank(\Se+\bar\Se)=2$, and there does not exist a $C^1$-coordinate chart $F:U\subseteq\R^3\to\C^1_z\times\R^1_s$ near $0\in\R^3$ such that $F^*\Coorvec z$ spans $\Se|_U$ and $F^*\Coorvec z\in\Co^\alpha(U;\C^3)$.
\end{thm}
See Theorem \ref{Thm::Final::SharpSum} for a general version.


\subsection{Historical remarks}
There are many previous discussions of complex Frobenius theorem in the nonsmooth setting.
Nijenhuis \& Woolf \cite{NijenhuisWoolf} studied Newlander-Nirenberg theorem with  nonsmooth parameters, which is  a special case of  nonsmooth complex Frobenius theorem.

Later Hill \& Taylor \cite{RoughComplex} studied nonsmooth complex Frobenius structures. They assumed the complex Frobenius structure to be at least $C^{1,1}$ in order to achieve the condition called ``Hypothesis V'' in their paper.





Recently, Gong \cite{Gong} proved a better estimate for complex Frobenius structure. He showed that for a non-integer $\alpha>1$ and a $\Co^\alpha$-complex Frobenius structure $\Se\le \C T\Mf$, locally one can find a $\Co^{\alpha-}$-coordinate chart $F:U\subseteq\Mf\to\R^r\times\C^m\times\R^{N-r-2m}$ such that $F^*\Coorvec{t^1},\dots,F^*\Coorvec{t^r},F^*\Coorvec{z^1},\dots,F^*\Coorvec{z^m}$ span $\Se$. Though he did not have the sharp result, the statement $F\in\Co^{\alpha-}$ is endpoint sharp (up to a loss of $\eps>0$ derivatives).
The proof technique in his paper is quite different for ours. He used homotopy formulae with Nash-Moser iterations.

Our result is stronger because we have not only the optimal regularity of the coordinate chart that $F\in\Co^\alpha$, but also the optimal regularity for other ingredients like the components of the coordinates and the coordinate vector fields. See Section \ref{Section::SharpGen}. 

\medskip
As the origin of the complex Frobenius theorem, the sharp estimate to Newlander-Nirenberg is due to Malgrange \cite{Malgrange}. He proved that for $\alpha>1$  one can always obtain a $\Co^{\alpha+1}$ local complex coordinate system for a $\Co^\alpha$ integrable almost complex structure. His argument also works for $\frac12<\alpha\le1$, see  \cite{RoughNN}. We remark that we cannot replace $(\Co^\alpha,\Co^{\alpha+1})$ by $(C^k,C^{k+1})$, see \cite{Liding}.

Later Street used Malgrange's method and obtained a sharp estimate for elliptic structure.
\begin{prop}[\cite{SharpElliptic}, Sharp estimate for elliptic structures]\label{Prop::Intro::SharpE}
	Let $\alpha>1$, let $\Mf$ be a $\Co^{\alpha+1}$-manifold, and let $r,m$ be integers satisfying $\dim \Mf=r+2m$. Endow $\R^r\times\C^m$ with standard coordinate system $(t,z)=(t^1,\dots,t^r,z^1,\dots,z^m)$.
	
	Let $\Se\le\C T\Mf$ be a $\Co^\alpha$-elliptic structure of rank $r+m$. Then for any $p\in \Mf$ there exists a $\Co^{\alpha+1}$-coordinate system $F:U\subseteq \Mf\to\R^r\times\C^m$ near $p$, such that $\Se|_U$ is spanned by $F^*\Coorvec{t^1},\dots,F^*\Coorvec{t^r},F^*\Coorvec{z^1},\dots,F^*\Coorvec{z^m}$.
\end{prop}

Note that in Proposition \ref{Prop::Intro::SharpE} we have $F^*\Coorvec{t^1},\dots,F^*\Coorvec{t^r},F^*\Coorvec{z^1},\dots,F^*\Coorvec{z^m}\in\Co^\alpha$ because $F\in\Co^{\alpha+1}$. However in a more general setting (Theorem \ref{Thm::ThmCoor1}), due to the influence of the parameters, we only have $F^*\Coorvec{t^1},\dots,F^*\Coorvec{t^r}$, $F^*\Coorvec{z^1},\dots,F^*\Coorvec{z^m}\in\Co^{\alpha-}$. Both results for $F$ and $F^*\Coorvec{z^1},\dots,F^*\Coorvec{z^m}$ are sharp.

\subsection{Idea of the proof}

Let $\Mf$ be a (smooth) manifold. In this section we use the following terminologies:
\begin{defn}\label{Defn::Intro::InvStr}
    We call a $C^1$ complex tangent subbundle $\Se\subseteq\C T\Mf$ is \textbf{involutive}, if for every $C^1$-vector fields $X, Y$ that are sections of $\Se$, the Lie bracket $[X,Y]$ is a $C^0$-section of $\Se$. (See also Definition \ref{Defn::ODE::InvMix}.)
\begin{itemize}[parsep=-0.3ex]
    \item An \textbf{elliptic structure} on $\Mf$ is an involutive complex subbundle $\Se\subseteq\C T\Mf$ such that $\C T\Mf=\Se+\bar\Se$. 

    \item A \textbf{complex structure}\footnote{In this paper we use ``complex structure'' to refer the almost complex structure satisfying involutivity condition. See also \cite[Section I.8]{Involutive}.} on $\Mf$ is an involutive complex subbundle $\Se\subseteq\C T\Mf$ such that $\C T\Mf=\Se\oplus\bar\Se$.
    \item An \textbf{essentially real structure} is a complex subbundle $\Se\subseteq\C T\Mf$ such that $\Se=\bar\Se$.
\end{itemize}
\end{defn}
Here $\bar\Se=\{(p,u-iv):p\in \Mf,\ u,v\in T_p\Mf,\ u+iv\in\Se_p\}$ is the complex conjugate subbundle of $\Se$.

\medskip
To prove Theorem \ref{Thm::ThmCoor1}, notice that $\Se+\bar\Se$ is an essentially real structure. By applying the real Frobenius theorem, we get a foliation of $\Co^{\alpha+1}$-manifolds, such that by restricted to each leaf $\Se$ becomes an elliptic structure. 
The theorem is proved once we have the estimate of the coordinate system and the coordinate vector fields for elliptic structures with parameters:
\begin{thm}\label{Thm::Intro::KeyOV}
Let $\Mf$ and $\Nf$ be two smooth manifolds such that $\dim\Mf=r+2m$ and $\dim\Nf=q$. Let $\Se\le(\C T\Mf)\times\Nf$ be a rank-$(r+m)$ $\Co^\alpha$-involutive subbundle such that $\Se+\bar\Se=(\C T\Mf)\times \Nf$. Then for any $(p_0,q_0)\in\Mf\times\Nf$ there is a $\Co^\alpha$-coordinate chart $F:U\times V\subseteq\Mf\times\Nf\to\R^r_t\times\C^m_z\times\R^q_s$ such that $\Se|_{U\times V}$ is spanned by $\Co^{\alpha-}$-vector fields $F^*\Coorvec{t^1},\dots,F^*\Coorvec{t^r},F^*\Coorvec{z^1},\dots,F^*\Coorvec{z^m}$.
\end{thm}
See Theorem \ref{Thm::Key} for a more general version, where we consider the case of mixed regularity $\Se\in\Co^{\alpha,\beta}$ with $0<\beta\le\alpha+1$ (see Definitions \ref{Defn::Hold::MixHold}, \ref{Defn::Hold::MoreMixHold}, \ref{Defn::ODE::MixHoldMaps},  \ref{Defn::ODE::CpxPaSubbd} and \ref{Defn::ODE::InvMix}). Theorem \ref{Thm::Key} implies Proposition \ref{Prop::Intro::SharpE} which can be viewed as the ``parameter free'' case. See also Remarks \ref{Rmk::Key::KeyFirst}, \ref{Rmk::Key::KeySpecial} and \ref{Rmk::Key::KeyImpliesGong}.

The idea is based on Malgrange's factorization method for the sharp Newlander-Nirenberg theorem, whose sketch can be found in, for example \cite[Page 47]{Involutive} and \cite[Section 3]{RoughNN}. Malgrange's idea is to find an intermediate coordinate change $H$ such that the generators for $H_*\Se$ satisfy an elliptic PDE system. Using the PDE we can show that $H_*\Se$ is a $\Co^\alpha$-family of real-analytic elliptic structures. Finally we construct another coordinate change $G$ using holomorphic Frobenius theorem, now with parameter, so that $G_*H_*\Se$ is the span of $\Coorvec{t^1},\dots,\Coorvec{t^r},\Coorvec{z^1},\dots,\Coorvec{z^m}$. Thus $F=G\circ H$ is the desired coordinate chart.

See Section \ref{Section::KeyProofOver} for a detailed sketch to the proof of Theorem \ref{Thm::Intro::KeyOV}.

Note that in Theorem \ref{Thm::Intro::KeyOV} we work on the elliptic structures rather than the complex structures. In this way we do not have to pick generators for $\Se$ that satisfy the ``Hypothesis V'' in \cite{RoughComplex}.

\medskip
For the proof of Theorem \ref{Thm::ThmCoor2}, we follow Nirenberg's original proof by applying real Frobenius theorem twice on the essentially real structures $\Se\cap\bar\Se$ and $\Se+\bar\Se$, see \cite[the proof of Theorem I.12.1]{Involutive}. This ends up with a family of complex structures. We complete the proof by applying Theorem \ref{Thm::Intro::KeyOV} with the special case $r=0$.

To get the optimal H\"older-Zygmund regularity during this reduction, we combine the two real Frobenius theorems into one step. We call it the \textbf{bi-layer Frobenius theorem}. 
\begin{thm}\label{Thm::Intro::BLFro}
Let $\Mf$ be a smooth manifold with dimension $(r+m+q)$. Let $\alpha,\beta>1$ and let $\V\le \E\le T\Mf$ be two involutive real subbundles such that $\V\in\Co^\alpha$, $\E\in\Co^\beta$, $\rank\V=r$ and $\rank \E=r+m$.

Then for any $p\in\Mf$ there is a $\Co^{\min(\alpha,\beta)}$-coordinate system $F=(F',F'',F'''):U\subseteq\Mf\to\R^r_t\times\R^m_x\times\R^q_s$ near $p$, such that $\V|_U$ is spanned by $F^*\Coorvec{t^1},\dots,F^*\Coorvec{t^r}$ and $\E|_U$ is spanned by $F^*\Coorvec{t^1},\dots,F^*\Coorvec{t^r},F^*\Coorvec{x^1},\dots,F^*\Coorvec{x^m}$.

Moreover for their regularities we have $F'\in\Co^\infty$, $F''\in\Co^\alpha$, $F'''\in\Co^\beta$, $F^*\Coorvec{t^1},\dots,F^*\Coorvec{t^r}\in\Co^\alpha$ and $F^*\Coorvec{x^1},\dots,F^*\Coorvec{x^m}\in\Co^{\min(\alpha-1,\beta)}$. 
\end{thm}
See Theorem \ref{Thm::ODE::BLFro} for a more general version. In Theorem \ref{Thm::ThmCoor2} we take $\alpha+1\ge\beta$, $\V=\Se\cap T\Mf$ and $\E=(\Se+\bar\Se)\cap T\Mf$.

\medskip
Unlike the sharp Newlander-Nirenberg theorem or Street's result where $F^*\Coorvec{z^1},\dots,F^*\Coorvec{z^m}$ are $\Co^\alpha$-vector fields, in our result we show that $F^*\Coorvec{z^j}\in\Co^{\alpha-}$ is optimal. 

One reason for the regularity loss comes from solving elliptic pdes. Say we have a $\Co^\alpha$-subbundle for some $\alpha>0$. In the first step of Malgrange's proof of the sharp Newlander-Nirenberg theorem we need to solve an elliptic pde system for $f=f(x)$ that has the following form:

\begin{equation}\label{Eqn::Intro::ElliPDE}
    \Delta f=\phi(f,\nabla f,\nabla^2f,a,\nabla a).
\end{equation}
Here $a=a(x)\in\Co^\alpha$ is the vector-valued coefficient that comes from the subbundle, and $\phi$ is a rational function which is defined and small when $f$ and $a$ are both small.

Since Laplace's equation gains two derivatives in H\"older-Zygmund spaces, we can find a $f\in\Co^{\alpha+1}$ solving \eqref{Eqn::Intro::ElliPDE}. This is why we can find a $\Co^{\alpha+1}$-coordinate system for a $\Co^\alpha$-integrable almost complex structure. In particular the coordinate vector fields are $\Co^\alpha$ because $\nabla f\in\Co^\alpha$.

When we do estimate on the complex Frobenius structure, we need to solve a pde system that has similar expression to \eqref{Eqn::Intro::ElliPDE} but with a parameter: we need $f=f(x,s)$ for the system $\Delta_xf=\phi(f,\nabla_xf,a,\nabla_xa)$ with $\Co^\alpha$-coefficient $a=a(x,s)$. 

We need to introduce some suitable function spaces that capture the property that the pde system ``gains 2 derivatives on $x$-variable''. Note that by comparing to $a(x,s)$, the function $\phi(f,\nabla_xf,a,\nabla_xa)$ lose one derivative in $x$, and by taking inverse Laplacian, $f$ should gain 1 derivative in $x$. Thus we can ask what is the good function space $\Xs=\Xs_{x,s}$ to work on, such that if $a\in\Xs_{x,s}$ then $\nabla_xf\in\Xs_{x,s}$.

Unfortunately the space $\Co^\alpha_{x,s}$ does not satisfy our need. When $a\in\Co^\alpha_{x,s}$ we may not find a solution $f$ such that $\nabla f\in\Co^\alpha_{x,s}$. This is roughly because the $x$-variable Riesz transform (say $x\in\R^n$ and $s\in V\subseteq\R^m$)
$$R_jf(x,s):=\frac{\Gamma(\frac{n+1}2)}{\pi^\frac{n+1}2}\lim\limits_{\eps\to 0}\int_{\R^n\backslash B^n(0,\eps)}\frac{y_j-x_j}{|y-x|^{n+1}}f(y,s)dy, \quad j=1,\dots,n,$$
does not map $\Co^\alpha(\R^n\times V)$ into itself. That is why $F^*\Coorvec{z^j}\in\Co^{\alpha-}$ is the best one we can get.

To see that why $R_j:\Co^\alpha_{x,s}\not\to \Co^\alpha_{x,s}$ holds, we consider the bi-parameter decomposition  $\Co^\alpha_{x,s}=\Co^\alpha_xL^\infty_s\cap L^\infty_x\Co^\alpha_s$ (see Lemma \ref{Lem::Hold::CharMixHold}): we have $\|f\|_{\Co^\alpha_{x,s}}\approx\sup_s\|f(\cdot,s)\|_{\Co^\alpha_x}+\sup_x\|f(x,\cdot)\|_{\Co^\alpha_s}$. We see that $R_j:\Co^\alpha_xL^\infty_s\to \Co^\alpha_xL^\infty_s$ is still bounded because Riesz transforms are bounded in $\Co^\alpha$-spaces. However $R_j:L^\infty_x\Co^\alpha_s\not\to L^\infty_x\Co^\alpha_s$ is not bounded because Riesz transforms are unbounded on $L^\infty$-spaces.

To construct the example of complex Frobenius structure such that $F^*\Coorvec z\notin\Co^\alpha$, we follow a similar but more sophisticated construction to \cite{Liding}. See Propositions \ref{Prop::Final::SharpddzRed} and \ref{Prop::Final::Expa} in Section \ref{Section::Sharpddz}.

\subsection{Notations and organizations}\label{Section::Convention}

We use \textbf{distributions} for generalized functions, rather than tangential subbundles on manifolds.

Given a bijection $f:U\to V$, we use $f^\Inv:V\to U$ as the inverse function of $f$. We do not use the notion $f^{-1}$ in order to reduce confusion when $(\cdot)^{-1}$ is used for inverting a matrix.

We use $A\Subset B$ to denote that $A$ is a relatively compact subset of $B$.

By $A \lesssim B$ we mean that $A \leq CB$ where $C$ is a constant independent of $A,B$. We use $A \approx B$ for ``$A \lesssim B$ and $B \lesssim A$''. And we use $A\lesssim_xB$ to emphasize the constant dependence on quantity $x$.

We use $\B^n=B^n(0,1)$ as the unit ball in $\R^n$ and $r\B^n=B^n(0,r)$ for $r>0$. We use a complex cone $\Hb^n:=\{x+iy\in\C^n:x\in\B^n,4|y|<1-|x|\}\subseteq\C^n$.

For functions $f:U\to\R$, $g:V\to\R$, we use the tensor notation $f\otimes g:U\times V\to\R$ as $(f\otimes g)(x,y):=f(x)g(y)$.

For the index of H\"older regularity, we use the  convention $\infty=\infty+1=\infty-1=\infty+=\infty-=\frac\infty2$.

\medskip
We denote by $\Co^\alpha$ the class of H\"older-Zygmund functions:
\begin{defn}\label{Defn::Intro::DefofHold}
    Let $\alpha\in(0,\infty]$, let $\Omega\subset\R^n$ be a convex open subset and let $X=\R^m$ or $\C^m$. The space $\Co^\alpha(\Omega;X)$ of bounded ($X$ vector-valued) H\"older-Zygmund functions on $\Omega$, is given recursively as follows:
    \begin{enumerate}[parsep=-0.3ex,label=(\roman*)]
    \item For $0<\alpha<1$: $\Co^\alpha(\Omega;X)$ consists of those continuous functions $f:\Omega\to X$ such that:
    \begin{equation*}
         \textstyle\|f\|_{\Co^\alpha(\Omega;X)}:=\sup_{x\in \Omega} |f(x)|_X + \sup_{\substack{x,y\in\Omega;x\neq y}} |x-y|^{-\alpha}|f(x)-f(y)|_X<\infty.
    \end{equation*}
    \item For $\alpha=1$: $\Co^1(\Omega;X)$ consists of those continuous functions $f:\Omega\to X$ such that:
    \begin{equation*}
         \textstyle\|f\|_{\Co^1(\Omega;X)}:=\sup_{x\in \Omega} |f(x)|_X + \sup_{\substack{x,y\in\Omega;x\neq y}} |x-y|^{-1}|\frac{f(x)+f(y)}2-f(\frac{x+y}2)|_X<\infty.
    \end{equation*}
    \item\label{Item::Hold::DefofHold::>1} For $\alpha>1$: $\Co^\alpha(\Omega;X)$ consists of $f\in\Co^{\alpha-1}(\Omega;X)$ such that the following norm is finite: $$\textstyle\|f\|_{\Co^\alpha(\Omega;X)}:=\|f\|_{\Co^{\alpha-1}(\Omega;X)}+\sum_{j=1}^n\|\partial_jf\|_{\Co^{\alpha-1}(\Omega;X)}.$$
    \item We use $\Co^\infty(\Omega;X):=\bigcap_{\beta>0}\Co^\beta(\Omega;X)$ and $\Co^{\alpha-}(\Omega;X):=\bigcap_{\beta<\alpha}\Co^\beta(\Omega;X)$, both endowed with the standard limit Fr\'echet topologies. That is, we say $f_k\to f_0$ in $\Co^{\alpha-}$ if $f_k\to f_0$ in $\Co^\beta$ for all $\beta<\alpha$.
    \item We use $\Co^{\alpha+}(\Omega):=\bigcup_{\beta>\alpha}\Co^\beta(\Omega)$, whose topology is used in the paper.
    \end{enumerate}

    We use $\Co^\beta(\Omega)=\Co^\beta(\Omega;\R)$ for $\beta\in\{\alpha,\alpha-,\alpha+:\alpha>0\}\cup\{\infty\}$.
    
     We use $\|f\|_{\Co^\alpha(\Omega)}$ either when $f:\Omega\to\R$ is a function or when the codomain of $f$ is clear. And we use $\|f\|_{\Co^\alpha}$ when both the domain and the codomain are clear.
\end{defn}
We define $\Co^\alpha(\Omega;X)$ for $\alpha\le0$ in Definition \ref{Defn::Hold::NegHold} using the Besov space $\Bs_{\infty\infty}^\alpha$.

Let $U\subseteq\R^m_x$ and $V\subseteq\R^n_s$ be two convex open subsets. For $\alpha,\beta\in(0,\infty]$, we denote by $\Co^{\alpha,\beta}(U,V)$ the space of functions on $U\times V$ that is bounded $\Co^\alpha$ in $x$ and bounded $\Co^\beta$ in $s$ separately, see Definition \ref{Defn::Hold::MixHold}. For $\alpha,\beta\in(-\infty,\infty]$ we denote by $\Co^\alpha\Co^\beta(U,V)$ the space of functions on $U\times V$ that is bounded $\Co^\alpha$ in $x$ and bounded $\Co^\beta$ in $s$ simultaneously, see Definition \ref{Defn::Hold::BiHold}. In most of the cases we work on the spaces like $\Co^\alpha L^\infty\cap\Co^\beta\Co^\gamma(U,V)$.

Given a coordinate chart $x=(x^1,\dots,x^n)$, we write  $dx:=[dx^1,\dots,dx^n]$ as a row vector and $\Coorvec{x}:=\begin{bmatrix}\partial_{x^1}\\\vdots\\\partial_{x^n}\end{bmatrix}$ as a column vector. We may use $X$ sometimes as a single vector field, and sometimes as a (column) collection of vector fields $X=[X_1,\dots,X_r]^\top$.  We use these conventions when there is no ambiguity.

In a mixed real and complex domain $\Omega\subseteq\R^r\times\C^m$ with standard (real and complex) coordinate system $(t,z)=(t^1,\dots,t^r,z^1,\dots,z^m)$, 
for a function $f$ on $\Omega$, we use $\nabla f,\nabla_zf,\partial_zf$ as functions taking values in column vectors, which have rows $r+2m$, $2m$ and $m$ respectively as follows,
\begin{equation}\label{Eqn::Intro::ColumnNote}
    \begin{gathered}\textstyle
    \nabla f=\nabla_{t,z}f=\partial_{t,z,\bar z}f=\big[\frac{\partial f}{\partial t^1},\dots,\frac{\partial f}{\partial t^r},\frac{\partial f}{\partial z^1},\dots,\frac{\partial f}{\partial z^m},\frac{\partial f}{\partial \bar z^1},\dots,\frac{\partial f}{\partial \bar z^m}\big]^\top,
\\\textstyle
\nabla_zf=\partial_{z,\bar z}f=\big[\frac{\partial f}{\partial z^1},\dots,\frac{\partial f}{\partial z^m},\frac{\partial f}{\partial \bar z^1},\dots,\frac{\partial f}{\partial \bar z^m}\big]^\top,
\qquad \partial_zf=\frac{\partial f}{\partial z}=\big[\frac{\partial f}{\partial z^1},\dots,\frac{\partial f}{\partial z^m}\big]^\top.
\end{gathered}
\end{equation}

For a map $F:\Mf\to\R^r\times\C^m$, we usually use the notation $F=(F',F'')$ where $F':\Mf\to\R^r$ and $F'':\Mf\to\C^m$ are the corresponding components of $F$.

\medskip
The paper is organized as follows:

In Section \ref{Section::HoldSec} we give discussions of H\"older-Zygmund spaces in the bi-parameter setting.
We prove some results for multiplications and compositions of bi-parameter H\"older-Zygmund spaces in Sections \ref{Section::HoldMult} and \ref{Section::HoldComp} respectively. These results will be used in Section \ref{Section::SecPDE}.

In Section \ref{Section::SecODE} we review the involutive structures and ODE flows. We prove the bi-layer Frobenius theorem Theorem \ref{Thm::Intro::BLFro} (see Theorem \ref{Thm::ODE::BLFro}) and a version of holomorphic Frobenius theorem (Proposition \ref{Prop::ODE::ParaHolFro}) in Section \ref{Section::PfFro}. Proposition \ref{Prop::ODE::ParaHolFro} will be used in the last step of Malgrange's factorization in Section \ref{Section::KeyProofOver}.

In Section \ref{Section::SecPDE} we prove two results for the nonlinear elliptic PDE system with parameters: one for existence (Section \ref{Section::ExistPDE}) and one for interior real-analyticity (Section \ref{Section::AnalPDE}).

In Section \ref{Section::SecKey} we state and prove the general version of the key result Theorem \ref{Thm::Intro::KeyOV} (Theorem \ref{Thm::Key}). We first sketch the idea using Malgrange's method in Section \ref{Section::KeyProofOver} and then prove them in Section \ref{Section::KeyLems}. To get a desired coordinate chart, we first pick some good generators for the given family of elliptic structures. Then we use the existence result in Section \ref{Section::ExistPDE} to construct an coordinate change such that in the new coordinates the generators satisfy some extra conditions. By the real-analyticity result in Section \ref{Section::AnalPDE} such generators admit a holomorphic extension to a complex domain. Finally, we use the holomorphic Frobenius theorem to complete the proof of Theorem \ref{Thm::Intro::KeyOV}. We finish the proof of Theorem \ref{Thm::Key} in Section \ref{Section::KeyPf}.

In Section \ref{Section::PfThm} we complete the proof of Theorems \ref{Thm::ThmCoor1} and \ref{Thm::ThmCoor2} (see Theorems \ref{Thm::TrueThm1} and \ref{Thm::TrueThm2}). 

In Section \ref{Section::SharpGen} we explain which parts of the results in the main theorems are sharp. For the sharpness of $F^*\Coorvec z\in\Co^{\alpha-}$ we give a complete construction of the complex Frobenius structure (Propositions \ref{Prop::Final::SharpddzRed} and \ref{Prop::Final::Expa}) and give the proof of Theorem \ref{Thm::Intro::Sharpddz} in Section \ref{Section::Sharpddz} (see Theorems \ref{Thm::Final::ProofExampleddz} and \ref{Thm::Final::SharpSum}).

In Section \ref{Section::SecHolLap} we include a result of the holomorphic extension for a right inverse Laplacian from $\B^n\subset\R^n$ to $\Hb^n\subset\C^n$ (Proposition \ref{Prop::HolLap}) based on \cite{Analyticity}.

\subsection*{Acknowledgement}
I would like to express my great appreciation for my advisor Prof. Brian Street for his help, not only  giving the idea of the project but also for his essential guidance throughout the project.
I would also like to thanks to Prof. Xianghong Gong for the invaluable advise for the paper.

\section{The Bi-parameter H\"older-Zygmund Structures}\label{Section::HoldSec}

In this section we set up some bi-parameter H\"older-Zygmund spaces that are used to discuss the H\"older regularity of some geometric objects. The results in this section are mostly for technical purposes.
\subsection{H\"older-Zygmund spaces of single and bi-parameters}\label{BasicHold}
Recall the classical (single-parameter) H\"older-Zygmund spaces in Definition \ref{Defn::Intro::DefofHold}. 
In applications we require the Littlewood-Paley characterizations for H\"older-Zygmund spaces.
\begin{defn}\label{Defn::Hold::DyadicResolution}
A (Fourier) dyadic resolution of unity for  $\R^n$ is a sequence of real-valued Schwartz functions $\phi=(\phi_j)_{j=0}^\infty$ on $\R^n$ such that
\begin{enumerate}[nolistsep,label=(\alph*)]
    \item The Fourier transform $\hat\phi_0(\xi)=\int_{\R^n}\phi(x)e^{-2\pi ix\xi}dx$ satisfies $\supp\hat\phi_0\subset\{|\xi|<2\}$ and $\hat\phi_0\big|_{\{|\xi|<1\}}\equiv1$.
    \item For $j\ge1$, $\phi_j(x):=2^{nj}\phi_0(2^jx)-2^{n(j-1)}\phi_0(2^{j-1}x)$ for $x\in\R^n$.
\end{enumerate}
\end{defn}
Immediately we see that
\begin{equation}\label{Eqn::Hold::RmkDyaSupp}
    \supp\phi_0\subset\{|\xi|<2\},\quad\supp\phi_j\subset\{2^{j-1}<|\xi|<2^{j+1}\},\quad j\ge1.
\end{equation}

\begin{lem}\label{Lem::Hold::HoldChar}
    Let $\alpha>0$ and let $\phi=(\phi_j)_{j=0}^\infty$ be a dyadic resolution.
    \begin{enumerate}[parsep=-0.3ex,label=(\roman*)]
        \item\label{Item::Hold::HoldChar::LPHoldChar} $f\in\Co^\alpha(\R^n)$ if and only if $\phi_j\ast f\in C^0(\R^n)$ for every $j\ge0$ and $\sup_{j\ge0}2^{j\alpha}\|\phi_j\ast f\|_{L^\infty(\R^n)}<\infty$. Moreover there is a $C=C(\phi,\alpha)>0$ that does not depend on $f$ such that
        \begin{equation*}\label{Eqn::Hold::HoldChar::LPHoldChar}
            \textstyle C^{-1}\|f\|_{\Co^\alpha(\R^n)}\le\sup_{j\ge0}2^{j\alpha}\|\phi_j\ast f\|_{L^\infty(\R^n)}\le C\|f\|_{\Co^\alpha(\R^n)},\quad\forall f\in\Co^\alpha(\R^n).
        \end{equation*}
        \item\label{Item::Hold::HoldChar::DomainChar} Let $\Omega\subset\R^n$ be a bounded convex open set with smooth boundary. Then $\Co^\alpha(\Omega)$ has an equivalent norm
        \begin{equation*}
            f\mapsto\inf\{\|\tilde f\|_{\Co^\alpha(\R^n)}:\tilde f\in\Co^\alpha(\R^n),\tilde f\big|_\Omega=f\}.
        \end{equation*}
        \item\label{Item::Hold::HoldChar::02} Let $\Omega\subseteq\R^n$ be either the whole space or a bounded convex open set with smooth boundary. Suppose $0<\alpha<2$, then $\Co^\alpha(\Omega)$ has an equivalent norm
        \begin{equation*}
            \textstyle f\mapsto\sup_{x\in\Omega}|f(x)|+\sup_{x,y\in\Omega;x\neq y}\big|\frac{f(x)+f(y)}2-f\big(\frac{x+y}2\big)\big|\cdot|x-y|^{-\alpha}.
        \end{equation*}
    \end{enumerate}
\end{lem}
See \cite[Chapters 2.2.2, 2.5.7 and 3.4.2]{Triebel1} for the proof.

\begin{remark}\label{Rmk::Hold::RmkforLPHoldChar}
    Using Lemma \ref{Lem::Hold::HoldChar} and Definition \ref{Defn::Intro::DefofHold} we know that for positive noninteger value, the H\"older-Zygmund spaces coincide with the classical H\"older spaces, that is $\Co^{m+s}(\Omega)=C^{m,s}(\Omega)$ for $m\ge0$, $0<s<1$. However for integer cases $C^m(\Omega)\subsetneq C^{m-1,1}(\Omega)\subsetneq \Co^m(\Omega)$.
\end{remark}

Following the notations in \cite[Definition 2.3.1/2(i)]{Triebel1} for $\alpha>0$ the H\"older-Zygmund spaces $\Co^\alpha(\Omega)$ are indeed the Besov spaces $\Bs_{\infty\infty}^\alpha(\Omega)$. Thus it is natural to define $\Co^\alpha$-spaces for $\alpha\le0$.

\begin{defn}\label{Defn::Hold::NegHold}Let $\alpha\le0$, let $X=\R^m$ or $\C^m$, and let $(\phi_j)_{j=0}^\infty$ be a dyadic resolution of unity in $\R^n$. We define $\Co^\alpha:=\Bs_{\infty\infty}^\alpha$. More precisely, $\Co^\alpha(\R^n;X)$ is the space of all temper distributions $f\in\Sc'(\R^n;X)$ such that
\begin{equation*}
\textstyle    \|f\|_{\Co^\alpha(\R^n;X)}:=\sup_{j\ge0}2^{j\alpha}\|\phi_j\ast f\|_{L^\infty(\R^n;X)}<\infty.
\end{equation*}
Let $\Omega\subset\R^n$ be a bounded convex domain, we define $\Co^\alpha(\Omega):=\{\tilde f|_\Omega:\tilde f\in\Co^\alpha(\R^n)\}$ with norm 
$$\|f\|_{\Co^\alpha(\Omega;X)}:=\inf\{\|\tilde f\|_{\Co^\alpha(\R^n;X)}:\tilde f|_\Omega=f\}.$$
\end{defn}
The $\Co^\alpha$-norms depend on  $\phi=(\phi_j)_{j=0}^\infty$, while the spaces themselves are not. See \cite[Chapter 2.3.2]{Triebel1}.

We use Littlewood-Paley decomposition along different variables to define bi-parameter Zygmund structures. 

\begin{note}\label{Note::Hold::ConvVar}
For functions $f(x,s)$ and $\rho(x)$, we use $\ast_x$ as the convolution acting on $x$-variable: $$\rho\ast_xf(x,s)=(\rho\ast f(\cdot,s))(x).$$ 

\end{note}
Note that if $f\in\Sc'(\R^n_x\times \R^q_s)$ and $\rho\in\Sc(\R^n)$, then $\rho\ast_xf$ is a well-defined tempered distribution in $\R^n_x\times\R^q_s$, since $\rho\ast_xf=(\rho\otimes\delta_{0^q})\ast f$. Here $\delta_{0^q}$ is the Direc delta measure of the origin in $\R^q$.
\begin{defn}\label{Defn::Hold::BiHold}Let $\alpha,\beta\in\R$. Let $x=(x^1,\dots,x^n)$ and $s=(s^1,\dots,s^q)$ be the coordinate system of $\R^n$ and $\R^q$ respectively. Let $\phi=(\phi_j)_{j=0}^\infty\subset\Sc(\R^n)$ and $\psi=(\psi_k)_{k=0}^\infty\subset\Sc(\R^q)$ be dyadic resolutions for $\R^n$ and $\R^q$. 
\begin{enumerate}[label=(\roman*),parsep=-0.3ex]
    \item\label{Item::Hold::BiHold::CalphaLLinfty} We define $\Co^\alpha_x L^\infty_s(\R^n,\R^q)=L^\infty_s\Co^\alpha_x(\R^q,\R^n)$ to be the set of all $f\in\Sc'(\R^n\times\R^q)$ such that $\phi_j\ast_xf\in L^\infty(\R^n\times\R^q)$ for every $j\ge0$ and 
    $$\textstyle\|f\|_{\Co^\alpha_x L^\infty_s(\R^n,\R^q)}=\|f\|_{\Co^\alpha_x L^\infty_s(\R^n,\R^q)_{\phi}}:=\sup_{j\ge0}2^{j\alpha}\|\phi_j\ast_xf\|_{L^\infty(\R^n\times\R^q)}<\infty.$$
    \item\label{Item::Hold::BiHold::CalphaLCbeta} We define $\Co^\alpha_x\Co^\beta_s(\R^m,\R^q)$ to be the set of all $f\in\Sc'(\R^n\times\R^q)$ such that $$\textstyle\|f\|_{\Co^\alpha_x\Co^\beta_s(\R^n,\R^q)}=\|f\|_{\Co^\alpha_x\Co^\beta_s(\R^n,\R^q)_{\phi,\psi}}:=\sup_{j,k\ge0}2^{j\alpha+k\beta}\|\phi_j\ast_x(\psi_k\ast_sf)\|_{L^\infty(\R^n\times\R^q)}<\infty.$$
    \item\label{Item::Hold::BiHold::Domain}  For $\Xs\in\{\Co^\alpha_xL^\infty_s,\Co^\alpha_x\Co^\beta_s\}$, let $U\subset\R^m$ and $V\subset\R^n$ be two open sets, define $\Xs(U_x,V_s)$ by
    $$\textstyle\Xs( U,V):=\{\tilde f|_{U\times V}:\tilde f\in \Xs(\R^n,\R^q)\},\quad\|f\|_{\Xs( U,V)}:=\inf_{\tilde f|_{U\times V}=f}\|\tilde f\|_{\Xs(\R^n,\R^q)}.$$
\end{enumerate}
    Let $\Xs,\Y$ be the function classes in \ref{Item::Hold::BiHold::CalphaLLinfty} or \ref{Item::Hold::BiHold::CalphaLCbeta}, we have convention $$\Xs\cap\Y( U,V):=\Xs( U, V)\cap \Y( U,V), \quad \|f\|_{\Xs\cap\Y(U,V)}:=\|f\|_{\Xs}+\|f\|_{\Y}.$$
    
    For vector valued functions, we use $\Xs\cap \Y(U,V;\R^m)$ as the spaces of all vector valued functions (or distributions) $f=(f^1,\dots,f^m):U\times V\to\R^m$ such that $f^j\in\Xs\cap\Y(U,V)$ for all $j=1,\dots,m$. And we use 
    $$\textstyle\|f\|_{\Xs\cap\Y(U,V;\R^m)}=\sup_{1\le j\le m}\|f^j\|_{\Xs\cap\Y(U,V)}.$$
    
    We may omit the variables, the domains and the codomains in the norms when they are clear.
    \end{defn}
\begin{remark}\label{Rmk::Hold::RmkforBiHold}
\begin{enumerate}[parsep=-0.3ex,label=(\roman*)]
    \item By Lemma \ref{Lem::Hold::HoldChar} \ref{Item::Hold::HoldChar::LPHoldChar} and \cite[Chapter 2.3.2]{Triebel1}, different choices of $(\phi_j)$ and $(\psi_k)$ result to the equivalent norms. We would use the norm notation $\|f\|_{\Xs(U,V)}$ through fixing a choice of $(\phi,\psi)$ implicitly. 
    \item\label{Item::Hold::RmkforBiHold::CxCs=CsCx} We have $\Co^\alpha_x\Co^\beta_s(U_x,V_s)=\Co^\beta_s\Co^\alpha_x(V_s,U_x)$. This is due to the fact that
    \begin{equation}\label{Eqn::Hold::RmkforBiHold::ConvComm}
        \phi_j\ast_x(\psi_k\ast_s f)=\psi_k\ast_s(\phi_j\ast_x f)=(\phi_j\otimes \psi_k)\ast f,\quad\text{where }(\phi_j\otimes\psi_k)(x,s)=\phi_j(x)\psi_k(s).
    \end{equation}
    \item When $\alpha>0$ and when $U,V$ are smooth domains, we have $\Co^\alpha\Co^\beta(U,V)=\Co^\alpha(U;\Co^\beta(V))$ in the sense that we take $X=\Co^\beta(V)$ in Definition \ref{Defn::Intro::DefofHold}. We leave the proof to readers.
    \item\label{Item::Hold::RmkforBiHold::notVectMeas} However $L^\infty\Co^\alpha(U,V)$ does not coincide with the vector-valued space $L^\infty(V;\Co^\alpha( U))$. In classical definitions $L^\infty(U;X)$ is the set of all $X$-valued strong measurable function (quotiented by almost everywhere) which is essentially bounded. Such functions in $L^\infty(U;\Co^\alpha(V))$ must be \textit{almost everywhere separable valued}. But $\Co^\alpha(U)$ is \textit{not} a separable space. So $L^\infty(U;\Co^\alpha(V))\subsetneq L^\infty\Co^\alpha(U,V)$ and the inclusion is strict. See for example \cite[Remark 1.2.16]{AnalBana}.
    \item\label{Item::Hold::RmkforBiHold::Extinthesametime}
    In the definition $f\in\Co^\alpha_xL^\infty_s\cap \Co^\beta_x\Co^\gamma_s( U,V)$ means there are extensions $\tilde f_1\in \Co^\alpha_xL^\infty_s(\R^n,\R^q)$ and $\tilde f_2\in \Co^\beta_x\Co^\gamma_s(\R^n,\R^q)$ such that $\tilde f_1|_{U\times V}=\tilde f_2|_{U\times V}=f$. When $U$ and $V$ are smooth domains, by Lemma \ref{Lem::Hold::CommuteExt} below we can choose a suitable extension $\tilde f_1=\tilde f_2$ by taking $\tilde f=E_xE_sf=E_sE_xf$ that extends $\Co^\alpha_xL^\infty_s$ and $L^\infty_x\Co^\beta_s$ simultaneously.
    \item\label{Item::Hold::RmkforBiHold::Interpo} Let $\alpha_0,\alpha_1,\beta_0,\beta_1\in\R$. Clearly $\min(2^{j\alpha_0+k\beta_0},2^{j\alpha_1+k\beta_1})\le2^{j((1-\theta)\alpha_0+\theta\alpha_1)+k((1-\theta)\beta_0+\theta\beta_1)}$ for $\theta\in[0,1]$. Thus we have inclusion: for $U,V$ either the total space or bounded smooth convex domains, 
    \begin{equation}\label{Eqn::Hold::RmkforBiHold::Interpo}
        \Co^{\alpha_0}\Co^{\beta_0}\cap\Co^{\alpha_1}\Co^{\beta_1}(U,V)\subseteq \Co^{\alpha_\theta}\Co^{\beta_\theta}(U,V),\quad\alpha_\theta:=(1-\theta)\alpha_0+\theta\alpha_1,\ \beta_\theta:=(1-\theta)\beta_0+\theta\beta_1.
    \end{equation}
\end{enumerate}
\end{remark}

\begin{lem}[Common extensions]\label{Lem::Hold::CommuteExt}
Let $R>0$, and let $U\subset\R^n$, $V\subset\R^q$ be two bounded domain with smooth boundaries.
Then there are extension operators $E_x:L^\infty(U)\to L^\infty(\R^n)$ and $E_s:L^\infty(V)\to L^\infty(\R^q)$ such that \begin{enumerate}[nolistsep,label=(\alph*)]
    \item\label{Item::Hold::CommuteExt::BddofExt} $E_x:\Co^\alpha(U)\to\Co^\alpha(\R^n)$ and $E_s:\Co^\alpha(V)\to\Co^\alpha(\R^q)$ are both bounded linear for all $-R<\alpha<R$.
    \item\label{Item::Hold::CommuteExt::Comm} $E_xE_s,E_sE_x:\Co^{(-R)+}\Co^{(-R)+}(U,V)\to \Co^{(-R)+}\Co^{(-R)+}(\R^n,\R^q)$ are well-defined and equal.
\end{enumerate}
Moreover $\tilde E:=E_xE_s$ is an extension operator for $U\times V$ that has boundedness 
\begin{equation}\label{Eqn::Hold::CommuteExt::Bdd}
    \tilde E: \Xs\Y(U,V)\to  \Xs\Y(\R^n,\R^q),\quad \Xs,\Y\in\{L^\infty,\Co^\alpha:-R<\alpha<R\}.
\end{equation}
\end{lem}
\begin{proof}
The construction $E_x$ and $E_s$ for \ref{Item::Hold::CommuteExt::BddofExt} can be found in \cite[Theorems 2.9.2, 2.9.4 (i) and 3.3.4 (i)]{Triebel1}, which are both $L^\infty$ and $\Co^\alpha$-bounded for $-R<\alpha<R$. 

For continuous functions $u\in C^0(U)$, $v\in C^0(V)$ we have equality $E_xE_s(u\otimes v)=E_xu\otimes E_sv=E_sE_x(u\otimes v)$. Thus by taking finite linear combinations of simple tensors and taking $C^0$-limit, we see that $E_xE_s=E_sE_x:C^0(U\times V)\to L^\infty (\R^n\times\R^q)$ is an extension operator for $U\times V$.

Let $\eps>0$ and let $(\phi_j),(\psi_k)$ be dyadic resolutions for $\R^n_x$, $\R^q_s$ respectively. Clearly for $\tilde f\in\Co^{-R+\eps}\Co^{-R+\eps}(\R^n,\R^q)$, the double sum $\sum_{j,k=0}^\infty(\phi_j\otimes\psi_k)\ast \tilde f$ converges to $\tilde f$ in $\Co^{-R+\frac\eps2}\Co^{-R+\frac\eps2}(\R^n,\R^q)$. Since $((\phi_j\otimes\psi_k)\ast\tilde f)|_{U\times V}\in C^0(U\times V)$ for each $j,k$, by Definition \ref{Defn::Hold::BiHold} \ref{Item::Hold::BiHold::Domain} we see that $E_xE_sf=E_sE_xf\in \Co^{-R+\frac\eps2}\Co^{-R+\frac\eps2}(\R^n,\R^q)$ whenever $f\in \Co^{-R+\eps}\Co^{-R+\eps}(\R^n,\R^q)$. Therefore $E_xE_s,E_sE_x\Co^{(-R)+}\Co^{(-R)+}(U,V)\to \Co^{(-R)+}\Co^{(-R)+}(\R^n,\R^q)$ are both defined  and equal. This proves \ref{Item::Hold::CommuteExt::Comm}.

Note that the same argument also shows that
\begin{equation}\label{Eqn::Hold::CommuteExt::LPandExt}
    \phi_j\ast_x(E_sg)(x,s)=E_s(\phi_j\ast_xg)(x,s),\quad\text{for }(x,s)\in \R^n\times\R^q,\quad j\ge0,\quad g\in\Co^{(-R)+}\Co^{(-R)+}(U,\R^q).
\end{equation}

To prove \eqref{Eqn::Hold::CommuteExt::Bdd}, we only prove the case $\Co^\alpha L^\infty$ for $-R<\alpha<R$. The argument for the rest of the cases are similar and we omit the details.

Using the boundedness $E_x:\Co^\alpha(U)\to\Co^\alpha(\R^n)$ and $E_s:L^\infty(V)\to L^\infty(\R^q)$ we have
\begin{align*}
    &\|\tilde E(\tilde f|_{U\times V})\|_{\Co^\alpha L^\infty(\R^n,\R^q)}=\sup_{j\ge0}2^{j\alpha}\|\phi_j\ast_x (E_sE_x(\tilde f|_{U\times V}))\|_{L^\infty(\R^n\times\R^q)}=\sup_{j\ge0}2^{j\alpha}\|E_s(\phi_j\ast_x E_x(\tilde f|_{U\times V}))\|_{L^\infty(\R^n\times\R^q)}
    \\
    \le&\|E_s\|_{L^\infty}\sup_{j\ge0}2^{j\alpha}\essup_{s\in V}\|\phi_j\ast_x E_x(\tilde f(\cdot,s)|_U)\|_{L^\infty(\R^n)}\le \|E_s\|_{L^\infty}\|E_x\|_{\Co^\alpha}\essup_{s\in \R^q}\|\tilde f(\cdot,s)\|_{\Co^\alpha(U)}\lesssim_{\tilde E}\|\tilde f\|_{\Co^\alpha L^\infty(\R^n,\R^q)}.
\end{align*}

For a function $f\in \Co^\alpha L^\infty(U,V)$, taking the infimum of all extensions $\tilde f\in \Co^\alpha L^\infty(\R^n,\R^q)$ over $f$, we get $\|\tilde Ef\|_{\Co^\alpha L^\infty(\R^n,\R^q)}\lesssim\|f\|_{\Co^\alpha L^\infty(U,V)}$ from the above inequality.
\end{proof}

More generally, for a linear operator $T=T_x$ acting on functions in $U\subseteq\R^n_x$, we have natural extension $\tilde T=T\otimes\id_s$ acting on functions in $U\times V\subseteq\R^n_x\times\R^q_s$.
\begin{lem}\label{Lem::Hold::TOtimesId}
Let $R_1<R_2$ and $r\in\R$. Let $U\subseteq\R^n$ and $V\subseteq\R^q$ be either the total spaces or bounded smooth domains. Suppose $T$ is a linear operator on functions of $U$ such that $T:\Co^\alpha(U)\to\Co^{\alpha+r}(U)$ for all $R_1<\alpha<R_2$. Then  $(\tilde T f)(x,s):=T(f(\cdot,s))(x)$ defines a linear operator that has boundedness $\tilde T:\Co^\alpha\Xs(U,V)\to\Co^\alpha\Xs(U,V)$ for all $R_1<\alpha<R_2$ and $\Xs\in\{L^\infty,\Co^\beta:\beta\in\R\}$.
\end{lem}
\begin{proof}
By passing to an extension $\tilde E=E_xE_s$ given in Lemma \ref{Lem::Hold::CommuteExt}, it suffices to consider the case $U=\R^n$ and $V=\R^q$. We only prove the case $\Xs=\Co^\beta$ since the case $\Xs=L^\infty$ is similar but simpler.

Clearly for $u\in \Co^{R_1+}(\R^n)$ and $v\in C^0(\R^q)$ we have $\tilde T(u\otimes v)=Tu\otimes v$ and $\psi_k\ast_s\tilde T(u\otimes v)=Tu\otimes(\psi_k\ast v)$ for all $k\ge0$. Taking finite linear combinations of simple tensors and then taking limit similar to the proof of Lemma \ref{Lem::Hold::CommuteExt} \ref{Item::Hold::CommuteExt::Comm}, we get $\tilde T:\Co^{R_1+\eps}\Co^{-R}(\R^n,\R^q)\to\Co^{R_1+r+\frac\eps2}\Co^{-R-\eps}(\R^n,\R^q)$ is defined and bounded for all $R>0$ and $0<\eps<R_2-R_1$, and we have similar to \eqref{Eqn::Hold::CommuteExt::LPandExt} that
\begin{equation*}
    \psi_k\ast_s(\tilde Tf)=\tilde T(\psi_k\ast_s f),\quad k\ge0,\quad f\in \Co^{R_1+}\Co^{-R}(\R^n,\R^q).
\end{equation*}

Therefore, for $R_1<\alpha<R_2$ and $\beta\in\R$,
\begin{align*}
    &\|\tilde Tf\|_{\Co^{\alpha+r}\Co^\beta(\R^n,\R^q)}=\sup_{j,k\ge0}2^{j(\alpha+r)+k\beta}\|(\phi_j\otimes\psi_k)\ast\tilde Tf\|_{L^\infty(\R^n\times \R^q)}=\sup_{j,k\ge0}2^{j(\alpha+r)+k\beta}\|\phi_j\ast_x\tilde T(\psi_k\ast_sf)\|_{L^\infty(\R^n\times \R^q)}
    \\
    \le&\|\tilde T\|_{\Co^\alpha\to\Co^{\alpha+r}}\sup_{j,k\ge0}2^{j(\alpha+r)+k\beta}\essup_{s\in\R^q}\|\phi_j\ast((\psi_k\ast_sf)(\cdot,s))\|_{L^\infty(\R^n)}=\|\tilde T\|_{\Co^\alpha\to\Co^{\alpha+r}}\|f\|_{\Co^\alpha\Co^\beta(\R^n,\R^q)}.
\end{align*}
This completes the proof.
\end{proof}






We define the mixed H\"older-Zygmund space $\Co^{\alpha,\beta}_{x,s}$ as the following
\begin{defn}\label{Defn::Hold::MixHold}
Let $\alpha,\beta>0$, and let $U\subseteq\R^n$, $V\subseteq\R^q$ be two convex open sets. We define $\Co^{\alpha,\beta}_{x,s}(U,V)$ to be the set of all continuous functions $f:U\times V\to\R$ satisfying $$\textstyle\|f\|_{\Co^{\alpha,\beta}_{x,s}(U,V)}:=\sup_{x\in U}\|f(x,\cdot)\|_{\Co^\beta(V)}+\sup_{s\in V}\|f(\cdot,s)\|_{\Co^\alpha(U)}<\infty.$$

We use abbreviation $f\in \Co^{\alpha,\beta}_{x,s}$ when the domain is clear. We use $f\in\Co^{\alpha,\beta}$ when the variables are also clear.
\end{defn}
\begin{remark}
    \begin{enumerate}[parsep=-0.3ex,label=(\roman*)]
        \item The norm definition is symmetric and we have $\Co^{\alpha,\beta}_{x,s}=\Co^{\beta,\alpha}_{s,x}$. But in applications we mostly refer the left variable as space variable, and the right variable as parameter. For the notation $\Co^{\alpha,\beta}$, we usually use $\alpha\ge\beta$.
        \item When $\alpha=\beta$, we have $\Co^{\alpha,\alpha}=\Co^\alpha$. See Lemma \ref{Lem::Hold::CharMixHold} \ref{Item::Hold::CharMixHold::HoldbyComp} below. So our $\Co^{\alpha,\beta}$-spaces is a generalization to the single variable cases. 
        \item We can define three variables mixed Zygmund spaces and the version on manifold analogously. See Definition \ref{Defn::Hold::MoreMixHold} \ref{Item::Hold::MoreMixHold::3Mix} and Definition \ref{Defn::ODE::MixHoldMaps}. 
    \end{enumerate}
\end{remark}
Also see Definition \ref{Defn::Hold::MoreMixHold} when $\alpha,\beta$ are not real numbers.

\begin{lem}\label{Lem::Hold::CharMixHold}
Let $\alpha,\beta>0$, and let $U\subset\R^m$, $V\subset\R^n$ be two bounded convex open sets with smooth boundaries.
\begin{enumerate}[nolistsep,label=(\roman*)]
    \item\label{Item::Hold::CharMixHold::Mix=Bi} $\Co^{\alpha,\beta}(U,V)=\Co^\alpha L^\infty\cap L^\infty\Co^\beta(U,V)$ with equivalent norms.
    \item\label{Item::Hold::CharMixHold::HoldbyComp} $\Co^{\alpha,\alpha}(U,V)=\Co^\alpha(U\times V)$ with  equivalent norms.
\end{enumerate}
\end{lem}
\begin{proof}
\ref{Item::Hold::CharMixHold::Mix=Bi}: By Lemma \ref{Lem::Hold::CommuteExt} we have an extension operator $\tilde E:=E_xE_s=E_sE_x:L^\infty\Co^\beta(U,V)\to L^\infty\Co^\beta(\R^n,\R^q)$, thus
\begin{equation}\label{Eqn::Hold::CharMixHold::Tmp}
    \textstyle\sup_{x\in U}\|f(x,\cdot)\|_{\Co^\beta(V)}\ge\essup_{x\in U}\|f(x,\cdot)\|_{\Co^\beta(V)}=\|E_sf\|_{L^\infty\Co^\beta(U,\R^q)}\approx\|\tilde Ef\|_{L^\infty\Co^\beta(\R^n,\R^q)}\approx\|f\|_{L^\infty\Co^\beta(U,V)}.
\end{equation}
Similarly $\sup_{s\in V}\|f(\cdot,s)\|_{\Co^\alpha(U)}\gtrsim\|f\|_{\Co^\alpha L^\infty(U,V)}$. Therefore $\|f\|_{\Co^{\alpha,\beta}(U,V)}\gtrsim\|f\|_{\Co^\alpha L^\infty\cap L^\infty\Co^\beta(U,V)}$.

For the converse direction, it suffices to show the first inequality in \eqref{Eqn::Hold::CharMixHold::Tmp} is indeed an equality.
By Remark \ref{Rmk::Hold::RmkforBiHold} \ref{Item::Hold::RmkforBiHold::Interpo} we have $\Co^\alpha L^\infty\cap L^\infty\Co^\beta(U,V)\subset\Co^\alpha \Co^0\cap \Co^0\Co^\beta(U,V)\subset\Co^{\frac\alpha2}\Co^\frac\beta2(U,V)\subset C^0(U,V)$. Thus $f\in \Co^\alpha L^\infty\cap L^\infty\Co^\beta(U,V)$ can be pointwise defined, and the essential supremum is indeed the supremum. This completes the proof.

\smallskip
\noindent\ref{Item::Hold::CharMixHold::HoldbyComp}: Clearly $\|f(x,\cdot)\|_{\Co^\alpha(V)}\le\|f\|_{\Co^\alpha(U\times V)}$ and $\|f(\cdot,s)\|_{\Co^\alpha(U)}\le\|f\|_{\Co^\alpha(U\times V)}$ for each $x\in U$ and $s\in V$. Therefore $\|f\|_{\Co^{\alpha,\alpha}(U,V)}\le 2\|f\|_{\Co^\alpha(U\times V)}$.

For the converse, we can choose $(\phi_j)_{j=0}^\infty$ and $(\psi_k)_{k=0}^\infty$ as dyadic resolutions for $\R^n_x$ and $\R^q_s$ respectively, such that $\supp\hat\phi_0\subset B^n(0,\sqrt2)$ and $\supp\hat\psi_0\subset B^q(0,\sqrt2)$ additionally. Define $(\pi_l)_{l=0}^\infty$ as $\pi_0=\phi_0\otimes\psi_0$ and $\pi_l=2^{l(n+q)}\pi_0(2^l\cdot)-2^{(l-1)(n+q)}\pi_0(2^{l-1}\cdot)$ for $l\ge1$. We see that $(\pi_l)_l$ is a dyadic resolutions for $\R^{n+q}$ and
\begin{equation*}
    \textstyle\pi_l=\sum_{\max(j,k)=l}\phi_j\otimes\psi_k=\phi_l\otimes\delta_{0\in\R^q}+\delta_{0\in\R^n}\otimes\psi_l-\sum_{\min(j,k)=l}\phi_j\otimes\psi_k,\quad\forall\ l\ge0.
\end{equation*}
Therefore for $\tilde f\in\Co^\alpha(\R^{n+q})$, by Lemma \ref{Lem::Hold::HoldChar} \ref{Item::Hold::HoldChar::LPHoldChar} and Remark \ref{Rmk::Hold::RmkforBiHold} \ref{Item::Hold::RmkforBiHold::CxCs=CsCx},
\begin{align*}
    &\textstyle\|\pi_l\ast \tilde f\|_{L^\infty}\le\|\phi_l\ast_x\tilde f\|_{L^\infty}+\|\psi_l\ast_s\tilde f\|_{L^\infty}+\sum_{j=l}^\infty\|\psi_l\ast_s(\phi_j\ast_x\tilde f)\|_{L^\infty}+\sum_{k=l+1}^\infty\|\phi_l\ast_x(\psi_k\ast_s\tilde f)\|_{L^\infty}
    \\
    \lesssim&_\alpha\|\tilde f\|_{\Co^\alpha_xL^\infty_s}2^{-l\alpha}+\|\tilde f\|_{L^\infty_x\Co^\alpha_s}2^{-l\alpha}+\|\psi_l\|_{L^1}\|\tilde f\|_{\Co^\alpha_xL^\infty_s}\sum_{j=l}^\infty2^{-j\alpha}+\|\phi_l\|_{L^1}\|\tilde f\|_{L^\infty_x\Co^\alpha_s}\sum_{k=l+1}^\infty2^{-k\alpha}\lesssim\|\tilde f\|_{\Co^\alpha_xL^\infty_s\cap L^\infty_x\Co^\alpha_s}2^{-l\alpha}.
\end{align*}
Thus $\|\tilde f\|_{\Co^\alpha(\R^n\times\R^q)}\lesssim_\alpha\|\tilde f\|_{\Co^\alpha L^\infty\cap L^\infty\Co^\alpha(\R^n,\R^q)}$. Taking restrictions to the domain $U\times V$ we finish the proof.
\end{proof}




\begin{remark}\label{Rmk::Hold::SimpleHoldbyCompFact}Let $\alpha,\beta>0$ and let $U,V$ be two bounded smooth domains. Since $\Co^0\subset L^\infty\subset\Co^1$, we have $\Co^\alpha L^\infty\cap\Co^1\Co^\beta(U,V)\subset \Co^{\alpha,\beta}(U,V)\subset \Co^\alpha L^\infty\cap\Co^0\Co^\beta(U,V)$.
\end{remark}

\begin{remark}[Comparison of different working spaces]\label{Rmk::Hold::ComparisonSpaces}
    The bi-parameter structures used in \cite[Section 2]{Gong} and \cite[Section 4.4]{NijenhuisWoolf} are the $\Cc^{\alpha,\beta}$-spaces, which are similar to our $\Co^{\alpha,\beta}$-spaces. We recall them here:
    
    \textit{Let $k\ge l\ge0$ be integers, let $r,s\in[0,1)$ and let $U\subseteq\R^n_x$, $V\subseteq\R^q_s$ be two open (bounded smooth) subsets. The space $\Cc^{k+r,l+s}(U,V)$ is the set of all continuous functions $f:U\times V\to\R$ whose norm below is finite:}
    \begin{equation*}
        \textstyle\|f\|_{\Cc^{k+r,l+s}(U,V)}:=\sum_{j=0}^l\sum_{i=0}^{k-j}\big(\sup_{s\in V}\|\nabla_x^i\nabla_s^jf(\cdot,s)\|_{C^{0,r}(U)}+\sup_{x\in U}\|\nabla_x^i\nabla_s^jf(x,\cdot)\|_{C^{0,s}(V)}\big).
    \end{equation*}
    \textit{Here $C^{0,\gamma}$ is the standard H\"older space $(0\le\gamma<1)$.}

    One can see that $\Cc^{k+r,l+s}(U,V)=\Co^{k+r,l+s}(U,V)$ when $r,s\neq0$ and $\Cc^{k+r,l+s}(U,V)\subsetneq\Co^{k+r,l+s}(U,V)$ when $r$ or $s=0$. We leave the proof to readers. 
    
    The $\Co^{\alpha,\beta}$-spaces is better than the $\Cc^{\alpha,\beta}$-spaces in some sense. On one hand $\Cc^{\alpha,\beta}$ are only defined for $\lfloor\alpha\rfloor\ge\lfloor\beta\rfloor$, while $\Co^{\alpha,\beta}$ are defined for all positive $\alpha,\beta$. On the other hand, the discussions of $\Cc^{\alpha,\beta}$ are always separated between the integer cases and non-integer cases, while for $\Co^{\alpha,\beta}$-spaces the discussion can be unified as we have complex interpolations,
    \begin{equation}\label{Eqn::Hold::CpxInt}
        [\Co^{\alpha_0,\beta_0}(U,V),\Co^{\alpha_1,\beta_1}(U,V)]_\theta=\Co^{(1-\theta)\alpha_0+\theta\alpha_1,(1-\theta)\beta_0+\theta\beta_1}(U,V),\quad\text{for all }\alpha_0,\alpha_1,\beta_0,\beta_1>0,\quad\theta\in[0,1].
    \end{equation}
    We also leave the proof of \eqref{Eqn::Hold::CpxInt} to readers.
\end{remark}

Here are more definitions.
\begin{defn}\label{Defn::Hold::MoreMixHold}
Let $\alpha,\beta>0$ and let $U\in\R^n$, $V\in\R^q$ be two open convex sets. We define the following spaces:
\begin{enumerate}[nolistsep,label=(\roman*)]
    \item $\Co^{\alpha,\beta}_{\loc}(U,V):=\{f\in C^0_\loc(U\times V):f\in\Co^{\alpha,\beta}(U',V'),\ \forall U'\Subset U,\ V'\Subset V\}$.
    \item $\Co^{\alpha,\beta-}(U,V)=\Co^{\beta-,\alpha}(U,V):=\bigcap_{0<\gamma<\beta}\Co^{\alpha,\gamma}(U,V)$.
    \item $\Co^{\infty,\beta}(U,V):=\bigcap_{\gamma>0}\Co^{\gamma,\beta}(U,V)$ and $\Co^\infty\Co^\beta(U,V):=\bigcap_{\gamma>0}\Co^\gamma\Co^\beta(U,V)$.
    \item\label{Item::Hold::MoreMixHold::3Mix} Let $\gamma>0$ and let $W\subset\R^m$ be as well an open convex set. We define $\Co^{\alpha,\beta,\gamma}(U,V,W)$ to be the set of all continuous functions $f:U\times V\times W\to\R$ satisfying
\begin{align*}
    \|f\|_{\Co^{\alpha,\beta,\gamma}_{x,s,t}(U,V,W)}:=&\sup_{x\in U,s\in V}\|f(x,s,\cdot)\|_{\Co^\gamma(W)}+\sup_{x\in U,t\in W}\|f(x,\cdot,t)\|_{\Co^\beta(V)}+\sup\limits_{s\in V,t\in W}\|f(\cdot,s,t)\|_{\Co^\alpha(U)}<\infty.
\end{align*}
\end{enumerate}
\end{defn}

\begin{remark}\label{Rmk::Hold::RmkCAlpBetGam}By Lemma \ref{Lem::Hold::CharMixHold} \ref{Item::Hold::CharMixHold::Mix=Bi} we can say that $\Co^{\alpha,\beta,\gamma}_{x,s,t}=\Co^\alpha_xL^\infty_{s,t}\cap\Co^\beta_sL^\infty_{x,t}\cap\Co^\gamma_tL^\infty_{x,s}=\Co^\alpha_xL^\infty_{s,t}\cap L^\infty_x\Co^{\beta,\gamma}_{s,t}$, where $\Co^\alpha_xL^\infty_{s,t}\cap\Co^\beta_sL^\infty_{x,t}\cap\Co^\gamma_tL^\infty_{x,s} $ and $\Co^\alpha_xL^\infty_{s,t}\cap L^\infty_x\Co^{\beta,\gamma}_{s,t}$ are defined analogously by Definitions \ref{Defn::Hold::BiHold} and \ref{Defn::Hold::MixHold}.
\end{remark}

For $0<\alpha\le\infty$, $0<\beta<\infty$, that $f\in\Co^{\alpha+1,\beta}_{x,s}$ does not imply $\nabla_x f\in\Co^{\alpha,\beta}_{x,s}$.
\begin{example}\label{Exam::Hold::GradMixHold}
Let $\alpha,\beta>0$. We consider the following $f:\R_x\times\R_s\to\C$:
$$f(x,s):=\sum_{l=1}^\infty 2^{-l(\alpha+1)}\exp(i2^{l}x+i2^{\frac{\alpha+1}\beta l}s).\quad\text{ We have }\frac{\partial f}{\partial x}(x,s)=\sum_{l=1}^\infty i2^{-l\alpha}\exp(i2^{l}x+i2^{\frac{\alpha+1}\beta l}s).$$

One can check that $f\in \Co^{\alpha+1}_xL^\infty_s$ and $f\in L^\infty_x\Co^\beta_s$. So $f\in\Co^{\alpha+1,\beta}_{x,s}$.

Let $\psi=(\psi_k)_{k=0}^\infty$ be a dyadic resolution. By \eqref{Eqn::Hold::RmkDyaSupp} $\supp\hat\psi_k=\supp2^{k-1}\hat\psi_1\subsetneq(-2^{k+1},-2^{k-1})\cup(2^{k-1},2^{k+1})$ for $k\ge1$, so there is an $\eps_0>0$ such that $\supp\hat\psi_k\subseteq 2^k\cdot\big([-2+\eps_0,-\frac12-\eps_0]\cup[\frac12+\eps_0,2-\eps_0]\big)$. Note that there are infinitely many integer pairs $(k,l)$ such that $|k-\frac{\alpha+1}\beta l|<\frac{\eps_0}2$, and for such a pair $(\phi_k\ast_sf)(x,s)=i2^{-l\alpha}\exp(i2^{l}x+i2^{\frac{\alpha+1}\beta l}s)$. 

So  $\|\psi_k\ast_sf\|_{L^\infty(\R^2)}\approx2^{-k\frac{\alpha\beta}{\alpha+1}}$ holds for infinite many $k\ge0$, which means $f\notin L^\infty_x\Co^{\frac{\alpha\beta}{\alpha+1}+\eps}_s$ for any $\eps>0$.
\end{example}
In general, for $0<\alpha,\beta<\infty$ and a function $f\in\Co^{\alpha+1,\beta}_{x,s}$, the best we can guarantee is $\nabla_xf\in\Co^{\alpha,\frac{\alpha\beta}{\alpha+1}}_{x,s}$, where the index $\frac{\alpha\beta}{\alpha+1}$ is optimal.
For applications we need the following endpoint optimal result.
\begin{lem}\label{Lem::Hold::GradMixHold}
Let $\alpha,\beta>0$ and let $U\subseteq\R^n_x$, $V\subseteq\R^q_s$ be two bounded convex smooth domains. Then $\nabla_x:\Co^{\alpha+1,\beta}(U,V)\to\Co^{\alpha,\frac{\alpha\beta}{\alpha+1}-}(U,V;\R^n)$ is bounded.
In particular, if $f\in\Co^{\infty,\beta}(U,V)$, then $\nabla_xf\in\Co^{\infty,\beta-}(U,V;\R^n)$.
\end{lem}
\begin{proof}
Clearly $\nabla_x:\Co^{\alpha+1,\beta}_{x,s}\subset \Co^{\alpha+1}_xL^\infty_s\cap \Co^0_x\Co^\beta_s\to \Co^\alpha_xL^\infty_s\cap \Co^{-1}_x\Co^\beta_s$ is bounded. By Remark \ref{Rmk::Hold::RmkforBiHold} \ref{Item::Hold::RmkforBiHold::Interpo} we have embedding $\Co^\alpha_xL^\infty_s\cap \Co^{-1}_x\Co^\beta_s\subset\Co^{(1-\theta)\alpha-\theta}_x\Co^{\theta\beta}_s$ for all $\theta\in(0,1)$. Thus $\Co^\alpha_xL^\infty_s\cap \Co^{-1}_x\Co^\beta_s\subset\bigcap_{0<\theta<\frac\alpha{\alpha+1}}\Co^{(1-\theta)\alpha-\theta}_x\Co^{\theta\beta}_s\subset L^\infty_x\Co^{\frac{\alpha\beta}{\alpha+1}-}_s$. This proves the boundedness $\nabla_x:\Co^{\alpha+1,\beta}_{x,s}\to\Co^{\alpha,\frac{\alpha\beta}{\alpha+1}-}_{x,s}$.

Let $\alpha\to+\infty$ we get $\nabla_x:\Co^{\infty,\beta}_{x,s}\to\Co^{\infty,\beta-}_{x,s}$.
\end{proof}

On the other hand, if $f(x,s)$ is, say a harmonic function in $x$, then the assumption $f\in\Co^{\alpha+1,\beta}_{x,s}$ can imply $\nabla_xf\in\Co^{\alpha,\beta}_{x,s}$. In fact we have $f\in\Co^\infty_x\Co^\beta_s$ by the following.
\begin{lem}\label{Lem::Hold::NablaHarm}
Let $\beta>0$, and let $U\subset\R^n_x$ and $V\subset\R^q_s$ be two bounded open sets with smooth boundaries. Suppose $f\in \Co^{-1}\Co^\beta(U,V)$ satisfies $\Delta_xf=0$. Then $f\in\Co^\infty_\loc\Co^\beta(U,V)$. In particular $f\in\Co^{\beta+1,\beta}_\loc(U,V)$ and $\nabla_xf\in\Co^\beta_\loc(U\times V;\R^n)$.

As a corollary, let $\tilde U\subseteq\C^m_z$ be an open set. If $g\in \Co^\beta_\loc(\tilde U\times V;\C)$ is holomorphic in $z$, then $\partial_zg\in\Co^\beta_\loc(\tilde U\times V;\C^m)$ holds.
\end{lem}
\begin{proof}Since the result is local in $x\in U$, it is enough to consider the case where $U=\B^n$ is the unit ball.

Let $E_s:\Co^\beta(V)\to \Co^\beta(\R^q)$ be an $s$-variable extension as in Lemma \ref{Lem::Hold::CommuteExt}. Clearly $\nabla_x^kE_sf\in\Co^{-k-1}\Co^\beta(U,\R^n;\R^{2^k})$ for every $k\ge0$.

Let $(\psi_j)_{j=0}^\infty$ be a dyadic resolution for $\R^q_s$. By Lemma \ref{Lem::SecHolLap::HLLem} \ref{Item::SecHolLap::HLLem::Harm} we have
\begin{equation*}
    2^{j\beta}|\psi_j\ast_sE_s\nabla_x^kf(x,s)|\lesssim_{n,k}(1-|x|)^{-k-1}2^{j\beta}\|\psi_j\ast_sE_s\nabla_x^kf\|_{\Co^{-k-1}(\B^n)}\lesssim_{\beta,E}\|\nabla_x^kf\|_{\Co^{-k-1}\Co^\beta(\B^n,V)}.
\end{equation*}
Taking supremum in $s\in\R^q$ and $j\ge0$, by Lemmas \ref{Lem::Hold::HoldChar} \ref{Item::Hold::HoldChar::LPHoldChar} we get 
$$\sup_{|\nu|\le k}\|\partial_x^\nu f\|_{L^\infty\Co^\beta(r\B^n,V)}\lesssim_{V,n,k,\beta}(1-r)^{-k-1}\|f\|_{\Co^{-1}\Co^\beta(\B^n,V)},\quad\forall 0<r<1.$$
Therefore $f\in\Co^k_\loc\Co^\beta(\B^n,V)$ for all $k$, which means $f\in\Co^\infty_\loc\Co^\beta(\B^n,V)$. The results $f\in\Co^{\beta+1,\beta}_\loc$ and $\nabla_xf\in\Co^\beta_\loc$ follow immediately.

When $g(z,s)$ is holomorphic in $z$, we know $g$ is complex harmonic in $z$-variable. So we reduce it to the real case by considering $\tilde U\subseteq\R^{2m}$.
\end{proof}

\subsection{A proposition for bi-parameter products}\label{Section::HoldMult}
In this subsection we are going to compute the mixed regularity of products using bi-parameter paraproducts. 

Recall in Definition \ref{Defn::Hold::BiHold}, we fix two dyadic resolutions $(\phi_j)_{j=0}^\infty\subset\Sc(\R^m)$ and $(\psi_k)_{k=0}^\infty\subset\Sc(\R^n)$ to define the $\Co^\alpha L^\infty $-norms and the $\Co^\alpha\Co^\beta $-norms.

\begin{prop}\label{Prop::Hold::Mult}
Let $\alpha,\beta>0$ and $\eta,\theta\in[0,1]$ satisfy $\eta+\theta<\alpha$. There is a $C=C_{\alpha,\beta,\eta,\theta,\phi,\psi}>0$ such that 
$$\|fg\|_{\Co^{\alpha-\theta}L^\infty\cap\Co^{-\eta-\theta}\Co^\beta(\R^n,\R^q)}\le C\|f\|_{\Co^{\alpha}L^\infty\cap\Co^{-\eta}\Co^\beta(\R^n,\R^q)}\|g\|_{\Co^{\alpha-\theta}L^\infty\cap\Co^{-\eta-\theta}\Co^\beta(\R^n,\R^q)},$$
for all $f\in \Co^{\alpha}L^\infty\cap\Co^{-\eta}\Co^\beta(\R^n,\R^q)$ and $g\in\Co^{\alpha-\theta}L^\infty\cap\Co^{-\eta-\theta}\Co^\beta(\R^n,\R^q)$.
\end{prop}
We postpone the proof to the end of this subsection.

For applications, we use the cases $(\eta,\theta)=(1,0)$ and $(\eta,\theta)=(0,0)$ to prove Proposition \ref{Prop::PDE::ExistPDE} that the left hand side of \eqref{Eqn::PDE::ExistenceH} behave well. And we use $(\eta,\theta)=(0,1)$ to prove Proposition \ref{Prop::Hold::CompThm} which is required in Proposition \ref{Prop::PDE::ExistPDE} \ref{Item::PDE::ExistPDE::Phi}.

\begin{cor}\label{Cor::Hold::CorMult}
Let $\alpha,\beta>0$ and $\eta,\theta\in[0,1]$ satisfy $\eta+\theta<\alpha$. Let $U\subset\R^m$ and $V\subset\R^n$ be two bounded open convex subsets with smooth boundaries. Then the product operation is bilinear bounded as the following:
\begin{equation}\label{Eqn::Hold::CorMult}
    [(f,g)\to fg]:\Co^\alpha L^\infty \cap\Co^{-\eta} \Co^\beta(U,V)\times\Co^{\alpha-\theta} L^\infty \cap\Co^{-\eta-\theta} \Co^\beta(U,V)\to \Co^{\alpha-\theta} L^\infty \cap\Co^{-\eta-\theta} \Co^\beta(U,V).
\end{equation}
In particular it is bilinearly bounded between the following spaces: for $\alpha>1$,
\begin{enumerate}[nolistsep,label=(\roman*)]
    \item\label{Item::Hold::CorMult::Prin} $\Co^\alpha L^\infty \cap\Co^0 \Co^\beta(U,V)\times\Co^{\alpha-1} L^\infty \cap\Co^{-1} \Co^\beta(U,V)\to \Co^{\alpha-1} L^\infty \cap\Co^{-1} \Co^\beta(U,V)$.
    \item\label{Item::Hold::CorMult::0} $\Co^\alpha L^\infty \cap\Co^0 \Co^\beta(U,V)\times \Co^\alpha L^\infty \cap\Co^0 \Co^\beta(U,V)\to \Co^\alpha L^\infty \cap\Co^0 \Co^\beta(U,V)$.
    \item\label{Item::Hold::CorMult::-1} $\Co^\alpha L^\infty \cap\Co^{-1} \Co^\beta(U,V)\times \Co^\alpha L^\infty \cap\Co^{-1} \Co^\beta(U,V)\to \Co^\alpha L^\infty \cap\Co^{-1} \Co^\beta(U,V)$.
\end{enumerate}
\end{cor}
\begin{proof}
Indeed \ref{Item::Hold::CorMult::Prin}, \ref{Item::Hold::CorMult::0} and \ref{Item::Hold::CorMult::-1} are the cases where $(\eta,\theta)=(0,1)$, $(\eta,\theta)=(0,0)$ and $(\eta,\theta)=(1,0)$ in \eqref{Eqn::Hold::CorMult} respectively.

By Lemma \ref{Lem::Hold::CommuteExt} there is an extension operator such that $\tilde E:\Xs(U,V)\to\Xs(\R^n,\R^q)$ and  $\tilde E:\Y(U,V)\to\Y(\R^n,\R^q)$. Here we use $\Xs=\Co^\alpha L^\infty \cap\Co^{-\eta} \Co^\beta$ and $\Y=\Co^{\alpha-\theta} L^\infty \cap\Co^{-\eta-\theta} \Co^\beta$.
Applying Proposition \ref{Prop::Hold::Mult}, 
\begin{equation*}
    \|fg\|_{\Y(U,V)}\le\|\tilde Ef\tilde Eg\|_{\Y(\R^n,\R^q)}\lesssim\|\tilde Ef\|_{\Xs(\R^n,\R^q)}\|\tilde Eg\|_{\Y(\R^n,\R^q)}\lesssim\|f\|_{\Xs(U,V)}\|g\|_{\Y(U,V)}.
\end{equation*}
So we obtain \eqref{Eqn::Hold::CorMult}.
\end{proof}

\begin{cor}\label{Cor::Hold::CramerMixed}
Let $\alpha,\beta>0$ and $\eta\in[0,1]$ satisfy $\eta<\alpha$. Let $m,n,q\ge0$ and let $U\subset \R^n $, $V\subset\R^q $ be bounded convex open sets with smooth boundaries.
Then there is a $c_0=c_0(U,V,m,n,q,\alpha,\beta,\eta)>0$, such that for any matrix-valued function $A: U\times V\to\C^{m\times m}$ satisfying $\|A-I_m\|_{\Co^{\alpha} L^\infty \cap\Co^{-\eta} \Co^\beta(U,V;\C^{m\times m})}<c_0$, we have:
\begin{itemize}[nolistsep]
    \item  $A(x,s)\in\C^{m\times m}$ is an invertible matrix for every $(x,s)\in U\times V$.
    \item   $\|A^{-1}-I_m\|_{\Co^{\alpha} L^\infty \cap\Co^{-\eta} \Co^\beta(U,V;\C^{m\times m})}\le2\|A-I_m\|_{\Co^{\alpha} L^\infty \cap\Co^{-\eta} \Co^\beta(U,V;\C^{m\times m})}$.
\end{itemize}
Moreover, the map of inverting matrices
$$[A\mapsto A^{-1}]:\{A\in \Co^{\alpha} L^\infty \cap\Co^{-\eta} \Co^\beta(U,V;\C^{m\times m}):\|A-I_m\|_{\Co^{\alpha} L^\infty \cap\Co^{-\eta} \Co^\beta }<c_0\}\to \Co^{\alpha} L^\infty \cap\Co^{-\eta} \Co^\beta(U,V;\C^{m\times m}),$$ is real-analytic and in particular Fr\'echet differentiable.
\end{cor}
\begin{proof}
Applying Corollary \ref{Cor::Hold::CorMult} to the components of the matrix product $(A-I_m)^k=(A-I_m)^{k-1}(A-I_m)$ recursively in $k$, with the parameter $\theta=0$,  we know there is a $C=C(U,V,m,n,q,\alpha,\beta,\eta)>0$ such that
\begin{align*}
    &\|(A-I_m)^k\|_{\Co^{\alpha} L^\infty \cap\Co^{-\eta} \Co^\beta}\le C^{k-1}\|A-I_m\|_{\Co^{\alpha} L^\infty \cap\Co^{-\eta} \Co^\beta }^k,\quad\forall\ k\ge1,\ A\in\Co^{\alpha} L^\infty \cap\Co^{-\eta} \Co^\beta(U,V;\C^{m\times m}).
\end{align*}

Now take $c_0=\frac12 C^{-1}$, we see that the sum $\sum_{k=0}^\infty(I_m-A)^k$ is norm convergent in $\Co^{\alpha} L^\infty \cap\Co^{-\eta} \Co^\beta(U,V;\C^{m\times m})$ whose limit equals to $(I_m-(I_m-A))^{-1}=A^{-1}$. So $A(x,s)$ is an invertible matrix for all $(x,s)\in U\times V$.

That $\Co^{\alpha} L^\infty \cap\Co^{-\eta} \Co^\beta $ is closed under products also implies that $[A\mapsto A^{-1}=I+\sum_{k=1}^\infty(I-A)^k]$ is a real-analytic map in the open set $\{A:\|A-I_m\|_{\Co^{\alpha} L^\infty \cap\Co^{-\eta} \Co^\beta}<c_0\}$ with respect to the norm-topology.
\end{proof}
\begin{remark}\label{Rmk::Hold::CramerMixedforSing}
    By taking $\eta=0$ and forgetting the $s$-variable, we get the following classical result:
    
    For any $\alpha>0$ and bounded convex smooth open $U\subset\R^n $, there is a $c_0'(U,m,\alpha)>0$, such that
    \begin{equation*}
        \forall A\in\Co^\alpha(U;\C^{m\times m}),\quad\|A-I_m\|_{\Co^\alpha(U;\C^{m\times m})}<c_0'\quad\Rightarrow\quad \|A^{-1}-I_m\|_{\Co^\alpha(U;\C^{m\times m})}\le 2\|A-I_m\|_{\Co^\alpha(U;\C^{m\times m})}.
    \end{equation*}
\end{remark}

The proof of Proposition \ref{Prop::Hold::Mult} use paraproduct decompositions. The single parameter case is the following:
\begin{lem}[The paraproduct decomposition]\label{Lem::Hold::LemParaProd}
Let $(\phi_j)_{j=0}^\infty\subset\Sc(\R^n)$ be a dyadic resolution. For bounded continuous functions $u,v$ defined on $\R^n$, consider the following decomposition
\begin{equation}\label{Eqn::Hold::LemParaProd::ParaDecomp1}
    uv=\sum_{j=0}^\infty\sum_{j'=0}^{j-3}(\phi_j\ast u)(\phi_{j'}\ast v)+\sum_{j'=0}^\infty\sum_{j=0}^{j'-3}(\phi_j\ast u)(\phi_{j'}\ast v)+\sum_{|j-j'|\le 2}(\phi_j\ast u)(\phi_{j'}\ast v).
\end{equation}

Then for every $l\ge0$, 
    \begin{equation}\label{Eqn::Hold::LemParaProd::ParaDecomp2}
    \phi_l\ast(uv)=\phi_l\ast\bigg(\sum_{j=l-2}^{l+2}\sum_{j'=0}^{j-3}(\phi_j\ast u)(\phi_{j'}\ast v)+\sum_{j'=l-2}^{l+2}\sum_{j=0}^{j'-3}(\phi_j\ast u)(\phi_{j'}\ast v)+\sum_{\substack{j,j'\ge l-3\\|j-j'|\le 2}}(\phi_j\ast u)(\phi_{j'}\ast v)\bigg).
\end{equation}
\end{lem}
\begin{remark}\label{Rmk::Hold::LemParaProd}
    In \cite[Chapter 2.8.1, (2.41)]{Chemin} \eqref{Eqn::Hold::LemParaProd::ParaDecomp1} is denoted by $uv=T(u,v)+T(v,u)+R(u,v)$ where $T$ is the paraproduct operator associated with $\phi$.
\end{remark}
\begin{proof}
\eqref{Eqn::Hold::LemParaProd::ParaDecomp2} can be obtained by investigating the Fourier support condition in Definition \ref{Defn::Hold::DyadicResolution}. Note that $\phi_l\ast\left((\phi_j\ast u)(\phi_{j'}\ast v)\right)\neq0$ only when $\supp\hat\phi_l\cap(\supp\hat\phi_j+\supp\hat\phi_{j'})\neq\varnothing$, and by \eqref{Eqn::Hold::RmkDyaSupp} we have on the following:
$$
\supp\hat\phi_l\subset\begin{cases}\{2^{l-1}<|\xi|<2^{l+1}\},&l\ge1,\\\{|\xi|<2\},&l=0;\end{cases}\quad\supp\hat\phi_j+\supp\hat\phi_{j'}\subset\begin{cases}\{2^{\max(j,j')-2}<|\xi|<2^{\max(j,j')+2}\},&|j-j'|\ge3,\\\{|\xi|<2^{\max(j,j')+2}\},&|j-j'|\le2.\end{cases}$$

For details, see  \cite[Proof of Theorem 2.8.2]{Triebel1}, \cite[Chapter 4.4]{Grafakos} or \cite[Proof of Theorem 2.52]{Chemin}. Note that their choice of decompositions \eqref{Eqn::Hold::LemParaProd::ParaDecomp1} may have slight differences.
\end{proof}
\begin{lem}\label{Lem::Hold::Product}
Let $\alpha>0$ and $\beta\in(-\alpha,\alpha]$. Let $U\subset\R^n$ be a bounded convex open set with smooth boundary.
\begin{enumerate}[nolistsep,label=(\roman*)]
    \item\label{Item::Hold::Product::Hold1} There is a $C=C(U,\alpha,\beta)\ge0$ such that $\|fg\|_{\Co^\beta(U)}\le C\|f\|_{\Co^\alpha(U)}\|g\|_{\Co^\beta(U)}$ for all $f\in\Co^\alpha(U)$, $g\in\Co^\beta(U)$.
    \item\label{Item::Hold::Product::Hold2} There is a $C=C(U,\alpha)\ge0$ such that $\|fg\|_{\Co^\alpha(U)}\le C(\|f\|_{\Co^\alpha(U)}\|g\|_{L^\infty(U)}+\|f\|_{L^\infty(U)}\|g\|_{\Co^\alpha(U)})$ for all $f,g\in\Co^\alpha(U)$.
\end{enumerate}
\end{lem}
\begin{proof}
See \cite[Theorem 3.3.2(ii)]{Triebel1} for \ref{Item::Hold::Product::Hold1} and \cite[Corollary 2.86]{Chemin} for \ref{Item::Hold::Product::Hold2}. Note that by Lemma \ref{Lem::Hold::HoldChar} and Definition \ref{Defn::Hold::NegHold} H\"older-Zygmund spaces are a special case to Besov spaces $\Co^\alpha(U)=\Bs_{\infty\infty}^\alpha(U)$ for  all $\alpha\in\R$.
\end{proof}

We now prove Proposition \ref{Prop::Hold::Mult}.

\begin{proof}[Proof of Proposition \ref{Prop::Hold::Mult}]We fix $(\phi_j)_{j=0}^\infty\subset\Sc(\R^n)$ and $(\psi_k)_{k=0}^\infty\subset\Sc(\R^q)$ as two dyadic resolutions for their spaces.
Note that by assumption $f,g$ are both $L^\infty$-functions, so their product is already defined in $L^\infty(U\times V)$. What we need is to show that $fg$ lays in the desired function class.

Since $\alpha-\theta>0$, by Lemma \ref{Lem::Hold::Product} \ref{Item::Hold::Product::Hold1} for almost every $s\in\R^q$, $\|f(\cdot,s)g(\cdot,s)\|_{\Co^{\alpha-\theta} (\R^n)}\lesssim_{\alpha-\theta}\|f(\cdot,s)\|_{\Co^{\alpha}}\|g(\cdot,s)\|_{\Co^{\alpha-\theta}}$,
with the implied constant being independent of $f$ and $g$. So by taking essential supremum of $s$, we get 
\begin{equation}\label{Eqn::Hold::Mult::CxLsControl}
    \begin{aligned}
    \|fg\|_{\Co^{\alpha-\theta} L^\infty (\R^n,\R^q)}=&\essup_{s\in\R^q}\|f(\cdot,s)g(\cdot,s)\|_{\Co^{\alpha-\theta} (\R^n)}\lesssim \|f\|_{\Co^\alpha L^\infty (\R^n,\R^q)}\|g\|_{\Co^{\alpha-\theta} L^\infty (\R^n,\R^q)}
    \\
    \le&\|f\|_{\Co^\alpha L^\infty \cap\Co^{-\eta} \Co^\beta(\R^n,\R^q)}\|g\|_{\Co^{\alpha-\theta} L^\infty \cap\Co^{-\eta-\theta} \Co^\beta(\R^n,\R^q)}.
\end{aligned}
\end{equation}

Our goal is to show the bilinear control $\|fg\|_{\Co^{-\eta-\theta} \Co^\beta}\lesssim\|f\|_{\Co^\alpha L^\infty \cap\Co^{-\eta} \Co^\beta}\|g\|_{\Co^{\alpha-\theta} L^\infty \cap\Co^{-\eta-\theta} \Co^\beta }$.

\medskip
Before we move into the computations, we first upgrade \eqref{Eqn::Hold::LemParaProd::ParaDecomp2} to the bi-parameter case.

For bounded functions $f$ and $g$ defined on $\R^n \times\R^q $, we denote
\begin{equation*}
    f_{jk}:=(\phi_j\otimes \psi_k)\ast f,\quad g_{jk}:=(\phi_j\otimes\psi_k)\ast g,\quad(fg)_{jk}:=(\phi_j\otimes\psi_k)\ast (fg),\quad j,k\ge0.
\end{equation*}

We can write $\displaystyle fg=\bigg(\sum\limits_{j'+3\le j}+\sum\limits_{j+3\le j'}+\sum\limits_{|j-j'|\le2}\bigg)\bigg(\sum\limits_{k'+3\le k}+\sum\limits_{k+3\le k'}+\sum\limits_{|k-k'|\le2}\bigg)(f_{jk}g_{j'k'})$ as $3\times 3$ sums:

\begin{equation}\label{Eqn::Hold::BigSum}
    \begin{aligned}
    fg=\sum_{\mu,\nu=1}^3 S^{\mu\nu}=&\sum\limits_{j=3}^\infty\sum\limits_{j'=0}^{j-3}\sum\limits_{k=3}^\infty\sum\limits_{k'=0}^{k-3}f_{jk}g_{j'k'}+\sum\limits_{j'=3}^\infty\sum\limits_{j=0}^{j'-3}\sum\limits_{k=3}^\infty\sum\limits_{k'=0}^{k-3}f_{jk}g_{j'k'}+\sum\limits_{|j-j'|\le2}\sum\limits_{k=3}^\infty\sum\limits_{k'=0}^{k-3}f_{jk}g_{j'k'}\\
    &+\sum\limits_{j=3}^\infty\sum\limits_{j'=0}^{j-3}\sum\limits_{k'=3}^\infty\sum\limits_{k=0}^{k'-3}f_{jk}g_{j'k'}+\sum\limits_{j'=3}^\infty\sum\limits_{j=0}^{j'-3}\sum\limits_{k'=3}^\infty\sum\limits_{k=0}^{k'-3}f_{jk}g_{j'k'}+\sum\limits_{|j-j'|\le2}\sum\limits_{k=3}^\infty\sum\limits_{k=0}^{k'-3}f_{jk}g_{j'k'}\\
    &+\sum\limits_{j=3}^\infty\sum\limits_{j'=0}^{j-3}\sum\limits_{|k-k'|\le2}f_{jk}g_{j'k'}+\sum\limits_{j'=3}^\infty\sum\limits_{j=0}^{j'-3}\sum\limits_{|k-k'|\le2}f_{jk}g_{j'k'}+\sum\limits_{|j-j'|\le2}\sum\limits_{|k-k'|\le2}f_{jk}g_{j'k'}.
\end{aligned}
\end{equation}
Here $S^{\mu\nu}$,  $1\le \mu,\nu\le 3$, is the $\mu$-th row $\nu$-th column term in \eqref{Eqn::Hold::BigSum}. We denote
\begin{equation*}
    S^{\mu\nu}_{lq}:=\|(\phi_j\otimes\psi_k)\ast S^{\mu\nu}\|_{L^\infty(\R^n\times\R^q)},\quad l,q\ge0.
\end{equation*}

Applying Lemma \ref{Lem::Hold::LemParaProd} on $x$-variable and $s$-variable separately we get
\begin{equation}\label{Eqn::Hold::BiParaProductDecomp}
    (fg)_{lq}=(\phi_l\otimes\psi_q)\ast\Bigg(\sum\limits_{j=l-2}^{l+2}\sum\limits_{j'=0}^{j-3}+\sum\limits_{j'=l-2}^{l+2}\sum\limits_{j=0}^{j'-3}+\sum\limits_{\substack{|j-j'|\le2\\j,j'\ge l-3}}\Bigg)\Bigg(\sum\limits_{k=q-2}^{q+2}\sum\limits_{k'=0}^{k-3}+\sum\limits_{k'=q-2}^{q+2}\sum\limits_{k=0}^{k'-3}+\sum\limits_{\substack{|k-k'|\le2\\k,k'\ge q-3}}\Bigg)(f_{jk}g_{j'k'}).
\end{equation}

The assumption $f\in\Co^\alpha L^\infty\subset \Co^\alpha \Co^0$ implies that $|f_{jk}|\lesssim_{\phi,\psi}\|f\|_{\Co^\alpha \Co^0}2^{-j\alpha}$, and the assumption $f\in\Co^{-\eta} \Co^\beta$ implies that $|f_{jk}|\lesssim_{\phi,\psi}\|f\|_{\Co^{-\eta} \Co^\beta}2^{j\eta-k\beta}$. Similar control works for $g$. To conclude this, we have
\begin{gather}
    |f_{jk}|\lesssim_{\phi,\psi}a_{jk}\cdot\|f\|_{\Co^\alpha L^\infty\cap\Co^{-\eta} \Co^\beta},\quad|g_{jk}|\lesssim_{\phi,\psi}b_{jk}\cdot\|g\|_{\Co^{\alpha-\theta} L^\infty\cap\Co^{-\eta-\theta} \Co^\beta},\notag
    \\\label{Eqn::Hold::PfMult::AB}
    \text{where }a_{jk}:=\begin{cases}2^{-j\alpha},&
    j\ge\frac{\beta k}{\alpha+\eta}\\2^{j\eta-k\beta},&j\le\frac{\beta k}{\alpha+\eta},\end{cases}\quad b_{jk}:=\begin{cases}2^{-j(\alpha-\theta)},&j\ge\frac{\beta k}{\alpha+\eta}\\2^{j(\eta+\theta)-k\beta},&j\le\frac{\beta k}{\alpha+\eta},\end{cases}\quad\text{for }j,k\ge0.
\end{gather}

For convenience we can normalize $f$ and $g$ so that $\|f\|_{\Co^\alpha L^\infty\cap\Co^{-\eta} \Co^\beta}=\|g\|_{\Co^{\alpha-\theta} L^\infty\cap\Co^{-\eta-\theta} \Co^\beta}=1$.

Note that when $l\ge\frac{\beta q}{\alpha+\eta}$, by \eqref{Eqn::Hold::Mult::CxLsControl} we have $$|(fg)_{lq}|\lesssim 2^{-l(\alpha-\theta)}\|fg\|_{\Co^{\alpha-\theta} \Co^0}\lesssim 2^{-l(\alpha-\theta)}\|f\|_{\Co^\alpha L^\infty\cap\Co^{-\eta} \Co^\beta}\|g\|_{\Co^{\alpha-\theta} L^\infty\cap\Co^{-\eta-\theta} \Co^\beta}=2^{-l(\alpha-\theta)}\le  2^{l(\eta+\theta)-q\beta}.$$
So we only need to show that
\begin{equation}\label{Eqn::Hold::PfMult::FinalGoal}
    \textstyle\sup_{x\in\R^n,s\in\R^q}|(fg)_{lq}(x,s)|\le\sum_{\mu,\nu=1}^3S^{\mu\nu}_{lq}\lesssim_{\phi,\psi,\alpha,\beta,\theta,\eta} 2^{l(\eta+\theta)-q\beta},\qquad\text{for }l\le\tfrac{\beta q}{\alpha+\eta}.
\end{equation}

We prove the control by estimating $S^{\mu\nu}_{lq}$ one by one. We use the following facts for geometric sums,
\begin{equation}\label{Eqn::Hold::PfMult::DyaSumTmp}
    \sum_{j=u}^v2^{j\gamma}\le\begin{cases}C_\gamma 2^{u\gamma},&\gamma<0,\\v-u,&\gamma=0,\\ C_\gamma 2^{v\gamma},&\gamma>0,\end{cases}\quad\Rightarrow\quad\sum_{j=0}^v2^{j\gamma}\le C_\gamma \min(v2^{v\gamma},1),\quad\text{for } u,v\ge0.
\end{equation}

For the first term $S^{11}_{lq}$, by \eqref{Eqn::Hold::PfMult::AB} we have $S^{11}_{lq}\le\sum_{j=l-2}^{l+2}\sum_{j'=0}^{j-3}\sum_{k=q-2}^{q+2}\sum_{k'=0}^{k-3}a_{jk}b_{j'k'}$. By \eqref{Eqn::Hold::PfMult::DyaSumTmp} we know $\sum_{j=l-2}^{l+2}\sum_{k=q-2}^{q+2}a_{jk}\le C_{\alpha,\eta,\theta}\cdot a_{lq}$, so we only need to consider the sums for $j',k'$ and omit $j,k$, as the following
{\small\begin{equation}
    \begin{aligned}
    &\sum_{j=l-2}^{l+2}\sum_{j'=0}^{j-3}\sum_{k=q-2}^{q+2}\sum_{k'=0}^{k-3}a_{jk}b_{j'k'}\lesssim a_{lq}\sum_{j'=0}^l\sum_{k'=0}^qb_{j'k'}\le 2^{l\eta-q\beta}\sum_{j'=0}^l\Big(\sum_{k'\le\frac{\alpha+\eta}\beta j'}2^{-j'(\alpha-\theta)}+\sum_{\frac{\alpha+\eta}\beta j'\le k'\le q}2^{j'(\eta+\theta)-k'\beta}\Big)
    \\&\lesssim2^{l\eta-q\beta}\sum_{j'=0}^l (j'2^{-j'(\alpha-\theta)}+2^{j'(\eta+\theta)-j'(\alpha+\eta)})\lesssim 2^{l\eta-q\beta}\le 2^{l(\eta+\theta)-q\beta}.
\end{aligned}
\end{equation}}

The deductions for the rest of the terms are similar, we sketch them as below. Note that the assumption $\alpha-\eta-\theta>0$ implies $\frac{2\alpha-\theta}{\alpha+\eta}>1$:
{\small\begin{equation}
    \begin{aligned}
    S_{lq}^{12}\lesssim&\sum_{j\le l}\sum_{k'\le q}a_{jq}b_{lk'}\le\sum_{j\le l}2^{j\eta-q\beta}\Big(\sum_{k'\le \frac{\alpha+\eta}\beta l}2^{-l(\alpha-\theta)}+\sum_{\frac{\alpha+\eta}\beta l\le k'\le q}2^{l(\eta+\theta)-k'\beta}\Big)
    \\
    &\lesssim l2^{l\eta-q\beta}(l2^{-l(\alpha-\theta)}+2^{l(\eta+\theta)-l(\alpha+\eta)})=2^{l(\eta+\theta)-q\beta}(l^22^{-l\alpha}+l2^{-l\alpha})\lesssim 2^{l(\eta+\theta)-q\beta}.
    \\
    S_{lq}^{13}\lesssim&\sum_{l'\ge l}\sum_{k'\le q}a_{l'q}b_{l'k'}\le\sum_{l'\le\frac{\beta q}{\alpha+\eta}}2^{l'\eta-q\beta}\Big(\sum_{k'\le\frac{\alpha+\eta}\beta l'}2^{-l'(\alpha-\theta)}+\sum_{\frac{\alpha+\eta}\beta l'\le k'\le q}2^{l'(\eta+\theta)-k'\beta}\Big)+\sum_{l'\ge\frac{\beta q}{\alpha+\eta}}2^{-l'\alpha}\sum_{k'\le q}2^{-l'(\alpha-\theta)}
    \\
    &\lesssim\sum_{l'\le\frac{\beta q}{\alpha+\eta}}2^{l'\eta-q\beta}(l'2^{-l'(\alpha-\theta)}+2^{-l'(\alpha-\theta)})+\sum_{l'\ge\frac{\beta q}{\alpha+\eta}}2^{-l'(2\alpha-\theta)}\lesssim 2^{-q\beta}\sum_{l'\le\frac{\beta q}{\alpha+\eta}}l'2^{-l'(\alpha-\eta-\theta)}+2^{-\frac{2\alpha-\theta}{\alpha+\eta}\beta q}\lesssim 2^{-q\beta}.
    \\
    S_{lq}^{21}\lesssim&\sum_{j'\le l}\sum_{k\le q}a_{lk}b_{j'q}\le\Big(\sum_{k\le\frac{\alpha+\eta}\beta l}2^{-l\alpha}+\sum_{\frac{\alpha+\eta}\beta l\le k\le q}2^{l\eta-k\beta}\Big)\sum_{j'\le l}2^{j'(\eta+\theta)-q\beta}\lesssim (l2^{-l\alpha}+2^{-l\alpha})\cdot l2^{l(\eta+\theta)-q\beta}\lesssim 2^{l(\eta+\theta)-q\beta}.
    \\
    S_{lq}^{22}\lesssim&\sum_{j\le l}\sum_{k\le q}a_{jk}b_{lq}\le\sum_{j,k=0}^\infty a_{jk}\cdot 2^{l(\eta+\theta)-q\beta}=\sum_{k=0}^\infty\Big(\sum_{j\ge\frac{\beta k}{\alpha+\eta}}2^{-j\alpha}+\sum_{j\le\frac{\beta k}{\alpha+\eta}}2^{j\eta-k\beta}\Big) 2^{l(\eta+\theta)-q\beta}\lesssim2^{l(\eta+\theta)-q\beta}.
    \\
    S_{lq}^{23}\lesssim&\sum_{l'\ge l}\sum_{k\le q}a_{l'k}b_{l'q}\le\sum_{l'\le\frac{\beta q}{\alpha+\eta}}\Big(\sum_{k\le\frac{\alpha+\eta}\beta l'}2^{-l'\alpha}+\sum_{\frac{\alpha+\eta}\beta l'\le k\le q}2^{l'\eta-k\beta}\Big)2^{l'(\eta+\theta)-q\beta}+\sum_{l'\ge\frac{\beta q}{\alpha+\eta}}\sum_{k\le q}2^{-l'\alpha} 2^{-l'(\alpha-\theta)}
    \\
    &\lesssim\sum_{l'\le\frac{\beta q}{\alpha+\eta}}(l'2^{-l'\alpha}+2^{-l'\alpha})\cdot 2^{l'(\eta+\theta)-q\beta}+\sum_{l'\ge\frac{\beta q}{\alpha+\eta}}2^{-l'(2\alpha-\theta)}\lesssim 2^{-q\beta}\sum_{l'\le\frac{\beta q}{\alpha+\eta}}l'2^{-l'(\alpha-\eta-\theta)}+2^{-\frac{2\alpha-\theta}{\alpha+\eta}\beta q}\lesssim 2^{-q\beta}.
    \\
    S_{lq}^{31}\lesssim&\sum_{j'\le l}\sum_{q'\ge q}a_{lq'}b_{j'q'}\le\sum_{j'\le l}\sum_{q'\ge q}2^{l\eta-q'\beta}2^{j'(\eta+\theta)-q'\beta}\lesssim l2^{l(2\eta+\theta)}\sum_{q'\ge q}2^{-2q'\beta}\lesssim 2^{l(\eta+\theta)}2^{-q\beta}(l2^{l\eta-q\beta})\overset{l\le\frac{\beta q}{\alpha+\eta}}{\lesssim}2^{l(\eta+\theta)-q\beta}.
    \\
    S_{lq}^{32}\lesssim&\sum_{j\le l}\sum_{q'\ge q}a_{jq'}b_{lq'}\le\sum_{j\le l}\sum_{q'\ge q}2^{j\eta-q'\beta}2^{l(\eta+\theta)-q'\beta}\lesssim l2^{l(2\eta+\theta)}\sum_{q'\ge q}2^{-2q'\beta}\lesssim2^{l(\eta+\theta)}2^{-q\beta}(l2^{l\eta-q\beta})\lesssim 2^{l(\eta+\theta)-q\beta}.
    \\
    S_{lq}^{33}\lesssim&\sum_{l'\ge l}\sum_{q'\ge q}a_{l'q'}b_{l'q'}\le\sum_{q'\ge q}\Big(\sum_{l\le l'\le \frac{\beta q'}{\alpha+\eta}}2^{l'\eta-q'\beta}2^{l'(\eta+\theta)-q'\beta}+\sum_{l'\ge\frac{\beta q'}{\alpha+\eta}}2^{-l'\alpha}2^{-l'(\alpha-\theta)}\Big)
    \\
    &\lesssim\sum_{q'\ge q}(q'2^{q'(\frac{2\eta+\theta}{\alpha+\eta}\beta-2\beta)}+2^{-\frac{2\alpha-\theta}{\alpha+\eta}\beta q'})=\sum_{q'\ge q}(q'2^{-\frac{2\alpha-\theta}{\alpha+\eta}\beta q'}+2^{-\frac{2\alpha-\theta}{\alpha+\eta}\beta q'})\lesssim q2^{-\frac{2\alpha-\theta}{\alpha+\eta}\beta q}\lesssim 2^{-q\beta}\le 2^{l(\eta+\theta)-q\beta}.
\end{aligned}
\end{equation}}

We conclude that $S^{\mu\nu}_{lq}\lesssim 2^{l(\eta+\theta)-q\beta}$ for all non-negatives $l\le\frac{\beta q}{\alpha+\eta}$ and all $1\le \mu,\nu\le 3$. Thus \eqref{Eqn::Hold::PfMult::FinalGoal} holds and we have finished the proof.
\end{proof}

\subsection{Compositions and inverse mapping with parameter}\label{Section::HoldComp}
In this subsection we work on the concept of compositions between parameter dependent maps. We are going to prove an auxiliary result, Proposition \ref{Prop::Hold::CompThm}, that is used in Proposition \ref{Prop::PDE::ExistPDE} \ref{Item::PDE::ExistPDE::Phi}. We recall that for a bijection $f:U_1\to U_2$, we denote its inverse map by $f^\Inv:U_2\to U_1$.

In this subsection, for maps $F(x,s):U\times V\subseteq\R^n\times\R^q\to\R^n$, usually we would denote by $\tilde F$  the map $\tilde F(x,s):=(F(x,s),s)$. Note that if $F(\cdot,s)$ has inverse for every $s$, then in this notation we have $$\tilde F^\Inv(y,s)=(F(\cdot,s)^\Inv(y),s).$$

Let $x=(x^1,\dots,x^n)$ and $y=(y^1,\dots,y^n)$ be two standard coordinate systems of $\R^n$, and let $s=(s^1,\dots,s ^q)$ be the standard coordinate system of $\R ^q$. We let $\Ic(x,s):=x$ for $x\in\R^n$ and $s\in\R ^q$.
\begin{prop}\label{Prop::Hold::CompThm}
Let $\alpha>1$ and $\beta\in(0,\alpha+1]$. Then for any $\eps>0$ there exists a $\delta=\delta(n,q,\alpha,\beta,\eps)>0$ such that the following is true:

Suppose $F\in\Co^{\alpha+1}L^\infty\cap\Co^1\Co^\beta(\B^n,\B ^q;\R^n)$ and $g\in\Co^\alpha L^\infty\cap\Co^0\Co^\beta(\B^n,\B ^q)$ satisfy
\begin{equation}\label{Eqn::Hold::CompThm::Assum}
    (F-\Ic)|_{(\partial\B^n)\times\B^q}\equiv0\quad\text{and}\quad\|F-\Ic\|_{\Co^{\alpha+1}L^\infty\cap\Co^1\Co^\beta(\B^n,\B ^q;\R^n)}+\|g\|_{\Co^\alpha L^\infty\cap\Co^0\Co^\beta(\B^n,\B ^q)}<\delta.
\end{equation}
Then $\tilde F(x,s):=(F(x,s),s)$ is bijective on $\B^n\times\B^q$. 
Moreover, by setting $\tilde\Phi(y,s):=\tilde F^\Inv(y,s)$ for $(y,s)\in\B^n\times\B^q$, we have $g\circ\tilde\Phi\in \Co^\alpha L^\infty\cap\Co^{-1}\Co^\beta(\B^n,\B^q)$ and 
\begin{equation}\label{Eqn::Hold::CompThm::Concl}
    \|g\circ\tilde\Phi\|_{\Co^\alpha L^\infty\cap\Co^{-1}\Co^\beta(\B^n,\B^q)}<\eps.
\end{equation}
\end{prop}
We postpone the proof to the end of this subsection.
\begin{remark}In Proposition \ref{Prop::Hold::CompThm} we may have $g\circ\tilde\Phi\in\Co^\alpha L^\infty \cap\Co^0\Co^\beta(\B^n,\B^q)$.
    But in applications, the Proposition \ref{Prop::PDE::AnalyticPDE}, our result in Proposition \ref{Prop::Hold::CompThm} is enough.
\end{remark}

The proof is a combination of quantitative controls for composition and inverse mapping. Recall the quantitative results for the one-parameter cases:
\begin{lem}\label{Lem::Hold::QCompIFT}
Let $m,n\ge1$ be integers. Let $\alpha>0$.
\begin{enumerate}[parsep=-0.3ex,label=(\roman*)]
    \item\label{Item::Hold::QComp} Let $M>0$ and $\beta>0$, such that $\max(\alpha,\beta)>1$. There is a $K_1=K_1(m,n,\alpha,\beta,M)>0$ satisfying the following:
    
    \smallskip\quad
    Suppose $g\in\Co^{\alpha}(\B^n;\R^m)$ satisfies $g(\B^n)\subseteq \B^m$ and $\|g\|_{\Co^{\alpha}(\B^n;\R^m)}\le M$. Then for every $f\in\Co^\beta(\B^m)$, we have $f\circ g\in \Co^{\min(\alpha,\beta)}(\B^n)$ with $\|f\circ g\|_{\Co^{\min(\alpha,\beta)}(\B^n)}\le K_1\|f\|_{\Co^\beta(\B^m)}$.
    
    \item\label{Item::Hold::QIFT} There is a $K_2(n,\alpha)>0$ satisfying the following:
    
    \smallskip\quad
    Suppose $F\in\Co^{\alpha+1}(\B^n;\R^n)$ satisfy  $\|F-\id\|_{\Co^{\alpha+1}(\B^n;\R^n)}<K_2^{-1}$ and $F|_{\partial\B^n}\equiv \id$. Then $F$ is a $\Co^{\alpha+1}$-diffeomorphism on $\B^n$, and for the inverse map $F^\Inv:\B^n\to \B^n$ we have $\|F^\Inv\|_{\Co^{\alpha+1}(\B^n;\R^n)}<K_2$.
\end{enumerate}
\end{lem}
\begin{proof}For the proof of \ref{Item::Hold::QIFT}, see \cite[Proposition 4.10]{StreetYaoVectorFields}.

For \ref{Item::Hold::QComp}, see \cite[Lemma 5.8]{CoordAdaptedII} for the case $\alpha>1$. We prove the case $0<\alpha\le1<\beta$ below. 

By Lemma \ref{Lem::Hold::HoldChar} \ref{Item::Hold::HoldChar::02} we have $\| f\circ g\|_{\Co^\alpha(\B^n)}\approx_{n,\alpha}\| f\circ g\|_{C^0(\B^n)}+\sup\limits_{x_0\neq x_1}|\tfrac{ f( g(x_0))+ f( g(x_1))}2- f( g(\tfrac{x_0+x_1}2))|\cdot|x_0-x_1|^{-\alpha}$ and $|\tfrac{ f( g(x_0))+ f( g(x_1))}2- f(\tfrac{ g(x_0)+ g(x_1)}2)|\lesssim_{m,\beta}\| f\|_{\Co^{\min(\beta,\frac32)}}|g(x_0)-g(x_1)|^{\min(\beta,\frac32)}$ for $x_0,x_1\in\B^n$. Therefore
\begin{align*}
    &|\tfrac{ f( g(x_0))+ f( g(x_1))}2- f( g(\tfrac{x_0+x_1}2))|\le|\tfrac{ f( g(x_0))+ f( g(x_1))}2- f(\tfrac{ g(x_0)+ g(x_1)}2)|+| f(\tfrac{ g(x_0)+ g(x_1)}2)- f( g(\tfrac{x_0+x_1}2))|
    \\
    \lesssim&_{m,\beta}\| f\|_{\Co^{\min(\beta,\frac32)}}| g(x_0)- g(x_1)|^{\min(\beta,\frac32)}+\| f\|_{C^{0,1}}|\tfrac{ g(x_0)+ g(x_1)}2- g(\tfrac{x_0+x_1}2)|
    \\
    \lesssim&_{n,\alpha}\| f\|_{\Co^{\min(\beta,\frac32)}}\| g\|_{\Co^{\alpha/\min(\beta,\frac32)}}^{\min(\beta,\frac32)}|x_0-x_1|^{\alpha}+\| f\|_{C^{0,1}}\| g\|_{\Co^\alpha}|x_0-x_1|^\beta\lesssim_M\| f\|_{\Co^{\beta}}|x_0-x_1|^{\alpha}.
\end{align*}

In particular we can choose $K_1(m,n,\alpha,\beta,M)\approx_{m,n,\alpha,\beta} (1+M)^{\min(\beta,\frac32)}$.
\end{proof}

\begin{lem}\label{Lem::Hold::QPIFT}
Let $n,q\in\Z_+$, $\alpha>0$ and $\beta\in(0,\alpha+1]$. Write $\Ic:\R^n\times\R^q\to\R^n$ as $\Ic(x,s):=x$. There is a $K_3=K_3(n,q,\alpha,\beta)>1$ that satisfies the following:

\medskip\quad Let $F\in\Co^{\alpha+1,\beta}(\B^n,\B^q;\R^n)$ satisfies $\|F-\Ic\|_{\Co^{\alpha+1,\beta}(\B^n,\B^q;\R^n)}<K_3^{-1}$ and $(F-\Ic)|_{(\partial\B^n)\times\B^q}\equiv0$. Then \begin{enumerate}[nolistsep,label=(\roman*)]
    \item\label{Item::Hold::QPIFT::Fbij} For each $s\in\B^q$, $F(\cdot,s):\B^n\to\B^n$ is bijective.
    \item\label{Item::Hold::QPIFT::Control} Define $\Phi(y,s):=F(\cdot,s)^\Inv(y)$ for $y\in\B^n$, $s\in\B^q$. Then $\Phi\in\Co^{\alpha+1,\beta}(\B^n,\B^q;\R^n)$ and satisfies 
    \begin{equation}\label{Eqn::Hold::QPIFT::Control}
        \|\Phi\|_{\Co^{\alpha+1,\beta}(\B^n,\B^q;\R^n)}\le K_3.
    \end{equation}
\end{enumerate}
\end{lem}


\begin{proof}We let $K_3$ be a large constant which may change from line to line.

Let $K_2=K_2(n,\alpha)>0$ be the constant in Lemma \ref{Lem::Hold::QCompIFT} \ref{Item::Hold::QIFT}. So we can choose $C_1>0$ such that 
\begin{equation*}
    \|F-\Ic\|_{\Co^{\alpha+1,\beta}(\B^n,\B^q;\R^n)}<C_1^{-1}\quad\Rightarrow\quad\|F(\cdot,s)-\id_{\R^n}\|_{\Co^{\alpha+1}(\B^n;\R^n)}<K_2^{-1},\ \forall s\in\B^q.
\end{equation*}
By assumption $(F(\cdot,s)-\Ic(\cdot,s))|_{\partial\B^n}=(F(\cdot,s)-\id_{\R^n})|_{\partial\B^n}\equiv0$, so by Lemma \ref{Lem::Hold::QCompIFT} \ref{Item::Hold::QIFT} $F(\cdot,s):\B^n\to\B^n$ is a $\Co^{\alpha+1}$-diffeomorphism for each $s\in\B^n$ and we have 
\begin{equation}\label{Eqn::Hold::PfQPIFT::C1}
    \textstyle\|F-\Ic\|_{\Co^{\alpha+1,\beta}(\B^n,\B^q;\R^n)}<C_1^{-1}\quad\Rightarrow\quad\sup_{s\in\B^q}\|F(\cdot,s)^\Inv\|_{\Co^{\alpha+1}(\B^n;\R^n)}<K_2.
\end{equation} In particular $\Phi(y,s)$ is pointwise defined for $(y,s)\in\B^n\times\B^q$. This finishes the proof of \ref{Item::Hold::QPIFT::Fbij}.

\medskip
Let $c'_0=c'_0(\B^n,n,\alpha)\in(0,1)$ be the constant in Remark \ref{Rmk::Hold::CramerMixedforSing}, we have, for every $s\in\B^q$,
\begin{equation}\label{Eqn::Hold::PfQPIFT::InvMat}
    \|\nabla_xF(\cdot,s)-I_n\|_{\Co^\alpha(\B^n;\R^{n\times n})}<c'_0\ \Rightarrow\ \|(\nabla_x F(\cdot,s))^{-1}-I_n\|_{\Co^\alpha(\B^n;\R^{n\times n})}<2c'_0.
\end{equation}

Thus we can find a $C_2\ge\max(2c'_0{}^{-1},C_1,K_2,1)$, such that
\begin{equation}\label{Eqn::Hold::PfQPIFT::C2}
    \|F-\Ic\|_{\Co^{\alpha+1,\beta}(\B^n,\B^q;\R^n)}<C_2^{-1}\ \Rightarrow\ \sup_{x\in\B^n,s\in\B^q}|(\nabla_x F(x,s))^{-1}|\le\sup_{s\in\B^q}\|(\nabla_x F(\cdot,s))^{-1}\|_{\Co^\alpha(\B^n;\R^{n\times n})}<C_2.
\end{equation}

In particular $\inf_{x\in\B^n,s\in\B^q}|\det(\nabla_xF(x,s))^{-1}|>C_2^{-n}$. Thus when $0<\beta<\min(\frac{2\alpha+1}{\alpha+1},\frac32)$ (in particular when $0<\beta\le\frac{3\alpha+1}{2\alpha+1}$), we get $\|\Phi\|_{\Co^{\alpha+1,\beta}}\lesssim_{n,q,\alpha,\beta}1$ by applying Proposition \ref{Prop::AppQPIFT}.

It remains to prove the case $\beta>\frac{3\alpha+1}{2\alpha+1}$. By chain rule, we have
\begin{equation*}
    0_{q\times n}=\nabla_s(F(\Phi(y,s),s))=\big((\nabla_xF)(\Phi(y,s),s)\big)\cdot\nabla_s\Phi(y,s)+(\nabla_sF)(\Phi(y,s),s).
\end{equation*}
Therefore
\begin{equation}
    \label{Eqn::Hold::PfQPIFT::Chain}
    \nabla_s\Phi(y,s)=-((\nabla_xF)^{-1}\cdot\nabla_sF)\circ(\Phi(y,s),s).
\end{equation}
What we need is to show $\|\nabla_s\Phi(y,\cdot)\|_{\Co^{\beta-1}(\B^q;\R^{q\times n})}\lesssim_{n,q,\alpha,\beta}1$. We will proceed by induction on $k=\lceil\beta-\frac{3\alpha+1}{2\alpha+1}\rceil$. The base case $k=1$ has done since we have obtained $K_3(n,q,\alpha,\beta)>0$ for $\beta\le\frac{3\alpha+1}{2\alpha+1}$ from the above discussion.

Suppose the case $k\ge1$ is done. Let $\beta$ satisfies $\lceil\beta-\frac{3\alpha+1}{2\alpha+1}\rceil=k+1$. We denote $\gamma(\beta):=\max(\beta-1,\frac{3\alpha+1}{2\alpha+1})>1$. By induction hypothesis and $\|\Phi\|_{\Co^{\gamma(\beta)}_{(x,s)}}\lesssim \|\Phi\|_{\Co^{\alpha+1,\gamma(\beta)}_{x,s}}$  there is a $C_3=C_3(n,q,\alpha,\beta)>0$ such that
\begin{equation}\label{Eqn::Hold::PfQPIFT::Assume}
    \|F-\Ic\|_{\Co^{\alpha+1,\gamma(\beta)}}\le K_3(n,q,\alpha,\gamma(\beta))^{-1}\quad\Rightarrow\quad\|\Phi\|_{\Co^{\gamma(\beta)}(\B^n\times\B^q;\R^n)}\le C_3.
\end{equation}

 Applying Lemma \ref{Lem::Hold::QCompIFT} \ref{Item::Hold::QComp} to \eqref{Eqn::Hold::PfQPIFT::Chain}, taking $f$ to be the components of $(\nabla_xF)^{-1}\cdot\nabla_sF\in\Co^{\beta-1}(\B^n\times\B^q;\R^{q\times n})$, we have, when the assumption of \eqref{Eqn::Hold::CompThm::Assum} is satisfied,
\begin{equation}\label{Eqn::Hold::PfQPIFT::Induction}
\begin{aligned}
\|\nabla_s\Phi\|_{\Co^{\beta-1}(\B^n,\B^q)}
\lesssim&\sup_{x,s}\|((\nabla_xF)^{-1}\cdot\nabla_sF)(\Phi(y,\cdot),s)\|_{\Co^{\beta-1}(\B^q)}+\sup_{x,s}\|((\nabla_xF)^{-1}\cdot\nabla_sF)(\Phi(y,s),\cdot)\|_{\Co^{\beta-1}(\B^q)}
\\
\le&K_1(n,q,\gamma(\beta),\beta-1,C_3)\sup_{s\in\B^q}\|(\nabla_xF)^{-1}\nabla_sF(\cdot,s)\|_{\Co^{\beta-1}(\B^n)}
\\
&+K_1(q,q,\gamma(\beta),\beta-1,\|\id_{\B^q}\|_{\Co^{\gamma(\beta)}})\sup_{x\in\B^n}\|(\nabla_xF)^{-1}\nabla_sF(x,\cdot)\|_{\Co^{\beta-1}(\B^q)}.
\end{aligned}
\end{equation}

Similar to the argument of \eqref{Eqn::Hold::PfQPIFT::C2}, by considering the constants $c'_0(\B^n,n,\beta-1)$ and $c'_0(\B^q,n,\beta-1)$ in Remark \ref{Rmk::Hold::CramerMixedforSing}, we can find a $C_4=C_4(n,q,\alpha,\beta)>\max(C_2,C_3,K_3(n,q,\alpha,\gamma(\beta)))$, such that
\begin{equation*}
    \|F-\Ic\|_{\Co^{\alpha+1,\beta}(\B^n,\B^q;\R^n)}<C_4^{-1}\ \Rightarrow\ \sup_{s\in\B^q}\|(\nabla_xF(\cdot,s))^{-1}\|_{\Co^{\beta-1}(\B^n;\R^{n\times n})}+\sup_{x\in\B^n}\|(\nabla_xF(x,\cdot))^{-1}\|_{\Co^{\beta-1}(\B^q;\R^{n\times n})}<C_4.
\end{equation*}
Thus $\|(\nabla_xF)^{-1}\nabla_sF\|_{\Co^{\beta-1}(\B^n\times\B^q;\R^{q\times n})}\lesssim_{n,q,\alpha,\beta}1$. Combining this with \eqref{Eqn::Hold::PfQPIFT::Induction} we see that $\sup\limits_{y\in\B^n}\|\Phi(y,\cdot)\|_{\Co^\beta(\B^q;\R^n)}\lesssim_{n,q,\alpha,\beta}1$. Thus we get a $K_3(n,q,\alpha,\beta)$ for such $\beta$, finishing the induction and hence the whole proof.
\end{proof}

To prove Proposition \ref{Prop::Hold::CompThm} we also need an auxiliary result for Laplacian.

\begin{lem}\label{Lem::Hold::LapInvBdd}
Let $\alpha>-1$. Let $U\subset\R^n_x$ and $V\subset\R^q_s$ be two bounded convex open sets with smooth boundaries. Then the $x$-variable Laplacian $\Delta_x=\sum_{j=1}^n\partial_{x^j}^2$ has a bounded linear inverse $\Pc$ as the following map
$$\Delta_x:\{f\in\Co^{\alpha+1}_x\Xs_s(U,V):f|_{(\partial U)\times V}=0\}\to\Co^{\alpha-1}_x\Xs_s(U,V),\quad\Xs\in\{L^\infty,\Co^\beta:\beta\in\R\}.$$
In particular $\Pc:\Co^{\alpha-1}\Xs(U,V)\to\Co^{\alpha+1}\Xs(U,V)$ is bounded linear.
\end{lem}
\begin{proof}
Indeed, for any $v\in\Co^{(-2)+}(U)$ there is a unique continuous $u\in C^0(\overline{U})$ satisfying $\Delta u=v$ and $u\big|_{\partial U}=0$. Moreover the solution operator $[v\mapsto u]:\Co^{\mu-1}(U)\to\Co^{\mu+1}(U)$ is bounded linear for all $\mu>-1$. See \cite[Theorem 15]{DirichletBoundedness}. One can also see \cite[Theorems 6.19 and 8.33]{GilbargTrudinger} for the cases where $\mu$ is positive non-integer.

We can denote the above solution operator as $\Pc_0$. So $\Pc_0:\Co^{\alpha-1}(U)\to\Co^{\alpha+1}(U)$ is bounded linear. Therefore by Lemma \ref{Lem::Hold::TOtimesId}, $\Pc f(x,s):=\Pc_0(f(\cdot,s))(x)$ gives the bounded operator $\Pc:\Co^{\alpha-1}\Xs(U,V)\to \Co^{\alpha+1}\Xs(U,V)$.
\end{proof}

\begin{proof}[Proof of Proposition \ref{Prop::Hold::CompThm}]We have $\|F-\Ic\|_{\Co^{\alpha+1,\beta}(\B^n,\B^q;\R^n)}\lesssim\|F-\Ic\|_{\Co^{\alpha+1}L^\infty\cap\Co^1\Co^\beta(\B^n,\B^q;\R^n)}$ by Remark \ref{Rmk::Hold::SimpleHoldbyCompFact}. Therefore there is a $\delta_1>0$ such that if \eqref{Eqn::Hold::CompThm::Assum} is satisfied with $\delta=\delta_1$, then $\|F-\Ic\|_{\Co^{\alpha+1,\beta}(\B^n,\B^q;\R^n)}\le K_3^{-1}$, wh ere $K_3=K_3(n,q,\alpha,\beta)>0$ is the constant in Lemma \ref{Lem::Hold::QPIFT}. Thus by Lemma \ref{Lem::Hold::QPIFT} \ref{Item::Hold::QPIFT::Fbij} $\Phi(y,s)=F(\cdot,s)^\Inv(y)$ is well-defined and $\|\Phi\|_{\Co^{\alpha+1,\beta}}\le K_3$ holds.

Applying Lemma \ref{Lem::Hold::QCompIFT} \ref{Item::Hold::QComp} along with its constant $K_1$, we have for every $f\in\Co^{\alpha+1,\beta}(\B^n,\B^q)$,
\begin{equation*}
\begin{aligned}
    \textstyle\sup_{s\in\B^q}\|f\circ(\Phi(\cdot,s),s)\|_{\Co^{\alpha+1}(\B^n)}&
    \textstyle\le K_1(n,n,\alpha+1,\alpha+1,K_3)\cdot\sup_{s\in\B^q}\|f(\cdot,s)\|_{\Co^{\alpha+1}(\B^n)};
    \\
    \textstyle\sup_{y\in\B^n,s\in\B^q}\|f\circ(\Phi(y,\cdot),s)\|_{\Co^\beta(\B^n)}&
    \textstyle\le K_1(n,q,\alpha+1,\beta,K_3)\cdot\sup_{s\in\B^q}\|f(\cdot,s)\|_{\Co^{\alpha+1}(\B^n)};
    \\
    \textstyle\sup_{y\in\B^n,s\in\B^q}\|f\circ(\Phi(y,s),\cdot)\|_{\Co^\beta(\B^n)}&
    \textstyle\le K_1(q,q,\alpha+1,\beta,\|\id_{\B^q}\|_{\Co^{\alpha+1}})\cdot\sup_{x\in\B^n}\|f(x,\cdot)\|_{\Co^\beta(\B^n)}.
\end{aligned}
\end{equation*}
Therefore we can find a $C_1=C_1(n,q,\alpha,\beta)>0$ such that
\begin{equation}\label{Eqn::Hold::PfCompThm::EstComp2}
    \|f\circ\tilde\Phi\|_{\Co^{\alpha+1,\beta}(\B^n,\B^q)}\le C_1\|f\|_{\Co^{\alpha+1,\beta}(\B^n,\B^q)},\quad\forall  f\in\Co^{\alpha+1,\beta}(\B^n,\B^q).
\end{equation}
In particular by taking $f$ to be the components of $F-\Ic$, we have $\|\Ic-\Phi\|_{\Co^{\alpha+1,\beta}(\B^n,\B^q;\R^n)}\lesssim\|F-\Ic\|_{\Co^{\alpha+1,\beta}(\B^n,\B^q;\R^n)}$. By $\Co^{\alpha+1,\beta}\subset\Co^{\alpha+1}L^\infty\cap\Co^0\Co^\beta$ in Remark \ref{Rmk::Hold::SimpleHoldbyCompFact}  and that $\nabla_y\Ic(y,s)=I_n$, we know $\|\nabla_y\Phi-I_n\|_{\Co^\alpha_y L^\infty_s\cap\Co^{-1}_y\Co^\beta_s}\lesssim\|F-\Ic\|_{\Co^{\alpha+1,\beta}_{x,s}}$. Thus, we can find a $C_2=C_2(n,q,\alpha,\beta)>0$ such that
\begin{equation*}
    \|\nabla_y\Phi-I_n\|_{\Co^\alpha L^\infty\cap\Co^{-1}\Co^\beta(\B^n,\B^q;\R^{n\times n})}\le C_2\|F-\Ic\|_{\Co^{\alpha+1,\beta}(\B^n,\B^q;\R^n)},\text{ when }\eqref{Eqn::Hold::CompThm::Assum}\text{ is satisfied with }\delta=\delta_1.
\end{equation*}

By Corollary \ref{Cor::Hold::CramerMixed} (with $\eta=1$), we see that if $C_2\|F-\Ic\|_{\Co^{\alpha+1,\beta}(\B^n,\B^q;\R^n)}$ is small enough, then $(\nabla_y\Phi)^{-1}\in\Co^\alpha L^\infty\cap\Co^{-1}\Co^\beta(\B^n,\B^q;\R^{n\times n})$. Therefore we can find a $\delta_2\in(0,\delta_1)$ such that
\begin{equation}\label{Eqn::Hold::PfCompThm::EstInvMat}
    \|(\nabla_y\Phi)^{-1}\|_{\Co^\alpha L^\infty\cap\Co^{-1}\Co^\beta(\B^n,\B^q;\R^{n\times n})}\le\delta_2^{-1},\text{ when }\eqref{Eqn::Hold::CompThm::Assum}\text{ is satisfied with }\delta=\delta_2.
\end{equation}

Now for $g\in\Co^\alpha L^\infty\cap\Co^0\Co^\beta(\B^n,\B^q)$, we take $G_j=\partial_{x^j}\Pc g$ where $\Pc:\Co^\alpha L^\infty\cap\Co^0\Co^\beta(\B^n,\B^q)\to\Co^{\alpha+2} L^\infty\cap\Co^2\Co^\beta(\B^n,\B^q)$ is the operator in Lemma \ref{Lem::Hold::LapInvBdd}. Thus $G=(G_1,\dots,G_n)\in \Co^{\alpha+1} L^\infty\cap\Co^1\Co^\beta(\B^n,\B^q;\R^n)$ and
\begin{equation*}
    g\circ\tilde\Phi=(\divg_xG)\circ\tilde\Phi,\qquad(\nabla_x G)\circ\tilde\Phi=\nabla_y(G\circ\tilde\Phi)\cdot(\nabla_y\Phi)^{-1}.
\end{equation*}

By Corollary \ref{Cor::Hold::CorMult} \ref{Item::Hold::CorMult::-1} we have
\begin{align*}
    \|g\circ\tilde\Phi\|_{\Co^\alpha L^\infty\cap\Co^{-1}\Co^\beta(\B^n,\B^q)}&\lesssim\|\nabla_y(G\circ\tilde\Phi)\|_{\Co^\alpha L^\infty\cap\Co^{-1}\Co^\beta(\B^n,\B^q;\R^{n\times n})}\|(\nabla_y\Phi)^{-1}\|_{\Co^\alpha L^\infty\cap\Co^{-1}\Co^\beta(\B^n,\B^q;\R^{n\times n})}.
\end{align*}

Therefore by \eqref{Eqn::Hold::PfCompThm::EstInvMat}, \eqref{Eqn::Hold::PfCompThm::EstComp2} and Remark \ref{Rmk::Hold::SimpleHoldbyCompFact}, we have: when \eqref{Eqn::Hold::CompThm::Assum} is satisfied with $\delta=\delta_2$,
\begin{align*}
    \|g\circ\tilde\Phi\|_{\Co^\alpha L^\infty\cap\Co^{-1}\Co^\beta(\B^n,\B^q)}&\lesssim\|G\circ\tilde\Phi\|_{\Co^\alpha L^\infty\cap\Co^0\Co^\beta(\B^n,\B^q;\R^n)}\lesssim\|G\circ\tilde\Phi\|_{\Co^{\alpha+1,\beta}(\B^n,\B^q;\R^n)}
    \\
    &\lesssim\|G\|_{\Co^{\alpha+1,\beta}(\B^n,\B^q;\R^n)}\lesssim\|G\|_{\Co^{\alpha+1} L^\infty\cap\Co^1\Co^\beta(\B^n,\B^q;\R^n)}\lesssim\|g\|_{\Co^{\alpha} L^\infty\cap\Co^0\Co^\beta(\B^n,\B^q)}.
\end{align*}

In other words, there is a $C_3=C_3(n,q,\alpha,\beta)>0$, such that
\begin{equation*}
    \|g\circ\tilde\Phi\|_{\Co^\alpha L^\infty\cap\Co^{-1}\Co^\beta(\B^n,\B^q)}\le C_3\|g\|_{\Co^{\alpha} L^\infty\cap\Co^0\Co^\beta(\B^n,\B^q)},\quad\forall\ g\in \Co^{\alpha} L^\infty\cap\Co^0\Co^\beta(\B^n,\B^q).
\end{equation*}
Now for every $\eps>0$, take $\delta=\min(C_3^{-1}\eps,\delta_2)$, we get \eqref{Eqn::Hold::CompThm::Concl} and complete the proof.
\end{proof}

In applications, we also need a local version of the composition estimate.
\begin{lem}\label{Lem::Hold::CompofMixHold}
Let $\alpha,\beta,\gamma,\delta>0$. Let $U_1\subseteq\R^n_x$, $V_1\subseteq\R^p_t$, $U_2\subseteq\R^q_y$ and $V_2\subseteq\R^r_s$ be four open sets. 
\begin{enumerate}[parsep=-0.3ex,label=(\roman*)]
    \item\label{Item::Hold::CompofMixHold::Comp} Suppose $\max(\alpha,\gamma)>1$ and $\max(\beta,\gamma)>1$. Let $g\in\Co^{\alpha,\beta}_\loc(U_1, V_1;U_2)$ and $f\in\Co^{\gamma,\delta}_\loc(U_2, V_2)$. We define $h:U_1\times V_1\times V_2\to\R $ as $h(x,t,s):=f(g(x,t),s)$. Then $h\in\Co^{\min(\alpha,\gamma),\min(\beta,\gamma),\delta}_{\loc}(U_1,V_1, V_2)$.
    
    \item\label{Item::Hold::CompofMixHold::InvMat} Let $A\in\Co^{\alpha,\beta}_{\loc}(U_1,V_1;\C^{m\times m})$ be a matrix function such that $A(x,s)\in\R^{n\times n}$ is invertible for every $(x,s)$. Then $A^{-1}\in\Co^{\alpha,\beta}_{\loc}(U_1,V_1;\C^{m\times m})$ as well.
    \item\label{Item::Hold::CompofMixHold::InvFun} Let $F\in\Co^{\alpha+1,\beta}_{\loc}(U_1,V_1;\R^n)$ and $(x_0,t_0)\in U_1\times V_1$. Suppose $\nabla_xF(x_0,t_0)\in \R^{n\times n}$ is an invertible matrix. Then there are neighborhoods $U_1'\subseteq U_1$ of $x_0$, $V_1'\subseteq V_1$ of $t_0$ and $\Omega_1\subseteq\R^n$ of $F(x_0)$ such that
    \begin{itemize}[nolistsep]
        \item For each $t\in V'_1$, $F(\cdot,t):U'_1\to F(U'_1,t)\subseteq\R^n$ is bijective and $F(U'_1,t)\supseteq \Omega_1$.
        \item Set $\Phi(x,t):=F(\cdot,t)^\Inv(x)$, then $\Phi\in\Co^{\alpha+1,\min(\alpha+1,\beta)}_{\loc}(\Omega_1,V'_1;U'_1)$.
    \end{itemize}
    In particular for the map $\tilde F(x,t):=(F(x,t),t))$ we have that $\tilde F:U_1'\times V_1'\to \tilde F(U'_1\times V'_1)$ is homeomorphism.
\end{enumerate}
\end{lem}

\begin{proof}
By passing to local and taking scaling, we can assume $U_1,V_1,U_2,V_2$ are all unit balls, and every maps have finite norms. 

For \ref{Item::Hold::CompofMixHold::Comp}, what we need to show is that $\sup_{t\in V_1,s\in V_2}\|h(\cdot,t,s)\|_{\Co^{\min(\alpha+1,\gamma)}(U_1)}$, $\sup_{x\in U_1,s\in V_2}\|h(x,\cdot,s)\|_{\Co^{\min(\beta,\gamma)}(V_1)}$ and $\sup_{x\in U_1,t\in V_1}\|h(x,y,\cdot)\|_{\Co^{\delta}(V_2)}$ are all finite.

Clearly $h\in\Co^\delta_sL^\infty_{x,t}$ because $\sup_{x,t}\|h(x,t,\cdot)\|_{\Co^\delta(V_2)}\le\sup_{y\in U_2}\|f(y,\cdot)\|_{\Co^\delta(V_2)}$.

By Lemma \ref{Lem::Hold::QCompIFT} \ref{Item::Hold::QComp} (since we assume $V_1$ and $U_2$ are balls) we have a $C>0$ that depends on $U_2,V_1,\beta,\gamma$ and the upper bound of $\sup_{s'\in V_2}\|g(\cdot,s')\|_{\Co^{\alpha}(U_1)}$ but not on $t\in V_1$ or $s\in V_2$, such that
$$\|h(\cdot,t,s)\|_{\Co^{\min(\alpha,\gamma)}(U_1)}\le C\|f(\cdot,t)\|_{\Co^\gamma(U_2)}.$$
Taking supremum over $t\in V_1$ and $s\in V_2$ we get $\sup_{t\in V_1,s\in V_2}\|h(\cdot,t,s)\|_{\Co^{\min(\alpha,\gamma)}(U_1)}<\infty$.

Switching the role of $x$ and $t$ we get $\sup\limits_{x\in U_1,s\in V_2}\|h(x,\cdot,s)\|_{\Co^{\min(\beta,\gamma)}(V_2)}<\infty$, finishing the proof of \ref{Item::Hold::CompofMixHold::Comp}.

\smallskip For \ref{Item::Hold::CompofMixHold::InvMat}, by assumption $\det A:U_1\times V_1\to\C$ is non-vanishing in the domain. By continuity and by shrinking $U_1,V_1$ we can assume that $\inf_{U\times V}|\det A|>c$ for some $c>0$. By \cite[Lemma 5.7]{CoordAdaptedII} (where we take $X=\nabla$ in the reference), we get $\sup_x\|\det A(x,\cdot)^{-1}\|_{\Co^\beta_s}<\infty$ and $\sup_s\|\det A(\cdot,s)^{-1}\|_{\Co^\alpha_x}<\infty$. So $(\det A)^{-1}\in\Co^{\alpha,\beta}_{x,s}$.

By Cramer's rule, the entries of $A^{-1}$ are the products of $(\det A)^{-1}$ and the entries of $A$, thus $A^{-1}\in\Co^{\alpha,\beta}_{x,s}$.

\smallskip For \ref{Item::Hold::CompofMixHold::InvFun}, apply the Inverse Function Theorem on $F(\cdot,t_0)$, we can find a neighborhood $\tilde U_1\subseteq U_1$ of $x_0$ such that $F(\cdot,t_0):\tilde U_1\to F(\tilde U_1,t_0)$ is a $\Co^{\alpha+1}$-diffeomorphism. Shrinking $\tilde U_1$ if necessary, we can assume $F(\cdot,t_0),F(\cdot,t_0)^\Inv\in\Co^{\alpha+1}(\tilde U_1;\R^n)$ both have bounded $C^1$-norm.

By \eqref{Eqn::Hold::RmkforBiHold::Interpo} we have $F\in\Co^{\alpha+1}\Co^0\cap\Co^0\Co^\beta\subset\Co^{\frac\alpha2+1}_x\Co^{\frac{\alpha\beta}{2(\alpha+1)}}_s$, thus $[t\mapsto F(\cdot,t)]:V_1\to C^1(\tilde U_1;\R^n)$ is a continuous map (in fact H\"older) with respect to the $C^1$-norm topology. Since the set of bounded $C^1$-embedding map in $C^1(\tilde U_1;\R^n)$ is open, we can now find a neighborhood $\tilde V_1\subseteq V_1$ of $t_0$ such that $F(\cdot,t)$ is $C^1$-embedding for all $t\in\tilde V_1$. By Inverse Function Theorem again, the inverse map of $F(\cdot,t):\tilde U_1\to F(\tilde U_1,t)$ is also $\Co^{\alpha+1}$.

Take $U'_1:=\tilde U_1$ and take $\Omega_1$ such that $F(x_0,t_0)\in\Omega_1\Subset F(U'_1,t_0)$. By continuity there is a neighborhood $V'_1\subseteq\tilde V_1$ of $t_0$ such that $\Omega_1\Subset F(U'_1,t) $ for all $t\in V'_1$ as well. Now we get $U'_1,V'_1,\Omega_1$ as desired and $\Phi:\Omega_1\times V'_1\to U'_1$ is pointwise defined.

Since $F(\cdot,t)$ is $C^1$-embedding which is homeomorphic to its image, the result $\sup_{t\in V'_1}\|\Phi(\cdot,t)\|_{\Co^{\alpha+1}(\Omega_1;\R^n)}<\infty$ then follows from \cite[Lemma 5.9]{CoordAdaptedII}.

When $\beta\le1$, the result $\sup_{x\in \Omega_1}\|\Phi(x,\cdot)\|_{\Co^\beta(V_1;\R^n)}<\infty$ follows from Proposition \ref{Prop::AppQPIFT} via a scaling. When $\beta>1$, applying Inverse Function Theorem to the $\Co^{\min(\alpha+1,\beta)}$-map $[(x,t)\in U'_1\times V'_1\mapsto(F(x,t),t)]$, we see that its inverse $[(x,t)\in \Omega_1\times V'_1\mapsto(\Phi(x,t),t)]$ is also $\Co^{\min(\alpha+1,\beta)}$, which means $\sup_{x\in \Omega_1}\|\Phi(x,\cdot)\|_{\Co^{\min(\alpha+1,\beta)}(V_1;\R^n)}<\infty$.

Therefore either case we have $\Phi\in\Co^{\alpha+1}_xL^\infty_t\cap L^\infty_x\Co^{\min(\alpha+1,\beta)}_t=\Co^{\alpha+1,\min(\alpha+1,\beta)}_{x,t}$.

Now $\tilde F$ and $\tilde F^\Inv$ are both $\Co^{\alpha+1,\beta}\subset C^0$ so $\tilde F$ is homeomorphic to its image.
\end{proof}

\section{The Vector Fields and Ordinary Differential Equations}\label{Section::SecODE}
\subsection{Review of tangent subbundles and involutivity}\label{SectionConvDiffGeom}
In this part we give some conventions of notations and terminologies. 

For $\kappa>1$, we can define $\Co^\alpha$-maps between two $\Co^\kappa$-manifolds, since we have compositions $\Co^\alpha_\loc\circ\Co^\kappa_\loc\subseteq\Co^\alpha_\loc$ and $\Co^\kappa_\loc\circ\Co^\alpha_\loc\subseteq\Co^\alpha_\loc$.
\begin{note}
    Let $\kappa>1$, $\alpha\in(0,\kappa]$, and let $\Mf,\Nf$ be two $\Co^\kappa$-differential manifolds.
    We use $\Co^\alpha_\loc(\Mf;\Nf)$ as the set of all (locally) $\Co^\alpha$-maps from $\Mf$ to $\Nf$.
\end{note}

We can define the mixed regularity $\Co^{\alpha,\beta}$ on product manifolds.

\begin{defn}\label{Defn::ODE::MixHoldMaps}
Let $\kappa>1$ and $\alpha,\beta\in(0,\kappa]$. Let $\Mf^m,\Nf^n,\Pf^q$ be three $\Co^\kappa$-manifolds.

We say a continuous map $f:\Mf\times\Nf\to\Pf$ is $\Co^{\alpha,\beta}_{\loc}$, denoted by $f\in\Co^{\alpha,\beta}_{\loc}(\Mf,\Nf;\Pf)$, if for every coordinate charts $\phi:U\subseteq\Mf\to\R^m$, $\psi:V\subseteq\Nf\to\R^n$ and $\rho:W\subseteq\Pf\to\R^q$ that satisfy $f(U\times V)\subseteq W $, we have 
\begin{equation}\label{Eqn::ODE::MixHoldMaps}
    \rho\circ f\circ(\phi^\Inv,\psi^\Inv)\in\Co^{\alpha,\beta}_{\loc}(\phi(U), \psi(V);\R^q).
\end{equation}

\end{defn}

\begin{remark}
\begin{enumerate}[parsep=-0.3ex,label=(\roman*)]
    \item The definition still has symmetry for indices: $\Co^{\alpha,\beta}_\loc(\Mf,\Nf;\Pf)=\Co^{\beta,\alpha}_\loc(\Nf,\Mf;\Pf)$.
    \item To check $f\in\Co^{\alpha,\beta}_\loc$ it suffices to show \eqref{Eqn::ODE::MixHoldMaps} is valid on a coordinate cover $\{(\phi_i,\psi_i,\rho_i)\}_{i\in I}$ for $\Mf\times\Nf\times f(\Mf\times\Nf)\subseteq\Mf\times\Nf\times\Pf$ such that $f(U_i\times V_i)\subseteq W_i$ for each $i\in I$. The proof can be done using Lemma \ref{Lem::Hold::CompofMixHold} \ref{Item::Hold::CompofMixHold::Comp}. Such cover always exists because of the continuity of $f$. We leave the proof to readers.
    \item In applications we only consider differentiable maps localized near a fixed point. By passing to a precompact coordinate chart near a point, everything is reduced to the case of Euclidean space, and by shrinking the domains every functions or vector fields would have bounded H\"older-Zygmund regularities.
\end{enumerate}
\end{remark}
Definition \ref{Defn::ODE::MixHoldMaps} is applied for labeling component-wise regularity for vector fields  and parameterizations with parameters.
\begin{defn}
    Let $1<\alpha\le\kappa$ and let $\Mf$ be a $m$-dimensional $\Co^\kappa$-manifold. We say $\Phi:\Omega\subseteq\R^m\to \Mf$ is a \textbf{$\Co^\alpha$-regular parameterization}, if $\Phi$ is injective and $\Phi^\Inv:\Phi(\Omega)\subseteq \Mf\to\R^m$ is a $\Co^\alpha$-coordinate chart.
\end{defn}

When $\kappa>2$, the (real) tangent bundle $T\Mf$ of a $\Co^\kappa$-manifold $\Mf$ has a natural $\Co^{\kappa-1}$-differential structure. Its complexification, the complex tangent bundle $\C T\Mf:=T\Mf\otimes\C= T\Mf\oplus iT\Mf$ is also a $\Co^{\kappa-1}$-differential manifold. So we can define (real or complex) $\Co^\alpha$-vector fields when $\alpha\in(0,\kappa-1]$.

Similarly, the cotangent bundle $T^*\Mf$ and the complex cotangent bundle $\C T^*\Mf$ are also $\Co^{\kappa-1}$-manifolds.

\begin{defn}\label{Defn::ODE::CpxSubbd}
    Let $\kappa>2$, $0<\alpha\le\kappa$ and let $\Mf$ be a $\Co^\kappa$-manifold. A \textbf{$\Co^\alpha$-complex tangent subbundle of rank $r$} is a subset $\Se\subseteq\C T\Mf$ satisfying that for any $p\in \Mf$,
    \begin{itemize}[parsep=-0.35ex]
        \item $\Se_p:=\Se\cap\C T_p\Mf$ is a $r$-dimensional complex vector subspace.
        \item There is an open neighborhood $U\subseteq \Mf$ of $p$ and $\Co^\alpha$-complex vector fields $X_1,\dots,X_r$ on $U$, such that $X_1|_q,\dots,X_r|_q$ form a complex linear basis for $\Se_q$ for all $q\in U$.
    \end{itemize}
    We denote it as $\Se\le\C T\Mf$. And we say $(X_1,\dots,X_r)$ form a \textbf{$\Co^\alpha$-local basis} for $\V$ (or $\Se$) on $U$.
\end{defn}
In the paper we do not deal with abstract vector bundle,  we will use ``subbundle'' as the abbreviation of (complex) tangent subbundle.

The real tangent subbundles can be defined in the same way through replacing every $\C T\Mf$ by $T\Mf$. 
    But we focus mostly on complex subbundles since every real subbundle $\V\le T\Mf$ can be identified as a complex subbundle $\V\otimes\C\le\C T\Mf$.
    
Alternatively a rank $r$ $\Co^\alpha$-subbundle $\Se\le \C T\Mf$ can be equivalently defined as a $\Co^\alpha$-section of the complex Grassmannian bundle $\Gr_\C(r,\C T\Mf)=\coprod_{p\in\Mf}\{\text{rank }r\text{ complex spaces of }\C T_p\Mf\}$. For their equivalence we leave the proof to readers.

In Theorem \ref{Thm::Key} we consider tangent subbundles with mixed regularities.
\begin{defn}\label{Defn::ODE::CpxPaSubbd}
    Let $0<\alpha,\beta\le\kappa$ and let $\Mf,\Nf$ be two $\Co^\kappa$-manifolds. A \textbf{$\Co^{\alpha,\beta}$-complex tangent subbundle of rank $r$} is a $\Co^{\min(\alpha,\beta)}$-subbundle $\Se\le \C T(\Mf\times\Nf)$ such that, for any $(u_0,v_0)\in\Mf\times\Nf$ there is a neighborhood $U\times V\subseteq \Mf\times\Nf$ of $(u_0,v_0)$ and complex vector fields $X_1,\dots,X_r\in\Co^{\alpha,\beta}_\loc(U,V;\C T(\Mf\times\Nf))$, such that $X_1|_{(u,v)},\dots,X_r|_{(u,v)}$ form a complex linear basis for $\Se_{(u,v)}$ for all $(u,v)\in U\times V$.
\end{defn}
Equivalently this is saying $\Se\in\Co^{\alpha,\beta}_\loc(\Mf,\Nf;\Gr_\C(r,\C T(\Mf\times \Nf)))$ as a section of the Grassmannian bundle. 

\medskip
Recall the notion of involutive structure in Definition \ref{Defn::Intro::InvStr}. Let $x=(x^1,\dots,x^n)$ be a local coordinate system, for (at least $C^1$) vector fields $X=:\sum_{j=1}^nX^j\Coorvec{x^j}$ and $Y=:\sum_{k=1}^nY^k\Coorvec{x^k}$, we recall that
\begin{equation}\label{Eqn::ODE::LieBracketRecap}
    [X,Y]=\sum_{j,k=1}^n\Big(X^j\frac{\partial Y^k}{\partial x^j}-Y^j\frac{\partial X^k}{\partial x^j}\Big)\Coorvec{x^k}.
\end{equation}
For completeness we give definition of the involutivity on subbundles with mixed regularities. 
\begin{defn}\label{Defn::ODE::InvMix}
Let $\kappa>2$, $\alpha\in(1,\kappa]$ and $\beta\in(0,\kappa]$. Let $\Mf,\Nf$ be two $\Co^\kappa$-manifolds and $\Se$ be a $\Co^{\alpha,\beta}$-complex subbundle such that $\Se\le(\C T\Mf)\times\Nf$. We say $\Se$ is involutive, if for every $C^0$-sections $X,Y:\Mf\times\Nf\to\Se$ of $\Se$ that are $C^1_\loc$ in $\Mf$, the commutator $[X,Y]$ is a $C^0$-section of $\Se$.
\end{defn}
\begin{remark}\label{Rmk::ODE::InvMix}
Definition \ref{Defn::ODE::InvMix} does not require $X$ or $Y$ to be differentiable on $\Nf$, since when $\Se\le(\C T\Mf)\times\Nf$ the differentiations in \eqref{Eqn::ODE::LieBracketRecap} are only taken on $\Mf$ but not on $\Nf$. This is critical for the case $\beta\le1$.
\end{remark}

An important fact for involutive structures is that we can pick a collection of good basis.

\begin{lem}[Canonical local basis]\label{Lem::ODE::GoodGen}
	Let $\alpha,\beta>0$, $\kappa\ge\max(\alpha,\beta)+1$. Let $\Mf$ and $\Nf$ be two $\Co^\kappa$-manifolds. Let $\Se$ be a $\Co^{\alpha,\beta}$-complex subbundle over $\Mf\times\Nf$ such that $\Se\le (\C T\Mf)\times\Nf$.
	
	Let $p\in \Mf$ and $\lambda_0\in\Nf$. Then there are numbers $r,m,q\ge0$ and a (mixed real and complex) $\Co^\kappa$-coordinate system $(t,z,s)=(t^1,\dots,t^r,z^1,\dots,z^m,s^1,\dots,s^q):\tilde U\subseteq \Mf\to\R^r\times\C^m\times\R^q$ near $p$, such that 
	\begin{equation}\label{Eqn::ODE::GoodGen::DirectSumAssumption}
	    \Se_{(p,\lambda_0)}\oplus\Span(\partial_{\bar z^1}|_p,\dots,\partial_{\bar z^m}|_p,\partial_{s^1}|_p,\dots,\partial_{s^q}|_p)=\C T_p\Mf
	\end{equation}
	Moreover there are a neighborhood $U\subseteq\tilde U$ of $p$, a neighborhood $V\subseteq\Nf$ of $\lambda_0$, and a $\Co^{\min(\alpha,\beta)}$-local basis $X=[X_1,\dots,X_{r+m}]^\top$ for $\Se$ on $U\times V$, such that the collection vectors $X'=[X_1,\dots,X_r]^\top$ and $X''=[X_{r+1},\dots,X_{r+m}]^\top$ are of the form 
	\begin{equation}\label{Eqn::ODE::GoodGen::GoodGenFormula}
	\begin{pmatrix}X'\\X''\end{pmatrix}=\begin{pmatrix}I_r&&A'&B'\\&I_m&A''&B''\end{pmatrix}\begin{pmatrix}\partial/\partial t\\\partial/\partial z\\\partial/\partial{\bar z}\\\partial/\partial s\end{pmatrix},\end{equation}
	where $A'\in\Co^{\alpha,\beta}_\loc(U,V;\C^{r\times m})$, $A''\in\Co^{\alpha,\beta}_\loc(U,V;\C^{m\times m})$, $B'\in\Co^{\alpha,\beta}_\loc(U,V;\C^{r\times q})$ and $B''\in\Co^{\alpha,\beta}_\loc(U,V;\C^{m\times q})$. In addition,
	\begin{enumerate}[parsep=-0.3ex,label=(\roman*)]
	    \item\label{Item::ODE::GoodGen::Uniqueness} Such $X_1,\dots,X_{r+m}$ are uniquely determined by the coordinate system $(t,z,s)$.
	    \item\label{Item::ODE::GoodGen::InvComm} If $\Se$ is involutive, then $X_1,\dots,X_{r+m}$ are pairwise commutative.
	    \item\label{Item::ODE::GoodGen::Real} Suppose $\Se=\bar\Se$, then we must have $m=0$, $A'=0$, $A''=0$, $B''=0$ and $B'$ is real-valued.
	\end{enumerate}

\end{lem}
\begin{proof}Clearly $\Se_{(p,\lambda_0)}\cap\bar\Se_{(p,\lambda_0)}$ and $\Se_{(p,\lambda_0)}+\bar\Se_{(p,\lambda_0)}$ are both complex linear subspaces of $\C T_p\Mf$. We have $r:=\rank(\Se_{(p,\lambda_0)}\cap\bar\Se_{(p,\lambda_0)})$, $m:=\rank\Se_{(p,\lambda_0)}-r$ and $q:=\dim\Mf-\rank(\Se_{(p,\lambda_0)}+\bar\Se_{(p,\lambda_0)})$. Note that $\dim\Mf=r+2m+q$. See \cite[Lemma I.8.5]{Involutive} for details of \eqref{Eqn::ODE::GoodGen::DirectSumAssumption}.

The formula \eqref{Eqn::ODE::GoodGen::GoodGenFormula} follows from the standard linear algebra argument. See \cite[Lemma 1]{ShortFrobenius} or \cite[Page 18]{Involutive}. Note that by Lemma \ref{Lem::Hold::CompofMixHold} \ref{Item::Hold::CompofMixHold::InvMat}, if $M\in\Co^{\alpha,\beta}_\loc(U,V;\C^{(r+m)\times (r+m)})$ is a matrix function invertible at every point in the domain, then $M^{-1}\in\Co^{\alpha,\beta}_\loc(U,V;\C^{(r+m)\times (r+m)})$ holds as well. We leave the details to readers.

\medskip
\noindent\ref{Item::ODE::GoodGen::Uniqueness}: If $\tilde X=[\tilde X_1,\dots,\tilde X_{r+m}]^\top$ is another local basis that has expression \eqref{Eqn::ODE::GoodGen::GoodGenFormula}, then  $\tilde X-X$ are linear combinations of $\partial_{\bar z}$ and $\partial_s$. But $\Se|_U\cap\Span(\partial_{\bar z},\partial_s)=\{0\}$, so $\tilde X=X$ is the unique collection.

\medskip\noindent\ref{Item::ODE::GoodGen::InvComm}: We use  $(\partial_{u^1},\dots,\partial_{u^{r+m}}):=(\partial_{t^1},\dots,\partial_{t^r},\partial_{z^1},\dots,\partial_{z^m})$, $(\partial_{v^1},\dots,\partial_{v^{m+q}}):=(\partial_{\bar z^1},\dots,\partial_{\bar z^m},\partial_{s^1},\dots,\partial_{s^q})$ and  $\begin{pmatrix}A'&B'\\A''&B''\end{pmatrix}=\left(f_j^k\right)_{\substack{1\le j\le r+m\\1\le k\le m+q}}$ for simplicity.

To show that $X_1,\dots,X_{r+m}$ are commutative, by direct computation, for $1\le j,j'\le r+m$,
	\begin{equation}\label{LieBraofX}
	    [X_j,X_{j'}]=\sum_{k=1}^{r+m}\bigg(\frac{\partial f_{j'}^k}{\partial u^j}-\frac{\partial f_j^k}{\partial u^{j'}}+\sum_{k'=1}^{m+q}\Big(f_j^{k'}\frac{\partial f_{j'}^k}{\partial v^{k'}}-f_{j'}^{k'}\frac{\partial f_j^k}{\partial v^{k'}}\Big)\bigg)\Coorvec{v^k}\in\Span\Big(\Coorvec{\bar z},\Coorvec{s}\Big).
	\end{equation}
	
	But by involutivity $[X_j,X_{j'}]$ are all sections of $\Se=\Span X$. Since $\Span X\cap\Span(\partial_{\bar z},\partial_s)=\{0\}$ by \eqref{Eqn::ODE::GoodGen::GoodGenFormula}, we get $[X_j,X_{j'}]=0$ for $1\le j,j'\le r+m$.
	
\medskip\noindent\ref{Item::ODE::GoodGen::Real}: When $\Se=\bar\Se$, the complex conjugate $\bar X_1,\dots,\bar X_{r+m}$ also form a local basis for $\Se$. By \ref{Item::ODE::GoodGen::Uniqueness} we see that $\bar X'=X'$, so $B'=\bar B'$ is a real-valued matrix.

By \eqref{Eqn::ODE::GoodGen::GoodGenFormula} we have $X''=\partial_z+A''\partial_{\bar z}+B''\partial_s$ and  $\bar X''=\partial_{\bar  z}+\bar A''\partial_{ z}+\bar B''\partial_s$. They are linearly independent since $\partial_z$ and $\partial_{\bar z}$ are linearly independent. So $X_1,\dots,X_{r+m},\bar X_{r+1},\dots,\bar X_{r+m}$ are all sections $\Se$ and are also linearly independent, which means $\rank\Se\ge r+2m$. Since $\rank\Se=r+m$ we must have $m=0$, which means $A'=0$, $A''=0$, $B''=0$ and $B'\in\Co^{\alpha,\beta}(U,V;\R^{r\times q})$.
\end{proof}

We recall the definition of complex Frobenius structure.
\begin{defn}\label{Defn::ODE::CpxStr}
	Let $\alpha>1$ and let $\Se\le\C T\Mf$ be a $\Co^\alpha$-complex subbundle. We call $\Se$ a \textbf{complex Frobenius structure} of $\Mf$, if $\Se+\bar \Se$ is a subbundle, and both $\Se$ and $\Se+\bar\Se$ are involutive.
\end{defn}
\begin{remark}
In \cite{Gong} $\Se$ is also called a \textit{Levi-flat CR vector bundle} on $\Mf$.
\end{remark}
\begin{defn}
    Let $\Se\le\C T\Mf$ be a complex tangent subbundle, we use $\Se^\bot$ as the \textbf{dual bundle} of $\Se$, and $\bar\Se$ as the \textbf{complex conjugate} of $\Se$. They are given by 
    \begin{equation}\label{Eqn::ODE::DualConjBundle}
        \Se^\bot:=\{(p,\theta)\in\C T^*\Mf:\theta(v)=0,\forall v\in\C T_p\Mf\}\le\C T^*\Mf,\quad\bar\Se:=\{(p,\bar v)\in\C T\Mf:(p,v)\in\Se\}\le\C T\Mf.
    \end{equation}
\end{defn}

\begin{remark}\label{Rmk::ODE::RmkInvStr}
\begin{enumerate}[label=(\roman*),parsep=-0.3ex]
    \item One can check that $\bar \Se$ is indeed a subbundle, since $\bar\Se_p\le\C T_p\Mf$ is  a complex subspace (which is closed under the scalar multiplication by $i$) for each $p\in \Mf$.
    \item\label{Item::ODE::RmkInvStr::NotBund} The fact $\Se\le\C T\Mf$ is a subbundle does not necessary implies $\Se\cap\bar \Se$ or $\Se+\bar\Se$ is a subbundle. For example, the \textit{Mizohata operator} $X=\partial_t+it\partial_x$ on $\R^2_{t,x}$ generates a rank 1 complex subbundle $\Se\le \C T\R^2$, but $\Se\cap\bar\Se$ and $\Se+\bar\Se$ do not have constant rank: $$(\Se\cap\bar\Se)_{(t,x)}=\begin{cases}0,&t\neq0,\\\C_t\times\{0\},&t=0,\end{cases}\quad (\Se+\bar\Se)_{(t,x)}=\begin{cases}\C_{t,x}^2,&t\neq0,\\\C_t\times\{0\},&t=0.\end{cases}$$
    \item\label{Item::ODE::RmkInvStr::MoreReg} Suppose $\Se+\bar\Se$ is a subbundle. Then it is possible that $\Se+\bar\Se$ is more regular than $\Se$. For example, in $\R^3$ with coordinates $(x,y,z)$, $\Se$ is spanned by $\partial_x+i(1+|y|)\partial_y$. $\Se$ is merely a Lipschitz subbundle, but $\Se+\bar\Se$ can be spanned by $\partial_x,\partial_y$. So $\Se+\bar\Se$ is a analytic subbundle. Similarly if $\Se\cap\bar\Se$ is a subbundle than it can be more regular than $\Se$ as well. In this case $\Se\cap\bar\Se=\{0\}$ is a rank 0 subbundle which formally speaking is real-analytic as well.
\end{enumerate}
\end{remark}

In the definition of complex Frobenius structure, we have no regularity assumption on $\Se+\bar\Se$. But the following lemma shows that once $\Se+\bar\Se$ has constant rank then $\Se\cap\bar\Se$ and $\Se+\bar \Se$ are both subbundles with at least $\Co^\alpha$-regularities.
\begin{lem}
	Let $\alpha>1$ and let $\Mf$ be a $\Co^{\alpha+1}$-manifold. Let $\Se\le\C T\Mf$ be a $\Co^\alpha$-complex subbundle such that $\Se+\bar \Se$ has constant rank at every point, then $\Se+\bar\Se$ and $\Se\cap\bar\Se$ are both $\Co^\alpha$-complex tangent subbundles.
\end{lem}
\begin{proof}Fix a point $p\in\Mf$. By Lemma \ref{Lem::ODE::GoodGen} $\Se$ has a $\Co^\alpha$-local basis $(X_1,\dots,X_{r+m})$ that has the form \eqref{Eqn::ODE::GoodGen::GoodGenFormula}. 

At the point $p$, by \eqref{Eqn::ODE::GoodGen::GoodGenFormula} we see that $X_1|_p,\dots,X_{r+m}|_p,\bar X_{r+1}|_p,\dots,\bar X_{r+m}|_p$ are linearly independent and span $\Se_p+\bar\Se_p=(\Se+\bar\Se)_p$. By assumption $\Se+\bar\Se$ has constant rank, so $\rank(\Se+\bar\Se)=r+2m$. Clearly by \eqref{Eqn::ODE::GoodGen::GoodGenFormula} $X_1,\dots,X_{r+m},\bar X_{r+1},\dots,\bar X_{r+m}$ are still linear independent near $p$. Therefore $\Se+\bar\Se$ is a $\Co^\alpha$-subbundle since $X_1,\dots,X_{r+m},\bar X_{r+1},\dots,\bar X_{r+m}$ form a $\Co^\alpha$-local basis.

Note that $[\bar X_1,\dots,\bar X_r,X_{r+1},\dots,X_{r+m},\bar X_{r+1},\dots,\bar X_{r+m}]^\top$ also have the form \eqref{Eqn::ODE::GoodGen::GoodGenFormula} and form a $\Co^\alpha$-local basis for $\Se+\bar\Se$. By Lemma \ref{Lem::ODE::GoodGen} \ref{Item::ODE::GoodGen::Uniqueness} we have $X_1=\bar X_1,\dots,X_r=\bar X_r$, which means $X_1,\dots,X_r$ are linear independent sections of $\Se\cap\bar\Se$.

Now $\rank(\Se\cap\bar\Se)=2\rank\Se-\rank(\Se+\bar\Se)=r$ we see that $X_1,\dots,X_r$ form a $\Co^\alpha$-local basis for $\Se\cap\bar\Se$. In particular $\Se\cap\bar\Se$ is also a $\Co^\alpha$-subbundle.
\end{proof}

If $\Mf$ is a complex manifold and $\Se$ is a holomorphic subbundle, we have a similar result to Lemma  \ref{Lem::ODE::GoodGen}.

\begin{defn}
Let $\Mf$ be a $m$-dimensional complex manifold. A holomorphic vector field on $\Mf$ is a complex vector field such that locally it has form $\sum_{k=1}^m f^k(z)\Coorvec{z^k}$ where $(z^1,\dots,z^m):U\subseteq\Mf\to\C^m$ is a complex (holomorphic) coordinate chart for $\Mf$, and $f^k$ are holomorphic functions on $U$.

A rank $r$ holomorphic tangent subbundle over $\Mf$ is a complex tangent subbundle $\Se\le \C T\Mf$ such that for any $p\in\Mf$ there is a neighborhood $U\subseteq\Mf$ of $p$ and holomorphic vector fields $Z_1,\dots,Z_r$ on $U$ that that form a local basis for $\Se|_U$.
\end{defn}

Equivalently $\Se$ is a holomorphic section of the Grassmannian bundle $\Gr_\C(r,T^{\Oh}\Mf)\subset \Gr_\C(r,\C T\Mf)$, where $T^\Oh \Mf=\Span(\Coorvec{z^1},\dots,\Coorvec{z^m})$ is the holomorphic tangent bundle. 

In the proof of Theorems \ref{Thm::ThmCoor1} and \ref{Thm::ThmCoor2}, after applying Frobenius theorem on $\Se+\bar\Se$, we need to show that the pullback complex Frobenius structure $\Phi^*\Se$ has the same regularity as $\Se$.

\begin{lem}\label{Lem::ODE::HoloGoodGen}
Let $\Mf$ is a $(r+q)$-dimensional complex manifold and $\Se\le T^\Oh \Mf$ be a rank $r$ holomorphic subbundle. Then for any point $p\in\Mf$ there is a holomorphic chart  $(z,w)=(z^1,\dots,z^r,w^1,\dots,w^q):\tilde U\subseteq \Mf\to\C^r\times\C^n$ such that $\Se_p\oplus\Span(\partial_w|_p)=T^\Oh_p \Mf$. Here $T^\Oh \Mf$ is the holomorphic tangent bundle.

Moreover for such chart $(z,w)$ there is a neighborhood $U$ of $p$, and a unique holomorphic local basis $Z=[Z_1,\dots,Z_r]^\top$ for $\Se$ on $U$ that has the form $Z=\partial_z+A\partial_w$ where $A\in\Oh(U;\C^{r\times n})$.
\end{lem}
\begin{proof}The existence of $(z,w)$ follows the same but simpler argument from the existence of the chart $(t,z,s)$ in Lemma \ref{Lem::ODE::GoodGen}.

Fix $(z,w)$, we can find a holomorphic local basis $W=[W_1,\dots,W_r]^\top$ for $\Se$ near $p$ that has expression $W=M'\partial_z+M''\partial_w$ for some holomorphic matrix functions $M',M''$ defined near $p$. Since $\Se_p\oplus\Span(\partial_w|_p)=T^\Oh_p \Mf$, $M'(p)\in\C^{r\times r}$ is invertible. By continuity $M'$ is invertible in some neighborhood $U\subseteq\Mf$ of $p$. And by cofactor representation of a holomorphic matrix we know $M'^{-1}$ is also a holomorphic matrix function on $U$.

Take $Z=M'^{-1}W=\partial_z+M'^{-1}M''\partial_w$ in $U$, so $A:=M'^{-1}M''$ is the desired holomorphic matrix function.

Since $\Se$ is a complex tangent subbundle of $\Mf$ and we $\Se_p\oplus\Span(\partial_{\bar z}|_p,\partial_{\re w}|_p,\partial_{\im w}|_p)=\C T_p \Mf$. By Lemma \ref{Lem::ODE::GoodGen} \ref{Item::ODE::GoodGen::Uniqueness}, $A$ is uniquely determined by the mixed real and complex coordinate system $(z,\re w,\im w)$, thus determined by $(z,w)$.
\end{proof}

Finally we give an auxiliary result to illustrate the regularity of the pullback subbundle. 
\begin{lem}\label{Lem::ODE::PullBackReg}
Let $\kappa>2$, $\alpha>0$ and $\beta\ge\alpha$ such that $1<\beta\le\kappa-1$. Let $0\le r\le  m\le n$, let $\Mf$ be a $n$-dimensional $\Co^\kappa$ manifold and $\Se\le \C T\Mf$ be a $\Co^\alpha$-complex tangent subbundle of rank $r$.

Suppose $\Phi:\Omega\subseteq\R^n_x\to\Mf$ is $\Co^\beta$-regular parameterization such that $\Phi_*\Coorvec{x^1},\dots,\Phi_*\Coorvec{x^m}\in\Co^\alpha_\loc(\Phi(\Omega);T\Mf)$ and $\Se|_{\Phi(\Omega)}\le\Span(\Phi_*\Coorvec{x^1},\dots,\Phi_*\Coorvec{x^m})$ (here $\Span$ is for the complex span). Then $\Phi^*\Se\le \C T\Omega$ is $\Co^\alpha$-subbundle as well.
\end{lem}
In general, if $\beta<\alpha+1$, it is possible that $\Phi_*\Coorvec{x^{m+1}},\dots,\Phi_*\Coorvec{x^n}$ are only $\Co^{\beta-1}(\supset\Co^\alpha)$.
\begin{proof}
Let $u_0\in \Omega$, we need to show that $\Phi^*\Se$ has a $\Co^\alpha$-local basis near $u_0$.

By assumption $\Se$ has $\Co^\alpha$-local basis $(X_1,\dots,X_r)$ near $\Phi(u_0)\in\Mf$. Thus $(\Phi^*X_1,\dots,\Phi^*X_r)$ is a local basis for $\Phi^*\Se$ near $u_0$.

Since $\Se|_{\Phi(\Omega)}\le\Span(\Phi_*\Coorvec{x^1},\dots,\Phi_*\Coorvec{x^m})$, we can write $X_j=\sum_{k=1}^mf_j^k\cdot\Phi_*\Coorvec{x^k}$ for $1\le j\le r$, for some functions $f_j^k$ defined near $\Phi(u_0)$. Since $f_j^k$ are uniquely determined by $X_1,\dots,X_r$ and $\Phi$, we see that $f_j^k\in\Co^{\min(\alpha,\beta)}=\Co^\alpha$.

Taking pullback by $\Phi$, we have $\Phi^*X_j=\sum_{k=1}^m(f_j^k\circ\Phi)\Coorvec{x^k}$, since $\Phi\in\Co^\beta\subseteq\Co^\alpha$, we get $f_j^k\circ\Phi\in\Co^\alpha$, so $(\Phi^*X_1,\dots,\Phi^*X_r)$ is a $\Co^\alpha$-local basis for $\Se$, finishing the proof.
\end{proof}
\subsection{Review of autonomous ODE flows}


In the real Frobenius theorem and the holomorphic Frobenius theorem we need ODE flows to construct the parameterizations.

\begin{defn}
Let $X$ be $C^1$-vector field on a $C^2$-manifold $\Mf$. The \textbf{flow map} of $X$, denoted by $\exp_X:\Do_X\subseteq\R\times \Mf\to \Mf$, is defined in the way that for any $p\in \Mf$, $t \mapsto\exp_X(t,p)$ is the unique maximal solution $\gamma=\gamma_p$ to the ODE 
$$\left\{\begin{array}{l}\dot\gamma(t)=X(\gamma(t))\in T_{\gamma(t)}\Mf,\\\gamma(0)=p.\end{array}\right.$$
Here $\Do_X\cap(\R\times\{p\})\hookrightarrow\R$ is the maximal existence interval of $\gamma$, which is an open interval containing $0$.

We denote $e^{tX}(p):=\exp_X(t,p)$.
\end{defn}
\begin{remark}There is no ambiguity to use $e^{tX}$ since $(t,p)\in\Do_X$ if and only if $(1,p)\in\Do_{tX}$ and $\exp_X(t,p)=\exp_{tX}(1,p)$ for all $(t,p)\in\Do_X$.
\end{remark}

From classical ODE theories, we know $\exp_X$ is well-defined, and $\Do_X$ is an open set in $\R\times \Mf$. Moreover $\Do_X=\R\times \Mf$ is the whole space when either $\Mf$ is a compact manifold or $X$ has compact support. For such cases $e^{tX}$ is a $C^1$-diffeomorphism of $\Mf$ for every $t$.

\begin{lem}\label{Lem::ODE::FlowComm}
    Let $X$ and $Y$ be $C^1$-vector fields defined on $\R^n$ with compact supports. Then 
    \begin{equation}\label{Eqn::ODE::LieDer}
    (e^{tX})^*Y=Y+\int_0^t(e^{uX})^*[X,Y]du,\quad\text{on }\R^n,\quad\forall\ t\in\R.
\end{equation}
    In particular, if $[X,Y]=0$ holds on some open subset $U_0\subseteq\R^n$, then for any $U_1\Subset U_0$ there is a $\eps>0$ such that $e^{tX}e^{sY}=e^{sY}e^{tX}=e^{tX+sY}$ holds in $U_1$ for $|t|,|s|<\eps$.
\end{lem}
Here for a diffeomorphism $F:\R^n\to\R^n$, the pullback vector field $F^*X:\R^n\to\R^n$ can be written as $(X\cdot\nabla F^\Inv)\circ F$, where we use row vectors $X=(X^1,\dots,X^n)$ and $(X\cdot\nabla G)^j=\sum_{k=1}^nX^k\frac{\partial G^j}{\partial x^k}$ for $j=1,\dots,n$.

\begin{proof}
The standard Lie derivative formula shows that $\lim\limits_{h\to 0}\frac1h((e^{hX})^*Y-Y)=[X,Y]$ as a continuous vector field. By $(e^{tX})^*X=X$, we have $\lim\limits_{h\to 0}\frac1h((e^{(t+h)X})^*Y-(e^{tX})^*Y)=[X,(e^{tX})^*Y]=(e^{tX})^*[X,Y]$ pointwisely. Taking the definite integral from $0$ to $t$ we get \eqref{Eqn::ODE::LieDer}.

Now assume $[X,Y]=0$ in $U_0\subseteq\R^n$. Take $\eps<\frac13\dist(U_1,\partial U_0)\max(\|X\|_{C^0(U_0;\R^n)},\|Y\|_{C^0(U_0;\R^n)})^{-1}$, so $e^{tX}$, $e^{sY}$ and $e^{sY}e^{tX}e^{-sY}$ are all defined in $U_1$ for $|t|,|s|<\eps$. 

Using \eqref{Eqn::ODE::LieDer} with $[X,Y]=0$ we see that $(e^{tX})_*Y=(e^{-tX})^*Y=Y$. Therefore for $|t|,|s|<\eps_0$, in  $U_1$
\begin{align*}
    &\Coorvec s(e^{sY}\circ e^{tX}\circ e^{-sY})=(Y-(e^{tX})_*Y)\circ e^{sY}e^{tX}e^{-sY}=0.
\end{align*}

Thus $e^{sY}\circ e^{tX}\circ e^{-sY}=e^{tX}$ as continuous maps.

On the other hand for $u\in[0,1]$, on the domain we have $\Coorvec u(e^{utX}e^{usY})=\Coorvec u(e^{utX})e^{usY}+\Coorvec u(e^{usY})e^{utX}=(tX+sY)\circ e^{utX}e^{usY}$, $\Coorvec u e^{u(tX+sY)}=(tX+sY)\circ e^{u(tX+sY)}$. So $u\mapsto e^{utX}e^{usY}$ and $u\mapsto  e^{u(tX+sY)}$ are both ODE flows of $tX+sY$, which are equal by the flow uniqueness. Taking $u=1$ we get $e^{tX}e^{sY}=e^{tX+sY}$, finishing the proof.
\end{proof}

Lemma \ref{Lem::ODE::FlowComm} fails if $t$ and $s$ are large. See the example below.
\begin{example}
Consider the vector fields $X(x,y)=(e^{-x}\cos y,-e^{-x}\sin y)$ and $Y(x,y)=(e^{-x}\sin y,e^{-x}\cos y)$ defined in $\R^2$. By direct computation, $[X,Y]=0$ so $X$ and $Y$ commute.

One can see that we have the following ODE solutions, for $t,s\in\R$:
\begin{align*}
    e^{tX}(\ln\sqrt2,\tfrac54\pi)=(\tfrac12\ln(t^2+2t+2),\tfrac32\pi+\arctan(t-1)),&\quad e^{2X}(\ln\sqrt2,\tfrac54\pi)=(\ln\sqrt2,\tfrac74\pi);
    \\
    e^{sY}(\ln\sqrt2,\tfrac54\pi)=(\tfrac12\ln(s^2+2s+2),\pi-\arctan(s-1)),&\quad e^{2Y}(\ln\sqrt2,\tfrac54\pi)=(\ln\sqrt2,\tfrac34\pi);
    \\
    e^{sY}(\ln\sqrt2,\tfrac74\pi)=(\tfrac12\ln(s^2+2s+2),2\pi+\arctan(s-1)),&\quad e^{2Y}(\ln\sqrt2,\tfrac74\pi)=(\ln\sqrt2,\tfrac94\pi);
    \\
    e^{tX}(\ln\sqrt2,\tfrac34\pi)=(\tfrac12\ln(t^2+2t+2),\tfrac12\pi-\arctan(t-1)),&\quad e^{2X}(\ln\sqrt2,\tfrac34\pi)=(\ln\sqrt2,\tfrac14\pi).
\end{align*}
Therefore $e^{2X}e^{2Y}(\ln\sqrt2,\tfrac54\pi)=(\ln\sqrt2,\tfrac94\pi)$ and $e^{2Y}e^{2X}(\ln\sqrt2,\tfrac34\pi)=(\ln\sqrt2,\tfrac14\pi)$ are not equal.

A related fact is that for the holomorphic vector field $Z(z)=e^{-z}\Coorvec z$ defined on $\C^1$, we cannot talk about its holomoprhic ODE flow globally.
\end{example}

In the proof of Frobenius theorem we need H\"older-Zygmund regularity estimates for ODE flows as follows: 

\begin{lem}[Flow regularity]\label{Lem::ODE::ODEReg} Let $\alpha\in(1,\infty]$ and let $X$ be a $\Co^\alpha$-vector field on an open set $U_0\subseteq\R^n$. Then for any $U_1\Subset U_0$ and $0<\eps_0<\frac12\dist(U_1,\partial U_0)\cdot\|X\|_{C^0(U_0)}^{-1}$, the map $\exp_X$ is defined on $(-\eps_0,\eps_0)\times U_1$ and $\exp_X\in\Co^{\alpha+1,\alpha}((-\eps_0,\eps_0),U_1;U_0)$.
\end{lem}
See Lemma \ref{Lem::AppODEReg} for a paramter version. Here for $\alpha=\infty$ we use the standard convention $\alpha+1=\infty$.
\begin{proof}We only need to prove the case $\alpha<\infty$.

Let $\chi\in C_c^\infty(U_0)$ be such that $\chi\equiv1$ in a neighborhood of $\{x:\dist(x,U_1)\le\frac12\dist(U_1,\partial U_0)\}$. We see that $e^{tX}|_{U_1}=e^{t(\chi X)}|_{U_1}$ whenever $|t|<\eps_0$. Therefore, in $U_1$ we can replace $X$ by $\chi X$ in the discussion. In other words, we can assume $X\in\Co^\alpha_c(\R^n)$.

Let $R>0$ be such that $\supp X\subseteq B^n(0,R)$ and let $\chi_1\in C_c^\infty(B^n(0,2R)) $ be such that $\chi_1|_{B^n(0,R)}\equiv1$. Clearly, $e^{tX}(x)=x$ for all $t\in\R$ if $x\notin B^n(0,R)$. Thus $f(t,x):=\chi_1(x)\cdot e^{tX}(x)$ solves the transport equation
\begin{equation*}\label{Eqn::ODE::ODEReg::TransEqn}
\partial_tf(t,x)-(Xf)(t,x)=0,\qquad f(0,x)=\chi_1(x)\cdot x,\quad t\in\R,\quad x\in\R^n.
\end{equation*}

By \cite[Theorem 3.14]{Chemin} we get $f\in L^\infty\Co^\alpha((-T,T),\R^n;\R^n)$ for every $T>0$. Clearly $f\in C^1([-T,T]\times \R^n;\R^n)$ for every $T>0$ and $f(t,x)\equiv e^{tX}(x)$ whenever $|x|<R$. Thus $\exp_X\in \Co^{1,\alpha}((-T,T),B(0,R);\R^n)$.
The results $\exp_X\in\Co^{\alpha+1,\alpha}((-T,T),B(0,R);\R^n)$ follows from the following recursive argument: locally in $(-T,T)\times B(0,R)$,
\begin{equation*}
    \exp_X\in\Co^{k,\alpha}\overunderset{\text{by Lemma \ref{Lem::Hold::QCompIFT} \ref{Item::Hold::QComp}}}{\text{or Lemma \ref{Lem::Hold::CompofMixHold} \ref{Item::Hold::CompofMixHold::Comp}}}{\Longrightarrow} \tfrac{\partial\exp_X}{\partial t}=X\circ \exp_X\in\Co^{\min(\alpha,k),\alpha}\Rightarrow \exp_X\in\Co^{\min(\alpha,k)+1}L^\infty\cap\Co^1\Co^\alpha\subset\Co^{\min(\alpha,k)+1,\alpha}.
\end{equation*}
We start with $k=1$ and ends at $k=\lceil\alpha\rceil$. Thus $\exp_X\in\Co^{\alpha+1,\alpha}((-T,T),B(0,R);\R^n)$, finishing the proof.
\end{proof}

In applications, we consider the multiflow associated with commutative vector fields.

\begin{cor}\label{Cor::ODE::ODERegCor} Let $\alpha>1$ and let $X=[X_1,\dots,X_r]^\top$ be a collection of commutative $\Co^\alpha$-vector fields on an open set $U_0\subseteq\R^n_x$. Then for any $U_1\Subset U_0$ there is a $\eps_0>0$ such that we the following equality, both of whose sides are defined:
\begin{equation*}
    e^{t^1X_1+\dots+t^rX_r}(x)=e^{t^1X_1}\circ\dots\circ e^{t^rX_r}(x),\quad\text{for } t=(t^1,\dots,t^r)\in(-\eps_0,\eps_0)^r,\quad x\in U_1.
\end{equation*}
Moreover by denoting $\exp_X(t,x):=e^{t^1X_1}\dots e^{t^rX_r}(x)$, we have
$\exp_X\in\Co^{\alpha+1,\alpha}(B^r(0,\eps_0),U_1;U_0)$ and satisfies $\Coorvec{t^j}\exp_X(t,x)=X_j\circ\exp_X(t,x)$ for $j=1,\dots,r$.
\end{cor}
\begin{proof}
We see that $\exp_X$ is defined whenever $\eps_0<\frac1r\dist(U_1,\partial U_0)\cdot(\max_{1\le j\le r}\|X_j\|_{\Co^\alpha(U_0;\R^n)})^{-1}$. 

The property $e^{t^1X_1+\dots+t^rX_r}(x)=e^{t^1X_1}\dots e^{t^rX_r}$ and $\Coorvec{t^j}\exp_X=X_j\circ\exp_X$ follows from Lemma \ref{Lem::ODE::FlowComm}. The regularity statement $\exp_X\in\Co^{\alpha+1,\alpha}$ follows from Lemma \ref{Lem::ODE::ODEReg}.
\end{proof}

\begin{remark}
    By Corollary \ref{Cor::ODE::ODERegCor} we can now use the multi-flow notation $e^{t\cdot X}=\exp_X(t,\cdot)$, where $t=(t^1,\dots,t^r)\in\R^r$ is a small vector and $X=[X_1,\dots,X_r]^\top$ is a collection of (at least $C^1$) commutative vector fields.
\end{remark}

All these properties of ODE flows also work for holomorphic vector fields on complex manifolds.

\begin{lem}[Holomorphic ODE flow with parameter]\label{Lem::ODE::HoloFlow}
    Let $U_0\subseteq\C^m$ and $V_0\subseteq\R^q$ be bounded open sets that are endowed with standard complex coordinate system $z=(z^1,\dots,z^m)$ and real coordinate system $s=(s^1,\dots,s^q)$ respectively. 
    
    Let $\beta>0$, and let $Z=(Z^1,\dots,Z^m):U_0\times V_0\to\C^m$ be a $\Co^\beta$-map that is holomorphic in $z$-variable. Then
    \begin{enumerate}[parsep=-0.3ex,label=(\roman*)]
        \item\label{Item::ODE::HoloFlow::DefReg} For any $U_1\times V_1\Subset U_0\times V_0$ there are an $\eps_0>0$ and a unique continuous map $\Phi^Z:\eps_0\B^2\times U_1\times V_1\subseteq\C^1_\tau\times\C^m_z\times\R^q_s\to U_0$ such that $\Phi^Z$ is $C^1$ in $(\tau,z)$ and satisfies
    \begin{equation}\label{Eqn::ODE::HoloFlow::EqnDef}
        \textstyle\frac{\partial\Phi^Z}{\partial\tau}(\tau,z,s)=Z(\Phi^Z(\tau,z,s)),\quad\frac{\partial\Phi^Z}{\partial\bar\tau}(\tau,z,s)=0,\quad\Phi^Z(0,z,s)=z,\quad\forall\tau\in \eps_0\B^2,\ z\in U_1,\ s\in V_1.
    \end{equation}
    In addition, $\Phi^Z$ is holomorphic in $(\tau,z)$ and $\Phi^Z\in\Co^\infty\Co^\beta(\eps_0\B^2\times U_1,V_1;\C^m)$.
    \item\label{Item::ODE::HoloFlow::Comm} Moreover, suppose $W:U_0\times V_0\to\C^m$ is another such map as $Z$, such that two vector fields $\sum_{j=1}^mZ^j(z,s)\Coorvec{z^j}$ and $\sum_{j=1}^mW^j(z,s)\Coorvec{z^j}$  are commutative. Let $\Phi^Z$ and $\Phi^W$ be as in \eqref{Eqn::ODE::HoloFlow::EqnDef}. Then for any for any $U_1\times V_1\Subset U_0\times V_0$ there are an $\eps_0>0$ such that
    \begin{equation}\label{Eqn::ODE::HoloFlow::FlowComm}
        \Phi^Z(\tau_1,\Phi^W(\tau_2,z,s),s)=\Phi^W(\tau_2,\Phi^Z(\tau_1,z,s),s),\quad\forall \tau_1,\tau_2\in\eps_0\B^2,\quad z\in U_1,\quad s\in V_1.
    \end{equation}
    \end{enumerate}
\end{lem}
\begin{remark}
    If we forget the $s$-variable, the result is indeed the well-definedness of the holomorphic ODE flow $e^{\tau Z}(z)$. In this way, we can write $\Phi^Z$ as 
    \begin{equation}\label{Eqn::ODE::HoloFlow::DefExpZ}
        \Phi^Z(\tau,z,s)=\exp_{Z(\cdot,s)}(\tau,z)=e^{\tau Z(\cdot,s)}(z).
    \end{equation}
    
    And thus \eqref{Eqn::ODE::HoloFlow::FlowComm} is saying that $e^{\tau_1Z(\cdot,s)}e^{\tau_2Z(\cdot,s)}=e^{\tau_2Z(\cdot,s)}e^{\tau_1Z(\cdot,s)}$.
\end{remark}

\begin{proof}[Proof of Lemma \ref{Lem::ODE::HoloFlow}]Write $(x,y):=(\re z,\im z)$. We define real vector fields $X$ and $Y$ on $U_0\subseteq\R^{2m}_{x,y}$ with parameter on $V_0\subseteq\R^{2m+q}_s$ as
\begin{equation}\label{Eqn::ODE::PfHoloFlow::XY}
    X(x,y;s)=\sum_{j=1}^m\re Z^j(z,s)\Coorvec{x^j}+\im Z^j(z,s)\Coorvec{y^j}, \quad Y(x,y;s)=\sum_{j=1}^m-\im Z^j(z,s)\Coorvec{x^j}+\re Z^j(z,s)\Coorvec{y^j}.
\end{equation}
By assumption, $Z$ is holomorphic in $z$, so $[X,Y]=0$ is commutative.

\smallskip
\noindent\ref{Item::ODE::HoloFlow::DefReg}: By the well-definedness of $C^1$ ODE flows and Corollary \ref{Cor::ODE::ODERegCor}, we can find a $\eps_0>0$ such that  $\tilde \Phi(u,v,x,y,s):=e^{uX(\cdot,s)+vY(\cdot,s)}(x,y)$ is the unique map defined in $(u,v)\in(-\eps_0,\eps_0)^2$, $(x,y,s)\in U_1\times V_1$ such that $\tilde\Phi$ is $C^1$ in $(u,v)$ and satisfies $\partial_u\tilde\Phi=X\circ\tilde\Phi$, $\partial_v\tilde\Phi=Y\circ\tilde\Phi$ and $\tilde\Phi(0,0,x,y,s)=(x,y,s)$.

By Lemma \ref{Lem::AppODEReg} $(t,x,y,s)\mapsto\exp_{X(\cdot,s)}(t,(x,y))$ and $(t,x,y,s)\mapsto \exp_{Y(\cdot,s)}(t,(x,y))$ both belongs to $\Co^{\beta+1,\beta}_{(t,x,y),s}$. Since $\tilde \Phi(u,v,x,y,s):=\exp_{X(\cdot,s)}(u,\exp_{Y(\cdot,s)}(v,(x,y)))$, by Lemma \ref{Lem::Hold::CompofMixHold} \ref{Item::Hold::CompofMixHold::Comp}, we get $\tilde \Phi\in\Co^{\beta+1,\beta}_{(u,v,x,y),s}$ as well.

We can write $\tilde\Phi=(\tilde\Phi^R,\tilde\Phi^I)$ where $\tilde\Phi^R,\tilde\Phi^I:(-\eps_0,\eps_0)^2\times U_1\times V_1\to\R^m$ are the components of $x$ and $y$-coordinates respectively. Thus, by direct computation we see that $\Phi^Z:=\tilde\Phi^R+i\tilde\Phi^I$ solves \eqref{Eqn::ODE::HoloFlow::EqnDef}. 

The uniqueness of $\Phi^Z$ follows from the uniqueness of $\tilde \Phi$.
On the other hand, by the uniqueness of the holomorphic ODE flow (see \cite[Theorem 1.1]{HoloODE} for example), we know $\Phi^Z(\cdot,s)$ is the holomoprhic ODE flow for the holomorphic vector field $Z(\cdot,s)$, for each $s$. Thus $\Phi^Z$ is holomorphic in $z$ as well.

Now $\tilde\Phi\in\Co^{\beta}_{u,v,x,y,s}$, so $\Phi^Z$ is a $\Co^\beta$-map. By Lemma \ref{Lem::Hold::NablaHarm} we get $\Phi\in\Co^\infty_{\tau,z}\Co^\beta_{s}$.

\medskip\noindent \ref{Item::ODE::HoloFlow::Comm}: We can fix a $s\in V_1$ in the discussion below, thus $Z$ and $W$ are holomorphic vector fields in $U_0$.

 Let $W=(W^1,\dots,W^m)$ be from the assumption, we define vector fields $X'$ and $Y'$ to be the ``real and imaginary part of $W$'' by replacing $Z^j$ with $W^j$ in \eqref{Eqn::ODE::PfHoloFlow::XY}. We have $[X,Y]=0$ and $[X',Y']=0$ since $Z$ and $W$ are holormorphic in $z$. The commutativity of $\sum_{j=1}^mZ^j(z,s)\Coorvec{z^j}$ and $\sum_{j=1}^mW^j(z,s)\Coorvec{z^j}$ implies that $X,X',Y,Y'$ are all pairwise commutative. 

By Corollary \ref{Cor::ODE::ODERegCor} we have $e^{u_1X}e^{v_1Y}e^{u_2X'}e^{v_2Y'}=e^{u_2X'}e^{v_2Y'}e^{u_1X}e^{v_1Y}$ in $U_1\times V_1$ for $u_1,u_2,v_1,v_2\in\R$ small. Repeating the argument for the construction of $\Phi^Z$ above we get \eqref{Eqn::ODE::HoloFlow::FlowComm}, finishing the proof.
\end{proof}

\subsection{The bi-layer Frobenius theorem and a holomorphic analogy}\label{Section::PfFro}
The classical Frobenius theorem states that an involutive real tangent subbundle is locally spanned by coordinate vector fields in a suitable real coordinate system. In this subsection, we will prove some generalizations along with the H\"older-Zygmund regularity estimates.

In the proof of Theorem \ref{Thm::ThmCoor2} (or Theorem \ref{Thm::TrueThm2}), we need to apply the real Frobenius theorem on $\Se\cap\bar\Se$ and on $\Se+\bar\Se$ simultaneously. In order to achieve the best possible regularity, we need a slightly different proof. 



In the following, we endow $\R^r$, $\R^m$ and $\R^q$ with standard coordinate system $t=(t^1,\dots,t^r)$, $x=(x^1,\dots,x^m)$ and $s=(s^1,\dots,s^q)$ respectively. For the case $\kappa=\infty$, we use the convention $\kappa-1=\infty$.

\begin{thm}[Bi-layer Frobenius theorem]\label{Thm::ODE::BLFro}
	Let $\kappa\in(2,\infty]$ and $\beta,\gamma\in(1,\kappa-1]$. Let $r,m,q$ be non-negative integers. Let $\Mf$ be a $(r+m+q)$-dimensional $\Co^\kappa$-manifold. Let $\V\le\E\le T\Mf$ be two involutive real tangent subbundles such that $\V\in\Co^\gamma$  has rank $r$, and $\E\in\Co^\beta$ has rank $r+m$.

	Then for any $p\in \Mf$, there is a neighborhood $\Omega=\Omega'\times\Omega''\times\Omega'''\subseteq\R^r\times\R^m\times\R^q$ of $(0,0,0)$ and a $\Co^{\min(\gamma,\beta)}$-regular parameterization $\Phi(t,x,s):\Omega\to \Mf$, such that:
	\begin{enumerate}[parsep=-0.3ex,label=(\arabic*)]
		\item\label{Item::ODE::BLFro1} $\Phi(0)=p$.
		\item\label{Item::ODE::BLFro2} $\Phi\in\Co_{t,x,s}^{\gamma+1,\min(\gamma,\beta+1),\min(\gamma,\beta)}(\Omega',\Omega'',\Omega''';\Mf)$, $\frac{\partial\Phi}{\partial t^1},\dots,\frac{\partial\Phi}{\partial t^r}\in\Co^{\gamma,\min(\gamma,\beta+1),\min(\gamma,\beta)}(\Omega',\Omega'',\Omega''';T\Mf)$ and $\frac{\partial\Phi}{\partial x^1},\dots,\frac{\partial\Phi}{\partial x^m}\in\Co^{\min(\gamma-1,\beta)}(\Omega;T\Mf)$.
		\item\label{Item::ODE::BLFro3} For every $(t,x,s)\in\Omega$, $\V_{\Phi(t,x,s)} $ is spanned by $\frac{\partial\Phi}{\partial t^1}(t,x,s),\dots,\frac{\partial\Phi}{\partial t^r}(t,x,s)\in T_{\Phi(t,x,s)}\Mf$, and $\E_{\Phi(t,x,s)}$ is spanned by  $\frac{\partial\Phi}{\partial t^1}(t,x,s),\dots,\frac{\partial\Phi}{\partial t^r}(t,x,s),\frac{\partial\Phi}{\partial x^1}(t,x,s),\dots,\frac{\partial\Phi}{\partial x^m}(t,x,s)\in T_{\Phi(t,x,s)}\Mf$.
	\end{enumerate}
	
    On the coordinates side, set $U:=\Phi(\Omega)\subseteq\Mf$ and let $F=(F',F'',F'''):=\Phi^\Inv:U\subseteq \Mf\to \R^r_t\times\R^m_x\times\R^q_s$ be the inverse map of $\Phi$. Then we have
	\begin{enumerate}[parsep=-0.3ex,label=(\arabic*)]\setcounter{enumi}{3}
		\item\label{Item::ODE::BLFro4} $F(p)=0$.
		\item\label{Item::ODE::BLFro5} $F'\in\Co^\kappa(U;\R^r)$, $F''\in\Co^\gamma(U;\R^m)$ and $F'''\in\Co^\beta(U;\R^q)$.
		\item\label{Item::ODE::BLFro6} $F^*\Coorvec{t^1},\dots,F^*\Coorvec{t^r}\in\Co^\gamma(U;T\Mf)$ and $F^*\Coorvec{x^1},\dots,F^*\Coorvec{x^m}\in\Co^{\min(\gamma-1,\beta)}(U;T\Mf)$.
		\item\label{Item::ODE::BLFro7} $\V|_U$ is spanned by $F^*\Coorvec{t^1},\dots,F^*\Coorvec{t^r}$, and $\E|_U$ is spanned by $F^*\Coorvec{t^1},\dots,F^*\Coorvec{t^r},F^*\Coorvec{x^1},\dots,F^*\Coorvec{x^m}$.
	\end{enumerate}
\end{thm}

For definition of $\Co^{\alpha,\beta,\gamma}$-type maps, see Definitions \ref{Defn::Hold::MoreMixHold} \ref{Item::Hold::MoreMixHold::3Mix} and \ref{Defn::ODE::MixHoldMaps}. 

\begin{remark}
    Here the regularity results for $F$ and those for $\Phi$ do not directly imply each other.
\end{remark}


If we take $m=0$ in Theorem \ref{Thm::ODE::BLFro}, i.e. $\V=\E$, we get the H\"older estimate immediately for the classical real Frobenius theorem.
\begin{cor}[Real Frobenius theorem]\label{Cor::ODE::RealFroCor}Let $\kappa\in(2,\infty]$, $\alpha\in(1,\kappa-1]$, let $r,q\ge0$, and let $\Mf$ be an $(r+q)$-dimensional $\Co^\kappa$-manifold. Assume $\V\le T\Mf$ is a $\Co^\alpha$-involutive real tangent subbundle of rank $r$.

Then for any $p\in \Mf$ there is a neighborhood $\Omega=\Omega'\times\Omega''\subseteq\R^r\times\R^n$ of $0$ and a $\Co^\alpha$-regular parameterization $\Phi(t,s):\Omega\to \Mf$, such that:
	\begin{enumerate}[nolistsep,label=(\arabic*)]
		\item $\Phi(0)=p$.
		\item $\Phi\in\Co_{t,s}^{\alpha+1,\alpha}(\Omega;\Mf)$, and $\frac{\partial\Phi}{\partial t^1},\dots,\frac{\partial\Phi}{\partial t^r}\in\Co^{\alpha}_{(t,s)}(\Omega;T\Mf)$.
		\item For any $(t,s)\in\Omega$, $\V_{\Phi(t,s)}\le T_{\Phi(t,s)}\Mf$ is spanned by $\frac{\partial\Phi}{\partial t^1}(t,s),\dots,\frac{\partial\Phi}{\partial t^r}(t,s)$.
	\end{enumerate}
    On the coordinate side, set $U:=\Phi(\Omega)\subseteq\Mf$ and let $F=(F',F''):=\Phi^\Inv:U\subseteq \Mf\to \R^r_t\times\R^q_s$, we have:
	\begin{enumerate}[nolistsep,label=(\arabic*)]\setcounter{enumi}{3}
		\item $F(p)=0$.
		\item $F'\in\Co^\kappa(U;\R^r)$ and $F''\in\Co^\alpha(U;\R^q)$.
		\item $F^*\Coorvec{t^1},\dots,F^*\Coorvec{t^r}\in\Co^\alpha(U;T\Mf)$.
		\item $\V|_U$ has a local basis $(F^*\Coorvec{t^1},\dots,F^*\Coorvec{t^r})$.
	\end{enumerate}
    
\end{cor}
Corollary \ref{Cor::ODE::RealFroCor} is sharp, see Section \ref{Section::SharpGen} and Proposition \ref{Prop::Final::RealFroisSharp}.

\begin{proof}[Proof of the Bi-layer Frobenius, Theorem \ref{Thm::ODE::BLFro}]We can pick an initial $\Co^\kappa$-coordinate system near $p$
\begin{equation}\label{Eqn::ODE::BLFro::H}
    (u,v,w)=(u^1,\dots,u^r,v^1,\dots,v^m,w^1,\dots,w^q):U_0\subseteq \Mf\to\R^r\times\R^m\times\R^q,
\end{equation}
such that
\begin{itemize}[parsep=-0.3ex]
    \item $u(p)=0\in\R^r$, $v(p)=0\in\R^m$, $w(p)=0\in\R^q$;
    \item $\Coorvec{u^1}|_p,\dots,\Coorvec{u^r}|_p$ form a real basis for $\V_p\le T_p\Mf$;
    \item $\Coorvec{u^1}|_p,\dots,\Coorvec{u^r}|_p,\Coorvec{v^1}|_p,\dots,\Coorvec{v^m}|_p$  form a real basis for $\E_p\le T_p\Mf$.
\end{itemize}

Thus, we can identify $\V,\E$ with their pushforwards on the open neighborhood of  $(0,0,0)\in\R^r\times\R^m\times\R^q$. In other words, we can assume $\Mf\subseteq\R^r_u\times\R^m_v\times\R^q_w$ and $p=(0,0,0)$.

Applying Lemma \ref{Lem::ODE::GoodGen}, on $\V$ and on $\E$ respectively, we can find a smaller neighborhood $U_1\subseteq U_0$ of $p=0$, a $\Co^\gamma$-local basis $T=[T_1,\dots,T_r]^\top$ for $\V$ on $U_1$, and a $\Co^\beta$-local basis $\tilde X=[X',X'']^\top=[X_1,\dots,X_{r+m}]^\top$ for $\E$ on $U_1$, that have the following form:
\begin{equation}\label{Eqn::ODE::BLFro::GenforTX}
    T=\Coorvec u+A'\Coorvec v+A''\Coorvec w,\quad \tilde X=\begin{pmatrix}X'\\X''\end{pmatrix}=\begin{pmatrix}I_r&&B'\\&I_m&B''\end{pmatrix}\begin{pmatrix}\Coorvec u\\\Coorvec v\\\Coorvec w\end{pmatrix},\quad\text{on }U_1\subseteq\R^r_u\times\R^m_v\times\R^q_w.
\end{equation}
Here $A'\in\Co^\gamma(U_1;\R^{r\times m})$, $A''\in\Co^\gamma(U_1;\R^{r\times n})$, $B'\in\Co^\beta(U_1;\R^{r\times n})$ and $B''\in\Co^\beta(U_1;\R^{m\times n})$. 

By \eqref{Eqn::ODE::BLFro::GenforTX} we see that $T_1,\dots,T_r,X_1,\dots,X_m$ are linearly independent, so they form a $\Co^{\min(\gamma,\beta)}$-local basis for $\E$ on $U_1$.

We define a natural map $\Gamma$ as:
\begin{equation}\label{Eqn::ODE::BLFro::DefGamma}
    \Gamma:\R^q_s\to\R^r_u\times\R^m_v\times\R^q_w, \qquad\Gamma(s):=(0,0,s).
\end{equation}
We have $\Gamma(0)=(0,0,0)=p$, and $\frac{\partial\Gamma}{\partial s^j}(s)=\Coorvec{w^j}\big|_{\Gamma(s)}$ for $j=1,\dots,q$.

\medskip
We can now define our map $\Phi:\Omega\subseteq\R^r_t\times\R^m_x\times\R^q_s\to \Mf$ for a smaller $\Omega$, as
\begin{equation}\label{Eqn::ODE::BLFro::DefPhi}
    \Phi(t,x,s):=\exp_T(t,\exp_{X''}(x,\Gamma(s)))=e^{tT}e^{xX''}(\Gamma(s))=e^{t^1T_1}\dots e^{t^rT_r}e^{x^1X_{r+1}}\dots e^{x^mX_{r+m}}(0^{r+m},s^1,\dots,s^q).
\end{equation}

The notation $\exp_T,\exp_{X''}$ are given in Corollary \ref{Cor::ODE::ODERegCor}. By taking $\eps_0$ in Lemma \ref{Lem::ODE::ODEReg} small enough we see that $\Phi$ is defined is when $\Omega$ is small enough.

We are going to see that $\Phi$ and $F:=\Phi^\Inv$ are the desired maps.

By shrinking $\Omega$ if necessary, we can define an auxiliary map $\Psi:\Omega\to \Mf$ as the following
    \begin{gather}
    \label{Eqn::ODE::BLFro::DefPsi}
    \Psi(t,x,s):=\exp_{\tilde X}((t,x),(0,0,s))=e^{tX'+xX''}(0,0,s)=e^{t^1X_1}\dots e^{t^rX_r}e^{x^1X_{r+1}}\dots e^{x^mX_{r+m}}(0^{r+m},s^1,\dots,s^q).
    \end{gather}

\ref{Item::ODE::BLFro1} and \ref{Item::ODE::BLFro4} are immediately since $\Phi(0)=\Psi(0)=(0,0,0)=p$ and $F(p)=F(0)=(0,0,0)$.

By construction \eqref{Eqn::ODE::BLFro::H} \eqref{Eqn::ODE::BLFro::DefPhi} and \eqref{Eqn::ODE::BLFro::DefPsi}, we know 
$$\textstyle\frac{\partial\Phi}{\partial t}(0)=T(0)=\frac{\partial\Psi}{\partial t}(0)=X'(0)=\Coorvec{u}\big|_p,\quad\frac{\partial\Phi}{\partial x}(0)=\frac{\partial\Psi}{\partial x}(0)=X''(0)=\Coorvec{v}\big|_p,\quad\frac{\partial\Phi}{\partial s}(0)=\frac{\partial\Psi}{\partial s}(0)=\frac{\partial\Gamma}{\partial s}(0)=\Coorvec{w}\big|_p.$$
So $\nabla\Phi(0)$ and $\nabla\Psi(0)$ both have full rank $r+m+q$. By the Inverse Function Theorem and shrinking $\Omega$ if necessary, we see that $\Phi:\Omega\to\Mf$ and $\Psi:\Omega\to\Mf$ are both $C^1$-regular parameterizations.

Therefore, $F=\Phi^\Inv$ is well-defined coordinate chart on $U:=\Phi(\Omega)\subseteq\Mf$. We write $F=(F',F'',F''')$ as the components to $\R^r_t$, $\R^m_x$ and $\R^q_s$ respectively.

We define $\tilde U:=\Psi(\Omega)$ and $G:=\Psi^\Inv:\tilde U\to\R^{r+m}_{(t,x)}\times\R^q_s$. Similarly, we write $G=(G',G'')$ where $G':\tilde U\to\R^{r+m}_{(t,x)}$ and $G'':\tilde U\to\R^q_s$.

\smallskip\noindent
\textit{Proof of \ref{Item::ODE::BLFro2}}: By Corollary \ref{Cor::ODE::ODERegCor} we know $\exp_T\in\Co^{\gamma+1,\gamma}_{t,(u,v,w)}$ and $\exp_{X''}\in\Co^{\beta+1,\beta}_{x,(u,v,w)}$. By Lemma \ref{Lem::Hold::CompofMixHold} \ref{Item::Hold::CompofMixHold::Comp}, since $\Gamma\in\Co^\kappa$, we know $\exp_{X''}(x,\Gamma(s))\in\Co^{\beta+1,\beta}_{x,s}$ and thus $\Phi\in\Co^{\gamma+1,\min(\gamma,\beta+1),\beta}_{t,x,s}$, in a neighborhood of $(0,0,0)$. 

Similarly, using $\exp_{\tilde X}\in\Co^{\beta+1,\beta}_{(t,x),(u,v,w)}$, we get $\Psi\in\Co^{\beta+1,\beta}_{(t,x),s}$. 

We have $\frac{\partial\Phi}{\partial t^j}=T_j\circ\Phi$, $j=1,\dots,r$. By Lemma \ref{Lem::Hold::CompofMixHold} \ref{Item::Hold::CompofMixHold::Comp} with $T\in\Co^\gamma$ and $\Phi\in\Co^{\gamma+1,\min(\gamma,\beta+1),\min(\gamma,\beta)}_{t,x,s}$ we get $\frac{\partial\Phi}{\partial t}=T\circ\Phi\in\Co^{\gamma,\min(\gamma,\beta+1),\min(\gamma,\beta)}_{t,x,s}$. 

We have $\frac{\partial\Phi}{\partial x^j}(t,x,s)=((e^{tT})_*X_j)\circ\Phi(t,x,s)$, $j=1,\dots,m$. For the pushforward formula we have
\begin{equation}\label{Eqn::ODE::BLFro::LieDer}
    ((e^{tT})_*X_j)(u,v,w)=(X_j\cdot\nabla_{u,v,w}(e^{tT}))(e^{-tT}(u,v,w)).
\end{equation}

We see that $(t,u,v,w)\mapsto (e^{tT}_*X_j)(u,v,w)$ is a $\Co^{\min(\gamma-1,\beta)}$-map. By Lemma \ref{Lem::Hold::CompofMixHold} \ref{Item::Hold::CompofMixHold::Comp} we get $\frac{\partial\Phi}{\partial x}\in\Co^{\min(\gamma-1,\beta)}$. Now the proof of \ref{Item::ODE::BLFro2} is complete.

\smallskip\noindent
\textit{Proof of \ref{Item::ODE::BLFro6}}: Since $\Phi=F^\Inv$, we have $F^*\Coorvec{t^j}=\frac{\partial\Phi}{\partial t^j}\circ F$ and $F^*\Coorvec{x^j}=\frac{\partial\Phi}{\partial x^j}\circ F$. By \eqref{Eqn::ODE::BLFro::DefPhi} we have $\frac{\partial\Phi}{\partial t^j}=T_j\circ\Phi$, so $F^*\Coorvec{t^j}=T_j\in\Co^\gamma(U;T\Mf)$ for $j=1,\dots,r$. By \ref{Item::ODE::BLFro2} we know $\Phi\in\Co^{\min(\gamma,\beta)}$ and $\frac{\partial\Phi}{\partial x^j}\in\Co^{\min(\gamma-1,\beta)}$, so $F\in\Co^{\min(\gamma,\beta)}$ and by composition we get $F^*\Coorvec{x^j}\in\Co^{\min(\gamma-1,\beta)}$ for $j=1,\dots,m$.

\medskip
\noindent \textit{Proof of \ref{Item::ODE::BLFro3} and \ref{Item::ODE::BLFro7}}: Again by $F^*\Coorvec{t^i}=\frac{\partial\Phi}{\partial t^i}\circ \Phi^\Inv$ and $F^*\Coorvec{x^j}=\frac{\partial\Phi}{\partial x^j}\circ \Phi^\Inv$ for $i=1,\dots,r$, $j=1,\dots,m$, we know \ref{Item::ODE::BLFro3} and \ref{Item::ODE::BLFro7} are equivalent. It suffices to prove \ref{Item::ODE::BLFro7}.

By \eqref{Eqn::ODE::BLFro::DefPhi} we have $F^*\Coorvec{t^i}=T_i$, $i=1,\dots, r$. Since by \eqref{Eqn::ODE::BLFro::GenforTX} $T_1,\dots,T_r$ form a local basis of $\V$ on $U\subseteq U_1$, we know $\V|_U$ is spanned by $F^*\Coorvec{t^1},\dots,F^*\Coorvec{t^r}$.

To prove $F^*\Coorvec{t^1},\dots,F^*\Coorvec{t^r},F^*\Coorvec{x^1},\dots,F^*\Coorvec{x^m}$ span $\E$ in a neighborhood of $p=0\in\Mf$, we need to show that
\begin{equation}\label{Eqn::ODE::BLFro::F'''=G''}
    F'''=G'',\quad\text{ in a neighborhood of }p=(0,0,0)\in\Mf.
\end{equation}
Once \eqref{Eqn::ODE::BLFro::F'''=G''} is done, since $F$ and $G$ are both coordinate chart, $(dF'''^1,\dots,dF'''^q)$ and $(dG''^1,\dots,dG''^q)$ are both collection of $q$-differentials that are linearly independent at every point in the domain,  we have that, in a neighborhood of $p\in\Mf$,
\begin{equation*}
    \textstyle\Span (F^*\Coorvec t,F^*\Coorvec x)=(\Span dF''')^\bot=(\Span dG'')^\bot=\Span (G^*\Coorvec t,G^*\Coorvec x)=\Span(X_1,\dots,X_{r+m})=\E.
\end{equation*}
This would prove \ref{Item::ODE::BLFro3} by shrinking $\Omega$ (and thus $U=\Phi(\Omega)$).

To prove \eqref{Eqn::ODE::BLFro::F'''=G''}, let $(t_0,x_0,s_0)\in\Omega$ be any point such that $\Phi(t_0,x_0,s_0)$ lays in the domain of $G$. It suffices to show $G''(t_0,x_0,s_0)=s_0$. To see this, by \eqref{Eqn::ODE::BLFro::DefPhi} and \eqref{Eqn::ODE::BLFro::DefPsi} we have $\Phi(0,x_0,s_0)=\Psi(0,x_0,s_0)$, so $G''(\Phi(0,x_0,s_0))=s_0$ since $G=\Psi^\Inv$.

Now for $1\le i\le r$ and for $t\in\R^r$ small, $\Coorvec{t^i}\Phi(t,x_0,s_0)=T_i\circ\Phi(t,x_0,s_0)$ is the linear combination of $X_1|_{\Phi(t,x_0,s_0)},\dots,X_{r+m}|_{\Phi(t,x_0,s_0)}\in \E_{\Phi(t,x_0,s_0)}=(\Span dG'')^\bot|_{\Phi(t,x_0,s_0)}$. Therefore $(\Phi_*\Coorvec {t^i})G''^j\equiv0$ for $1\le i\le r$, $1\le j\le q$, in a neighborhood of $0\in\Mf$, which means $t\mapsto G''(\Phi(t,x_0,s_0))$ is constant.

Since $G''(\Phi(0,x_0,s_0))=s_0$ and $t_0$ is also small, we conclude that $G''(\Phi(t_0,x_0,s_0))=s_0$, completing the proof of \eqref{Eqn::ODE::BLFro::F'''=G''} and hence of \ref{Item::ODE::BLFro3}.

\medskip\noindent
\textit{Proof of \ref{Item::ODE::BLFro5}}:
Clearly $F'\in\Co^\kappa(U;\R^r)$, because by \eqref{Eqn::ODE::BLFro::GenforTX} and \eqref{Eqn::ODE::BLFro::DefPhi} we have
\begin{equation}\label{Eqn::ODE::BLFro::F'=u}
    F'(u,v,w)=u,\quad\text{ on }U\subseteq\Mf.
\end{equation}

Similarly, $G'(u,v,w)=(u,v)$ as well.

To see $F''\in\Co^\gamma(U;\R^m)$, we write $F''$ in terms of $\exp_T$. Here we denote $\tilde v:\Mf\to\R^m_v$, $\tilde v(u,v,w):=v$ as the natural projection\footnote{We abuse of notation here. It is more convenient to use $u,v,w$ as variables in the proof. Thus we use $\tilde v$ for $v$ in \eqref{Eqn::ODE::BLFro::H}.}. Using the proven fact $G'(u,v,w)=(u,v)$ we have $\tilde v(\Psi(0,x,s))=\tilde v(e^{xX''}(\Gamma(s)))=x$. 

Note that by \eqref{Eqn::ODE::BLFro::DefPhi} and \eqref{Eqn::ODE::BLFro::DefPsi} we have $\Phi(0,x,s)=\Psi(0,x,s)$ and $\Phi(0,x,s)=e^{-t\cdot T}(\Phi(t,x,s))$. Therefore \begin{align*}
    F''(\Phi(t,x,s))=x=\tilde v(\Phi(0,x,s))=\tilde v(e^{-t\cdot T}(\Phi(t,x,s))=\tilde v(e^{-F'(\Phi(t,x,s))\cdot T}(\Phi(t,x,s)).
\end{align*}
Using $(u,v,w)=\Phi(t,x,s)$ and that $F'(u,v,w)=u$ from \eqref{Eqn::ODE::BLFro::F'=u}, we conclude that in the domain
\begin{equation*}
    F''(u,v,w)=\tilde v(e^{-u\cdot T}(u,v,w))=\tilde v\circ\exp_T(-u,(u,v,w)).
\end{equation*}
Thus $F''\in\Co^\gamma$ since $\exp_T\in\Co^\gamma$.

Finally by \eqref{Eqn::ODE::BLFro::F'''=G''} we have $F'''=G''$ in a neighborhood of $0\in\Mf$. Since $G=\Psi^\Inv\in\Co^\beta$ and $G''$ is a component of $G$, we see that $F'''\in\Co^\beta$. 

By shrinking $U$ if necessary, we conclude that $F'\in\Co^\kappa$, $F''\in\Co^\gamma$ and $F'''\in\Co^\beta$, finishing the proof of \ref{Item::ODE::BLFro5} and thus the whole proof.
\end{proof}
\begin{remark}
    It would be natural to generalize the result to the multi-layers case, where we have a flag of involutive tangent subbundles $0\le\V_1\le\dots\le \V_k\le T\Mf$. And the parameterization $\Phi(t^{(1)},\dots,t^{(k)},s)=e^{t^{(1)}T_{(1)}}\dots e^{t^{(k)}T_{(k)}}(\Gamma(s))$ gives the parameterization with desired regularities, where $T_{(j)}$ are collections of vector fields such that $T_{(1)},\dots,T_{(k)}$ are linearly independent and $\V_j=\Span_\R(T_{(1)},\dots,T_{(j)})$ for each $j$. We will not use this generalization in the paper.
\end{remark}

Using a similar but simpler construction from Theorem \ref{Thm::ODE::BLFro}, we can prove a holomorphic Frobenius theorem with parameter dependence. Here we endow $\C^n$ and $\R^q$ with standard (real and complex) coordinate system $\zeta=(\zeta^1,\dots,\zeta^n)$ and $s=(s^1,\dots,s^q)$ respectively.

\begin{prop}[Holomorphic Frobenius theorem with parameter]\label{Prop::ODE::ParaHolFro}
    Let $\tilde U_0\subseteq\C^n$ and $V_0\subseteq\R^q$. Let $\beta>0$ and $1\le m\le n$. And let $Z_j=\sum_{k=1}^nb_j^k(\zeta,s)\Coorvec{\zeta^k}$, $j=1,\dots,m$ be complex vector fields on $\tilde U_0\times V_0$ such that 
    \begin{itemize}[nolistsep]
        \item $Z_1,\dots,Z_m$ are linearly independent and pairwise commutative at every point in $\tilde U_0\times V_0$
        \item $b_j^k\in\Co^\beta_\loc(\tilde U_0\times V_0;\C^n)$ are holomorphic in $w$-variable.
    \end{itemize}
    
    Then for any $(\zeta_0,s_0)\in \tilde U_0\times V_0$ there is a neighborhood $\tilde U\times V\subseteq \tilde U_0\times V_0$ of $(\zeta_0,s_0)$ and a $\Co^\beta$-map $\widetilde G'':\tilde U\times V\to\C^{n-m}$, such that 
    \begin{enumerate}[parsep=-0.3ex,label=(\roman*)]
        \item\label{Item::ODE::ParaHolFro::0}$\widetilde G''(\zeta_0,s_0)=0$.
        \item\label{Item::ODE::ParaHolFro::Reg} $\widetilde G''\in\Co^\infty\Co^\beta(\tilde U,V;\C^{n-m})$ is holomorphic in $\zeta$-variable. In particular $\widetilde G''\in\Co^{\infty,\beta}_{\zeta,s,\loc}$ and $\nabla_\zeta \widetilde G''\in\Co^{\infty,\beta}_{\zeta,s,\loc}$. 
        \item\label{Item::ODE::ParaHolFro::Span} For each $s\in V$, the contangent subbundle $\Span(Z_1(\cdot,s),\dots,Z_m(\cdot,s))^\bot|_{\tilde U}\le \C T^*\tilde U$ has a holomorphic local basis $(d\widetilde G''^1(\cdot,s),\dots,d\widetilde G''^{n-m}(\cdot,s),d\bar \zeta^1,\dots,d\bar \zeta^n)$.
    \end{enumerate}
    
    Let  $U:=\tilde U\cap\R^n\subseteq\R^n$ be the open set in the real space endowed with the real coordinates $(\xi^1,\dots,\xi^n)$ where $\xi=\re\zeta$. Let $X_j:=\sum_{k=1}^nb_j^k(\xi+i0,s)\Coorvec{\xi^k}$ for $1\le j\le m$ be the ``real domain restriction'' of $Z_1,\dots,Z_m$. Then
    \begin{enumerate}[parsep=-0.3ex,label=(\roman*)]\setcounter{enumi}{3}
        \item\label{Item::ODE::ParaHolFro::SpanR} For each $s\in V$, the cotangent subbundle $\Span(X_1(\cdot,s),\dots,X_m(\cdot,s))^\bot\le \C T^*U$ has a real-analytic local basis $\big(d(\widetilde G''^1(\cdot,s)|_U),\dots,d(\widetilde G''^{n-m}(\cdot,s)|_U)\big)$.
    \end{enumerate}
\end{prop}
\begin{remark}
We state \ref{Item::ODE::ParaHolFro::Span} and \ref{Item::ODE::ParaHolFro::SpanR} in this way because the pointwise span of $dG''$ or $d(G''|_{U\times V})$ does not make sense when $\beta\le1$.

    When $\beta>1$, the result \ref{Item::ODE::ParaHolFro::SpanR} is the same as to say that the $\Co^\beta$-subbundle $\Span(X_1,\dots,X_m)\le \C T^*(U\times V)$ has a local basis $(d(\widetilde G''^1|_{U\times V}),\dots,d(\widetilde G''^{n-m}|_{U\times V}),ds^1,\dots,ds^q)$. Similar result holds for \ref{Item::ODE::ParaHolFro::Span}.
\end{remark}
\begin{proof}[Proof of Proposition \ref{Prop::ODE::ParaHolFro}]
Endow two complex spaces $\C^m$ and $\C^{n-m}$ with standard complex coordinate $z=(z^1,\dots,z^m)$ and $w=(w^1,\dots,w^{n-m})$.

Since $Z_1,\dots,Z_m$ are linearly independent, the span of $Z_1,\dots,Z_m$ is a rank $m$ complex subbundle on $\tilde U_0\times V_0$. Hence we can find an affine complex linear map $\Gamma:\C^{n-m}_w\to\C^n_\zeta$ such that $\Gamma(0)=\zeta_0$ and 
\begin{equation}\label{Eqn::ODE::PfParaHolFro::LinInd}
    \partial_{w^1}\Gamma,\dots,\partial_{w^{n-m}}\Gamma,Z_1|_{(\zeta_0,s_0)},\dots,Z_m|_{(\zeta,s)}\in\C T_{\zeta_0} \tilde U_0\text{ are $\C$-linearly independent}.
\end{equation}

We define $\Co^\beta$-maps $\Psi(z,w,s)$ to $\tilde U_0$ and $\widehat\Psi(z,w,s)$ to $\tilde U_0\times V_0$ as
\begin{equation*}
    \Psi(z,w,s):=e^{z\cdot Z}(\Gamma(w),s+s_0)=\big(e^{z^1 Z_1(\cdot,s+s_0)}\dots e^{z^m Z_m(\cdot,s+s_0)}(\Gamma(w))\big),\quad\widehat\Psi(z,w,s):=(\Psi(z,w,s),s+s_0).
\end{equation*}
Here $e^{z^jZ_j(\cdot,s)}(\zeta)=\Phi^{Z_j}(z_j,\zeta,s)$ is given by \eqref{Eqn::ODE::HoloFlow::EqnDef} and \eqref{Eqn::ODE::HoloFlow::DefExpZ}.

Clearly $\Psi(0,0,0)=(\zeta_0,s_0)$. By Lemma \ref{Lem::ODE::HoloFlow} \ref{Item::ODE::HoloFlow::DefReg} and taking compositions, $\Psi$ is defined in a neighborhood $\Omega'\times\Omega''\times\Omega'''\subseteq\C^m_z\times\C^m_w\times\R^q_s$ of $0$, such that $\Psi\in\Co^\infty\Co^\beta(\Omega'\times\Omega'',\Omega''';\tilde U_0)$ and  $\Psi$ is holomorphic in $(z,w)$.

Since $\Coorvec{z^j}e^{z^jZ_j}=Z_j\circ e^{z^jZ_j}$ from \eqref{Eqn::ODE::HoloFlow::EqnDef}, by Lemma \ref{Lem::ODE::FlowComm} we have
\begin{equation}\label{Eqn::ODE::PfParaHolFro::DPsi}
    \textstyle\Coorvec{z^j}\widehat\Psi(z,w,s)=Z_j\circ\widehat\Psi(z,w,s),\quad 1\le j\le m,\quad (z,w,s)\in\Omega.
\end{equation}

By condition \eqref{Eqn::ODE::PfParaHolFro::LinInd} we see that $\nabla_{z,w}\Psi(\cdot,0)$ is of full rank in a neighborhood of $0$. 
Therefore, by Lemma \ref{Lem::Hold::CompofMixHold} \ref{Item::Hold::CompofMixHold::InvFun} and shrinking $\Omega'\times\Omega''\times\Omega'''$ if necessary, $\Psi(\cdot,s):\Omega'\times\Omega''\to\tilde U_0$ is locally biholomorphic for each $s\in\Omega'''$ and $\widehat\Psi^\Inv\in\Co^{\infty,\beta}_\loc(\tilde U,V;\C^n\times\R^q)$ for some neighborhood $\tilde U\times V\subseteq\tilde U_0\times V_0$ of $(\zeta_0,s_0)$. Thus we get the well-defined map $\widetilde G(\zeta,s):=\Psi(\cdot,s-s_0)^\Inv(\zeta)$ for $(\zeta,s)\in\tilde U\times V$ with $\widetilde G\in\Co^{\infty,\beta}_{\zeta,s}$.

Write $\widetilde G=(\widetilde G',\widetilde G'',\widetilde G''')$ with respect to the $z$, $w$, $s$-coordinate components. We are going to show that $\widetilde G'':\tilde U\times V\to\C^{n-m}$ is as desired.

Immediately $\widetilde G''(\zeta_0,s_0)=0$ because $\widehat\Psi^\Inv(\zeta,s)=(\widetilde G''(\zeta,s),s-s_0)$ and $\widehat\Psi(0)=(\zeta_0,s_0)$. We have \ref{Item::ODE::ParaHolFro::0}.

Since $\widetilde G\in\Co^{\infty,\beta}_{\zeta,s}(\tilde U,V;\C^n)$ and $\Psi$ is holomorphic in $z,w$, we know $\widetilde G$ is holomorphic in $\zeta$. By Lemma \ref{Lem::Hold::NablaHarm} we get $\widetilde G\in\Co^\infty_\zeta\Co^\beta_s$. In particular $\widetilde G''\in\Co^\infty_\zeta\Co^\beta_s$ is holomorphic in $\zeta$, giving \ref{Item::ODE::ParaHolFro::Reg}.

\medskip To prove \ref{Item::ODE::ParaHolFro::Span} and \ref{Item::ODE::ParaHolFro::SpanR} we can fix a $s\in V$. Thus without loss of generality $V=\{\ast\}$ is a singleton, and we can assume $Z(\zeta,s)=Z(\zeta)$ and $\widetilde G(\zeta,s)=\widetilde G(\zeta)$ in the following argument.

\noindent\smallskip
\textit{Proof of \ref{Item::ODE::ParaHolFro::Span}}: Now $\widetilde G:\tilde U\to\C^m$ is a complex (holomorphic) coordinate chart, so $d\widetilde G''^1,\dots,d\widetilde G''^{n-m},d\bar \zeta^1,\dots,d\bar \zeta^n$ are linearly independent differentials on $\tilde U$.

By \eqref{Eqn::ODE::PfParaHolFro::DPsi} we have $Z_j=\Psi_*\Coorvec{z^j}=\widetilde G^*\Coorvec{z^j}$, thus $Z_j\widetilde G''^k=(\widetilde G^*\Coorvec{z^j})(\widetilde G^*dw^k)=0$ all vanish for $1\le j\le m$ and $1\le k\le n-m$.
This finishes the proof of \ref{Item::ODE::ParaHolFro::Span}.

\medskip\noindent\textit{Proof of \ref{Item::ODE::ParaHolFro::SpanR}}: By \ref{Item::ODE::ParaHolFro::Span}, we know the $n\times (n-m)$ matrix $(\partial_{\zeta^j}\widetilde G''^k)_{\substack{1\le j\le n\\1\le k\le n-m}}$ has full rank $n-m$ at every point in $\tilde U$.
By \ref{Item::ODE::ParaHolFro::Reg} $\widetilde G''$ is holomorphic in $\zeta$, so $\partial_\xi \widetilde G''=\partial_\zeta \widetilde G''$. Therefore, the matrix-valued function $(\partial_{\xi^j}\widetilde G''^k)_{n\times(n-m)}:U\to\C^{n\times (n-m)}$ still has rank $n-m$ at every point. Since $d(\widetilde G''^k|_{U})=\sum_{j=1}^n\frac{\partial \widetilde G''^k}{\partial\xi^j}d\xi^j$, we conclude that $d(\widetilde G''|_{U})$ is a collection of $(n-m)$-differentials that are linearly independent on $U$.

Since $\partial_\xi \widetilde G''=\partial_\zeta \widetilde G''$ and $Z_j\widetilde G''=0$ from above, we know $X_j(\widetilde G''|_{U})=0$ as well, for $1\le j\le m$. Therefore, $d(\widetilde G''|_{U})$ annihilates $X_1,\dots,X_m$ on $U$. This completes the proof of \ref{Item::ODE::ParaHolFro::SpanR}.
\end{proof}

\section{Some Elliptic PDEs}\label{Section::SecPDE}
In this section we will prove two results which occur in the key steps of Malgrange's factorization.
\subsection{Existence of a particular PDE}\label{Section::ExistPDE}
In this subsection we use $\tau=(\tau^1,\dots,\tau^r)$, $w=(w^1,\dots,w^m)$, $s=(s^1,\dots,s^q)$ as the standard (real or complex) coordinate system for $\R^r$, $\C^m$ and $\R^q$ respectively. We consider the unit balls $\B^{r+2m}\subset\R^r_\tau\times\C^m_w$ and $\B^q\subset\R^q_s$.

\begin{conv}\label{Conv::PDE::ConvofExtPDE}
We denote $\Ic:\B^{r+2m}\times\B^q\to\B^{2m}$, $\Ic(\tau,w,s):=w$ as the projection map. 

For a matrix $B=(b_j^k)$, we use $b_j^k$ to denote the $j$-th row $k$-th column entry of $B$.

We use $A=\begin{pmatrix}	A'\\A''	\end{pmatrix}:\B^{r+2m}_{\tau,w}\times \B^q_s\to\C^{(r+m)\times m}$ as a matrix map, $H''=(H''^1,\dots,H''^m):\B^{r+2m}_{\tau,w}\times \B^q_s\to \C^m$ as a vector-valued map whose image is set to be a $m$-dimensional row vector. We use 
\begin{equation}\label{Eqn::PDE::GradH}
    H''_\tau=\begin{pmatrix}\frac{\partial H''^1}{\partial \tau^1}&\cdots&\frac{\partial H''^m}{\partial \tau^1}\\\vdots&\ddots&\vdots\\\frac{\partial H''^1}{\partial \tau^r}&\cdots&\frac{\partial H''^m}{\partial \tau^r}\end{pmatrix},\quad H''_w=\begin{pmatrix}\frac{\partial H''^1}{\partial w^1}&\cdots&\frac{\partial H''^m}{\partial w^1}\\\vdots&\ddots&\vdots\\\frac{\partial H''^1}{\partial w^m}&\cdots&\frac{\partial H''^m}{\partial w^m}\end{pmatrix},\quad\nabla_{\tau,w} H''=\begin{pmatrix}H''_\tau\\H''_w\\H''_{\bar w}\end{pmatrix}
\end{equation}

Let $\Psi,\Theta,\Lambda$ be matrix-valued functions of $H''$ and $A$ as follow:
\begin{equation}\label{Eqn::PDE::matrixfun}
\Psi[H'']:=\begin{pmatrix}I&H''_\tau&\bar H''_\tau\\&H''_ w&\bar H''_ w\\&H''_{\bar w}&\bar H''_{\bar w}\end{pmatrix},\ \Theta[A;H'']:=\begin{pmatrix}I&H''_\tau+ A'H''_{\bar w}\\&H''_ w+ A''H''_{\bar w}\end{pmatrix},\  \Lambda[A;H'']:=\Theta[A;H'']^{-1}\begin{pmatrix}\bar H''_\tau+ A'\bar H''_{\bar w}\\\bar H''_ w+ A''\bar H''_{\bar w}\end{pmatrix}.\end{equation}

\end{conv}

\begin{remark}One can see that $\Psi,\Theta,\Lambda$ are all rational functions in the components of $A$ and $\nabla_{\tau,w}H''$.
\end{remark} 

We are going to consider the following PDE system, whose deduction is in Lemma \ref{Lem::Key::NewGen}: for $l=1,\dots,m$, $(\tau,w)\in\B^{r+2m}$ and $s\in\B^q$,
\begin{equation}\label{Eqn::PDE::ExistenceH}\begin{aligned}
		&\sum_{j=1}^r\left(\begin{pmatrix}I&H''_\tau&\bar H''_\tau\\&H''_ w&\bar H''_ w\\&H''_{\bar w}&\bar H''_{\bar w}\end{pmatrix}^{-1}\!\!\begin{pmatrix}\partial_{\tau}\\\partial_{ w}\\\partial_{\bar w}\end{pmatrix}\right)_j\left(\begin{pmatrix}I&H''_\tau+ A'H''_{\bar w}\\&H''_ w+ A''H''_{\bar w}\end{pmatrix}^{-1}\!\!\begin{pmatrix}\bar H''_\tau+ A'\bar H''_{\bar w}\\\bar H''_ w+ A''\bar H''_{\bar w}\end{pmatrix}\right)_j^l\Bigg|_{(\tau, w;s)}
		\\
		+&\sum_{k=1}^m\left(\begin{pmatrix}I&H''_\tau&\bar H''_\tau\\&H''_ w&\bar H''_ w\\&H''_{\bar w}&\bar H''_{\bar w}\end{pmatrix}^{-1}\!\!\begin{pmatrix}\partial_{\tau}\\\partial_{ w}\\\partial_{\bar w}\end{pmatrix}\right)_{k+r+m}\!\!\left(\begin{pmatrix}I&H''_\tau+ A'H''_{\bar w}\\&H''_ w+ A''H''_{\bar w}\end{pmatrix}^{-1}\!\!\begin{pmatrix}\bar H''_\tau+ A'\bar H''_{\bar w}\\\bar H''_ w+ A''\bar H''_{\bar w}\end{pmatrix}\right)_{k+r}^l\Bigg|_{(\tau, w;s)}=0.
		\end{aligned}\end{equation}

Recall $\partial_\tau=[\partial_{\tau^1},\dots,\partial_{\tau^r}]^\top$ in Section \ref{Section::Convention} and the space $\Co^{\alpha+1}_{\tau,w}L^\infty_s\cap\Co^1_{\tau,w}\Co^\beta_s$ in Definition \ref{Defn::Hold::BiHold}. In this subsection we prove the following: 
\begin{prop}\label{Prop::PDE::ExistPDE}
	Using Convention \ref{Conv::PDE::ConvofExtPDE}, for any $\alpha\in(1,\infty)$, $\beta\in(0,\alpha+1]$ and $\eps>0$, there is a $\delta=\delta(\eps,\alpha,\beta,r,m,q)>0$, such that if $A\in\Co^{\alpha,\beta}(\B^{r+2m},\B^q;\C^{(r+m)\times m})$ satisfies $\|A\|_{\Co^{\alpha,\beta}_{(\tau,w),s}}<\delta$, then there is a $H''\in\Co^{\alpha+1}_{\tau,w}L^\infty_s\cap\Co^1_{\tau,w}\Co^\beta_s(\B^{r+2m}, \B^q;\C^m)$ such that 
	\begin{enumerate}[parsep=-0.3ex,label=(\roman*)]
	    \item\label{Item::PDE::ExistPDE::PDEEqn}$H''$ solves the following PDE system \eqref{Eqn::PDE::ExistenceH} and satisfies the boundary condition $H''|_{(\partial\B^{r+2m})\times\B^q}=\Ic$.
	    \item\label{Item::PDE::ExistPDE::HisDiffeo} Let $H(\tau,w,s):=(\tau,H''(\tau,w,s),s)$. Then $H:\B^{r+2m}\times\B^q\to\B^{r+2m}\times\B^q$ is homeomorphism, and $H(\frac14\B^{r+2m}\times\frac14\B^q)\subseteq\frac12\B^{r+2m}\times\frac12\B^q$.
	    \item\label{Item::PDE::ExistPDE::Phi} Endow $\R^r$ and $\C^m$ with another (real or complex) coordinate system $\sigma=(\sigma^1,\dots,\sigma^r)$ and $\zeta=(\zeta^1,\dots,\zeta^m)$. Let $\tilde\Phi:=H^\Inv:\B^{r+2m}_{\sigma,\zeta}\times\B^q_s\to\B^{r+2m}_{\tau,w}\times\B^q_s$ be the inverse map. Then  $\Lambda[A;H'']\circ\tilde\Phi\in\Co^\alpha_{\sigma,\zeta} L^\infty_s\cap\Co^{-1}_{\sigma,\zeta}\Co^\beta_s(\B^{r+2m},\B^q;\C^{(r+m)\times m})$. Moreover
	    \begin{equation}\label{Eqn::PDE::ExistPDE::BddLambda}
	        \|\Lambda[A;H'']\circ\tilde\Phi\|_{\Co^\alpha L^\infty \cap\Co^{-1} \Co^\beta(\B^{r+2m},\B^q;\C^{(r+m)\times m})}<\eps.
	    \end{equation}
	\end{enumerate}
\end{prop}

As an immediate corollary to Proposition \ref{Prop::PDE::ExistPDE} we have:
\begin{cor}\label{Cor::PDE::CorofExistPDE}
	The map $H''$ in the consequence of Proposition \ref{Prop::PDE::ExistPDE} satisfies $H''\in\Co^{\alpha+1,\beta}(\B^{r+2m},\B^q;\C^m)$ and $\nabla_{\tau,w} H''\in\Co^{\alpha,\beta-}(\B^{r+2m},\B^q;\C^{(r+2m)\times m})$.
\end{cor}
\begin{proof}By assumption $H''\in\Co^{\alpha+1}L^\infty\cap\Co^1\Co^\beta$.
Remark \ref{Rmk::Hold::SimpleHoldbyCompFact} we have $\Co^{\alpha+1}L^\infty\cap\Co^1\Co^\beta\subset\Co^{\alpha+1,\beta}$, thus $H''\in\Co^{\alpha+1,\beta}_{(\tau,w),s}$. 
By Remark \ref{Rmk::Hold::RmkforBiHold} \ref{Item::Hold::RmkforBiHold::Interpo} we have $\Co^\alpha L^\infty\cap\Co^0\Co^\beta\subset\Co^\eps\Co^{\beta-\eps}\subset L^\infty\Co^{\beta-\eps}$ for every $0<\eps<\alpha$. Thus $\nabla_{\tau,w} H''\in \Co^{\alpha}_{\tau,w}L^\infty_s\cap(\bigcap_{0<\eps<\beta} L^\infty_{\tau,w}\Co^{\beta-\eps}_s)=\Co^{\alpha,\beta-}_{(\tau,w),s}$.
\end{proof}

Before we prove the Proposition \ref{Prop::PDE::ExistPDE}, we need to show the PDE \eqref{Eqn::PDE::ExistenceH} is well-defined, in the sense that $\Psi[H''](\tau,w,s)$ and $\Theta[A;H''](\tau,w,s)$ are invertible matrices for all $(\tau,w,s)\in\B^{r+2m}\times \B^q$. 

For $\alpha,\beta>0$ and $\eps>0$, we define spaces
	\begin{equation}\label{Eqn::PDE::MatisInv::SpaceUV}
	    \begin{aligned}
    &\U_{\alpha,\beta,\eps}:=\{A\in\Co^{\alpha,\beta}(\B^{r+2m},\B^q;\C^{(r+m)\times m}):\|A\|_{\Co^{\alpha,\beta}}<\eps\},\\
    &\V_{\alpha,\beta,\eps}:=\{H''\in\Co^{\alpha+1}L^\infty\cap\Co^1\Co^\beta(\B^{r+2m},\B^q;\C^m):H''|_{(\partial\B^{r+2m})\times\B^q}\equiv\Ic,\ \|H''-\Ic\|_{\Co^{\alpha+1}L^\infty\cap\Co^1\Co^\beta}<\eps\},
        \end{aligned}
	\end{equation}
	with metrics induced by the norm structures of their respective ambient Banach spaces.
\begin{lem}\label{Lem::PDE::MatisInv}
    For any $\alpha>1$ and $\beta>0$ there is a $\eps_1=\eps_1(r,m,q,\alpha,\beta)>0$ such that
    \begin{enumerate}[parsep=-0.3ex,label=(\roman*)]
        \item\label{Item::PDE::MatisInv::Inv}For every $A\in\U_{\alpha,\beta,\eps_1}$, $H''\in \V_{\alpha,\beta,\eps_1}$, $(\tau,w)\in\B^{r+2m}$ and $s\in\B^q$, the matrices $\Psi[H''](\tau,w,s)$ and $\Theta[A;H''](\tau,w,s)$ given in \eqref{Eqn::PDE::matrixfun} are both invertible.
        \item\label{Item::PDE::MatisInv::BddLambda} For every $A\in\U_{\alpha,\beta,\eps_1}$ and $H''\in\V_{\alpha,\beta,\eps_1}$,
        \begin{equation}\label{Eqn::PDE::MatisInv::BddLambda}
	    \|\Lambda[A;H'']\|_{\Co^{\alpha}_{\tau,w}L^\infty_s\cap\Co^0_{\tau,w}\Co^\beta_s}\le\eps_1^{-1}(\|H''-\Ic\|_{\Co^{\alpha+1}_{\tau,w}L^\infty_s\cap\Co^1_{\tau,w}\Co^\beta_s}+\|A\|_{\Co^{\alpha,\beta}_{(\tau,w),s}}).
	\end{equation}
	
	    \item\label{Item::PDE::MatisInv::Diff} Denote by $T[A;H'']^l$ the left hand side of \eqref{Eqn::PDE::ExistenceH} for $l=1,\dots,m$, and $T[A;H'']:=(T[A;H'']^l)_{l=1}^m$ as the vector-valued function. Then $T:\U_{\alpha,\beta,\eps_1}\times\V_{\alpha,\beta,\eps_1}\to \Co^{\alpha-1}L^\infty\cap\Co^{-1}\Co^\beta(\B^{r+2m},\B^q;\C^m)$ is Fr\'echet differentiable. Moreover for every $H''\in\V_{\alpha,\beta,\eps_1} $,
	    \begin{equation}\label{Eqn::PDE::MatisInv::TangT}
	        \|T[0;H'']-(\Delta_\tau+\Box_w)\bar H''\|_{\Co^{\alpha-1}L^\infty\cap\Co^{-1}\Co^\beta(\B^{r+2m},\B^q;\C^m)}\le\eps_1^{-1}\|H''-\Ic\|_{\Co^{\alpha+1}L^\infty\cap\Co^1\Co^\beta(\B^{r+2m},\B^q;\C^m)}^2.
	    \end{equation}
	    Here $\Delta_\tau+\Box_w=\sum_{j=1}^r\frac{\partial^2}{\partial\tau^j\partial\tau^j}+\sum_{k=1}^m\frac{\partial^2}{\partial w^k\partial\bar w^k}$ is the mixed real and complex Laplacian.
    \end{enumerate}
\end{lem}
\begin{proof} We let $\eps_1$ be a small constant which may change from line to line.

Note that when $H''=\Ic$ and $A=0$, we have
$$\nabla_{\tau,w}\Ic=\begin{pmatrix}0_{r\times m}\\I_m\\0_{m\times m}\end{pmatrix},\quad\Psi[\Ic]=\begin{pmatrix}I_r\\&I_m\\&&I_m\end{pmatrix}=I_{r+2m},\quad\Theta[0;\Ic]=\begin{pmatrix}I_r+ 0_{r\times r}&0_{r\times m}\\0_{m\times r}&I_m+0_{m\times m}\end{pmatrix}=I_{r+m}.$$

By cofactor representations of the matrices, $\Psi^{-1}$, $\Theta^{-1}$ and $\Lambda$ are rational functions in the components of $A$, $\nabla_{\tau,w}H''$ and $\nabla_{\tau,w}\bar H''$, which can be expressed as power series expansions in the components of $A$, $\nabla_{\tau,w}H''$ and $\nabla_{\tau,w}\bar H''$ at the point $A=0$ and $\nabla_{\tau,w}H''=\nabla_{\tau,w}\Ic$ and $\nabla_{\tau,w}\bar H''=\nabla_{\tau,w}\Ic$.

By Corollary \ref{Cor::Hold::CorMult} \ref{Item::Hold::CorMult::0} $\Co^{\alpha}_{\tau,w}L^\infty_s\cap\Co^0_{\tau,w}\Co^\beta_s$ is closed under products. By Corollary \ref{Cor::Hold::CramerMixed}, inverting a  $\Co^{\alpha}_{\tau,w}L^\infty_s\cap\Co^0_{\tau,w}\Co^\beta_s$-matrix is a real-analytic map near the identity matrix. Therefore there is a $\tilde\eps_1>0$ such that if $\|A\|_{\Co^\alpha L^\infty\cap\Co^0\Co^\beta}<\tilde\eps_1$ and $\|\nabla_{\tau,w}(H''-\Ic)\|_{\Co^\alpha L^\infty\cap\Co^0\Co^\beta}<\tilde\eps_1$, then $\Psi[H'']$ and $\Theta[A;H'']$ are invertible matrices at every point in $\B^{r+2m}\times\B^q$, and moreover as maps between $\Co^{\alpha}_{\tau,w}L^\infty_s\cap\Co^0_{\tau,w}\Co^\beta_s$-spaces,
\begin{align*}
&\Psi^{-1}:\{\|\nabla_{\tau,w}(H''-\Ic)\|_{\Co^{\alpha}_{\tau,w}L^\infty_s\cap\Co^0_{\tau,w}\Co^\beta_s}<\tilde\eps_1\}\to\Co^{\alpha}_{\tau,w}L^\infty_s\cap\Co^0_{\tau,w}\Co^\beta_s,
\\
    &\Lambda:\{\|A\|_{\Co^{\alpha}L^\infty\cap\Co^0\Co^\beta}<\tilde\eps_1\}\times\{\|\nabla_{\tau,w}(H''-\Ic)\|_{\Co^{\alpha}L^\infty\cap\Co^0\Co^\beta}<\tilde\eps_1\}\to\Co^{\alpha}_{\tau,w}L^\infty_s\cap\Co^0_{\tau,w}\Co^\beta_s,
\end{align*}
are both real-analytic.

Clearly $\Co^{\alpha,\beta}_{(\tau,w),s}\hookrightarrow \Co^{\alpha}_{\tau,w}L^\infty_s\cap\Co^0_{\tau,w}\Co^\beta_s $ (see Remark \ref{Rmk::Hold::SimpleHoldbyCompFact}) and $\nabla_{\tau,w}:\Co^{\alpha+1}_{\tau,w}L^\infty_s\cap\Co^1_{\tau,w}\Co^\beta_s\to\Co^{\alpha}_{\tau,w}L^\infty_s\cap\Co^0_{\tau,w}\Co^\beta_s$. Thus there is a  $\tilde\eps_2<C^{-1}\tilde\eps_1$ where $C=C(\alpha,\beta,r,m,q)>0$ is a large constant, such that
\begin{equation}\label{Eqn::PDE::MatisInv::Tmp1}
    A\in\U_{\alpha,\beta,\tilde\eps_2},\ H''\in\V_{\alpha,\beta,\tilde\eps_2}\quad\Longrightarrow\quad\|A\|_{\Co^{\alpha}_{\tau,w}L^\infty_s\cap\Co^0_{\tau,w}\Co^\beta_s}+\|\nabla_{\tau,w}(H''-\Ic)\|_{\Co^{\alpha}_{\tau,w}L^\infty_s\cap\Co^0_{\tau,w}\Co^\beta_s}<\tilde\eps_1.
\end{equation}

Therefore when $\eps\le\tilde\eps_2$ we see that $\Psi[H'']$ and $\Theta[A;H'']$ are pointwise invertible in $\B^{r+2m}\times\B^q$, finishing the proof of \ref{Item::PDE::MatisInv::Inv}. Moreover we have the following Fr\'echet differentiable maps (which are in fact real-analytic):
\begin{equation}\label{Eqn::PDE::MatisInv::PsiLambdaDiff}
    \begin{aligned}
    \Psi^{-1}:\U_{\alpha,\beta,\tilde\eps_2}\to\Co^{\alpha}L^\infty\cap\Co^0\Co^\beta(\B^{r+2m},\B^q;\C^{(r+2m)\times (r+2m)}),
    \\
    \Lambda:\U_{\alpha,\beta,\tilde\eps_2}\times\V_{\alpha,\beta,\tilde\eps_2}\to\Co^{\alpha}L^\infty\cap\Co^0\Co^\beta(\B^{r+2m},\B^q;\C^{(r+2m)\times m}).
\end{aligned}
\end{equation}

Since $\Psi[\Ic]=I_{r+2m}$ and $\Lambda[0;\Ic]=0$, taking the first order terms of the power expansion we have
\begin{gather}
\label{Eqn::PDE::MatisInv::PfBddPsi}
\|\Psi[H'']^{-1}-I_{r+2m}\|_{\Co^{\alpha}_{\tau,w}L^\infty_s\cap\Co^0_{\tau,w}\Co^\beta_s}\lesssim\|\nabla_{\tau,w}H''-\nabla_{\tau,w}\Ic\|_{\Co^{\alpha}_{\tau,w}L^\infty_s\cap\Co^0_{\tau,w}\Co^\beta_s},
\\
\label{Eqn::PDE::MatisInv::PfBddLambda}
    \|\Lambda[A;H'']\|_{\Co^{\alpha}_{\tau,w}L^\infty_s\cap\Co^0_{\tau,w}\Co^\beta_s}\lesssim_{\alpha}\|A\|_{\Co^{\alpha}_{\tau,w}L^\infty_s\cap\Co^0_{\tau,w}\Co^\beta_s}+\|\nabla_{\tau,w}H''-\nabla_{\tau,w}\Ic\|_{\Co^{\alpha}_{\tau,w}L^\infty_s\cap\Co^0_{\tau,w}\Co^\beta_s},
\end{gather}
whenever $A\in\U_{\alpha,\beta,\tilde\eps_2}$ and $H''\in\V_{\alpha,\beta,\tilde\eps_2}$. Therefore by \eqref{Eqn::PDE::MatisInv::PfBddLambda} with possibly shrinking $\eps_1$, we get \eqref{Eqn::PDE::MatisInv::BddLambda} and prove \ref{Item::PDE::MatisInv::BddLambda}.

\medskip
We now prove \ref{Item::PDE::MatisInv::Diff}. From the expression \eqref{Eqn::PDE::ExistenceH}, $T[A;H'']$ is the linear combinations to the components of $\Psi[H'']^{-1}\otimes\nabla_{\tau,w}\Lambda[A;H'']$. As the maps in \eqref{Eqn::PDE::MatisInv::PsiLambdaDiff}, $\Psi^{-1}$ and $\Lambda$ are both Fr\'echet differentiable. Therefore 
$$\big[(A,H'')\mapsto \nabla_{\tau,w}\Lambda[A;H'']\big]:\U_{\alpha,\beta,\eps_1}\times\V_{\alpha,\beta,\eps_1}\to\Co^{\alpha-1}L^\infty\cap\Co^{-1}_{\tau,w}\Co^\beta(\B^{r+2m},\B^q;\C^{(r+m)\times m\times(r+2m)}),$$ is differentiable.

By Corollary \ref{Cor::Hold::CorMult} \ref{Item::Hold::CorMult::Prin} the product map $\Co^{\alpha}_{\tau,w}L^\infty_s\cap\Co^{0}_{\tau,w}\Co^\beta_s\times \Co^{\alpha-1}_{\tau,w}L^\infty_s\cap\Co^{-1}_{\tau,w}\Co^\beta_s\to \Co^{\alpha-1}_{\tau,w}L^\infty_s\cap\Co^{-1}_{\tau,w}\Co^\beta_s $ is bounded hence is differentiable, we see that
$$\big[(A,H'')\mapsto\Psi[H'']^{-1}\otimes \nabla_{\tau,w}\Lambda[A;H'']\big]:\U_{\alpha,\beta,\eps_1}\times\V_{\alpha,\beta,\eps_1}\to\Co^{\alpha-1}L^\infty\cap\Co^{-1}\Co^\beta(\B^{r+2m},\B^q;\C^{(r+2m)^2\times m(r+m)(r+2m)}),$$ is also differentiable. Taking linear combinations we get that $T:\U_{\alpha,\beta,\eps_1}\times \V_{\alpha,\beta,\eps_1}\to\Co^{\alpha-1}L^\infty\cap\Co^{-1}\Co^\beta$ is Fr\'echet differentiable.

Since $H''\mapsto\Lambda[0;H'']$ can a real-analytic function to the components of $\nabla_{\tau,w}H''$ and $\nabla_{\tau,w}\bar H''$, taking second order power expansion at $\nabla_{\tau,w}H''=\nabla_{\tau,w}\bar H''=\nabla_{\tau,w}\Ic$ we have
\begin{equation*}
    \Lambda[0;H'']=\begin{pmatrix}I_r&\\&I_m\end{pmatrix}\begin{pmatrix}\bar H''_\tau\\\bar H''_w\end{pmatrix}+O(|\nabla_{\tau,w}H''-\nabla_{\tau,w}\Ic|^2).
\end{equation*}

By  Corollary \ref{Cor::Hold::CorMult} \ref{Item::Hold::CorMult::Prin} again, since $\eps_1$ is small, we have
\begin{equation}\label{Eqn::PDE::MatisInv::PfBddLambda2}
    \bigg\|\Lambda[0;H'']-\begin{pmatrix}\bar H''_\tau\\\bar H''_w\end{pmatrix}\bigg\|_{\Co^\alpha_{\tau,w}L^\infty_s\cap\Co^0_{\tau,w}\Co^\beta_s}\lesssim\|\nabla_{\tau,w}(H''-\Ic)\|_{\Co^\alpha_{\tau,w}L^\infty_s\cap\Co^0_{\tau,w}\Co^\beta_s}^2\lesssim\|H''-\Ic\|_{\Co^{\alpha+1}_{\tau,w}L^\infty_s\cap\Co^1_{\tau,w}\Co^\beta_s}^2.
\end{equation}

Taking $\nabla_{\tau,w}$ on $\Lambda[0;H'']$, since by \eqref{Eqn::PDE::MatisInv::PfBddPsi} that $\Psi[H'']^{-1}=I_{r+2m}+O(|\nabla_{\tau,w}(H''-\Ic)|)$ near $H''=\Ic$, we have
\begin{align*}
    T[0;H'']^l=&\sum_{j=1}^r\left((I_{r+2m}+O(\nabla_{\tau,w}(H''-\Ic))\begin{pmatrix}\partial_{\tau}\\\partial_{ w}\\\partial_{\bar w}\end{pmatrix}\right)_j\left(\begin{pmatrix}\bar H''_\tau\\\bar H''_ w\end{pmatrix}+O(\nabla_{\tau,w}(H''-\Ic))^2\right)_j^l
		\\
		&+\sum_{k=1}^m\left((I_{r+2m}+O(\nabla_{\tau,w}(H''-\Ic))\begin{pmatrix}\partial_{\tau}\\\partial_{ w}\\\partial_{\bar w}\end{pmatrix}\right)_{k+r+m}\left(\begin{pmatrix}\bar H''_\tau\\\bar H''_ w\end{pmatrix}+O(\nabla_{\tau,w}(H''-\Ic))^2\right)_{k+r}^l
		\\
		=&\sum_{j=1}^r\Coorvec{\tau^j}\frac{\partial\bar H''^l}{\partial\tau^j}+\sum_{k=1}^m\Coorvec{w^k}\frac{\partial\bar H''^l}{\partial \bar w^k}+O(\nabla_{\tau,w}(H''-\Ic))^2+O(\nabla_{\tau,w}(H''-\Ic))O(\nabla^2_{\tau,w}(H''-\Ic))
		\\
		=&(\Delta_\tau+\Box_w)\bar H''+O(\nabla_{\tau,w}(H''-\Ic))^2+O(\nabla_{\tau,w}(H''-\Ic))O(\nabla^2_{\tau,w}(H''-\Ic)).
\end{align*}

Combining \eqref{Eqn::PDE::MatisInv::PfBddPsi}, \eqref{Eqn::PDE::MatisInv::PfBddLambda}, \eqref{Eqn::PDE::MatisInv::PfBddLambda2} and taking $\eps_1$ smaller, we therefore get \eqref{Eqn::PDE::MatisInv::TangT}, finishing the proof of \ref{Item::PDE::MatisInv::Diff}.
\end{proof}


\begin{proof}[Proof of Proposition \ref{Prop::PDE::ExistPDE}]
	By Lemma \ref{Lem::PDE::MatisInv} \ref{Item::PDE::MatisInv::Diff}, $T:\U_{\alpha,\beta,\eps_1}\times\V_{\alpha,\beta,\eps_1}\to\Co^{\alpha-1}_{\tau,w}L^\infty_s\cap\Co^{-1}_{\tau,w}\Co^\beta_s(\B^{r+2m},\B^q,\C^m)$ is a well-defined Fr\'echet differentiable map for some small $\eps_1>0$. By \eqref{Eqn::PDE::MatisInv::TangT}, since $T[0;\Ic]=0$, we know the tangential map of $T$ with respect to $H''$ at $(A,H'')=(0,\Ic)$ is
	\begin{equation}\label{Eqn::PDE::ExistPDE::PfTangT}
	    \begin{aligned}
	    \left.\frac{\partial T}{\partial H''}\right|_{A=0,H''=\Ic}=(\Delta_\tau+\Box_w)\overline{(\cdot)}:&\{F\in\Co^{\alpha+1}_{\tau,w}L^\infty_s\cap\Co^1_{\tau,w}\Co^\beta_s(\B^{r+2m},\B^q;\C^m):F|_{(\partial\B^{r+2m})\times\B^q}=0\}
	    \\
	    &\to\Co^{\alpha-1}L^\infty\cap\Co^{-1}\Co^\beta(\B^{r+2m},\B^q;\C^m).
	\end{aligned}
	\end{equation}
	By Lemma \ref{Lem::Hold::LapInvBdd} (passing to a scaling on $\B^{r+2m}$), we see that the operator \eqref{Eqn::PDE::ExistPDE::PfTangT} is invertible. Therefore by the standard Implicit Function Theorem on Banach space, we can find a $\delta_1$ and a continuous map $\mathfrak h:\U_{\alpha,\beta,\delta_1}\to\V_{\alpha,\beta,\eps_1}$ such that $T[A;\mathfrak h[A]]=0$ holds for all $A\in\U_{\alpha,\beta,\delta_1}$. By continuity of $\mathfrak h$, we see that for any $\eps>0$ there is a $\delta>0$ such that $\mathfrak h:\U_{\alpha,\beta,\delta}\to\V_{\alpha,\beta,\eps}$ holds. Taking $H''=\mathfrak h[A]$ we can finish the proof of \ref{Item::PDE::ExistPDE::PDEEqn}.
	
    The results \ref{Item::PDE::ExistPDE::HisDiffeo} and \ref{Item::PDE::ExistPDE::Phi} are the direct consequence to Lemma \ref{Lem::Hold::QPIFT} and Proposition \ref{Prop::Hold::CompThm} by considering $F((\tau,w),s):=(\tau,H''(\tau,w,s))$. Note that $H(\frac14\B^{r+2m}\times\frac14\B^q)\subseteq\frac12\B^{r+2m}\times\frac12\B^q$ is automatically satisfied if $\|H''-\Ic\|_{C^0}<\frac14$, which can be obtained whenever $\delta_0$ is suitably small.
\end{proof}
\begin{remark}\label{Rmk::PDE::RegBeta}
    By Lemma \ref{Lem::PDE::MatisInv}, one can see that no matter how large $\beta>0$ is (in particular for $\beta>\alpha+1$), we can find a $\delta_0=\delta_0(\alpha,\beta,r,m,q)>0$ such that: if $\|A\|_{\Co^{\alpha,\beta}}<\delta_0$ then there is a $H''\in\Co^{\alpha+1}L^\infty\cap\Co^1\Co^\beta$ satisfying results \ref{Item::PDE::ExistPDE::PDEEqn} and \ref{Item::PDE::ExistPDE::HisDiffeo} in Proposition \ref{Prop::PDE::ExistPDE}. We only consider the case $\beta\le\alpha+1$ because for the inverse map $\tilde \Phi=H^\Inv$ we only have $\tilde \Phi\in\Co^{\alpha+1,\min(\alpha+1,\beta)}_{(\zeta,\sigma),s}$, where the regularity $\Co^{\min(\alpha+1,\beta)}$ for parameters is optimal.
\end{remark}
\subsection{Real-analyticity of a particular PDE}
\label{Section::AnalPDE}
In this part on a complex space we denote the fixed open cone
\begin{equation}\label{Eqn::HCone}
    \Hb^n=\{x+iy:x,y\in\B^n,\ 4|y|<1-|x|\}\subset\C^n_z.
\end{equation}
Clearly $\Hb^n_z\cap\R^n_x=\B^n$ is a ``complex extension'' of the unit ball $\B^n\subset\R^n$ in the real domain. 

For a holomorphic function $f\in\Oh(\Hb^n)$, we denote $\partial f=\partial_zf=(\frac{\partial f}{\partial z^1},\dots,\frac{\partial f}{\partial z^n})$ as a $n$-dimensional vector valued function.

In this part we give the proof of the following regularity result for a quasilinear elliptic equation system:
\begin{prop}\label{Prop::PDE::AnalyticPDE}Let $m,n,p,q\in\Z_+$, let $\Theta:\C^m\times\C^{n\times m}\to\C^p$ be a complex bilinear map, and let $L=\sum_{j=1}^na^j\Coorvec{x^j}$ be a vector-valued first order differential operator on $\R^n$ with constant coefficients $a^j\in\C^{m\times p}$.

Assume $\alpha>\frac12$, $\beta>0$, there is a $\eps_0=\eps_0(m,n,p,q,\alpha,\beta,\Theta,L)>0$ that satisfies the following:

Suppose $f\in \Co^\alpha_xL^\infty_s\cap\Co^{-1}_x\Co^\beta_s(\B^n,\B^q;\C^m)$ satisfies $\|f\|_{\Co^\alpha_xL^\infty_s\cap\Co^{-1}_x\Co^\beta_s}<\eps_0$ and $\Delta_xf=L_x\Theta(f,\partial_x f)$, then $f$ admits extension $\mathfrak f:\Hb^n\times \B^q\subset\C^n_z\times \R^q_s\to\C^m$ such that $\mathfrak f$ is holomorphic in $z$ and $\mathfrak f(x+i0;s)=f(x;s)$ for $x\in\B^n,s\in \B^q$. 
Moreover $\mathfrak f\in\Co^\infty_\loc\Co^\beta(\Hb^n,\B^q;\C^m)$.
\end{prop}

The proof of Proposition \ref{Prop::PDE::AnalyticPDE} is based on the following Holomorphic extensibility for Laplacian. 
\begin{prop}[Holormorphic property of Laplacian]\label{Prop::HolLap}Let $\B^n\subset\R^n$ and $\Hb^n\subset\C^n$ be as above.
	\begin{enumerate}[parsep=-0.3ex,label={(\roman*)}]
		\item\label{Item::HolLap::P} There is a linear operator $\Pv$ acting on functions on $\B^n$ such that $\Delta \Pv=\id_{\B^n}$ and $\Pv:\Co^{\alpha-2}(\B^n)\to\Co^\alpha(\B^n)$ is bounded for all $-2<\alpha<1$.
		\item\label{Item::HolLap::TildeP} There is a linear operator  $\tilde \Pv$ acting on functions on $\Hb^n$, such that for every $-2<\alpha<1$:
		\begin{itemize}[nolistsep]
		    \item $\tilde \Pv:\Co^{\alpha-2}_\Oh(\Hb^n)\to \Co^\alpha_\Oh(\Hb^n)$ is bounded linear;
		    \item $\Delta \tilde \Pv u=u$ and $(\tilde \Pv u)|_{\B^n}=\Pv(u|_{\B^n})$ for all $u\in\Co^{\alpha-2}_\Oh(\Hb^n)$. Here $\Delta=\Delta_x=\sum_{j=1}^n\frac{\partial^2}{\partial x^j\partial x^j}$.
		\end{itemize}
		\item\label{Item::HolLap::E} If $\Delta u=0$ in $\B^n$, then $u$ extend holomorphically into $\Hb^n$, call it $\Ex u\in\Oh(\Hb^n)$. Moreover the holomorphic extension operator $\Ex :\Co^{\alpha}(\B^n;\C)\cap\ker\Delta\to \Co^{\alpha}_\Oh(\Hb^n)$ is bounded for all $-2<\alpha<1$.
	\end{enumerate}
\end{prop}
\begin{remark}
    Proposition \ref{Prop::HolLap} is essentially done in \cite{Analyticity}. However in our cases we need $\alpha$ to be negative. So we include a complete proof in Section \ref{Section::SecHolLap}.
\end{remark}
\begin{remark}
    In \ref{Item::HolLap::TildeP} since $u\in\Co^{\alpha-2}_\Oh$ is holomorphic, the statement $\Delta_x \tilde \Pv u=u$ is equivalent as $\Delta_z\tilde\Pv u=u$ where $\Delta_z=\sum_{j=1}^n\frac{\partial^2}{\partial z^j\partial z^j}$. In this case $\Delta_z$ is NOT an elliptic operator on $\Hb^n$: it is not the complex Laplacian $\square=\sum_{j=1}^n\frac{\partial^2}{\partial z^j\partial\bar z^j}$.
\end{remark}

We postpone the proof to Appendix \ref{Section::SecHolLap}. Here the space $\Co^\alpha_\Oh(\Hb^n)$ is given by the following:

\begin{defn}\label{Defn::PDE::HoloHoldSpace}
Let $\alpha<1$, we denote $\Co^\alpha_\Oh(\Hb^n)=\Co^\alpha_\Oh(\Hb^n;\C)$ as the space of holomorphic function $g\in\Oh(\Hb)$ such that 
\begin{equation*}
    \|g\|_{\Co^\alpha_\Oh}:=|g(0)|+\sup\limits_{z\in\Hb^n}\dist(z,\partial\Hb^n)^{1-\alpha}|\nabla g(z)|<\infty.
\end{equation*}

Let $\beta>0$ and let $\Xs\in\{L^\infty,\Co^\beta\}$. 
We define $\Co^\beta_{\Oh}\Xs(\Hb^n,\B^q)=\Co^\beta_{\Oh}\Xs(\Hb^n,\B^q;\C)$ as the set of continuous function $f:\Hb^n\times \B^q\to\C$ such that 
\begin{equation}\label{Eqn::HoloParaSpace::Norm}
    \|f\|_{\Co^\alpha_{\Oh}\Xs(\Hb^n,\B^q)}:=\|f(0,\cdot)\|_{\Xs(\B^q;\C)}+\sup\limits_{z\in\Hb^n}\dist(z,\partial\Hb^n)^{1-\alpha}\|\partial_z f(z,\cdot)\|_{\Xs(\B^q;\C)}<\infty.
\end{equation}

Also we define the space $\Co^\alpha_{\Oh}L^\infty\cap\Co^{-\eta}_{\Oh}\Co^\beta(\Hb^n,\B^q)$ with norm $\|f\|_{\Co^\alpha_{\Oh}L^\infty\cap\Co^{-\eta}_{\Oh}\Co^\beta(\Hb^n,\B^q)}:=\|f\|_{\Co^\alpha_{\Oh}L^\infty(\Hb^n,\B^q)}+\|f\|_{\Co^{-\eta}_{\Oh}\Co^\beta(\Hb^n,\B^q)}$, for $\eta<1$.
\end{defn}
\begin{remark}
    $\Co^\alpha_\Oh(\Hb^n)$ is exactly the space of $\Co^\alpha$-functions on $\Hb^n$ that are holomorphic in the interior. For part of the inclusion see Lemma \ref{Lem::SecHolLap::HLLem}.
\end{remark}
\begin{remark}\label{Rmk::PDE::HolLapPara}
    By the same argument to Lemma \ref{Lem::Hold::TOtimesId}, we see that $\Pv,\tilde\Pv,\Ex$ in Proposition \ref{Prop::HolLap} have natural extension to the parameter case. We still denote them as $\Pv,\tilde\Pv,\Ex$: let $V\subseteq\R^q$ be a bounded smooth subset, we have boundedness
    \begin{itemize}[nolistsep]
        \item $\Pv:\Co^{\alpha-2}\Xs(\B^n,V)\to\Co^\alpha\Xs(\B^n,V)$ for $-2<\alpha<1$ and $\Xs\in\{L^\infty,\Co^\beta:\beta\in\R\}$.
        \item $\tilde\Pv:\Co^{\alpha-2}_\Oh\Xs(\Hb^n,V)\to\Co^\alpha_\Oh\Xs(\Hb^n,V)$ for $-2<\alpha<1$ and $\Xs\in\{L^\infty,\Co^\beta:\beta\in\R\}$.
        \item $\Ex:\Co^{\alpha}\Xs(\B^n,V)\cap\Delta_x\to\Co^\alpha_\Oh\Xs(\Hb^n,V)$ for $-2<\alpha<1$ and $\Xs\in\{L^\infty,\Co^\beta:\beta\in\R\}$.
    \end{itemize}
    We omit the proof to readers.
\end{remark}

The following lemma is useful in the proof of Proposition \ref{Prop::PDE::AnalyticPDE}.
\begin{lem}\label{Lem::PDE::HoloParaLem}
    Let $\beta>0$ and let $\Xs\in\{L^\infty,\Co^\beta\}$.
    \begin{enumerate}[parsep=-0.3ex,label=(\roman*)]
        \item\label{Eqn::PDE::HoloParaLem::Grad} The differentiation $[\tilde f\mapsto\partial_z\tilde f]:\Co^{\alpha}_{\Oh}\Xs(\Hb^n,\B^q)\to\Co^{\alpha-1}_{\Oh}\Xs(\Hb^n,\B^q;\C^n)$ is bounded linear for all $\alpha<1$.
        \item\label{Item::PDE::HoloParaLem::EqvNorm} For $\alpha<0$, $\Co^\alpha_{\Oh}\Xs(\Hb^n,\B^q)$ has an equivalent norm
        \begin{equation}\label{Eqn::PDE::HoloParaLem::EqvNorm}
            \textstyle\tilde f\mapsto\sup_{z\in\Hb^n}\dist(z,\partial\Hb^n)^{-\alpha}\|\tilde f(z,\cdot)\|_{\Xs(\B^q;\C)}.
        \end{equation}
        
        \item\label{Item::PDE::HoloParaLem::Res} The restriction map $[\tilde f\mapsto \tilde f|_{\B^n\times\B^q}]:\Co^\alpha_\Oh\Xs(\Hb^n,\B^q)\to\Co^\alpha_x\Xs(\B^n,\B^q;\C)$ is bounded linear for $0<\alpha<1$.
        \item\label{Item::PDE::HoloParaLem::Inc} We have inclusion map $\Co^\alpha_\Oh\Xs(\Hb^n,\B^q)\hookrightarrow L^\infty\Xs(\Hb^n,\B^q;\C)$ for $0<\alpha<1$.
    \end{enumerate}
    
\end{lem}
\begin{proof}These results follow from Lemma \ref{Lem::SecHolLap::HLLem}, the parameter free cases, along with the same argument to Lemma \ref{Lem::Hold::TOtimesId}. 
\end{proof}

We need to show that in $\Hb^n\times \B^q$, $\mathfrak f\mapsto \Theta(\mathfrak f,\partial_z\mathfrak f)$ maps $\mathfrak f$ into the desired function space.

\begin{lem}\label{Lem::PDE::HoloMult}
	Let $0<\alpha<1$ and $\beta>0$. 
For $f\in\Co^\alpha_\Oh L^\infty\cap\Co^{-1}_\Oh\Co^\beta(\Hb^n,\B^q)$ and $g\in\Co^{\alpha-1}_\Oh L^\infty\cap\Co^{-2}_\Oh\Co^\beta(\Hb^n,\B^q)$, we have $fg\in\Co^{\alpha-1}_\Oh L^\infty\cap\Co^{-2}_\Oh\Co^\beta(\Hb^n,\B^q)$. Moreover there is a $C=C(n,q,\alpha,\beta)>0$ that does not depend on $f,g$ such that
$$\|fg\|_{\Co^{\alpha-1}_\Oh L^\infty\cap\Co^{-2}_\Oh\Co^\beta(\Hb^n,\B^q)}\le C\|f\|_{\Co^{\alpha}_\Oh L^\infty\cap\Co^{-1}_\Oh\Co^\beta(\Hb^n,\B^q)}\|g\|_{\Co^{\alpha-1}_\Oh L^\infty\cap\Co^{-2}_\Oh\Co^\beta(\Hb^n,\B^q)}.$$
\end{lem}

\begin{proof}Clearly the product $f(z,s)g(z,s)$ is locally continuous and holomorphic in $z$.
    
    By Lemma \ref{Lem::PDE::HoloParaLem} \ref{Item::PDE::HoloParaLem::EqvNorm} and \ref{Item::PDE::HoloParaLem::Inc}, we have almost every $s\in\B^q$,
    \begin{align*}
        |(fg)(z,s)|\lesssim_{n,\alpha}\|f\|_{L^\infty}\dist(z,\partial\Hb^n)^{\alpha-1}\|g\|_{\Co^{\alpha-1}_\Oh L^\infty}\lesssim_{n,\alpha}\|f\|_{\Co^\alpha_\Oh L^\infty}\dist(z,\partial\Hb^n)^{\alpha-1}\|g\|_{\Co^{\alpha-1}_\Oh L^\infty}.
    \end{align*}
    Dividing $\dist(z,\partial\Hb^n)^{\alpha-1}$ on both side, taking essential sup over $z$ and $s$, and using Lemma \ref{Lem::PDE::HoloParaLem} \ref{Item::PDE::HoloParaLem::EqvNorm} we get the first control $\|fg\|_{\Co^{\alpha-1}_\Oh L^\infty}\lesssim\|f\|_{\Co^{\alpha}_\Oh L^\infty\cap\Co^{-1}_\Oh\Co^\beta}\|g\|_{\Co^{\alpha-1}_\Oh L^\infty\cap\Co^{-2}_\Oh\Co^\beta}$.
    
    To control $\|fg\|_{\Co^{-2}_\Oh\Co^\beta}$. By Lemmas \ref{Lem::Hold::Product} \ref{Item::Hold::Product::Hold2} and \ref{Lem::PDE::HoloParaLem} \ref{Item::PDE::HoloParaLem::EqvNorm}, for every $z\in\Hb^n$,
    \begin{align*}
        &\|fg(z,\cdot)\|_{\Co^\beta(\B^q)}\lesssim_\beta\|f(z,\cdot)\|_{L^\infty}\|g(z,\cdot)\|_{\Co^\beta}+\|f(z,\cdot)\|_{\Co^\beta}\|g(z,\cdot)\|_{L^\infty}
        \\
        \le&\|f\|_{\Co^\alpha_\Oh L^\infty}\|g\|_{\Co^{-2}_\Oh\Co^\beta}\dist(z,\partial\Hb^n)^{-2}+\|f\|_{\Co^{-1}_\Oh\Co^\beta}\|g\|_{\Co^{\alpha-1}_\Oh L^\infty}\dist(z,\partial\Hb^n)^{-1}\cdot\dist(z,\partial\Hb^n)^{\alpha-1}
        \\
        \le&\|f\|_{\Co^{\alpha}_\Oh L^\infty\cap\Co^{-1}_\Oh\Co^\beta}\|g\|_{\Co^{\alpha-1}_\Oh L^\infty\cap\Co^{-2}_\Oh\Co^\beta}\dist(z,\partial\Hb^n)^{-2}.
    \end{align*}
	
	Dividing $\dist(z,\partial\Hb^n)^{-2}$ on both side, we get $\|fg\|_{\Co^{-2}_\Oh \Co^\beta}\lesssim\|f\|_{\Co^{\alpha}_\Oh L^\infty\cap\Co^{-1}_\Oh\Co^\beta}\|g\|_{\Co^{\alpha-1}_\Oh L^\infty\cap\Co^{-2}_\Oh\Co^\beta}$, completing the proof.
\end{proof}

\begin{proof}[Proof of Proposition \ref{Prop::PDE::AnalyticPDE}]It suffices to consider $\frac12<\alpha<1$.


We use $\Theta[u]=\Theta(u,\nabla_xu)$ for function $u$ either on $\B^n\times\B^q$ or on $\Hb^n\times\B^q$. Note that for $\tilde u\in\Co^{\alpha}_\Oh L^\infty\cap\Co^{-1}_\Oh\Co^\beta $ we have $\nabla_x\tilde u=\partial_z\tilde u$.
	
	The proof uses contraction mappings twice. The first one take place in $\Co^\alpha L^\infty(\B^n,\B^q;\C^m)$ and shows that $\Pv\Delta f$ has some particular properties. The second one take place in $\Co^\alpha_\Oh L^\infty\cap\Co^{-1}_\Oh\Co^\beta(\Hb^n,\B^q;\C^m)$, showing that $\Pv\Delta f$ is a restriction of a function in $\Hb^n\times\B^q$ which is holomorphic in $z$. 
	
	\medskip
	For $u\in\Co^\alpha  L^\infty(\B^n,\B^q;\C^m)$, by Lemma \ref{Lem::Hold::Product} (since $\alpha>\frac12$) we have $u\otimes\nabla_x u(\cdot,s)\in\Co^{\alpha-1}_x$ uniformly for almost every $s$, so $\Theta[u]\in\Co^{\alpha-1} L^\infty(\B^n,\B^q;\C^p)$ and thus $L\Theta[u]\in \Co^{\alpha-2}  L^\infty(\B^n,\B^q;\C^m)$. Moreover we have for every $u,v\in\Co^\alpha  L^\infty(\B^n,\B^q;\C^m)$
{\small\begin{equation}\label{Eqn::PDE::AnaPDE::ContMap1}
\begin{aligned}
    \|L\Theta[u]\|_{\Co^{\alpha-2}L^\infty}&\lesssim \|\Theta(u,\nabla_x u)\|_{\Co^{\alpha-1}L^\infty}\lesssim\|u\otimes\nabla_x u\|_{\Co^{\alpha-1}L^\infty}\lesssim\|u\|_{\Co^{\alpha-2}L^\infty}^2,
    \\
        \|L\Theta[u]-L\Theta[v]\|_{\Co^{\alpha-2}L^\infty}&\lesssim\|\Theta(u-v,\nabla_x u)\|_{\Co^{\alpha-1}L^\infty}+\|\Theta(v,\nabla_x (u-v))\|_{\Co^{\alpha-1}L^\infty}\lesssim(\|u\|_{\Co^{\alpha}L^\infty}+\|v\|_{\Co^{\alpha}L^\infty})\|u-v\|_{\Co^{\alpha}L^\infty}.
\end{aligned}
\end{equation}
}

Here the implicit constants depend only on $m,n,q,\alpha,L,\Theta$ but not on $u$ or $v$.

    For $\tilde u\in\Co^\alpha_\Oh L^\infty\cap\Co^{-1}_\Oh\Co^\beta(\Hb^n,\B^q;\C^m)$,  $L\Theta[\tilde u]$ is a function on $\Hb^n\times\B^q$ that is holormorphic in $z$. By Lemma \ref{Lem::PDE::HoloMult}, $\tilde u\otimes\partial_z \tilde u\in \Co^{\alpha-1}_\Oh L^\infty\cap\Co^{-2}_\Oh\Co^\beta(\Hb^n,\B^q;\C^{m^2 n})$, we know $L\Theta[\tilde u]\in \Co^{\alpha-2}_\Oh L^\infty\cap\Co^{-3}_\Oh\Co^\beta(\Hb^n,\B^q;\C^m)$. Moreover, similar to \eqref{Eqn::PDE::AnaPDE::ContMap2} we have for $\tilde u,\tilde v\in\Co^{\alpha}_\Oh L^\infty\cap\Co^{-1}_\Oh\Co^\beta(\Hb^n,\B^q;\C^m)$:
	\begin{equation}\label{Eqn::PDE::AnaPDE::ContMap2}
	\begin{aligned}
    \|L\Theta[\tilde u]\|_{\Co^{\alpha-2}_{\Oh}L^\infty\cap\Co^{-3}_{\Oh}\Co^\beta}&\lesssim\|\tilde u\|_{\Co^{\alpha-1}_{\Oh}L^\infty\cap\Co^{-1}_{\Oh}\Co^\beta}^2,\\
       \|L\Theta[\tilde u]-L\Theta[\tilde v]\|_{\Co^{\alpha-2}_{\Oh}L^\infty\cap\Co^{-3}_{\Oh}\Co^\beta}&\lesssim(\|\tilde u\|_{\Co^{\alpha-1}_{\Oh}L^\infty\cap\Co^{-1}_{\Oh}\Co^\beta}+\|\tilde v\|_{\Co^{\alpha-1}_{\Oh}L^\infty\cap\Co^{-1}_{\Oh}\Co^\beta})\|\tilde u-\tilde v\|_{\Co^{\alpha-1}_{\Oh}L^\infty\cap\Co^{-1}_{\Oh}\Co^\beta}.
    \end{aligned}
\end{equation}
Here the implicit constants depend only on $m,n,q,\alpha,\beta,L,\Theta$ but not on $\tilde u$ or $\tilde v$.

	Let $\Pv:\Co^{\alpha-2}(\B^n)\to\Co^\alpha(\B^n)$ be as in Proposition \ref{Prop::HolLap}. So the assumption $\Delta f=L\Theta[f]$ implies that 
	\begin{equation}\label{Eqn::PDE::AnaPDE::FixPtf}
	    f=f-\Pv\Delta f+\Pv\Delta f=(f-\Pv\Delta f)+\Pv L\Theta[\Pv\Delta f+(f-\Pv\Delta f)].
	\end{equation}
	
	Define $T_f$ on $\Co^\alpha  L^\infty(\B^n,\B^q;\C^m)$ as 
	$$T_f[u]:=(f-\Pv\Delta f)+\Pv L\Theta[u+(f-\Pv\Delta f)],\quad u\in\Co^\alpha L^\infty(\B^n,\B^q;\C^m).$$
	
	Since $\|f-\Pv\Delta f\|_{\Co^\alpha L^\infty(\B^n,\B^q;\C^m)}\lesssim\|f\|_{\Co^\alpha L^\infty(\B^n,\B^q;\C^m)}$. By \eqref{Eqn::PDE::AnaPDE::ContMap1} and Remark \ref{Rmk::PDE::HolLapPara}, we can find a $C_1=C_1(m,n,q,\alpha,L,\Theta,\Pv)>1$ that does not depend on $f$, that bounds \eqref{Eqn::PDE::AnaPDE::ContMap1} and satisfies $C_1>\|\id-\Pv\Delta \|_{\Co^\alpha L^\infty}$. More precisely,  for every $f,u,v\in \Co^\alpha L^\infty(\B^n,\B^q,\C^m)$, we have
\begin{equation*}
    \begin{aligned}
    \|T_f[u]\|_{\Co^\alpha L^\infty}&\le C_1\|f\|_{\Co^\alpha L^\infty}+ C_1\big(\|u\|_{\Co^\alpha L^\infty}+C_1\|f\|_{\Co^\alpha L^\infty}\big)^2;
    \\
    \|T_f[u]-T_f[v]\|_{\Co^\alpha L^\infty}&\le C_1\left(\|u\|_{\Co^\alpha L^\infty}+\|v\|_{\Co^\alpha L^\infty}+C_1\|f\|_{\Co^\alpha L^\infty}\right)\|u-v\|_{\Co^\alpha L^\infty};
    \\
    \|f-\Pv\Delta f\|_{\Co^\alpha L^\infty}&\le C_1\|f\|_{\Co^\alpha L^\infty}.
    \end{aligned}
\end{equation*}

Take $\eps_1>0$ such that $9\eps_1C_1+10\eps_1 C_1^2<1$. If $\|f\|_{\Co^\alpha  L^\infty}<\eps_1$ and $\|u\|_{\Co^\alpha  L^\infty},\|v\|_{\Co^\alpha  L^\infty}\le 2C_1\eps_1$, then 
\begin{gather*}
    \|T_f[u]\|_{\Co^\alpha L^\infty}\le C_1\eps_1+C_1(3C_1\eps_1)^2<2C_1\eps_1,\\
    \|T_f[u]-T_f[v]\|_{\Co^\alpha L^\infty}<C_1(5C_1\eps_1)\|u-v\|_{\Co^\alpha L^\infty}<\tfrac12\|u-v\|_{\Co^\alpha L^\infty}.
\end{gather*}

Therefore when $\|f\|_{\Co^\alpha  L^\infty}<\eps_1$, $T_f$ is a contraction mapping in $\{u\in\Co^\alpha   L^\infty(\B^n,\B^q;\C^m):\|u\|_{\Co^\alpha L^\infty}\le2C_1\eps_1\}$. In such case, there is a unique $u$ satisfying $\|u\|_{\Co^\alpha L^\infty}\le2C_1\eps_1$ such that $T_f[u]=u$. On the other hand, $f\in\{u\in\Co^\alpha   L^\infty(\B^n,\B^q;\C^m):\|u\|_{\Co^\alpha L^\infty}\le2C_1\eps_1\}$ satisfies \eqref{Eqn::PDE::AnaPDE::FixPtf}. Thus $f$ is that unique fixed point.

\medskip
Similarly, by Proposition \ref{Prop::HolLap} \ref{Item::HolLap::P} (also see Remark \ref{Rmk::PDE::HolLapPara}), we have boundedness $\Pv:\Co^{-3}(\B^n)\to\Co^{-1}(\B^n)$, so $\|f-\Pv\Delta f\|_{\Co^\alpha L^\infty\cap\Co^{-1} \Co^\beta}\lesssim\|f\|_{\Co^\alpha L^\infty\cap\Co^{-1} \Co^\beta}$. Note that $f-\Pv\Delta f$ is harmonic function in $x$-variable, so by Proposition \ref{Prop::HolLap}  \ref{Item::HolLap::E} we have $\Ex(f-\Pv\Delta f)\in \Co^\alpha_\Oh L^\infty\cap\Co^{-1}_\Oh\Co^\beta(\Hb^n,\B^q;\C^m)$ with 
\begin{equation}\label{Eqn::PDE::AnaPDE::BddExf}
    \|\Ex(f-\Pv\Delta f)\|_{\Co^\alpha_\Oh L^\infty\cap\Co^{-1}_\Oh \Co^\beta(\Hb^n,\B^q;\C^m)}\lesssim_{\alpha,\beta}\|f\|_{\Co^\alpha L^\infty\cap\Co^{-1}\Co^\beta(\B^n,\B^q;\C^m)}.
\end{equation}

For a $f\in \Co^\alpha L^\infty\cap\Co^{-1} \Co^\beta(\B^n,\B^q;\C^m)$, define $\tilde T_f$ on $\Co^\alpha_\Oh L^\infty\cap\Co^{-1}_\Oh\Co^\beta(\Hb^n,\B^q;\C^m)$ as 
$$\tilde T_f[\tilde u]=\Ex(f-\Pv\Delta f)+\tilde\Pv L\Theta[\tilde u+\Ex(f-\Pv\Delta f)],\quad\tilde u\in \Co^\alpha_\Oh L^\infty\cap\Co^{-1}_\Oh\Co^\beta(\Hb^n,\B^q;\C^m).$$
Here $\tilde\Pv$ is as in Proposition \ref{Prop::HolLap} \ref{Item::HolLap::TildeP}.

Take $C_2=C_2(m,n,q,\alpha,\beta,L,\Theta,\tilde\Pv,\Ex)>1$ that bounds \eqref{Eqn::PDE::AnaPDE::ContMap2} and \eqref{Eqn::PDE::AnaPDE::BddExf}. Therefore, for every $\tilde u,\tilde v\in\Co^\alpha_\Oh L^\infty\cap\Co^{-1}_\Oh\Co^\beta(\Hb^n,\B^q;\C^m)$,
{\small\begin{align*}
    \|\tilde T_f[\tilde u]\|_{\Co^\alpha_{\Oh} L^\infty\cap\Co^{-1}_{\Oh}\Co^\beta}&\le C_2\|f\|_{\Co^\alpha L^\infty\cap\Co^{-1}\Co^\beta}+ C_2\big(\|\tilde u\|_{\Co^\alpha_{\Oh} L^\infty\cap\Co^{-1}_{\Oh}\Co^\beta}+C_2\|f\|_{\Co^\alpha L^\infty\cap\Co^{-1}\Co^\beta}\big)^2;
    \\
    \|\tilde T_f[\tilde u]-\tilde T_f[\tilde v]\|_{\Co^\alpha_{\Oh} L^\infty\cap\Co^{-1}_{\Oh}\Co^\beta}&\le C_2\big(\|\tilde u\|_{\Co^\alpha_\Oh L^\infty\cap\Co^{-1}_\Oh\Co^\beta}+\|\tilde v\|_{\Co^\alpha_\Oh L^\infty\cap\Co^{-1}_\Oh\Co^\beta}+C_2\|f\|_{\Co^\alpha L^\infty\cap\Co^{-1}\Co^\beta}\big)\|\tilde u-\tilde v\|_{\Co^\alpha_\Oh L^\infty\cap\Co^{-1}_\Oh\Co^\beta}.
\end{align*}}

By Lemma \ref{Lem::PDE::HoloParaLem} \ref{Item::PDE::HoloParaLem::Res} we can find a $C_3=C_3(m,n,q,\alpha,\beta)>0$ such that,
\begin{equation}\label{Eqn::PDE::AnaPDE::RestrictionTmp}
    \|\tilde u|_{\B^n\times\B^q}\|_{\Co^\alpha  L^\infty(\B^n,\B^q;\C^m)}\le C_3\|\tilde u\|_{\Co^\alpha_\Oh L^\infty\cap\Co^{-1}_\Oh\Co^\beta(\Hb^n,\B^q;\C^m)},\quad \forall u\in\Co^\alpha_\Oh L^\infty\cap\Co^{-1}_\Oh\Co^\beta(\Hb^n,\B^q;\C^m).
\end{equation}
Take $\eps_0>0$ such that $C_3C_2\eps_0 <\eps_1$ and $9\eps_0 C_2+10\eps_0 C_2^2<1$. So $\|f\|_{\Co^\alpha L^\infty\cap\Co^{-1}\Co^\beta}<\eps_0$ implies $\|f\|_{\Co^\alpha L^\infty}<\eps_1$.
For such $f$, we know $\tilde T_f$ is a contraction mapping in $\{\tilde u\in\Co^\alpha_{\Oh} L^\infty\cap\Co^{-1}_{\Oh}\Co^\beta(\Hb^n,\B^q;\C^m):\|\tilde u\|_{\Co^\alpha_{\Oh} L^\infty\cap\Co^{-1}_{\Oh}\Co^\beta}\le2C_2\eps_0\}$.
Thus there is a unique $\tilde u\in\Co^\alpha_\Oh L^\infty\cap\Co^{-1}_\Oh\Co^\beta(\Hb^n,\B^q;\C^m) $ such that $\|\tilde u\|_{\Co^\alpha_\Oh}\le2C_2\eps_0$ and $\tilde u=\tilde T_f[\tilde u]$. 

By Proposition \ref{Prop::HolLap} \ref{Item::HolLap::TildeP}, 
$$\tilde u|_{\B^n\times \B^q}=\big(\Ex(f-\Pv\Delta f)+\tilde\Pv L\Theta[\tilde u+\Ex(f-\Pv\Delta f)]\big)\big|_{\B^n\times \B^q}=f-\Pv\Delta f+\Pv L\Theta[\tilde u|_{\B^n\times \B^q}+f-\Pv\Delta f]=T_f[\tilde u|_{\B^n\times \B^q}].$$
By \eqref{Eqn::PDE::AnaPDE::RestrictionTmp} and the assumption $C_3C_2\eps_0<\eps_1 <2C_1\eps_1$, we have  $\|\tilde u|_{\B^n\times \B^q}\|_{\Co^\alpha  L^\infty(\B^n,\B^q;\C^m)}\le 2C_1\eps_1$. By the uniqueness of the fixed point for $T_f$, we get $f=\tilde u|_{\B^n\times \B^q}$. 

So $\mathfrak f=\tilde u\in\Co^{-1}_\Oh\Co^\beta(\Hb^n,\B^q;\C^m)$ is the desired extension of $f$ to $\Hb^n\times\B^q$.

Finally, by Lemma \ref{Lem::Hold::NablaHarm} we have $\mathfrak f\in\Co^\infty_\loc \Co^\beta(\Hb^n,\B^q;\C^m)$, completing the proof.
\end{proof}

\section{Malgrange's Factorization with Parameters}\label{Section::SecKey}
In this section we prove the key part of Theorems \ref{Thm::ThmCoor1} and \ref{Thm::ThmCoor2}. 

Let $\Mf$  and $\Nf$ be two smooth manifolds. In this section we consider the product manifold $\Mf\times\Nf$ as the ambient space, so $(\C T\Mf)\times\Nf=\coprod_{q\in\Nf}\C T\Mf$ is a subbundle of $\C T(\Mf\times\Nf)$ whose (complex) rank equals to the (real) dimension of $\Mf$. 

We consider a $\Co^{\alpha,\beta}$ involutive complex tangential subbundle $\Se$ of $\Mf\times\Nf$ such that $\Se\le(\C T\Mf)\times\Nf$ and $\Se+\bar\Se=(\C T\Mf)\times \Nf$. Note that $(\C T\Mf)\times\Nf$ is automatically involutive, so $\Se$ is a complex Frobenius structure.

We recall the space $\Co^{\alpha,\beta}_\loc(\Mf,\Nf)$ ($\alpha>1$, $\beta>0$) in Definition \ref{Defn::ODE::MixHoldMaps}, the $\Co^{\alpha,\beta}$-subbundle $\Se$ in Definition \ref{Defn::ODE::CpxPaSubbd}, and the involutivity of $\Se\le(\C T\Mf)\times\Nf$ in Definition \ref{Defn::ODE::InvMix} (also see Remark \ref{Rmk::ODE::InvMix}). We endow $\R^r$, $\C^m$ and $\R^q$ with standard (real and complex) coordinate system $t=(t^1,\dots,t^r)$, $z=(z^1,\dots,z^m)$ and $s=(s^1,\dots,s^q)$ respectively. 

\begin{thm}[The key estimate]\label{Thm::Key}
Let $r,m,q\ge0$, let $\Mf$ and $\Nf$ be two smooth manifolds with dimensions $(r+2m)$ and $q$ respectively. Let $\alpha\in(1,\infty)$ and $\beta\in(0,\alpha+1]$. Let $\Se\le (\C T\Mf)\times \Nf$ be a $\Co^{\alpha,\beta}$-involutive subbundle of rank $r+m$ such that $\Se+\bar\Se=(\C T\Mf)\times \Nf$.

Then for any $u_0\in \Mf$ and $v_0\in\Nf$, there are neighborhoods $U\subseteq\Mf$ of $u_0$, $V\subseteq\Nf$ of $v_0$ and a map $F=(F',F'',F'''):U\times V\to\R^r_t\times\C^m_z\times\R^q_s$ such that 
\begin{enumerate}[parsep=-0.3ex,label=(\arabic*)]
    \item\label{Item::Key::F'''}$F''':V\to\R^q_s$ is a smooth coordinate chart of $\Nf$ that is independent of $U\subseteq\Mf$, such that $F'''(v_0)=0$; $F':U\to\R^r_t$ is a smooth map that is independent of $V\subseteq\Nf$, such that $F'(u_0)=0$.
    \item\label{Item::Key::FBase}$F:U\times V\to F(U\times V)$ is homeomorphism. And for each $v\in V$, $[u\mapsto (F'(u),F''(u,v))]:U\to\R^r_t\times\C^m_z$ is a $\Co^{\alpha+1}$-coordinate chart.
    \item\label{Item::Key::F''Reg} $F''(u_0,v_0)=0$, $F''\in\Co^{\alpha+1,\beta}_\loc(U,V;\C^m)$ and $XF''\in\Co^{\alpha,\beta-}_\loc(U,V;\C^m)$ for every $C^\infty$ vector field $X$ on $\Mf$.
    \item\label{Item::Key::Span} $F^*\Coorvec{t^1},\dots,F^*\Coorvec{t^r},F^*\Coorvec{z^1},\dots,F^*\Coorvec{z^m}$ are well-defined $\Co^{\min(\alpha,\beta-)}$-vector fields and span $\Se|_{U\times V}$.
\end{enumerate}

For the parameterization side, let $\Phi=(\Phi',\Phi''):F(U\times V)\to\Mf\times\Nf$ be the inverse map of $F$, and let $\Omega:=\Omega'\times\Omega''\times\Omega'''\subseteq\R^r_t\times\C^m_z\times\R^q_s$ be any neighborhood of $(0,0,0)$ which is contained in $F(U\times V)$. Then
\begin{enumerate}[parsep=-0.3ex,label=(\arabic*)]\setcounter{enumi}{4}
    \item\label{Item::Key::Phi''} $\Phi'':V\to \Nf$ is a smooth regular parameterization. Moreover $\Phi''=(F''')^\Inv$.
    \item\label{Item::Key::Phi0}  $\Phi:\Omega\to\Mf\times\Nf$ is homeomorphic onto its image,  and $\Phi(0,0,0)=(u_0,v_0)$. For each $s\in\Omega''$, $\Phi'(\cdot,s):\Omega'\times\Omega''\to\Mf$ is a $\Co^{\alpha+1}$-regular parameterization.
    \item\label{Item::Key::PhiReg}  $\Phi'\in\Co^{\alpha+1,\beta}_\loc(\Omega'\times\Omega'',\Omega''';\Mf)$ and $\frac{\partial\Phi'}{\partial t^1},\dots,\frac{\partial\Phi'}{\partial t^r},\frac{\partial\Phi'}{\partial z^1},\frac{\partial\Phi'}{\partial \bar z^1},\dots,\frac{\partial\Phi'}{\partial z^m},\frac{\partial\Phi'}{\partial \bar z^m}\in\Co^{\alpha,\beta-}_\loc(\Omega'\times\Omega'',\Omega''';\C T\Mf)$.
    \item\label{Item::Key::PhiSpan}  $\Se_{\Phi(t,z,s)}\le\C T_{\Phi'(t,z,s)}\Mf$ is spanned by $\frac{\partial\Phi}{\partial t^1}(t,z,s),\dots,\frac{\partial\Phi}{\partial t^r}(t,z,s),\frac{\partial\Phi}{\partial z^1}(t,z,s),\dots,\frac{\partial\Phi}{\partial z^m}(t,z,s)$,  $\forall (t,z,s)\in\Omega$.
\end{enumerate}
In particular,
\begin{enumerate}[parsep=-0.3ex,label=(\arabic*)]\setcounter{enumi}{8}
    \item\label{Item::Key::>1} When $1<\beta\le\alpha+1$, $F$ is a $\Co^{\beta}$-coordinate chart and $\Phi$ is a $\Co^{\beta}$-regular parameterization.
\end{enumerate}

\end{thm}
\begin{remark}\label{Rmk::Key::KeyFirst}
\begin{enumerate}[parsep=-0.3ex,label=(\roman*)]
    \item By taking some modifications in the proof using \cite[Proposition B.4]{SharpElliptic}, Theorem \ref{Thm::Key} are also true for the case $\alpha=\infty$ and $0<\beta<\infty$ (cf. \cite[Theorem 1.3(i)]{Gong}), we omit the detail to readers. 

\item We shall see in Section \ref{Section::Sharpddz} that the index ``$\beta-$'' for  $XF''\in\Co^{\alpha,\beta-}$ and $F^*\Coorvec z\in\Co^{\min(\alpha,\beta-)}$ are both sharp.

\item Theorem \ref{Thm::Key} implies Street's result \cite[Theorem 1.1]{SharpElliptic} for the sharp estimate on elliptic structure by considering $\alpha=\beta$ and forgetting the $s$-parameter. (Recall $\Co^{\alpha,\alpha}(U,V)=\Co^\alpha(U\times V)$ from Lemma \ref{Lem::Hold::CharMixHold}.)
\end{enumerate}
\end{remark}


\begin{remark}\label{Rmk::Key::KeySpecial}
Our results allow the range $\beta\in(0,1]$ and $\beta\in(\alpha,\alpha+1]$. 

For the case $\beta\le1$, it does not make sense to talk about $\Co^\beta$-coordinate charts nor $\Co^\beta$-parameterizations. Nevertheless by \ref{Item::Key::FBase} we have 
$F^*\Coorvec t|_{(u,v)}=(F',F''(\cdot,v))^*\Coorvec t|_u $ and $F^*\Coorvec z|_{(u,v)}=(F',F''(\cdot,v))^*\Coorvec z|_u $, since $(F',F''(\cdot,v))$ is a coordinate chart for a fix $v\in V$. Alternatively by \ref{Item::Key::Phi''} $F^*\Coorvec t=\frac{\partial\Phi}{\partial t}\circ F$ and $F^*\Coorvec z=\frac{\partial\Phi}{\partial t}\circ F$ are both continuous. Either way $F^*\Coorvec t,F^*\Coorvec z$ are collections of pointwise defined vector fields.

For the case $\beta>\alpha$, the parameter part is actually more regular than the part of spatial variable. In this case $F^*\Coorvec t,F^*\Coorvec z$ are $\Co^\alpha$-vector fields on $U\times V\subseteq\Mf\times\Nf$. See also Remark \ref{Rmk::PDE::RegBeta}.
\end{remark}
\begin{remark}\label{Rmk::Key::KeyImpliesGong}
If we only consider a family of elliptic structures rather than a family of complex Frobenius structures, and assume $\alpha,\beta$ are non-integers, then Theorem \ref{Thm::Key} is much stronger than \cite[Theorem 1.3]{Gong}:
\begin{itemize}[parsep=-0.3ex]
    \item We consider all $(\alpha,\beta)$ such that $\alpha>1$ and $0<\beta\le\alpha+1$. In \cite[Theorem 1.3]{Gong} only the case $\alpha>\beta+3\ge4$ and the case $\alpha=\beta$ is considered.
    \item For the regularity of the coordinate charts we have the sharp estimate $F\in\Co^{\alpha+1,\beta}(U,V)$. \cite[Theorem 1.3]{Gong} only gives $F\in\Co^{\alpha-,\beta-}(U,V)$.
\end{itemize}
\end{remark}



We give a detailed overview to the idea of its proof in Section \ref{Section::KeyProofOver} using Malgrange's factorization. 

\subsection{Overview of the key estimate}\label{Section::KeyProofOver}
 For vector fields $X_1,\dots,X_{r+m}$, we write $X=[X_1,\dots,X_{r+m}]^\top$ following from the convention in Section \ref{Section::Convention}.

The most important part is to find the map $F'':U\times V\to\C^m$ that has the desired regularity property. We factorize $F''$ to the following composition of maps:
\begin{equation}\label{Eqn::Key::FacterizationofF}
    F''=\overline{G''}\circ H\circ L.
\end{equation}

Endow $\R^r,\C^m,\R^q$ with standard coordinates $\tau=(\tau^1,\dots,\tau^r)$, $w=(w^1,\dots,w^m)$ and $s=(s^1,\dots,s^q)$ respectively.
There are smooth coordinate charts $(L',L''):U_0\subseteq\Mf\xrightarrow{\sim}\B^{r+2m}\subset\R^r_\tau\times\C^m_w$ and $L''':V_0\subseteq\Nf\xrightarrow{\sim}\B^q\subset\R^q_s$, such that by taking $L=(L',L'',L'''):U_0\times V_0\to\B^{r+2m}\times\B^q$,  $L_*\Se$ has a $\Co^{\alpha,\beta}_{(\tau,w),s}$-local basis on $\B^{r+2m}\times\B^q$ with the form 
\begin{equation}\label{Eqn::Key::DefofA}
    X=\begin{pmatrix}X'\\X''\end{pmatrix}=\begin{pmatrix}I_r&0&A'\\0&I_m&A''\end{pmatrix}\begin{pmatrix}\partial_\tau\\\partial_w\\\partial_{\bar w}\end{pmatrix},\quad A:=\begin{pmatrix}
    A'\\A''\end{pmatrix},
\end{equation}
where $\|A\|_{\Co^{\alpha,\beta}(\B^{r+2m},\B^q;\C^{(r+m)\times m})}$ is small enough. This can be done by some scaling argument, see Lemma \ref{Lem::Key::IniNorm}.

There is a ``$\Co^{\alpha+1,\beta}_{(\tau,w),s}$-coordinate change'' $(\sigma,\zeta,s)=H(\tau,w,s)$ where $H:\B^{r+2m}_{\tau,w}\times\B^q_s\to\B^{r+2m}_{\sigma,\zeta}\times\B^q_s$ is of the form $H(\tau,w,s)=(\tau,H''(\tau,w,s),s)$, such that the local generators for $(H\circ L)_*\Se$ are $\Co^\beta$ and  real-analytic in $(\sigma,\zeta)$-variable. More precisely, the canonical local basis of $(H\circ L)_*\Se$ has the the form
\begin{equation}\label{Eqn::Key::DefofB}
\begin{pmatrix}T\\Z\end{pmatrix}=\begin{pmatrix}I_r&0&B'\\0&I_m&B''\end{pmatrix}\begin{pmatrix}\partial_\sigma\\\partial_\zeta\\\partial_{\bar \zeta}\end{pmatrix},\ B= \left.\begin{pmatrix}B'\\B''\end{pmatrix}\right|_{(\sigma,\zeta,s)}:=\left.\begin{pmatrix}I&H''_\tau+ A'H''_{\bar w}\\&H''_ w+ A''H''_{\bar w}\end{pmatrix}^{-1}\!\!\begin{pmatrix}\bar H''_\tau+ A'\bar H''_{\bar w}\\\bar H''_ w+ A''\bar H''_{\bar w}\end{pmatrix}\right|_{H^\Inv(\sigma,\zeta,s)}.
\end{equation}
Here $B'$ and $B''$ are defined on $\B^{r+2m}_{\sigma,\zeta}\times\B^q_s\subset(\R^r\times\C^m)\times\R^q$ and admit $(\sigma,\zeta)$-variables holomorphic extensions to the domain $\frac12\Hb^{r+2m}\times\frac12\B^q\subset\C^{r+2m}\times\R^q$.
In this case we note that $H''_\tau,H''_{\bar w}\in\Co^{\alpha,\beta-}_{(\tau,w),s}$ has a regularity loss on the parameter, while $T,Z\in\Co^{\infty,\beta}_{(\sigma,\zeta),s}$ retrieve the full $\Co^\beta$ regularity on the parameter. 

We denote by $\Tf$ and $\Zf$ the  collections of holomorphic vector fields corresponding to $T$ and $Z$. Finally, applying Proposition \ref{Prop::ODE::ParaHolFro}, the holomorphic Frobenius theorem with parameter, on $\Tf,\Zf$ we can find a $\Co^{\infty}_{\sigma,\zeta}\Co^\beta_s$-map $G''(\sigma,\zeta,s)$ such that $(H\circ L)_*\Se^\bot$ is spanned by $dG'',d\bar\zeta,ds$. See Lemma \ref{Lem::Key::FinalModf}.

In this way, the composition map $\overline{G''}\circ H\circ L$ is as desired.

\medskip To make $B(\sigma,\zeta,s)$ real-analytic, we impose extra conditions on $H$: We require $H''\in\Co^{\alpha+1}_{\tau,w}L^\infty_s\cap\Co^1_{\tau,w}\Co^\beta_s$, and $H$ satisfies the following condition:
\begin{equation}\label{Eqn::Key::KeyEqn1}
\sum_{j=1}^r\frac{\partial T_j}{\partial \sigma^j}+\sum_{k=1}^m\frac{\partial Z_k}{\partial \bar \zeta^k}=0.
\end{equation}

Rewriting \eqref{Eqn::Key::KeyEqn1} in the $(\tau,w,s)$-space using the coordinate change $H$, we get \eqref{Eqn::PDE::ExistenceH}. See Lemma \ref{Lem::Key::NewGen} for this deduction. The existence of such $H$ is done in Proposition \ref{Prop::PDE::ExistPDE}.

We write $B'=(b_j^l)_{\substack{1\le j\le r\\1\le l\le m}}$ and $B''=(b_{r+k}^l)_{\substack{1\le k\le m\\1\le l\le m}}$.
By Lemma \ref{Lem::ODE::GoodGen} \ref{Item::ODE::GoodGen::InvComm} the involutivity of $\Se$ implies that $T_1,\dots,T_r,Z_1,\dots,Z_m$ are commutative. Writing out the conditions $[T_j,T_{j'}]=0$, $[T_j,Z_k]=0$ and $[Z_k,Z_{k'}]=0$ in terms of $B'$ and $B''$ and combining \eqref{Eqn::Key::KeyEqn1}, we have the following:
\begin{equation}\label{Eqn::Key::KeyEqnB}
\begin{aligned}
\frac{\partial b_{j'}^l}{\partial \sigma^j}-\frac{\partial b_j^l}{\partial \sigma^{j'}}&=\sum_{q=1}^m\Big(b_{j'}^q\frac{\partial b_j^l}{\partial\bar  \zeta^q}-b_j^q\frac{\partial b_{j'}^l}{\partial\bar \zeta^q}\Big),&1\le j<j'\le r,1\le l\le m;\\
\frac{\partial b_{r+k}^l}{\partial \sigma^j}-\frac{\partial b_j^l}{\partial \zeta^k}&=\sum_{q=1}^m\Big(b_{r+k}^q\frac{\partial b_j^l}{\partial\bar \zeta^q}-b_j^q\frac{\partial b_{r+k}^l}{\partial\bar \zeta^q}\Big),&1\le j\le r,1\le k,l\le m;\\
\frac{\partial b_{r+k'}^l}{\partial \zeta^k}-\frac{\partial b_{r+k}^l}{\partial \zeta^{k'}}&=\sum_{q=1}^m\Big(b_{r+k'}^q\frac{\partial b_{r+k}^l}{\partial\bar \zeta^q}-b_{r+k}^q\frac{\partial b_{r+k'}^l}{\partial\bar \zeta^q}\Big),&1\le k<{k'}\le m,1\le l\le m;\\
\sum_{j=1}^r\frac{\partial b_j^l}{\partial\sigma^j}+\sum_{k=1}^m\frac{\partial b_{r+k}^l}{\partial\bar\zeta^k}&=0,&1\le l\le m.
\end{aligned}\end{equation}

\eqref{Eqn::Key::KeyEqnB} is a quasilinear PDE system of the form $DB=\Theta(B,\nabla_{\sigma,\zeta}B)$, where $D$ is a first order differential operator with constant (vector-valued) coefficients and $\Theta$ is a product map with constant (vector-valued) coefficients. We can check that $D^*D=\Delta_\sigma+\Box_\zeta$ is the real and complex Laplacian acting on the components. So \eqref{Eqn::Key::KeyEqnB} is an elliptic PDE system and with some simple modification we can apply Proposition \ref{Prop::PDE::AnalyticPDE}. See Lemma \ref{Lem::Key::BSatAna}.

\subsection{Some transition lemmas}\label{Section::KeyLems}
We first construct the coordinate chart $L$.
\begin{lem}\label{Lem::Key::IniNorm}
Endow $\R^r,\C^m,\R^q$ with standard coordinate systems $\tau=(\tau^1,\dots,\tau^r)$, $w=(w^1,\dots,w^m)$, $s=(s^1,\dots,s^q)$ respectively.

Let $\alpha>1$, $\beta>0$. Let $\Se\le(\C T\Mf)\times\Nf$ and $(u_0,v_0)\in\Mf\times\Nf$ be as in the assumption of Theorem \ref{Thm::Key}. Then for any $\delta>0$ there are neighborhoods $U_0\subseteq\Mf$ of $p_0$, $V_0\subseteq\Nf$ of $q_0$ and smooth coordinate charts $(L',L''):U_0\xrightarrow{\sim}\B^{r+2m}\subset\R^r_\tau\times\C^m_w$, $L''':V_0\xrightarrow{\sim}\B^q_s$, such that for the product chart $L=(L',L'',L''')$, $L_*\Se$ has a local basis $X=[X_1,\dots,X_{r+m}]^\top$ with the form \eqref{Eqn::Key::DefofA} defined on $\B^{r+2m}_{\tau,w}\times\B^q_s$, and
\begin{equation}\label{Eqn::Key::IniNorm::A}
    \|A\|_{\Co^{\alpha,\beta}(\B^{r+2m},\B^q;\C^{(r+m)\times m})}<\delta.
\end{equation}
\end{lem}
\begin{proof}Let $\tilde L''':\tilde V_0\subseteq\Nf\to\R^q_s$ be an arbitrary smooth coordinate chart near $v_0$ such that $\tilde L'''(v_0)=0$. Since $\rank\Se=r+m$ and $\rank(\Se+\bar\Se)=r+2m$, by standard linear algebra argument (see also \cite[Page 18]{Involutive}) we can find a smooth coordinate chart $(\tilde L',\tilde L''):\tilde U_0\subseteq\Mf\to\R^r_\tau\times\C^m_w$ such that
\begin{equation}\label{Eqn::Key::PfIniNorm::GenAssum1}
    \textstyle(\tilde L',\tilde L'')(u_0)=(0,0)\text{ and }\Se_{(u_0,v_0)}=(\tilde L',\tilde L'')^*\Span\big(\Coorvec{\tau^1},\dots,\Coorvec{\tau^r},\Coorvec{w^1},\dots,\Coorvec{w^m}\big)\big|_{u_0}\le \C T_{u_0}\Mf.
\end{equation}
Therefore, by considering $\tilde L:=(\tilde L',\tilde L'',\tilde L''')$ and working on $\Se_{(u_0,v_0)}\le\C T_{(u_0,v_0)}(\Mf\times\Nf)$, we have
\begin{equation}\label{Eqn::Key::PfIniNorm::GenAssum2}
    \textstyle(\tilde L_*\Se)_{(0,0,0)}\oplus\Span\big(\Coorvec{\bar w^1}\big|_{(0,0,0)},\dots,\Coorvec{\bar w^m}\big|_{(0,0,0)},\Coorvec{s^1}\big|_{(0,0,0)},\dots\Coorvec{s^q}\big|_{(0,0,0)}\big)=\R^r\times\C^m\times\R^q.
\end{equation}

Applying Lemma \ref{Lem::ODE::GoodGen}, since we have \eqref{Eqn::Key::PfIniNorm::GenAssum2}, there is a $\lambda_0>0$ such that $\lambda_0\B^{r+2m}_{\tau,w}\times\lambda_0\B^q_s\subseteq\tilde L(U_0\times V_0)$ and we can find a $\Co^\alpha$-local basis $\tilde X=[\tilde X_1,\dots,\tilde X_{r+m}]^\top$ of $\tilde L_*\Se$ that is defined on $\lambda_0\B^{r+2m}_{\tau,w}\times\lambda_0\B^q_s$ and has the form,
\begin{equation*}
    \tilde X=\begin{pmatrix}\tilde X'\\\tilde X''\end{pmatrix}=\begin{pmatrix}I_r&&\tilde A'&\tilde A'''\\&I_m&\tilde A''&\tilde A''''\end{pmatrix}\begin{pmatrix}\partial_\tau\\\partial_w\\\partial_{\bar w}\\\partial_s\end{pmatrix},\quad\text{on }\lambda_0\B^{r+2m}_{\tau,w}\times\lambda_0\B^q_s.
\end{equation*}
Clearly $\tilde A'''=0$ and $\tilde A''''=0$ since $\Se\le(\C T\Mf)\times\Nf$. By \eqref{Eqn::Key::PfIniNorm::GenAssum1} we have $\tilde A'(0)=0$ and $\tilde A''(0)=0$.

Write $\tilde A=\begin{pmatrix}\tilde A'\\\tilde A''\end{pmatrix}$. By scaling (see \cite[Lemma 5.9 (i)]{SharpElliptic} or \cite[Lemma 5.16]{StreetYaoVectorFields}) there is a  $C=C(r,m,q,\lambda_0,\alpha,\beta)>0$ such that 
\begin{equation}\label{Eqn::Key::IniNorm::Scaling}
    \|\tilde A(\lambda\cdot)\|_{\Co^{\alpha,\beta}(\B^{r+2m},\B^q;\C^{(r+m)\times m})}\le C\lambda^{\min(\beta,\frac12)}\|\tilde A\|_{\Co^{\alpha,\beta}(\lambda_0\B^{r+2m},\lambda_0\B^q;\C^{(r+m)\times m})},\quad\forall \lambda\in(0,\lambda_0].
\end{equation}

Take $\lambda=(C\|\tilde A\|_{\Co^{\alpha,\beta}(\lambda_0\B^{r+2m},\lambda_0\B^q)})^{-2}\cdot(\frac\delta2)^2$. Then the right hand side of \eqref{Eqn::Key::IniNorm::Scaling} is less than $\delta$. 

Take $L:=\lambda\tilde L$, we see that $L_*\Se$ has a $\Co^{\alpha,\beta}$-local basis $X=[X_1,\dots,X_{r+m}]^\top$ with the form
\begin{equation*}
    X=\begin{pmatrix}I_r&&A'\\&I_m& A''\end{pmatrix}\begin{pmatrix}\partial_\tau\\\partial_w\\\partial_{\bar w}\end{pmatrix},\quad\text{where }\begin{pmatrix}
    A'(\tau,w,s)\\A''(\tau,w,s)
    \end{pmatrix}=\begin{pmatrix}
    \tilde A'(\lambda\tau,\lambda w,\lambda s)\\\tilde A''(\lambda\tau,\lambda w,\lambda s)
    \end{pmatrix}\text{ for }(\tau,w)\in\B^{r+2m},s\in\B^q.
\end{equation*}

So we have $\|A\|_{\Co^{\alpha,\beta}(\B^{r+2m},\B^q)}=\|\tilde A(\lambda\cdot)\|_{\Co^{\alpha,\beta}(\lambda_0\B^{r+2m},\lambda_0\B^q)}<\delta$ as desired.
\end{proof}

Next we show that \eqref{Eqn::Key::KeyEqn1} and \eqref{Eqn::PDE::ExistenceH} are associated. Here we endow $\R^r\times\C^m$ with two standard coordinate systems $(\tau,w)=(\tau^1,\dots,\tau^r,w^1,\dots,w^m)$ and $(\sigma,\zeta)=(\sigma^1,\dots,\sigma^r,\zeta^1,\dots,\zeta^m)$, and we still use $s=(s^1,\dots,s^q)$ as the standard coordinate system for $\R^q$.
\begin{lem}\label{Lem::Key::NewGen}Let $\eps,\delta>0$, $A:\B^{r+2m}_{\tau,w}\times\B^q_s\to\C^{(r+m)\times m}$ and $H:\B^{r+2m}_{\tau,w}\times\B^q_s\to\B^{r+2m}_{\sigma,\zeta}\times\B^q_s$ be as in (the assumptions and consequence) of Proposition \ref{Prop::PDE::ExistPDE}. Then,
\begin{enumerate}[parsep=-0.3ex,label=(\roman*)]
    \item\label{Item::Key::NewGen::B} If $L_*\Se$ has a local basis with the form \eqref{Eqn::Key::DefofA}, then $(H\circ L)_*\Se$ has a local basis with the form \eqref{Eqn::Key::DefofB}.
    \item\label{Item::Key::NewGen::Eqn} Suppose $\Se$ is involutive and $H''$ is a solution to \eqref{Eqn::PDE::ExistenceH}, then the coefficient matrix $B=[B',B'']^\top$ in \eqref{Eqn::Key::DefofB} is a solution to \eqref{Eqn::Key::KeyEqnB}.
\end{enumerate} 
\end{lem}
\begin{proof}
Following Convention \ref{Conv::PDE::ConvofExtPDE}, by assumption $H(\tau,w,s)=(\tau,H''(\tau,w,s),s)$, we have
\begin{equation}\label{Eqn::Key::NewGen::H}
    H^*(d\sigma,d\zeta,d\bar\zeta,ds)=(d\tau,dw,d\bar w,ds)\begin{pmatrix}I&H''_\tau&\bar H''_\tau\\&H''_w&\bar H''_w\\&H''_{\bar w}&\bar H''_{\bar w}\\& H''_s&\bar H''_s&I\end{pmatrix}\ \Rightarrow\ \begin{pmatrix}
    \Coorvec\tau\\\Coorvec w\\\Coorvec{\bar w}\end{pmatrix}=\begin{pmatrix}I&H''_\tau&\bar H''_\tau\\&H''_w&\bar H''_w\\&H''_{\bar w}&\bar H''_{\bar w}\end{pmatrix}\cdot H^*\begin{pmatrix}
    \Coorvec\sigma\\\Coorvec\zeta\\\Coorvec{\bar\zeta}\end{pmatrix}.
\end{equation}
Therefore taking pushforward of \eqref{Eqn::Key::DefofA} under $H$ we have, \begin{equation*}
    H_*X=H_*\left(\begin{pmatrix}I&&A'\\&I&A''\end{pmatrix}\begin{pmatrix}I&H''_\tau&\bar H''_\tau\\&H''_w&\bar H''_w\\&H''_{\bar w}&\bar H''_{\bar w}\end{pmatrix}\right)\begin{pmatrix}
    \Coorvec\sigma\\\Coorvec\zeta\\\Coorvec{\bar\zeta}\end{pmatrix}=\left(\begin{pmatrix}I&H''_\tau+A'H''_{\bar w}&\bar H''_\tau+A'\bar H''_{\bar w}\\&H''_w+A''\bar H''_{\bar w}&\bar H''_w+A''\bar H''_{\bar w}\end{pmatrix}\circ H^\Inv\right)\cdot\begin{pmatrix}
    \Coorvec\sigma\\\Coorvec\zeta\\\Coorvec{\bar\zeta}\end{pmatrix}.
\end{equation*}
The matrix $\begin{pmatrix}I&H''_\tau+A'H''_{\bar w}\\&H''_w+A''\bar H''_{\bar w}\end{pmatrix}$ is invertible (see Lemma \ref{Lem::PDE::MatisInv} \ref{Item::PDE::MatisInv::Inv}), we see that $H_*(L_*\Se)=(H\circ L)_*\Se$ has a local basis
\begin{equation*}
    H_*\left(\begin{pmatrix}I&H''_\tau+A'H''_{\bar w}\\&H''_w+A''\bar H''_{\bar w}\end{pmatrix}^{-1}\begin{pmatrix}X'\\X''\end{pmatrix}\right)=\left(\left(I_{r+m}+\begin{pmatrix}I&H''_\tau+ A'H''_{\bar w}\\&H''_ w+ A''H''_{\bar w}\end{pmatrix}^{-1}\begin{pmatrix}\bar H''_\tau+ A'\bar H''_{\bar w}\\\bar H''_ w+ A''\bar H''_{\bar w}\end{pmatrix}\right)\circ H^\Inv\right)\cdot\begin{pmatrix}
    \Coorvec\sigma\\\Coorvec\zeta\\\Coorvec{\bar\zeta}\end{pmatrix}.
\end{equation*}
This gives \eqref{Eqn::Key::DefofB} and finishes the proof of \ref{Item::Key::NewGen::B}.

\medskip Recall the notations $B'=(b_j^l)_{\substack{1\le j\le r;1\le l\le m}}$ and $B''=(b_{r+k}^l)_{\substack{1\le k,l\le m}}$. By \eqref{Eqn::Key::DefofB} we have,
\begin{equation}\label{Eqn::Key::NewGen::Tmp1}
    b_j^l\circ H=\left(\begin{pmatrix}I&H''_\tau+ A'H''_{\bar w}\\&H''_ w+ A''H''_{\bar w}\end{pmatrix}^{-1}\!\!\begin{pmatrix}\bar H''_\tau+ A'\bar H''_{\bar w}\\\bar H''_ w+ A''\bar H''_{\bar w}\end{pmatrix}\right)_j^l,\quad 1\le j\le m+r,\quad 1\le l\le m.
\end{equation}

By \eqref{Eqn::Key::NewGen::H} we have, for $1\le j\le r$ and $1\le k\le m$,
\begin{equation}\label{Eqn::Key::NewGen::Tmp2}
    \left.\Coorvec{\sigma^j}\right|_{H(\tau,w,s)}=\left(\begin{pmatrix}I&H''_\tau&\bar H''_\tau\\&H''_w&\bar H''_w\\&H''_{\bar w}&\bar H''_{\bar w}\end{pmatrix}^{-1}\;\begin{pmatrix}
    \Coorvec\tau\\\Coorvec w\\\Coorvec{\bar w}\end{pmatrix}\right)_j,
    \
    \left.\Coorvec{\bar\zeta^k}\right|_{H(\tau,w,s)}=\left(\begin{pmatrix}I&H''_\tau&\bar H''_\tau\\&H''_w&\bar H''_w\\&H''_{\bar w}&\bar H''_{\bar w}\end{pmatrix}^{-1}\;\begin{pmatrix}
    \Coorvec\tau\\\Coorvec w\\\Coorvec{\bar w}\end{pmatrix}\right)_{r+m+k}.
\end{equation}

Plugging \eqref{Eqn::Key::NewGen::Tmp1} and \eqref{Eqn::Key::NewGen::Tmp2} into \eqref{Eqn::PDE::ExistenceH} we get $\sum_{j=1}^r\frac{\partial b_j^l}{\partial\sigma^j}+\sum_{k=1}^m\frac{\partial b_{r+k}^l}{\partial\bar\zeta^k}=0$ for $1\le l\le m$, which is the last row of \eqref{Eqn::Key::KeyEqnB}.

Since $\Se$ is involutive, $(H\circ L)_*\Se$ is also involutive, by Lemma \ref{Lem::ODE::GoodGen} \ref{Item::ODE::GoodGen::InvComm}, $T_1,\dots,T_r,Z_1,\dots,Z_m$ given in \eqref{Eqn::Key::DefofB} are pairwise commutative. Writing out the condition $[T_j,T_{j'}]=0$ $(1\le j<j'\le r)$, $[T_j,Z_k]=0$ $(1\le j\le r,1\le k\le m)$ and $[Z_k,Z_{k'}]=0$ $(1\le k<k'\le m)$ in the coordinate components we get the first three rows of \eqref{Eqn::Key::KeyEqnB}, finishing the proof of \ref{Item::Key::NewGen::Eqn}.
\end{proof}

Now we have the PDE system \eqref{Eqn::Key::KeyEqnB}. We show that it satisfies the assumption of Proposition \ref{Prop::PDE::AnalyticPDE}.
\begin{lem}\label{Lem::Key::BSatAna}
Suppose $B:\B^{r+2m}_{\sigma,\zeta}\times\B^q_s\to\C^{(r+m)\times m}$ solves \eqref{Eqn::Key::KeyEqnB}, then $f(\sigma,\xi,\eta,s):=B(\sigma,\frac12(\xi+i\eta),s)$ is defined on $\B^{r+2m}_{\sigma,\xi,\eta}\times\B^q_s$ and $f$ solves a PDE with the form $Lf=\Delta_{\sigma,\xi,\eta}\Theta(f,\nabla_{\sigma,\xi,\eta}f)$ where $L$ and $\Theta$ satisfy the assumptions of Proposition \ref{Prop::PDE::AnalyticPDE}.
\end{lem}
\begin{proof}
We can write \eqref{Eqn::Key::KeyEqnB} as the form $DB=\tilde\Theta(B,\nabla_{\sigma,\xi,\eta}B)$ where the coefficients of $D$ and $\Theta$ take value in $\C^{m\times\frac{(r+m)(r+m-1)}2+m}$.

By direct computation (see also \cite[Lemma B.5]{SharpElliptic}) we have $D^*D=\sum_{j=1}^r\frac{\partial^2}{\partial\sigma^j\partial\sigma^j}+\sum_{k=1}^m\frac{\partial^2}{\partial\zeta^k\partial\bar\zeta^k}=\Delta_\sigma+\Box_\zeta$, acting on the components of the vector value functions. 

Using the real coordinates $\xi=\re\zeta$, $\eta=\im\zeta$, we see that $\Box_\zeta=\frac14\sum_{k=1}^m(\frac{\partial^2}{\partial\xi^k\partial\xi^k}+\frac{\partial^2}{\partial\eta^k\partial\eta^k})=\frac14\Delta_{\xi,\eta}$. So for $f(\sigma,\xi,\eta,s)=B(\sigma,\frac{\xi+i\eta}2,s)$ we have $(D^*DB)(\sigma,\frac{\xi+i\eta}2,s)=(\Delta_\sigma+\Delta_{\xi,\eta})f(\sigma,\xi,\eta,s)$ and $\nabla_{\xi,\eta}f(\sigma,\xi,\eta,s)=\frac12(\nabla_{\xi,\eta}B)(\sigma,\frac{\xi+i\eta}2,s)$. Therefore $f$ satisfies $\Delta_{\sigma,\xi,\eta}f=D^*\tilde \Theta(f,(\nabla_\sigma f,2\nabla_{\xi,\eta}f))$. Taking $L:=D^*$ and $\Theta(f,\nabla_{\sigma,\xi,\eta}f):=\tilde \Theta(f,(\nabla_\sigma f,2\nabla_{\xi,\eta}f))$ we finish the proof.
\end{proof}

Finally we find the map $G''$ in \eqref{Eqn::Key::FacterizationofF} using Proposition \ref{Prop::ODE::ParaHolFro}. We work on $\frac12\B^{r+2m}\subset\R^{r+2m}_{\sigma,\xi,\eta}$ and its ``complex extension'' $\frac12\Hb^{r+2m}\subset\C^{r+2m}$.

We endow $\C^{r+2m}$ with standard complex coordinates $(\boldsymbol\sigma,\boldsymbol\xi,\boldsymbol\eta)=(\boldsymbol\sigma^1,\dots,\boldsymbol\sigma^r,\boldsymbol\xi^1,\dots,\boldsymbol\xi^m,\boldsymbol\eta^1,\dots,\boldsymbol\eta^m)$ in the way that $\sigma=\re\boldsymbol\sigma$, $\xi=\re\boldsymbol\xi$ and $\eta=\re\boldsymbol\eta$. We denote $\Coorvec{\boldsymbol\zeta}:=\frac12(\Coorvec{\boldsymbol\xi}-i\Coorvec{\boldsymbol\eta})=[\frac12(\Coorvec{\boldsymbol\xi^1}-i\Coorvec{\boldsymbol\eta^1}),\dots,\frac12(\Coorvec{\boldsymbol\xi^m}-i\Coorvec{\boldsymbol\eta^m})]^\top$. 

Note that in this place, we temporarily forget the original complex structure of $\C^m_\zeta$. The objects $\boldsymbol\zeta=\boldsymbol\xi+i\boldsymbol\eta$ and $\Coorvec{\boldsymbol\zeta}$ are just conventions.

In this setting, $T_1,\dots,T_r,Z_1,\dots,Z_m$ given in \eqref{Eqn::Key::DefofB} are complex vector fields defined on $\frac12\B^{r+2m}\times\frac12\B^q\subset\R^{r+2m}_{\sigma,\xi,\eta}\times\R^q_s$.
\begin{lem}\label{Lem::Key::FinalModf}
Suppose $T_1,\dots,T_r,Z_1,\dots,Z_m$ in \eqref{Eqn::Key::DefofB} are commutative, and the coefficient matrix function $B:\frac12\B^{r+2m}_{\sigma,\xi,\eta}\times\frac12\B^q_s\to\C^{(r+m)\times m}$ has a $(\sigma,\xi,\eta)$-variable holomorphic extension $\Bf\in\Co^\infty_\loc\Co^\beta(\frac12\Hb^{r+2m},\frac12\B^q;\C^{(r+m)\times m})$. Then for any $(\sigma_0,\xi_0,\eta_0,s_0)\in \frac12\B^{r+2m}\times\frac12\B^q$ there is a neighborhood $U_1\times V_1\subset\frac12\B^{r+2m}\times\frac12\B^q$ of $(\sigma_0,\xi_0,\eta_0,s_0)$ and a map $G''\in\Co^\infty\Co^\beta(U_1,V_1;\C^m)$ such that $G''(\sigma_0,\xi_0,\eta_0,s_0)=0$ and
\begin{equation*}
    (\Span(T_1(\cdot,s),\dots,T_r(\cdot,s),Z_1(\cdot,s),\dots,Z_m(\cdot,s)))^\bot|_{U_1}=\Span(dG''(\cdot,s)^1,\dots,dG''(\cdot,s)^m)\le\C T^*U_1,\quad\forall\ s\in V_1.
\end{equation*}
\end{lem}
Recall the notation of dual bundle in \eqref{Eqn::ODE::DualConjBundle}.
\begin{proof}
We define collections of vector fields $\Tf=[\Tf_1,\dots,\Tf_r]^\top$ and $\Zf=[\Zf_1,\dots,\Zf_m]^\top$  as
\begin{equation*}
    \begin{pmatrix}\Tf\\\Zf\end{pmatrix}:=\begin{pmatrix}I_r&&\Bf'\\&I_m&\Bf''\end{pmatrix}\begin{pmatrix}\partial_{\boldsymbol\sigma}\\\partial_{\boldsymbol\zeta}\\\partial_{\bar{\boldsymbol\zeta}}\end{pmatrix}.
\end{equation*}
Here $\Bf'$ is the upper $r\times m$ block of $\Bf$, and $\Bf''$ is the lower $m\times m$ block of $\Bf$.

Thus $(T_1,\dots,T_r,Z_1,\dots,Z_m)$  is the ``real domain restriction'' of $(\Tf_1,\dots,\Tf_r,\Zf_1,\dots,\Zf_m)$ in the sense of Proposition \ref{Prop::ODE::ParaHolFro}. We see that the coefficients of $[\Tf_j,\Tf_{j'}]$, $[\Tf_j,\Zf_k]$ and $[\Zf_k,\Zf_{k'}]$ are the holomorphic extension those of $[T_j,T_{j'}]$, $[T_j,Z_k]$ and $[Z_k,Z_{k'}]$. 

By assumption $T_1,\dots,T_r,Z_1,\dots,Z_m$ are commutative, so $[T_j,T_{j'}]=[T_j,Z_k]=[Z_k,Z_{k'}]=0$. Since a holomorphic function in $\Hb^{r+2m}$ that vanishes in $\B^{r+2m}$ must be a zero function, we know $[\Tf_j,\Tf_{j'}]=[\Tf_j,\Zf_k]=[\Zf_k,\Zf_{k'}]=0$ as well. Thus $\Tf_1,\dots,\Tf_r,\Zf_1,\dots,\Zf_m$ are commutative.

Now the assumptions of Proposition \ref{Prop::ODE::ParaHolFro} are satisfied. By Proposition \ref{Prop::ODE::ParaHolFro} we can find a neighborhood $\tilde U_1\times V_1\subseteq\frac12\Hb^{r+2m}\times\frac12\B^q$ of $(\sigma_0,\xi_0,\eta_0,s_0)$ and a map $\widetilde G''\in\Co^\infty\Co^\beta(\tilde U_1,V_1;\C^m)$ such that $\widetilde G''(\sigma_0,\xi_0,\eta_0,s_0)=0$ and for $U_1:=\tilde U_1\cap\R^{r+2m}$ we have for each $s\in V_1$,
$$(\Span(T_1(\cdot,s),\dots,T_r(\cdot,s),Z_1(\cdot,s),\dots,Z_m(\cdot,s)))^\bot|_{U_1}=\Span(d(\widetilde G''(\cdot,s)^1|_{U_1}),\dots,d(\widetilde G''(\cdot,s)^m|_{U_1}))\le \C T^*U_1.$$

Clearly $\widetilde G''|_{U_1\times V_1}\in\Co^\infty\Co^\beta(U_1,V_1;\C^m)$ holds. Thus, $G'':=\widetilde G''|_{U_1\times V_1}$ is as desired.
\end{proof}

\subsection{Proof of Theorem \ref{Thm::Key}}\label{Section::KeyPf}

In the proof we often consider the local bases of cotangent subbundles. 
\begin{conv}\label{Conv::Key::CoSpan}
Let $M$ and $N$ be two smooth manifolds, let $\Tc\le (\C TM)\times N$ be a continuous subbundle. Let $\mu=(\mu^1,\dots,\mu^n):M\to\R^n$ and $\nu=(\nu^1,\dots,\nu^q):N\to\R^q$ be two smooth coordinate charts and let $f=(f^1,\dots,f^{n-r}):M\times N\to\C^{n-r}$ be a continuous map which is $C^1$ in $M$. We say $\Tc^\bot$ is spanned by $df^1,\dots,df^{n-r},d\nu^1,\dots,d\nu^q$, denoted by $\Tc^\bot=\Span(df,d\nu)$, if
\begin{equation}\label{Eqn::Key::CoSpan}
    \Span\Big(\sum_{j=1}^n\frac{\partial f^i}{\partial\mu^j}(x,s)d\mu^j|_x\Big)_{i=1}^r=\{\theta\in T^*_xM:\langle\theta,X\rangle=0,\ \forall X\in\Tc_{(x,s)}\},\quad\forall\ u\in M,\ s\in N.
\end{equation}
\end{conv}

\begin{remark}
When $\beta>1$ and when $f\in\Co^\beta$, \eqref{Eqn::Key::CoSpan} coincides with the classical notion $\Tc^\bot=\Span(df,d\nu)$ since $\Span(df,d\nu)=\Span(\frac{\partial f}{\partial\mu}d\mu+\frac{\partial f}{\partial\nu}d\nu,d\nu)=\Span(\frac{\partial f}{\partial\mu}d\mu,d\nu)$. 

The convention is used specifically for the case $\beta\le1$, since when $f\in\Co^{\alpha+1,\beta}$, the collection of differentials $df$ is no longer pointwise defined, while $\frac{\partial f}{\partial\mu}\cdot d\mu\equiv df \pmod{d\nu}$ and the left hand side is still pointwise.
\end{remark}

We now begin the proof.
On the space $\R^r_\sigma\times\C^m_\zeta$, we write $\xi=\re\zeta$ and $\eta=\im\zeta$.

By Lemma \ref{Lem::Key::BSatAna} and Proposition \ref{Prop::PDE::AnalyticPDE}, along with a scaling, we can find a $\eps_0=\eps_0(r,m,q,\alpha)>0$ such that 
\begin{enumerate}[parsep=-0.3ex,label=(B.1)]
    \item\label{Item::Key::PfKey::B} If $B\in\Co^\alpha L^\infty\cap\Co^{-1}\Co^\beta(\B^{r+2m}_{\sigma,\eta,\xi},\B^q_s;\C^{(r+m)\times m})$ satisfies \eqref{Eqn::Key::KeyEqnB} with 
    \begin{equation}\label{Eqn::Key::PfKey::BddB}
        \|B\|_{\Co^\alpha L^\infty\cap\Co^{-1}\Co^\beta(\B^{r+2m},\B^q)}<\eps_0,
    \end{equation}then $B(\sigma,\frac{\xi+i\eta}2,s)$ has a   $(\sigma,\xi,\eta)$-holomorphic extension $\Bf(\boldsymbol\sigma,\frac12\boldsymbol\xi,\frac12\boldsymbol\eta,s)\in\Co^\infty_\loc\Co^\beta(\Hb^{r+2m}_{\boldsymbol\sigma,\boldsymbol\xi,\boldsymbol\eta},\B^q_s;\C^{(r+m)\times m})$.
\end{enumerate}
For such $\eps_0>0$, we can find a $\delta_0>0$ such that the results of Proposition \ref{Prop::PDE::ExistPDE} are satisfied with $\eps=\eps_0$, $\delta=\delta_0$.

For such $\delta_0>0$, by Lemma \ref{Lem::Key::IniNorm} there is a  smooth coordinate chart $L:=((L',L''),L'''):U_0\times V_0\subseteq\Mf\times\Nf\xrightarrow{\sim}\B^{r+2m}_{\tau,w}\times\B^q_s$ near $(u_0,v_0)$ such that,
\begin{enumerate}[parsep=-0.3ex,label=(L.\arabic*)]
    \item\label{Item::Key::PfKey::LBasic} $(L',L''):U_0\subseteq\Mf\xrightarrow{\sim}\B^{r+2m}_{\tau,w}$ and $L''':V_0\subseteq\Nf\xrightarrow{\sim}\B^q_s$ are both smooth coordinate charts, and $L(u_0,v_0)=(0,0)$.
    \item $L_*\Se$ has a $\Co^{\alpha,\beta}_{(\tau,w),s}$-local basis with form \eqref{Eqn::Key::DefofA}, where \eqref{Eqn::Key::IniNorm::A} is satisfied with $\delta=\delta_0$.
\end{enumerate}

Now $\|A\|_{\Co^{\alpha,\beta}(\B^{r+2m},\B^q;\C^{(r+m)\times m})}<\delta_0$ from \eqref{Eqn::Key::IniNorm::A}. By Proposition \ref{Prop::PDE::ExistPDE} and Corollary \ref{Cor::PDE::CorofExistPDE} we can find a $\Co^{\alpha+1,\beta}$-map $H:\B^{r+2m}_{\tau,w}\times\B^q_s\to\R^r_\sigma\times\C^m_\zeta\times\B^q_s$ that has the form $H(\tau,w,s)=(\tau,H''(\tau,w,s),s)$ such that
\begin{enumerate}[parsep=-0.3ex,label=(H.\arabic*)]
    \item\label{Item::Key::PfKey::HDiffeo} $H:\B^{r+2m}_{\tau,w}\times\B^q_s\to\B^{r+2m}_{\sigma,\zeta}\times\B^q_s$ is homeomorphism, and $H,H^\Inv\in\Co^{\alpha+1,\beta}(\B^{r+2m},\B^q;\R^r\times\C^m)$.
    \item\label{Item::Key::PfKey::HTmp} $H(\frac14\B^{r+2m}\times\frac14\B^q)\subseteq\frac12\B^{r+2m}\times\frac12\B^q$.
    \item\label{Item::Key::PfKey::HReg}$H''\in\Co^{\alpha+1,\beta}(\B^{r+2m},\B^q;\C^m)$ and $\nabla_{\tau,w}H''\in\Co^{\alpha,\beta-}(\B^{r+2m},\B^q;\C^{(r+2m)\times m})$.
    \item\label{Item::Key::PfKey::HPDE} $H''$ solves \eqref{Eqn::PDE::ExistenceH}.
    \item\label{Item::Key::PfKey::BddLambda} \eqref{Eqn::PDE::ExistPDE::BddLambda} is satisfied with $\eps=\eps_0$.
\end{enumerate}

Here \ref{Item::Key::PfKey::HDiffeo} shows that $(H\circ L)_*\Se$ is defined on $\B^{r+2m}_{\sigma,\zeta}\times\B^q_s$.

By Lemma \ref{Lem::Key::NewGen} \ref{Item::Key::NewGen::Eqn},  \ref{Item::Key::PfKey::HPDE} implies that $B$ solves \eqref{Eqn::Key::KeyEqnB}. By comparing \eqref{Eqn::PDE::matrixfun} and \eqref{Eqn::Key::DefofB} we see that $B=\Lambda[A,H'']\circ H^\Inv$, so \ref{Item::Key::PfKey::BddLambda} implies \eqref{Eqn::Key::PfKey::BddB} with $\eps=\eps_0$. Thus we get the result \ref{Item::Key::PfKey::B}: a matrix-valued function $\Bf\in\Co^\infty_\loc\Co^\beta(\frac12\Hb^{r+2m},\frac12\B^q;\C^{(r+m)\times m})$.

Applying Lemma \ref{Lem::Key::FinalModf} on $H\circ L(u_0,v_0)= H(0,0)\in\frac12\B^{r+2m}\times\frac12\B^q$, we can find a neighborhood $U_1\times V_1\subseteq \frac12\B^{r+2m}\times\frac12\B^q$ of $H(0,0)$ and a $G'':U_1\times V_1\to\C^m$ such that

\begin{enumerate}[parsep=-0.3ex,label=(G.\arabic*)]
    \item\label{Item::Key::PfKey::GReg} $G''\in\Co^\infty\Co^\beta(U_1,V_1;\C^m)$ and $G''(H(0,0))=0$. In particular $G''\in\Co^{\infty,\beta}_{(\sigma,\zeta),s}$ and $\nabla_{\sigma,\zeta}G''\in\Co^{\infty,\beta}_{(\sigma,\zeta),s}$.
    \item\label{Item::Key::PfKey::GSpan} For each $s\in V_1$, $\Span d(G''(\cdot,s))=\Span(T(\cdot,s),Z(\cdot,s))|_{U_1}^\bot=(H\circ L)_*\Se|_{U_1\times \{s\}}^\bot$. Equivalently we have $\Span(dG'',ds)=(H\circ L)_*\Se|_{U_1\times V_1}^\bot $ in the sense of Convention \ref{Conv::Key::CoSpan}.
\end{enumerate}

Therefore, $(\Span d(G''\circ H),ds)|_{H^{-1}(U_1\times V_1)}=L_*\Se|_{H^{-1}(U_1\times V_1)}$ in the sense of Convention \ref{Conv::Key::CoSpan}, which means
\begin{equation}\label{Eqn::Key::PfKey::Span}
    \Span(d(G''\circ H\circ L),dL''')|_{(H\circ L)^{-1}(U_1\times V_1)}=\Se|_{(H\circ L)^{-1}(U_1\times V_1)}^\bot\text{ in the sense of Convention \ref{Conv::Key::CoSpan}}.
\end{equation}

Next, we claim that 
\begin{enumerate}[parsep=-0.3ex,label=(L.3)]
    \item\label{Item::Key::PfKey::Claim} $\big(dL',d(\re G''\circ H\circ L(\cdot,v_0)),d(\im G''\circ H\circ L(\cdot,v_0))\big)$ is a collection of $(r+2m)$-real differentials that are linearly independent at $u_0\in\Mf$.
\end{enumerate} 

Indeed, we have $\Span(d(\overline{G''}\circ H),ds)=L_*\bar\Se^\bot$ in the sense of Convention \ref{Conv::Key::CoSpan}, which means for each $s\in V_1$,
$$\Span(d(\re G''\circ H(\cdot,s)),d(\im G''\circ H(\cdot,s))^\bot=\Span(d(G''\circ H(\cdot,s)),d(\overline{G''}\circ H(\cdot,s)))^\bot=L_*(\Se\cap\bar \Se)|_{H^{-1}(U_1\times\{s\})},$$ has rank $r=\dim\Mf-(r+2m)$. Therefore, $(d(\re G''\circ H(\cdot,s),d(\im G''\circ H(\cdot,s))$ is a linearly independent collection in the domain for all $s\in V_1$.

On the other hand, by \eqref{Eqn::Key::DefofA} we have 
$L_*\Se=\Span(\Coorvec \tau+A'\Coorvec w,\Coorvec w+A''\Coorvec{\bar w})$, so $L_*(\Se\cap\bar\Se)=\Span(\Coorvec\tau+\frac12A'\Coorvec w+\frac12\bar A'\Coorvec{\bar w}$, which means $L_*(\Se\cap\bar\Se)^\bot\cap\Span d\tau=\{0\}$. We conclude that $(d\tau,d(\re G''\circ H(\cdot,s)),d(\im G''\circ H(\cdot,s))$ is a linearly independent collection in the domain for all $s\in V_1$. Since $dL'=L^*d\tau$ and $dL'''=L^*ds$, taking compositions with $L$ we obtain \ref{Item::Key::PfKey::Claim}.

\medskip
Now endow $\R^r\times\C^m\times\R^q$ with standard coordinates $(t,z,s)=(t^1,\dots,t^r,z^1,\dots,z^m,s^1,\dots,s^q)$. We take
\begin{equation}\label{Eqn::Key::PfKey::DefF}
    F=(F',F'',F'''):=(L',\overline{G''}\circ H\circ L,L'''):(H\circ L)^{-1}(U_1\times V_1)\subseteq\Mf\times\Nf\to\R^r_t\times\C^m_z\times\R^q_s.
\end{equation}

By \ref{Item::Key::PfKey::LBasic}, \ref{Item::Key::PfKey::HReg}, \ref{Item::Key::PfKey::GReg} and Lemma \ref{Lem::Hold::CompofMixHold} \ref{Item::Hold::CompofMixHold::Comp} we have 

\begin{enumerate}[parsep=-0.3ex,label=(F.1)]
    \item\label{Item::Key::PfKey::FReg1} $F'(u_0)=0$, $F''(u_0,v_0)=0$, $F'''(v_0)=0$. And for every product neighborhood $U\times V\subseteq (H\circ L)^{-1}(U_1\times V_1)$ of $(u_0,v_0)$, we have $F'\in\Co^\infty_\loc(U;\R^r)$, $F''\in\Co^{\alpha+1,\beta}_\loc(U,V;\C^m)$ and $F'''\in\Co^\infty_\loc(V;\R^q)$.
\end{enumerate}
In particular \ref{Item::Key::F'''} is satisfied, since $L'''$ is already a smooth coordinate chart.

By \ref{Item::Key::PfKey::FReg1} $(F',F'')\in\Co^{\alpha+1,\beta}(U,V;\R^r\times\C^m)$ for every $U\times V\subseteq (H\circ L)^{-1}(U_1\times V_1) $. By \ref{Item::Key::PfKey::Claim}, $(dF',dF''(\cdot,v_0))$ has full rank $r+2m$ in the domain.
Therefore, by Lemma \ref{Lem::Hold::CompofMixHold} \ref{Item::Hold::CompofMixHold::InvFun}, we can find a small enough neighborhood $U\times V\subseteq (H\circ L)^{-1}(U_1\times V_1)$ of $(u_0,v_0)$, such that $F:U\times V\to F(U\times V)\subseteq\R^r_t\times\C^m_z\times\R^q_s$ is homeomorphism, and 
\begin{enumerate}[parsep=-0.3ex,label=(F.2)]
    \item\label{Item::Key::PfKey::FReg2} $F^\Inv\in\Co^{\alpha+1,\beta}_{(t,z),s}(\Omega'\times\Omega'',\Omega''';\Mf\times\Nf)$ whenever $\Omega'_t\times\Omega''_z\times\Omega'''_s\subseteq F(U\times V)$.
\end{enumerate}

Define $\Phi:=F^\Inv:\Omega\to U\times V$ where $\Omega=\Omega'_s\times\Omega''_z\times\Omega'''_s$ is an arbitrary product neighborhood of $(0,0,0)$ in the assumption of Theorem \ref{Thm::Key}. We are going to show that $F$ and $\Phi$ are as desired.

\medskip
Firstly, by \ref{Item::Key::PfKey::FReg1}, we get $\Phi''=(L''')^\Inv=(F''')^\Inv$ is a smooth parameterization, which is the result \ref{Item::Key::Phi''}.

By \ref{Item::Key::PfKey::FReg1} and \ref{Item::Key::PfKey::FReg2} we know $F\in\Co^{\beta}$, $\Phi\in\Co^{\beta}$. Thus \ref{Item::Key::>1} holds and $F$, $\Phi$ are both homeomorphic to their images respectively. Since $F(\cdot,v)\in\Co^{\alpha+1}$, $\Phi(\cdot,s)\in\Co^{\alpha+1}$ for each $v$ and $s$, we know $(F',F''(\cdot,v))$ is a $\Co^{\alpha+1}$-chart and $\Phi'(\cdot,s)$ is a $\Co^{\alpha+1}$-parameterization, which gives \ref{Item::Key::FBase} and \ref{Item::Key::Phi0}.

\medskip
To prove \ref{Item::Key::F''Reg} and \ref{Item::Key::PhiReg}, we use $\widehat U:=(L',L'')(U)\subseteq\R^r_\tau\times\C^m_w$, $\widehat V:=L'''(V)\subseteq\R^q_s$, $\widehat F'':=F''\circ L^{-1}$ and $\widehat\Phi:=L\circ\Phi$.
Since $(L',L'')$ and $L'''$ are smooth charts, it is equivalent but more convenient to work on the chart $\widehat F:\widehat U\times\widehat V\to\R^r_t\times\C^m_z\times\R^q_s$ and the parameterization $\widehat\Phi:\Omega\to\widehat U\times\widehat V$:
 it suffices to show the following
\begin{equation}\label{Eqn::Key::PfKey::AllReg}
    \begin{gathered}
    \widehat F''\in\Co^{\alpha+1,\beta}_{(\tau,w),s}(\widehat U,\widehat V;\C^m), \quad\nabla_{\tau,w}\widehat F''\in \Co^{\alpha,\beta-}(\widehat U,\widehat V;\C^{(r+2m)\times m}),
    \\
    \widehat\Phi\in\Co^{\alpha+1,\beta}_{(t,z),s}(\Omega'\times\Omega'',\Omega''';\widehat U\times \widehat V),
    \quad\nabla_{t,z}\widehat\Phi\in\Co^{\alpha,\beta-}(\Omega'\times\Omega'',\Omega''';\C^{(r+2m)\times(r+2m)}).
\end{gathered}
\end{equation}

Here for every vector field $X\in\Co^\infty_\loc(\Mf;T\Mf)$, $(L',L'')_*X$ must be the smooth linear combinations of $\Coorvec\tau,\Coorvec w,\Coorvec{\bar w}$ on $\widehat U$. Thus $\nabla_{\tau,w}\widehat F''\in \Co^{\alpha,\beta-}_{(\tau,w),s}$ implies $XF''\in\Co^{\alpha,\beta-}_\loc$.

\medskip\noindent\textit{Proof of \eqref{Eqn::Key::PfKey::AllReg}}: By \eqref{Eqn::Key::PfKey::DefF} we have $\widehat F''(\tau,w,s)=\overline{G''}(H(\tau,w,s))=\overline{G''}(\tau,H''(\tau,w,s),s)$.
By \ref{Item::Key::PfKey::HReg} and \ref{Item::Key::PfKey::GReg} we have $G''\in\Co^{\infty,\beta}_{(\sigma,\zeta),s}$, $H\in\Co^{\alpha+1,\beta}_{(\tau,w),s}$, $\nabla_{\sigma,\zeta}G''\in\Co^{\infty,\beta}_{(\sigma,\zeta),s}$ and $\nabla_{\tau,w}H\in\Co^{\alpha,\beta-}_{(\tau,w),s}$. Therefore, by Lemma \ref{Lem::Hold::CompofMixHold} \ref{Item::Hold::CompofMixHold::Comp} we get $\widehat F''\in \Co^{\alpha+1,\beta}_{(\tau,w),s}$ and $\nabla_{\tau,w}\widehat F''=((\nabla_{\sigma,\zeta}\overline{G''})\circ H)\cdot\nabla_{\tau,w}H\in\Co^{\alpha+1,\beta}_{(\tau,w),s}\cdot \Co^{\alpha,\beta-}_{(\tau,w),s}\subseteq \Co^{\alpha,\beta-}_{(\tau,w),s}$.

Since $F\circ L^{-1}(\tau,w,s)=(\tau,\widehat F''(\tau,w,s),s)$, we see that $\widehat\Phi$ has the form
\begin{equation*}
    \widehat\Phi(t,z,s)=(t,\widehat\Psi(t,z,s),s),\quad\text{where }\widehat\Psi:\Omega'_t\times\Omega''_z\times\Omega'''_s\to\C^m_w,\quad\widehat\Psi(t,z,s):=\widehat F''(t,\cdot,s)^\Inv(z).
\end{equation*}

Since $\widehat F''\in\Co^{\alpha+1,\beta}_{(\tau,w),s}$, applying Lemma \ref{Lem::Hold::CompofMixHold} \ref{Item::Hold::CompofMixHold::InvFun} to the map $[(\tau,w),s)\mapsto(\tau,\widehat F''(\tau,w,s))]$ we get $\widehat\Psi\in\Co^{\alpha+1,\beta}_{(t,z),s}$. 

Using the chain rules on $z=\widehat F''(\widehat\Phi(t,z,s))=\widehat F''(t,\widehat\Psi(t,z,s),s)$ we have (recall $\nabla_w=[\Coorvec w,\Coorvec{\bar w}]^\top$),
\begin{align*}
    I_{2m}&=((\nabla_w\widehat F'')\circ\widehat\Phi)\cdot\nabla_z\widehat\Psi,& 0_{r\times m}&=\nabla_\tau \widehat F''\circ\widehat\Phi+((\nabla_w\widehat F'')\circ\widehat\Phi)\cdot\nabla_t\widehat\Psi;
    \\
    \Longrightarrow\quad \nabla_z\widehat\Psi&=(\nabla_w\widehat F'')^{-1}\circ\widehat\Phi,&\nabla_t\widehat\Psi&=((\nabla_w\widehat F'')^{-1}\cdot\nabla_\tau \widehat F'')\circ\widehat\Phi.
\end{align*}

Since $\nabla_{\tau,w}\widehat F''\in\Co^{\alpha,\beta-}_{(\tau,w),s}$, by Lemma \ref{Lem::Hold::CompofMixHold} \ref{Item::Hold::CompofMixHold::InvMat} we have $(\nabla_w\widehat F'')^{-1}\in\Co^{\alpha,\beta-}_{(\tau,w),s}$ thus $(\nabla_w\widehat F'')^{-1}\cdot\nabla_\tau \widehat F''\in\Co^{\alpha,\beta-}_{(\tau,w),s}$. Applying Lemma \ref{Lem::Hold::CompofMixHold} \ref{Item::Hold::CompofMixHold::Comp} with $\widehat\Phi\in\Co^{\alpha+1,\beta}_{(t,z),s}$, we get $\nabla_{t,z}\widehat\Psi\in\Co^{\alpha,\beta-}_{(t,z),s}$. Therefore $\nabla_{t,z}\widehat\Phi\in\Co^{\alpha,\beta-}_{(t,z),s}$ finishing the proof of \eqref{Eqn::Key::PfKey::AllReg} and hence the proof of \ref{Item::Key::F''Reg} and \ref{Item::Key::PhiReg}.

\medskip
Now by \ref{Item::Key::PfKey::FReg1} and \ref{Item::Key::PhiReg} we see that $F\in\Co^{\alpha+1,\beta}_{u,v}$ and $\nabla_{t,z}\Phi\in\Co^{\alpha,\beta-}_{(t,z),s}$. By Lemma \ref{Lem::Hold::CompofMixHold} \ref{Item::Hold::CompofMixHold::Comp} we see that $\frac{\partial\Phi}{\partial t^j}\circ F\in\Co^{\alpha,\beta-}_\loc(U,V;T\Mf)$ and $\frac{\partial\Phi}{\partial z^k}\circ F\in\Co^{\alpha,\beta-}_\loc(U,V;\C T\Mf)$ for $1\le j\le r$ and $1\le k\le m$. Thus $F^*\Coorvec{t^j}=\frac{\partial\Phi}{\partial t^j}\circ F$, $F^*\Coorvec{z^k}=\frac{\partial\Phi}{\partial z^k}\circ F$ are both well-defined vector fields and have regularity $\Co^\alpha\cap\Co^{\beta-}=\Co^{\min(\alpha,\beta-)}$.

By definition, $dF'=F^*dt$, $dF'''=F^*ds$,  $dF''=F^*dz$ and $d\overline{F''}=F^*d\bar z$. So by \eqref{Eqn::Key::PfKey::Span} and \eqref{Eqn::Key::PfKey::DefF}, in the sense of Convention \ref{Conv::Key::CoSpan},
\begin{equation*}
    \textstyle\Span(F^*\Coorvec t,F^*\Coorvec z)=\Span(F^*d\bar z,F^*ds)^\bot=\Span(d\overline{F''},dF''')^\bot=\Se|_{U\times V}.
\end{equation*}

This proves \ref{Item::Key::Span}. Immediately \ref{Item::Key::PhiSpan} follows.
Now the whole proof is complete.\qed

\section{The Rough Complex Frobenius Theorem: Statements and Proofs}\label{Section::SecThm}\label{Section::PfThm}

Now we give more general statements for Theorems \ref{Thm::ThmCoor1} and \ref{Thm::ThmCoor2} with proofs. 

Recall  Definitions \ref{Defn::Hold::MixHold}, \ref{Defn::Hold::MoreMixHold} and \ref{Defn::ODE::MixHoldMaps} for the $\Co^{\alpha,\beta}$-spaces, Definition \ref{Defn::ODE::CpxSubbd} for a local basis, and \eqref{Eqn::Intro::ColumnNote} for the column and matrix convention of Jacobian matrices. 

We endow $\R^r\times\C^m\times\R^q$ with standard coordinate system $(t,z,s)=(t^1,\dots,t^r,z^1,\dots,z^m,s^1,\dots,s^q)$.

The refinement of Theorem \ref{Thm::ThmCoor1} is the following.
\begin{thm}\label{Thm::TrueThm1} 

Let $\alpha,\beta,\kappa\in(1,\infty]$ satisfy $\alpha\le\beta\le\kappa-1$, and let $r,m,q$ be nonnegative integers. Let $\Mf$ be a $(r+2m+q)$-dimensional $\Co^\kappa$-manifold. 

Let $\Se$ be a $\Co^\alpha$-complex Frobenius structure over $\Mf$ that has complex rank $r+m$, such that  $\Se+\bar\Se$ is a $\Co^\beta$-subbundle and has complex rank $r+2m$.

Then for any $p_0\in \Mf$ there is an open neighborhood $\Omega=\Omega'\times\Omega''\times\Omega'''\subseteq\R^r_t\times\C^m_z\times\R^q_s$ of $(0,0,0)$ and a map $\Phi:\Omega\to\Mf$ such that:
\begin{enumerate}[parsep=-0.3ex,label=(\arabic*)]
    \item\label{Item::TrueThm::Phi0} $\Phi:\Omega\to\Phi(\Omega)$ is a $\Co^\alpha$-diffeomorphism and $\Phi(0)=p_0$.
    \item\label{Item::TrueThm::PhiReg}  $\Phi\in\Co^{\alpha+1,\alpha}_\loc(\Omega'\times\Omega'',\Omega''';\Mf)$.
    \item\label{Item::TrueThm::PhiDReg}  $\frac{\partial\Phi}{\partial t^1},\cdots,\frac{\partial\Phi}{\partial t^r}\in\Co^{\alpha,\alpha-}_\loc(\Omega'\times\Omega'',\Omega''';T\Mf)$, and $\frac{\partial\Phi}{\partial z^1},\dots,\frac{\partial\Phi}{\partial z^m},\frac{\partial\Phi}{\partial\bar z^1},\dots,\frac{\partial\Phi}{\partial\bar z^m}\in\Co^{\alpha,\alpha-}_\loc(\Omega'\times\Omega'',\Omega''';\C T\Mf)$.
    \item\label{Item::TrueThm::PhiSpan} For any $u\in\Omega$, the complex subspace $\Se_{\Phi(u)}$ is spanned by $\frac{\partial\Phi}{\partial t^1}(u),\cdots,\frac{\partial\Phi}{\partial t^r}(u),\frac{\partial\Phi}{\partial z^1}(u),\dots,\frac{\partial\Phi}{\partial z^m}(u)\in\C T_{\Phi(u)}\Mf$. 
\end{enumerate}

For the coordinate side, set $U:=\Phi(\Omega)\subseteq\Mf$ and let $F=(F',F'',F'''):U\to\R^r\times\C^m\times\R^q$ be the inverse map of $\Phi$. Then:
\begin{enumerate}[parsep=-0.3ex,label=(\arabic*)]\setcounter{enumi}{4}
    \item\label{Item::TrueThm::F0} $F$ is a $\Co^\alpha$-coordinate chart of $\Mf$ near $p_0$ and $F(p_0)=(0,0,0)$.
    \item\label{Item::TrueThm::FReg} $F'\in\Co^\kappa_\loc(U;\R^r)$, $F''\in\Co^\alpha_\loc(U;\C^m)$ and $F'''\in\Co^\beta_\loc(U;\R^q)$.
    \item\label{Item::TrueThm::FDReg} $F^*\Coorvec{t^1},\dots,F^*\Coorvec{t^r},F^*\Coorvec{z^1},\dots,F^*\Coorvec{z^m}$ are $\Co^{\alpha-}$-vector fields defined on $U$.
    \item\label{Item::TrueThm::FSpan}$\Se|_U$ has a local basis $(F^*\Coorvec{t^1},\dots,F^*\Coorvec{t^r},F^*\Coorvec{z^1},\dots,F^*\Coorvec{z^m})$.
\end{enumerate}

In particular,
\begin{enumerate}[parsep=-0.3ex,label=(\arabic*)]\setcounter{enumi}{8}
    \item\label{Item::TrueThm::SBotSpan} $\Se^\bot|_U$ has a local basis $(F^*d\bar z^1,\dots,F^*d\bar z^m,F^*ds^1,\dots,F^*ds^q)$. 
    \item\label{Item::TrueThm::SCapSpan} $(\Se\cap\bar\Se)|_U$ has a local basis $(F^*\Coorvec{t^1},\dots,F^*\Coorvec{t^r})$.
    \item\label{Item::TrueThm::S+Span} $(\Se+\bar\Se)|_U$ has a local basis $(F^*\Coorvec{t^1},\dots,F^*\Coorvec{t^r},F^*\Coorvec{z^1},\dots,F^*\Coorvec{z^m},F^*\Coorvec{\bar z^1},\dots,F^*\Coorvec{\bar z^m})$.
\end{enumerate}
\end{thm}
The refinement of Theorem \ref{Thm::ThmCoor2} is the following
\begin{thm}\label{Thm::TrueThm2}
Let $\alpha,\beta,\kappa>1$, $r,m,q\ge0$, $\Mf$ and $\Se\le\C T\Mf$ be the same as in Theorem \ref{Thm::TrueThm1}.

Suppose in addition to the assumptions of Theorem \ref{Thm::TrueThm1}, there is a $\gamma\in(1,\infty]$ such that $\alpha+1\le\gamma\le\kappa-1$ and $\Se\cap\bar\Se$ is a  $\Co^\gamma$-subbundle.

Then for any $p_0\in \Mf$ there are an open neighborhood $\Omega=\Omega'\times\Omega''\times\Omega'''\subseteq\R^r_t\times\C^m_z\times\R^q_s$ of $(0,0,0)$, and a $\Co^\alpha$-map $\Phi:\Omega\to\Mf$, such that in additional to all the consequences in Theorem \ref{Thm::TrueThm1},
\begin{enumerate}[parsep=-0.3ex,label=(\arabic*)]\setcounter{enumi}{11}
    \item\label{Item::TrueThm::ImprovePhi} $\Phi\in\Co^{\gamma+1,\alpha+1,\alpha}_\loc(\Omega',\Omega'',\Omega''';\Mf)$.
    \item\label{Item::TrueThm::ImproveDPhi}$\frac{\partial\Phi}{\partial t^1},\cdots,\frac{\partial\Phi}{\partial t^r}\in\Co^{\gamma,\alpha+1,\alpha}_\loc(\Omega',\Omega'',\Omega''';T\Mf)$.
    \item\label{Item::TrueThm::ImproveDDt}For $F=\Phi^\Inv$, the vector fields $F^*\Coorvec{t^1},\dots,F^*\Coorvec{t^r}$ defined on $U=\Phi(\Omega)$ are all $\Co^\gamma$.
\end{enumerate}
\end{thm}
Some of the results are directly implied by others, we mark them as below.
\begin{remark}\label{Rmk::Final::RmkofTrueThm}
\begin{enumerate}[parsep=-0.3ex,label=(\roman*)]
    \item\label{Item::RmkofTrueThm::Infty} The case $\alpha=\infty$ (which implies $\beta=\gamma=\kappa=\infty$) is the classical result \cite{Nirenberg} (see Proposition \ref{Prop::Intro::Nirenberg}).
    \item\label{Item::RmkofTrueThm::FDReg} \ref{Item::TrueThm::FDReg} is implied by \ref{Item::TrueThm::PhiDReg} and \ref{Item::TrueThm::F0}.
\end{enumerate}

\smallskip
We have $F^*\Coorvec{t^j}=\frac{\partial\Phi}{\partial t^j}\circ F$ and $F^*\Coorvec{z^k}=\frac{\partial\Phi}{\partial z^k}\circ F$ for $1\le j\le r$ and $1\le k\le m$. \ref{Item::TrueThm::PhiDReg} shows that $\frac{\partial\Phi}{\partial t^j},\frac{\partial\Phi}{\partial z^k}\in\Co^{\alpha-}$ and \ref{Item::TrueThm::F0} shows that $F\in\Co^\alpha$. Taking compositions we get $\frac{\partial\Phi}{\partial t^j}\circ F,\frac{\partial\Phi}{\partial z^k}\circ F\in\Co^{\alpha-}$.\hfill\qedsymbol

\begin{enumerate}[parsep=-0.3ex,label=(\roman*)]\setcounter{enumi}{2}
    \item\label{Item::RmkofTrueThm::Span} \ref{Item::TrueThm::PhiSpan} and \ref{Item::TrueThm::FSpan} are equivalent.
\end{enumerate}

This is because $F^*\Coorvec{t^j}|_p=\frac{\partial\Phi}{\partial t^j}(F(p))$ and $F^*\Coorvec{z^k}|_p=\frac{\partial\Phi}{\partial z^k}(F(p))$ for $p\in U$, $1\le j\le r$ and $1\le k\le m$.\hfill\qedsymbol

\begin{enumerate}[parsep=-0.3ex,label=(\roman*)]\setcounter{enumi}{3}
    \item\label{Item::RmkofTrueThm::OtherSpan} \ref{Item::TrueThm::SBotSpan}, \ref{Item::TrueThm::SCapSpan} and \ref{Item::TrueThm::S+Span} are all implied by \ref{Item::TrueThm::FSpan}.
\end{enumerate} 

Since $(F^*dt,F^*dz,F^*d\bar z,F^*ds)$ is the dual basis of $(F^*\Coorvec t,F^*\Coorvec z,F^*\Coorvec{\bar z},F^*\Coorvec s)$,  and by \ref{Item::TrueThm::FSpan} $F^*\Coorvec t,F^*\Coorvec z$ span $\Se|_U$. Taking the dual bundle, we get that $F^*d\bar z,F^*ds$ span $\Se^\bot|_U$.

For the complex conjugate of \ref{Item::TrueThm::FSpan}, $\bar\Se|_U$ has local basis $F^*\Coorvec t,F^*\Coorvec{\bar z}$. Taking intersection  we get that $\Se\cap\bar\Se|_U$ is spanned by $F^*\Coorvec t$. Taking union we get that $\Se+\bar\Se|_U$ is spanned by $F^*\Coorvec t,F^*\Coorvec z,F^*\Coorvec{\bar z}$.\hfill\qedsymbol
\end{remark}

The proofs of Theorems \ref{Thm::TrueThm1} and \ref{Thm::TrueThm2} are combinations of the real Frobenius theorem, the Theorem \ref{Thm::ODE::BLFro} or Corollary \ref{Cor::ODE::RealFroCor}, and the estimate of the parameterized elliptic structures, the Theorem \ref{Thm::Key}. 

\begin{proof}[Proof of Theorem \ref{Thm::TrueThm1}]By Remark \ref{Rmk::Final::RmkofTrueThm} \ref{Item::RmkofTrueThm::Infty} only the case $\alpha<\infty$ needs a proof.

We fix the point $p_0\in\Mf$.
First we apply real Frobenius theorem, the Corollary \ref{Cor::ODE::ODERegCor} $\Se+\bar\Se$. We endow $\R^{r+2m}\times\R^q$ with standard coordinate system $(\nu,\lambda)=(\nu^1,\dots,\nu^{r+2m},\lambda^1,\dots,\lambda^q)$.

Since $\Se+\bar\Se$ is $\Co^\beta$ and has rank $r+2m$, we can find a neighborhood $\Omega_1=\Omega'_1\times\Omega'''_1\subseteq\R^{r+2m}_\nu\times\R^q_\lambda$ of $(0,0)$ and a $\Co^\beta$-parameterization $\Phi_1:\Omega_1\to\Mf$, such that by taking $U_1:=\Phi_1(\Omega_1)$ and $F_1=(F_1',F_1'''):U_1\to\R^{r+2m}\times\R^q$, $F_1:=\Phi_1^\Inv$, we have
\begin{enumerate}[parsep=-0.3ex,label=(S1.\arabic*)]
    \item\label{Item::PfTrueThm1::Phi10} $\Phi_1(0,0)=p_0$. 
    \item\label{Item::PfTrueThm1::Phi1Reg} $\Phi_1\in\Co^{\beta+1,\beta}_{\nu,\lambda,\loc}(\Omega_1',\Omega_1''';\Mf)$ and $\nabla_\nu\Phi_1\in\Co^\beta_\loc(\Omega_1;T\Mf\otimes\R^{r+2m})$.
    \item\label{Item::PfTrueThm1::F1} $F_1'\in\Co^\kappa_\loc(U_1;\R^{r+2m})$ and $F_1'''\in\Co^\beta_\loc(U_1;\R^q)$.
    \item\label{Item::PfTrueThm1::RFroSpan} $(\Se+\bar\Se)|_{U_1}$ has a $\Co^\beta$-local basis $(F_1^*\Coorvec{\nu^1},\dots,F_1^*\Coorvec{\nu^{r+2m}})$.
\end{enumerate}

Now $\Phi_1^*\Se=(F_1)_*\Se\le\C T\Omega_1$. By \ref{Item::PfTrueThm1::RFroSpan}, $\Phi_1^*\Se\le\Phi_1^*(\Se+\bar\Se)=\Span(\Coorvec{\nu^1},\dots,\Coorvec{\nu^{r+2m}})|_{\Omega_1}=(\C T\Omega_1')\times\Omega_1'''$.

Since $F_1^*\Coorvec{\nu^1},\dots,F_1^*\Coorvec{\nu^{r+2m}}\in\Co^\beta\subseteq\Co^\alpha$, by Lemma \ref{Lem::ODE::PullBackReg} $\Phi_1^*\Se$ is a $\Co^\alpha$-subbundle of $(\C T\Omega_1')\times\Omega_1'''$. 

Now Theorem \ref{Thm::Key} applies to $\Phi_1^*\Se\le (\C T\Omega_1')\times\Omega_1'''$.
We endow $\R^r\times\C^m\times\R^q$ with standard real and complex coordinate system $(t,z,s)=(t^1,\dots,t^r,z^1,\dots,z^m,s^1,\dots,s^q)$. By Theorem \ref{Thm::Key}, we can find a neighborhood $\Omega=\Omega'\times\Omega''\times\Omega'''\subseteq\R^r_t\times\C^m_z\times\R^q_s$ of $(0,0,0)$ and a $\Co^\alpha$-parameterization $\Phi_2=(\Phi_2',\Phi_2'''):\Omega\to\Omega_1'\times\Omega_1'''$, such that by taking $U_2:=\Phi_2(\Omega)$ and $F_2=(F_2',F_2'',F_2'''):=\Phi_2^\Inv:U_2\to\R^r\times\C^m\times\R^q$,
\begin{enumerate}[parsep=-0.3ex,label=(S1.\arabic*)]\setcounter{enumi}{4}
    \item\label{Item::PfTrueThm1::Phi20} $\Phi_2(0,0,0)=(0,0)$.
    \item\label{Item::PfTrueThm1::Phi2''} $\Phi_2'''=\Phi_2'''(s):\Omega'''\subseteq\R^q_s\to\Omega'''_1$ is a smooth parameterization that does not depend on $(t,z)$.
    \item\label{Item::PfTrueThm1::Phi2Reg} $\Phi_2'\in\Co^{\alpha+1,\alpha}_{(t,z),s}(\Omega'\times\Omega'',\Omega''';\Omega_1')$ and $\nabla_{t,z}\Phi_2'\in\Co^{\alpha,\alpha-}(\Omega'\times\Omega'',\Omega''';\C^{(r+2m)\times(r+2m)})$.
    \item\label{Item::PfTrueThm1::F2Chart} $F_2'=F_2'(\nu)\in C^\infty$ does not depend on $\lambda$; and $F_2'''=F_2'''(\lambda)\in C^\infty$ does not depend on $\nu$.
    \item\label{Item::PfTrueThm1::KeySpan} $\Phi_1^*\Se|_{U_2}$ has a $\Co^{\alpha-}$-local basis $(F_2^*\Coorvec{t^1},\dots,F_2^*\Coorvec{t^r},F_2^*\Coorvec{z^1},\dots,F_2^*\Coorvec{z^m})$.
\end{enumerate}

We now define $\Phi:=\Phi_1\circ\Phi_2:\Omega\to\Mf$, $U:=\Phi(\Omega)$ and $F=(F',F'',F'''):=\Phi^\Inv$.

Immediately by \ref{Item::PfTrueThm1::Phi10} and \ref{Item::PfTrueThm1::Phi20} we have \ref{Item::TrueThm::Phi0} and \ref{Item::TrueThm::F0}. 

By \ref{Item::PfTrueThm1::RFroSpan} and \ref{Item::PfTrueThm1::KeySpan} we have $\Se=\Phi_1^*\Span(F_2^*\Coorvec t,F_2^*\Coorvec z)=\Span(F_1^*F_2^*\Coorvec t,F_1^*F_2^*\Coorvec z)=\Span(F^*\Coorvec t,F^*\Coorvec z)$, which is \ref{Item::TrueThm::FSpan}. By Remark \ref{Rmk::Final::RmkofTrueThm} \ref{Item::RmkofTrueThm::Span} we get \ref{Item::TrueThm::PhiSpan}.

By \ref{Item::PfTrueThm1::Phi2''} we have $\Phi(t,z,s)=\Phi_1(\Phi_2'(t,z,s),\Phi_2'''(s))$. By chain rule we have
\begin{equation*}\label{Item::PfTrueThm1::NablaPhiReg}
    \nabla_{t,z}\Phi(t,z,s)=((\nabla_\nu\Phi_1)\circ \Phi_2)(t,z,s))\cdot\nabla_{t,z}\Phi_2'(t,z,s).
\end{equation*}
Thus \ref{Item::TrueThm::PhiReg} and \ref{Item::TrueThm::PhiDReg} follow by applying Lemma  \ref{Lem::Hold::CompofMixHold} \ref{Item::Hold::CompofMixHold::Comp} with \ref{Item::PfTrueThm1::Phi1Reg} and \ref{Item::PfTrueThm1::Phi2Reg}.

Clearly $F''=F_2''\circ F_1\in\Co^\alpha$. By \ref{Item::PfTrueThm1::F1} and \ref{Item::PfTrueThm1::F2Chart} we have $F'=F_2'\circ F_1'\in\Co^\kappa$ and $F'''=F_2'''\circ F_1'''\in\Co^\beta$, giving the result \ref{Item::TrueThm::FReg}.

The rest of the results in Theorem \ref{Thm::TrueThm1} follow from Remark \ref{Rmk::Final::RmkofTrueThm}.
\end{proof}

\begin{proof}[Proof of Theorem \ref{Thm::TrueThm2}]
We use a construction that is slightly different to the proof of Theorem \ref{Thm::TrueThm1}. Instead of applying Corollary \ref{Cor::ODE::RealFroCor}, we apply Theorem \ref{Thm::ODE::BLFro} in the first step.

We endow $\R^r\times\R^{2m}\times\R^q$ with standard coordinate system $(\mu,\nu,\lambda)=(\mu^1,\dots,\mu^r,\nu^1,\dots,\nu^{2m},\lambda^1,\dots,\lambda^q)$.

Since $\Se+\bar\Se\in\Co^\beta$ has rank $r+2m$ and $\Se\cap\bar\Se\in\Co^\gamma$ has rank $r$, we can find a neighborhood $\Omega_1=\Omega'_1\times\Omega''_1\times\Omega'''_1\subseteq\R^r_\mu\times\R^{2m}_\nu\times\R^q_\lambda$ of $(0,0,0)$ and a $\Co^{\min(\beta,\gamma)}$-parameterization $\Phi_1:\Omega_1\to\Mf$, such that by taking $U_1:=\Phi_1(\Omega_1)$ and $F_1=(F_1',F_1'',F_1'''):U_1\to\R^r\times\R^{2m}\times\R^q$ as $F_1:=\Phi_1^\Inv$, we have
\begin{enumerate}[parsep=-0.3ex,label=(S2.\arabic*)]
    \item\label{Item::PfTrueThm2::Phi10} $\Phi_1(0,0,0)=p_0$, 
    \item\label{Item::PfTrueThm2::Phi1Reg} $\Phi_1\in\Co^{\gamma+1,\min(\gamma,\beta+1),\min(\gamma,\beta)}_{\mu,\nu,\lambda,\loc}(\Omega_1',\Omega_1'',\Omega_1''';\Mf)$ and $\nabla_\mu\Phi_1\in\Co^{\gamma,\min(\gamma,\beta+1),\min(\gamma,\beta)}_\loc(\Omega_1',\Omega_1'',\Omega_1''';T\Mf\otimes\R^r)$. 
    \item\label{Item::PfTrueThm2::F1Reg} $F_1'\in\Co^\kappa_\loc(U_1;\R^r)$, $F_1''\in\Co^\gamma_\loc(U_1;\R^{2m})$ and $F_1'''\in\Co^\beta_\loc(U_1;\R^q)$.
    \item\label{Item::PfTrueThm2::F*Reg} On $U_1$, $F_1^*\Coorvec{\mu^1},\dots,F_1^*\Coorvec{\mu^r}$ are $\Co^\gamma$-vector fields and $F_1^*\Coorvec{\nu^1},\dots,F_1^*\Coorvec{\nu^{2m}}$ are $\Co^{\min(\gamma-1,\beta)}$-vector fields.
    \item\label{Item::PfTrueThm2::Phi1Span} $(\Se\cap\bar\Se)|_{U_1}=\Span(F_1^*\Coorvec{\mu^1},\dots,F_1^*\Coorvec{\mu^r})$, and $(\Se+\bar\Se)|_{U_1}=\Span(F_1^*\Coorvec{\mu^1},\dots,F_1^*\Coorvec{\mu^r},F_1^*\Coorvec{\nu^1},\dots,F_1^*\Coorvec{\nu^{2m}})$.
\end{enumerate}

By \ref{Item::PfTrueThm2::F*Reg} and \ref{Item::PfTrueThm2::Phi1Span}, $\Se$ is contained in the span of $F_1^*\Coorvec{\mu^1},\dots,F_1^*\Coorvec{\mu^r},F_1^*\Coorvec{\nu^1},\dots,F_1^*\Coorvec{\nu^{2m}}\in\Co^{\min(\beta,\gamma-1)}$. By assumption $\min(\gamma-1,\beta)\ge\alpha$, using Lemma \ref{Lem::ODE::PullBackReg} we see that $\Phi_1^*\Se\le \C T\Omega_1$ is a $\Co^\alpha$-subbundle

By the standard linear algebra argument we can find a linear complex coordinate system $w=(w^1,\dots,w^m)$ on $\R^{2m}_\nu$ such that
\begin{equation*}
    \textstyle(\Phi_1^*\Se)_{(0,0)}=\Span\big(\Coorvec{\mu^1}|_{(0,0)},\dots,\Coorvec{\mu^r}|_{(0,0)},\Coorvec{w^1}|_{(0,0)},\dots,\Coorvec{w^m}|_{(0,0)}\big)\le \C T_{(0,0)}(\R^r_\mu\times\C^m_w).
\end{equation*}
In particular $(\Phi_1^*\Se)_{(0,0)}\oplus\Span \big(\Coorvec{\bar w^1}|_{(0,0)},\dots,\Coorvec{\bar w^m}|_{(0,0)},\Coorvec{\lambda^1}|_{(0,0)},\dots,\Coorvec{\lambda^q}|_{(0,0)})=\C T_{(0,0)}(\R^r_\mu\times\C^m_w)$.
Applying Lemma \ref{Lem::ODE::GoodGen}, we can find a neighborhood $W=W'\times W''\times W'''\subseteq\Omega_1'\times\Omega_1''\times\Omega_1'''$ of $(0,0,0)\in\Omega$ such that $\Phi_1^*\Se$ has the following local basis on $W$,
\begin{equation*}
    X=\begin{pmatrix}X'\\X''
    \end{pmatrix}\begin{pmatrix}I&&A'&B'\\&I&A''&B''\end{pmatrix}\begin{pmatrix}\partial_\mu\\\partial_w\\\partial_{\bar w}\\\partial_\lambda
    \end{pmatrix}.
\end{equation*}where $A',A'',B',B''$ are $\Co^\alpha$-matrix functions defined in $W$.

Since $\Phi_1^*\Se\le \C T(\Omega_1'\times\Omega_1'')\times\Omega_1'''$, we have $B'\equiv0$ and $B''\equiv 0$. Since $\Coorvec{\mu^1},\dots,\Coorvec{\mu^r}$ are sections of $\Phi_1^*(\Se\cap\bar\Se)\subset\Phi_1^*\Se$, by Lemma \ref{Lem::ODE::GoodGen} \ref{Item::ODE::GoodGen::Uniqueness} we have $A'=0$, i.e. $X_j=\Coorvec{\mu^j}$ for $1\le j\le r$.

By Lemma \ref{Lem::ODE::GoodGen} \ref{Item::ODE::GoodGen::InvComm}, since $\Phi_1^*\Se$ is involutive, we see that $[X_j,X_k]=0$ for $1\le j\le r<k\le r+m$. So $\Coorvec{\mu^j}A''(\mu,w,s)=0$ for $1\le j\le r$ in the domain, which means $A''(\mu,w,s)=A''(w,s)$.

Define a complex tangent subbundle $\Tc\le\C T(W''\times W''')$ as $\Tc:=\Span(\Coorvec w+A''(w,s)\Coorvec{\bar w})$. We see that $\Tc$ is a $\Co^\alpha$-involutive subbundle of rank $m$ and satisfies
\begin{equation}\label{Eqn::PfTrueThm2::Tensor}
    \Tc\cap\bar\Tc=0,\quad\Tc+\bar \Tc=(\C TW'')\times W''',\quad(\Phi_1^*\Se)|_W=(\C TW')\otimes_\C\Tc.
\end{equation}

Endow $\C^m\times\R^q$ with standard coordinate system $(z,s)=(z^1,\dots,z^m,s^1,\dots,s^q)$. By Theorem \ref{Thm::Key} (by taking $r=0$ in this theorem), we can find a neighborhood $\Omega_2=\Omega_2''\times\Omega_2'''\subseteq\C^m_z\times\R^q_s$ of $(0,0)$ and a $\Co^\alpha$-parameterization $\Phi_2=(\Phi_2'',\Phi_2'''):\Omega_2\to W''\times W'''$, such that by taking $U_2:=\Phi_2(\Omega_2)$ and $F_2=(F_2'',F_2'''):U_2\to\C^m\times\R^q$, $F_2:=\Phi_2^\Inv$, we have
\begin{enumerate}[parsep=-0.3ex,label=(S2.\arabic*)]\setcounter{enumi}{5}
    \item\label{Item::PfTrueThm2::Phi20} $\Phi_2(0,0,0)=(0,0)$.
    \item\label{Item::PfTrueThm2::Phi2''} $\Phi_2'''=\Phi_2'''(s):\Omega_2'''\subseteq\R^q_s\to\Omega_1'''$ is a smooth parameterization that does not depend on $z$.
    \item\label{Item::PfTrueThm2::Phi2Reg} $\Phi_2''\in\Co^{\alpha+1,\alpha}_{z,s}(\Omega_2'',\Omega_2''';W'')$ and $\nabla_z\Phi_2''\in\Co^{\alpha,\alpha-}(\Omega_2'',\Omega_2''';\C^{2m\times 2m})$ 
    \item\label{Item::PfTrueThm2::F2Reg} $F_2'''=F_2'''(\lambda)\in C^\infty$ does not depend on $(\mu,\nu)$; and $F_2''\in\Co^\alpha(U_2;\C^m)$.
    \item\label{Item::PfTrueThm2::Phi2Span} $\Tc|_{U_2}$ has a $\Co^{\alpha-}$-local basis $(F_2^*\Coorvec{z^1},\dots,F_2^*\Coorvec{z^m})$.
\end{enumerate}

Define $\Omega':=W'$, $\Omega'':=\Omega_2''$, $\Omega''':=\Omega_2'''$ and $\Phi:\Omega\to\Mf$ as 
\begin{equation}\label{Eqn::PfTrueThm2::DefPhi}
    \Phi(t,z,s):=\Phi_1(t,\Phi_2(t,z))=\Phi_1(t,\Phi''_2(z,s),\Phi_2'''(s)),\quad t\in\Omega',z\in\Omega'',s\in\Omega'''.
\end{equation}

We define $U:=\Phi(\Omega)$ and $F=(F',F'',F'''):=\Phi^\Inv$.  Immediately by \ref{Item::PfTrueThm2::Phi10} and \ref{Item::PfTrueThm2::Phi20} we have \ref{Item::TrueThm::Phi0} and \ref{Item::TrueThm::F0}.

Taking the chain rules on \eqref{Eqn::PfTrueThm2::DefPhi} we have
\begin{equation}\label{Eqn::PfTrueThm2::NablaPhiReg}
    \nabla_t\Phi(t,z,s)=\nabla_\mu\Phi_1(t,\Phi_2''(z,s),\Phi_2'''(s)),\quad\nabla_z\Phi(t,z,s)=(\nabla_\nu\Phi_1)(t,\Phi_2''(z,s),\Phi_2'''(s))\cdot\nabla_z\Phi_2(z,s).
\end{equation}

Based on \ref{Item::PfTrueThm2::Phi1Reg} and \ref{Item::PfTrueThm2::Phi2Reg}, applying Lemma \ref{Lem::Hold::CompofMixHold} \ref{Item::Hold::CompofMixHold::Comp} to \eqref{Eqn::PfTrueThm2::DefPhi} and \eqref{Eqn::PfTrueThm2::NablaPhiReg}, we get \ref{Item::TrueThm::ImprovePhi}, \ref{Item::TrueThm::PhiDReg} and \ref{Item::TrueThm::ImproveDPhi}. Note that \ref{Item::TrueThm::ImprovePhi} is stronger than \ref{Item::TrueThm::PhiReg}.

\smallskip
By \eqref{Eqn::PfTrueThm2::DefPhi}, \ref{Item::PfTrueThm2::F1Reg} and \ref{Item::PfTrueThm2::F2Reg} we see that $F=(F',F'',F''')$ satisfies
\begin{equation}\label{Eqn::PfTrueThm2::EqnofF}
    F'=F_1'=\id_{\R^r}\circ F_1'\in\Co^\kappa_\loc(U;\R^r),\quad F''=F_2''\circ F_1''\in\Co^\alpha_\loc(U;\C^m),\quad F'''=F_2'''\circ F_1'''\in\Co^\beta_\loc(U;\R^q).
\end{equation}
This gives \ref{Item::TrueThm::FReg}.

By \eqref{Eqn::PfTrueThm2::NablaPhiReg}, we get $\Coorvec{t^j}\Phi(t,z,s)=\frac{\partial\Phi_1}{\partial\mu^j}(t,\Phi_2''(z,s),\Phi_2'''(s))=(\Phi_1)_*\Coorvec{\mu^j}|_{\Phi(t,z,s)}$ for $1\le j\le r$. So $F^*\Coorvec{t^j}=F_1^*\Coorvec{\mu^j}\in\Co^\gamma$ for $1\le j\le r$, finishing the proof of \ref{Item::TrueThm::ImproveDDt}.

By \eqref{Eqn::PfTrueThm2::Tensor} and \ref{Item::PfTrueThm2::Phi2Span} we see that $\Phi_1^*\Se=\Span(\Coorvec{\mu},F_2^*\Coorvec{z})$, so by writing $\id_{\R^r}:\R^r_\mu\to\R^r_t$ as the identity we have $\Phi_1^*\Se=\Span(\id_{\R^r}^*\Coorvec{t},F_2^*\Coorvec{z})$. Using \eqref{Eqn::PfTrueThm2::EqnofF} we have $F=(\id_{\R^r},F_2)\circ F_1$, taking pushforward of $\Phi_1$ we get $\Se|_U=F_1^*\Span(\id_{\R^r}^*\Coorvec{t},F_2^*\Coorvec{z})|_U=\Span(F^*\Coorvec t,F^*\Coorvec z)$, finishing the proof of \ref{Item::TrueThm::FSpan}.

The rest of the results follow from Remark \ref{Rmk::Final::RmkofTrueThm} as well. 
\end{proof}

\section{The Sharpness of Coordinates and Coordinate Vector Fields}\label{Section::SharpGen}

Fixed four numbers $\alpha,\beta,\gamma,\kappa>1$ such that $\alpha\le\min(\beta,\gamma)$ and $\kappa\ge\max(\beta,\gamma)+1$. Consider all possible complex Frobenius structures $\Se\le \C T\Mf$ among all $\Co^\kappa$-manifolds $\Mf$ such that $\Se\in\Co^\alpha$, $\Se+\bar\Se\in\Co^\beta$ and $\Se\cap\bar\Se\in\Co^\gamma$. 

Fix a point $p\in \Mf$, let $F=(F',F'',F'''):U\subseteq\Mf\to\R^r_t\times\C^m_z\times\R^{n}_s$ be any desired $C^1$-coordinate chart near a fixed point $p$ that represents $\Se$, that is, $F^*\Coorvec{t^1},\dots,F^*\Coorvec{t^r},F^*\Coorvec{z^1},\dots,F^*\Coorvec{z^m}$ span $\Se$ in $U$. 

Recall the notation $\Co^{\mu+}=\bigcup_{\nu>\mu}\Co^\nu$ for $\mu\in\R$.

We list the best H\"older-Zygmund regularities the six objects: $F'$, $F''$, $F'''$, $F^*\Coorvec t$, $F^*\Coorvec z$, $F^*\Coorvec s$, as follows.

\begin{itemize}[parsep=-0.1ex]
    \item $F'\notin\Co^{\kappa+}$: $\Mf$ is $\Co^\kappa$, its coordinate functions are at most $\Co^\kappa$, so $F'\in\Co^\kappa$ is optimal.
    
    \item
 $F'''\notin\Co^{\beta+}$: This is due to the sharpness of real Frobenius theorem that $\Se+\bar\Se\in\Co^\beta$ only guarantees $F'''$ to be $\Co^\beta$. See Proposition \ref{Prop::Final::RealFroisSharp} below.

    \item
$F^*\Coorvec t\notin\Co^{\gamma+}$: $F^*\Coorvec t$ spans $\Se\cap\bar\Se$ which is $\Co^\gamma$. If $F^*\Coorvec t\in\Co^{\gamma+}$ then $\Se\cap\bar\Se\in\Co^{\gamma+}$ as well.

    \item
$F''\notin\Co^{\alpha+}$ and $F^*\Coorvec s\notin\Co^{(\alpha-1)+}$: The argument is similar to the sharpness of real Frobenius theorem. See Proposition \ref{Prop::Final::Sharpdds} below.

    \item
$F^*\Coorvec z\notin\Co^\alpha$: This is the counter intuitive part in our results. See Section \ref{Section::Sharpddz} for the complete deduction and Proposition \ref{Prop::Final::SharpddzRed}, Theorem \ref{Thm::Final::ProofExampleddz} for the proof.
\end{itemize}

When $\gamma\ge\alpha+1$ (that $\Se\cap\bar\Se$ is at least 1 order more regular than $\Se$), by Theorem \ref{Thm::TrueThm2}, all six regularity estimates above are sharp. 

When $\alpha\le\gamma<\alpha+1$, from Theorem \ref{Thm::ThmCoor1} we only have $F^*\Coorvec t\in\Co^{\alpha-}$, the rest of five estimates are still sharp.

When $\alpha\le\gamma<\alpha+1$ and $\gamma>2$, using Theorem \ref{Thm::TrueThm1} we can find a $\Co^\alpha$-coordinate chart $F=(F',F'',F''')$ to $\R^r_t\times\C^m_z\times\R^n_s$ such that $F''\in\Co^\alpha$, $F^*\Coorvec z\in\Co^{\alpha-}$ but $F^*\Coorvec t\in\Co^{\alpha-}$. If we use Theorem \ref{Thm::TrueThm2}, then we can get $F^*\Coorvec t\in\Co^{\gamma}$ which is sharp, but neither $F''\in\Co^{\gamma-1}$, $F^*\Coorvec z\in\Co^{(\gamma-1)-}$ nor $F^*\Coorvec s\in\Co^{\gamma-2}$ is sharp. 

We do not know whether we can choose suitable coordinate chart $F$ so that $F^*\Coorvec t\in\Co^{\gamma}$ and $F''\in\Co^\alpha$ holds in the same time. 

\subsection{Sharpness of $F''$, $F'''$ and $F^*\Coorvec s$: some elementary examples}\label{Section::SharpSimp}
In this subsection we use the conventions $f_u=\partial_uf$, $f_v=\partial_vf$, etc.
\begin{prop}\label{Prop::Final::RealFroisSharp}
    Let $\beta>1$, and let $(u,v)$ be the standard coordinate system for $\R^2$. We define $\V\le T\R^2$ as
    $$\textstyle \V:=\Span_\R T,\quad T:=\Coorvec u+f(v)\Coorvec v,\quad\text{where }f(v):=\max(0,v)^\beta.$$
    
    Suppose $F=(F',F'''):U\subseteq\R^2_{u,v}\to\R_x\times\R_s$ is a $C^1$-coordinate chart near $(0,0)$ such that $F^*\Coorvec x$ spans $\V|_U$, then $\partial_vF'''\notin\Co^{(\beta-1)+}_\loc(U)$. In particular $F'''\notin\Co^{\beta+}_\loc(U;\R^2)$.
\end{prop}Note that $\V$ has rank 1 so is automatically involutive.
\begin{proof}
By direct computation
\begin{equation*}
    e^{tT}(u,v)=\begin{cases}\big(u+t,\frac v{\left(1-(\beta-1)v^{\beta-1}t\right)^\frac1{\beta-1}}\big)&\text{for }v>0\text{ and }-\infty<t<(\beta-1)v^{1-\beta},
    \\
    (u+t,v)&\text{for }v\le0\text{ and }t\in\R.
    \end{cases}
\end{equation*}

First we show that for every $\delta\neq0$, the function
\begin{equation}\label{Eqn::ODE::DefGDelta}
    g^\delta:I_\delta\to\R,\quad g^\delta(v):=\begin{cases}\frac v{\left(1-(\beta-1)\delta v^{\beta-1}\right)^\frac1{\beta-1}},&v>0,\\v,&v\le0,\end{cases}\quad\text{where }I_\delta=\begin{cases}
    \left(-\infty,(\frac\delta{\beta-1})^\frac1{1-\beta}\right),&\delta>0,\\\R,&\delta<0,
    \end{cases}
\end{equation}
 is not $\Co^{\beta+}$ near $v=0$. Note that $e^{u T}(0,v)=(u,g^u(v))$ for every $u$.

Indeed by Taylor's expansion,
\begin{equation*}
    \frac v{\left(1-(\beta-1)\delta v^{\beta-1}\right)^\frac1{\beta-1}}=v+\delta v^\beta+\sum_{k=2}^\infty{-\frac1{\beta-1}\choose k}\left(-(\beta-1)\delta\right)^k\cdot v^{k\beta-k+1},\quad\text{converging for } 0\le v<(\tfrac\delta{\beta-1})^\frac1{1-\beta}.
\end{equation*}
Hence $g^\delta(v)=v+\max(0,v)^\beta+\sum_{k=2}^\infty{-\frac1{\beta-1}\choose k}\left(-(\beta-1)\delta\right)^k\cdot\max(0,v)^{k\beta-k+1}$ holds near $v=0$.

Now near $v=0$ $\sum_{k=2}^\infty{-\frac1{\beta-1}\choose k}\left(-(\beta-1)\delta\right)^k\cdot\max(0,v)^{k\beta-k+1}$ is a $\Co^{2\beta-1}_\loc\subsetneq\Co^{\beta+}_\loc$-function, but $\max(0,v)^\beta\notin\Co^{\beta+}$. So $g^\delta\notin\Co^{\beta+}$ near $v=0$.
Therefore the function $(u,v)\mapsto g^u(v)$ is not $\Co^{\beta+}$ near $(u,v)=(0,0)$.

By assumption $F^*\Coorvec x$ spans $\V|_U$. So $\V^\bot|_U=\Span dF''$, which means that we have a transport equation
\begin{equation}\label{Eqn::Final::RealFroisSharp::EqnofF'''}
    \textstyle TF'''(u,v)=0\quad\text{i.e.}\quad\partial_uF'''(u,v)=-f(v)\cdot\partial_vF'''(u,v),\quad\text{for all }(u,v)\in U.
\end{equation}

Since $F$ is a $C^1$-chart, we have $(F_u'''(0,0),F_v'''(0,0))\neq(0,0)$. By \eqref{Eqn::Final::RealFroisSharp::EqnofF'''} and $f(0)=0$, we have $F_u'''(0,0)=0$. Thus $F_v'''(0,0)\neq0$. 

Therefore $h(v):=F'''(0,v)$ is a $C^1$-function that has non-vanishing derivatives in a neighborhood of $v=0$. By the Inverse Function Theorem, on a smaller neighborhood of $v=0$, $h$ is a $C^1$-diffeomorphism onto its image.

\medskip
Suppose by contrast $F'''_v\in\Co^{(\beta-1)+}$ in a neighborhood of $(u,v)=(0,0)$. By \eqref{Eqn::Final::RealFroisSharp::EqnofF'''} since $f\in\Co^\beta_\loc\subset\Co^{(\beta-1)+}$, we have $F'''_u\in\Co^{(\beta-1)+}$ as well, which means $F'''\in\Co^{\beta+}$ near $(u,v)=(0,0)$. Since $h(v)=F'''(0,v)$, we have $h\in\Co^{\beta+}$ in the domain and by the Inverse Function Theorem $h^\Inv\in\Co^{\beta+}$.  

So $h^\Inv\circ F''(u,v)$ is a $\Co^{\beta+}$-function defined in a neighborhood of $(u,v)=(0,0)$. By \eqref{Eqn::Final::RealFroisSharp::EqnofF'''} and \eqref{Eqn::ODE::DefGDelta}, $F'''(u,v)=F'''(e^{-uT}(u,v))=F'''(0,g^{-u}(v))$ holds a neighborhood of $0$.
So $h^\Inv\circ F'''(u,v)=g^{-u}(v)$, contradicting to the fact that $(u,v)\mapsto g^u(v)$ is not $\Co^{\beta+}$ near $(0,0)$. This concludes that $F_v'''\notin\Co^{(\beta-1)+}$ near $(u,v)=(0,0)$.
Therefore $F'''\notin\Co^{\beta+}$ and we finish the proof.
\end{proof}

\begin{prop}\label{Prop::Final::Sharpdds}Endow $\R^3$ with standard coordinate system $(u,v,\theta)$. Let $\alpha>0$. We define $\Se\le\C T\R^3$ as
\begin{equation*}
    \Se:=\Span_\C Z,\quad Z:=\partial_u+ie^{f(\theta)}\partial_v\quad\text{where }f(\theta):=\max(0,\theta)^\alpha.
\end{equation*}

Suppose $F=(F'',F'''):U\subseteq\R^3_{u,v,\theta}\to\C^1_z\times\R^1_s$ is a $C^1$-coordinate chart near $0\in\R^3$ such that that $F(0)=(0,0)$ and $F^*\Coorvec z$ spans $\Se|_U$. Then both $F''\notin\Co^{\alpha+}$ and $F^*\Coorvec s\notin\Co^{(\alpha-1)+}$ near $(0,0,0)$.
\end{prop}
By assumption $Z\in\Co^\alpha$, $\Se$ is involutive since it has rank 1, and $\Se+\bar\Se=\Span(\Coorvec u,\Coorvec v)$ is involutive as well. So $\Se$ is a $\Co^\alpha$-complex Frobenius structure.

\begin{proof}
By assumption $\Se+\bar\Se|_U=\Span(\partial_u,\partial_v)|_U=\Span(F^*\Coorvec z,F^*\Coorvec{\bar z})=(\Span F^*ds)^\bot=(\Span dF''')^\bot$, so $\partial_uF'''=\partial_vF'''=0$ which means $F'''(u,v,\theta)=F'''(\theta)$. So $dF'''=F'''_\theta d\theta$ and $F'''_\theta (\theta)$ is non-vanishing in a neighborhood of $\theta=0$.

Now $\Se^\bot|_U=(\Span Z)^\bot|_U=\Span(F^*d\bar z,F^*ds)=\Span (d\bar F'',dF''')$, so
\begin{equation*}
    Z\bar F''=\bar F''_u+ie^{f(\theta)}\bar F''_v=0,
    \quad\Rightarrow\quad F''_v=-ie^{-f(\theta)} F''_u,\quad\text{in }U.
\end{equation*}
Here we use $F''_u=\partial_u F''$, $F''_v=\partial_v F''$ and $F''_\theta=\partial_\theta F''$. So
\begin{equation}\label{Eqn::Final::Sharpdds::ClassicalExp1}
\begin{pmatrix}F^*dz\\F^*d\bar z\\F^*ds\end{pmatrix}=
\begin{pmatrix}
F''_u&-F''_u&F''_\theta\\
\bar F''_u&\bar F''_u&\bar F''_\theta\\
&&F'''_\theta 
\end{pmatrix}
\begin{pmatrix}du\\-ie^{-f(\theta)}dv\\d\theta\end{pmatrix}\ \Rightarrow\
\begin{pmatrix}
\partial_u\\ ie^{f(\theta)}\partial_v\\\partial_\theta
\end{pmatrix}=
\begin{pmatrix}
F''_u&\bar F''_u&\\
-F''_u&\bar F''_u\\
F''_\theta&\bar F''_\theta&F'''_\theta 
\end{pmatrix}
\begin{pmatrix}F^*\partial_z\\F^*\partial_{\bar z}\\F^*\partial_s\end{pmatrix}.
\end{equation}
Applying Cramer's rule for inverse matrices to the right hand equation of \eqref{Eqn::Final::Sharpdds::ClassicalExp1}, we have
\begin{equation}\label{Eqn::Final::RealFroisSharp::EqnofF*dds}
    F^*\Coorvec s=\frac1{F'''_\theta (\theta)}\bigg(\Coorvec\theta-\frac{\bar F''_\theta}{\bar F''_u}\Big(\Coorvec u+ie^{f(\theta)}\Coorvec v\Big)-\frac{F''_\theta}{F''_u}\Big(\Coorvec u-ie^{f(\theta)}\Coorvec v\Big)\bigg),\quad\text{in }U.
\end{equation}

In particular we see that $F''_u$ and $F''_v=-ie^{-f}F''_u$ are both non-vanishing in the domain $U$.

For fixed $\theta\in\R$,  $\partial_u+ie^{f(\theta)}\partial_v$ is scaled $\overline{\partial}$-operator on $\R^2_{u,v}$ that annihilates $F''(\cdot,\theta_0)$. Therefore by a scaling of the Cauchy Integral Formula, for every $\tilde V'\times V''\Subset U$, $(u,v)\in \tilde V'$, $\theta\in V''$ and $k\ge0$,
\begin{equation*}
    F''(u,v,\theta)=\frac1{2\pi i}\int_{\partial V'}\frac{F''(\xi,\eta,\theta)(d\xi+ie^{-f(\theta)}d\eta)}{(\xi-u)+ie^{-f(\theta)}(\eta-v)},\quad \frac{\partial^kF''(u,v,\theta)}{\partial u^k}=\frac{k!}{2\pi i}\int_{\partial V'}\frac{F''(\xi,\eta,\theta)(d\xi+ie^{-f(\theta)}d\eta)}{((\xi-u)+ie^{-f(\theta)}(\eta-v))^{k+1}}.
\end{equation*}

Since $F''\in C^0( \tilde V'\times V'';\C)$, for every precompact subset $ V'\Subset \tilde V'$ we have
\begin{equation*}
    \|\partial_u^kF''\|_{L^\infty(V'\times V'';\C^{2^k})}\le k!\dist(V',\partial \tilde V')^{-k-1}\|F''\|_{L^\infty( \tilde V'\times V'';\C)}<\infty,\quad\forall k\ge0.
\end{equation*}
Using $F_v''=-ie^{-f(\theta)}F_u''$, we see that $\|\nabla_{u,v}^kF''\|_{L^\infty( V'\times V'';\C^{2^k})}<\infty$ for every $V'\times V''\Subset U$, which means $F''\in \Co^\infty_{u,v}L^\infty_\theta(V',V'';\C)$.

Let $\mu>0$ be such that $F''\in\Co^\mu_\loc(U)$. Since $F''\in\Co^\infty_{u,v} L^\infty_\theta(V',V'';\C)$ for every $V'\times V''\Subset U$, we have:
\begin{enumerate}[parsep=-0.3ex,label=(7.$\mu$)]
    \item\label{Item::Final::Sharpdds::Claim} For every $\mu>0$, if $F''\in\Co^\mu_\loc(U)$, then $F''\in\Co^{\infty,\mu-}_{(u,v),s}(V',V'';\C)$ holds for all $V'\times V''\Subset U$.
\end{enumerate}

Now suppose by contrast that $F''\in\Co^{\alpha+}_\loc(U;\C)$. Say $F''\in\Co^{\alpha+2\eps}_\loc(U;\C)$ for some $\eps>0$. Take $\mu=\alpha+\eps$ in \ref{Item::Final::Sharpdds::Claim}, we have $F''_u,F''_v\in\Co^{\alpha+\eps}(V'\times V'';\C)$ for every precompact neighborhood $V'\times V''\Subset U$ of $0\in\R^2_{u,v}\times\R^1_\theta$.

Since $F''_u,F''_v$ are both non-vanishing, we have $-ie^{f(\theta)}=F''_v/F_u''\in\Co^{\alpha+\eps}$ near $(0,0,0)$. However $-ie^{f}\notin\Co^{\alpha+}$ near $\theta=0$ since $f\notin\Co^{\alpha+}$ near $\theta=0$. Contradiction!

This concludes the proof of $F''\notin\Co^{\alpha+}$ near $(0,0,0)$.

\medskip
To show $F^*\Coorvec s\notin\Co^{(\alpha-1)+}$ near $(0,0,0)$, we can assume $F'''\in\Co^{\alpha+}$, since otherwise $F'''_\theta \notin\Co^{(\alpha-1)+}$ in a neighborhood of $\theta=0$ and by \eqref{Eqn::Final::RealFroisSharp::EqnofF*dds} the $\Coorvec\theta$-component of $F^*\Coorvec s$ is not $\Co^{(\alpha-1)+}$ already.


Let $\mu>0$ be the largest number such that $F''\in\Co^{\mu-}$ near $(0,0,0)$. By assumption $F''\in C^1$ so $\mu\ge1$. Since $F''\notin\Co^{\alpha+}$ near $(0,0,0)$, we see that $\mu\le\alpha$.

By \ref{Item::Final::Sharpdds::Claim}, we know $F''_u,F''_v\in\Co^{\mu-\eps}$ near $(0,0,0)$ for every $\eps>0$, in particular $F''_u,F''_v\in\Co^{\mu-\frac12}$ near $(0,0,0)$.

We see that $F''_\theta\notin\Co^{(\mu-1)+}$ near $(0,0,0)$. Suppose otherwise, i.e. $F''_\theta\in\Co^{\mu-1+\eps}$ for some $0<\eps<\frac12$. Since $F''_u,F''_v\in\Co^{\mu-\frac12}$, we have $F''\in\Co^{\mu+\eps}$. This contradicts to the maximality of $\mu$.

Now we have $F''_\theta/F''_u\notin\Co^{(\mu-1)+}$ for such $\mu\ge1$, in particular $F''_\theta/F''_u\notin\Co^{(\alpha-1)+}$. Since $Z=\partial_u+ie^f\partial_v$ and $\bar Z=\partial_u-ie^f\partial_v$ are both $\Co^\alpha$-vector fields in $\R^3$ and we already assumed $F'''_\theta \in\Co^{(\alpha-1)+}$, using \eqref{Eqn::Final::RealFroisSharp::EqnofF*dds} again we conclude that $F^*\Coorvec s\notin\Co^{(\alpha-1)+}$ near $(0,0,0)$. This completes the whole proof.
\end{proof}

\subsection{Sharpness of $F^*\Coorvec z$: a $\bar\partial$-equation with parameter}\label{Section::Sharpddz}

In this subsection we will construct a rank-1 complex Frobenius structure $\Se$ on $\C^1\times\R^1$, such that $F^*\Coorvec z\notin\Co^\alpha$ whenever $F$ is a coordinate chart representing $\Se$ near the origin. 

We sketch the idea to the proof of Theorem \ref{Thm::Intro::Sharpddz} here. It is adapted from \cite{Liding}, where we prove that there is a $C^k$-integrable almost complex structure that does not admit a $C^{k+1}$-coordinate chart.

Endow $\C^1\times\R^1$ with standard coordinate system $(w,\theta)$. We consider a subbundle $\Se\le\C T(\C^1\times\R^1)$ as
\begin{equation*}
    \Se:=\Span(\partial_w+a(w,\theta)\partial_{\bar w}),\quad\text{where }a\in\Co^\alpha(\C^1\times\R^1;\C)\text{ such that }\|a\|_{L^\infty}<1.
\end{equation*}

Clearly $\Se$ is involutive and $\Se+\bar\Se=\Span(\partial_w,\partial_{\bar w})=(\C T\C_w^1)\times\R^1_\theta$ since $\rank \Se=1$ and $\sup_{w,\theta}|a(w,\theta)|<1$. So $\Se$ is a $\Co^\alpha$-complex Frobenius structure.

\smallskip
\noindent\textit{Step 1} (Proposition \ref{Prop::Final::SharpddzRed}): We show that, the existence of a coordinate chart $F$ near $(0,0)$ such that $F^*\Coorvec z\in\Co^\alpha$ spans $\Se$, is equivalent to the existence of $f\in\Co^\alpha$ near $(w,\theta)=(0,0)$ that solves
\begin{equation}\label{Eqn::Final::SharpddzRed::aPDE}
    (\partial_w+a(w,\theta)\partial_{\bar w})f=-\partial_{\bar w}a,\quad\text{or equivalently}\quad f_w+(af)_{\bar w}-a_{\bar w}f=-a_{\bar w}.
\end{equation}

\smallskip
\noindent\textit{Step 2} (Proposition \ref{Prop::Final::Expa}): We construct an $a(w,\theta)$ such that $a$ is smooth away from $\{w=0\}=\{0\}\times\R^1_\theta$, and the complex Hilbert transform (also known as the Beurling transform) $\partial_w^{-1}\partial_{\bar w} a\notin\Co^\alpha$ near $w=0$.

\smallskip
\noindent\textit{Step 3} (Theorem \ref{Thm::Final::ProofExampleddz}): We show that there is no $f\in\Co^\alpha$ defined near $(w,\theta)=(0,0)$ that solves \eqref{Eqn::Final::SharpddzRed::aPDE}. This can be done by showing that:
\begin{itemize}[parsep=-0.3ex]
    \item For every $f$ that solves the homogeneous equation $(\partial_w+a\partial_{\bar w})f=0$ with $f(0,\theta)=0$, we must have $f\in\Co^\alpha$.
    \item There exists a non-$\Co^\alpha$ function $f$ such that $f(0,\theta)=0$ and $f$ solves \eqref{Eqn::Final::SharpddzRed::aPDE}.
\end{itemize}
Thus, by the superposition principle, every solution to \eqref{Eqn::Final::SharpddzRed::aPDE} cannot be $\Co^\alpha$. By Step 1 we complete the proof.

To make the result more general, we consider the case of mixed regularity, where $a\in\Co^{\alpha,\beta}_{w,\theta}$ for arbitrary $\alpha>\frac12$ and $\beta>0$ (recall again from Lemma \ref{Lem::Hold::CharMixHold} that $\Co^\alpha_{(w,\theta)}=\Co^{\alpha,\alpha}_{w,\theta}$). See Theorem \ref{Thm::Final::SharpSum}.

\medskip
We now start the proof. We use the notation $f_w=\partial_wf$ etc.

\begin{prop}\label{Prop::Final::SharpddzRed}
Let $\alpha>\frac12,\beta>0$, let $a\in\Co^{\alpha,\beta}(\C^1_w,\R^1_\theta;\C)$ be such that $\sup|a|<1$, and let $\Se=\Span(\partial_w+a(w,\theta)\partial_{\bar w})$ be a $\Co^{\min(\alpha,\beta)}$-complex Frobenius structure with rank $1$. The following \ref{Item::Final::SharpddzRed::Chart} and \ref{Item::Final::SharpddzRed::PDE} are equivalent:
\begin{enumerate}[parsep=-0.3ex,label=(\Alph*)]
    \item\label{Item::Final::SharpddzRed::Chart} There is a continuous map $F:U_1\times V_1\subseteq\C^1_w\times\R^1_\theta\to\C^1_z\times\R^1_s$ near $(0,0)$ such that
    \begin{enumerate}[nolistsep,label=(A.\arabic*)]
        \item\label{Item::Final::SharpddzRed::Chart::Reg} $F$ is homeomorphic to its image, $\nabla_zF\in C^0(U_1\times V_1;\C^{2\times 2})$ and $\nabla_w(F^\Inv)\in C^0(U_1\times V_1;\C^{2\times 2})$.
        \item\label{Item::Final::SharpddzRed::Chart::Span}  $F^*\Coorvec z\in\Co^{\alpha,\beta}(U_1,V_1;\C^3)$ and $F^*\Coorvec z$ spans $\Se|_{U_1\times V_1}$.
    \end{enumerate}
    \item\label{Item::Final::SharpddzRed::PDE} There is a function $f\in\Co^{\alpha,\beta}(U_2, V_2;\C)$ defined in a neighborhood $U_2\times V_2\subseteq\C^1_w\times\R^1_\theta$ of $(0,0)$ that solves \eqref{Eqn::Final::SharpddzRed::aPDE}.
\end{enumerate}
\end{prop}
The results are also true for $\alpha$ or $\beta=\infty$, without any change of the proof.

For the application to Theorem \ref{Thm::Intro::Sharpddz}, we only need the contraposition of \ref{Item::Final::SharpddzRed::Chart} $\Rightarrow$ \ref{Item::Final::SharpddzRed::PDE}.

\begin{proof}
\ref{Item::Final::SharpddzRed::Chart} $\Rightarrow$ \ref{Item::Final::SharpddzRed::PDE}: Write $F=:(F',F'')$ and $\Phi=(\Phi',\Phi''):=F^\Inv:F(U_1\times V_1)\subseteq\C^1_z\times\R^1_s\to \C^1_w\times\R^1_\theta$. By assumption $\nabla_z\Phi=(\nabla_z\Phi',\nabla_z\Phi'')\in C^0(F(U_1\times V_1);\C^{2\times 2})$, so $F^*\Coorvec z=\frac{\partial(\re\Phi',\im\Phi',\Phi'')}{\partial z}\circ F$ and $F^*\Coorvec {\bar z}=\frac{\partial(\re\Phi',\im\Phi',\Phi'')}{\partial\bar z}\circ F$ are well-defined continuous vector fields on $\R^3_{(\re w,\im w,\theta)}$.

By assumption $(\Se+\bar\Se)|_{U_1\times V_1}=\Span(\partial_w,\partial_{\bar w})$ is spanned by $F^*\Coorvec{z}$ and $F^*\Coorvec z$, so $\partial_wF''=\partial_{\bar w}F''=0$ in the sense of distributions, which means $F''(w,\theta)=F''(\theta)$.

By assumption $d\bar F'=F^*d\bar z$ is continuous in $w$ and is annihilated by $\Span F^*\Coorvec z=\Span(\partial_w+a\partial_{\bar w})$. Therefore
\begin{equation}\label{Eqn::Final::SharpddzRed::Exp1}
    (\partial_w+a\partial_{\bar w})\bar F'=\bar F'_w+a\bar F'_{\bar w}=0,\quad\text{i.e.}\quad F'_{\bar w}=-\bar aF'_w.
\end{equation}

So for the transition matrix (cf. \eqref{Eqn::Final::Sharpdds::ClassicalExp1}) we have
\begin{equation}\label{Eqn::Final::SharpddzRed::Exp2}
\begin{pmatrix}F^*dz\\F^*d\bar z\\F^*ds\end{pmatrix}=
\begin{pmatrix}
F'_w&-\bar aF'_w&F'_\theta\\
-a\bar F'_{\bar w}&\bar F'_{\bar w}&\bar F'_\theta\\
&&F''_\theta
\end{pmatrix}
\begin{pmatrix}dw\\d\bar w\\d\theta\end{pmatrix}\quad \Rightarrow\quad
\begin{pmatrix}
\partial_w\\ \partial_{\bar w}
\end{pmatrix}=
\begin{pmatrix}
F'_w&-a\bar F'_{\bar w}\\
-\bar aF'_w&\bar F'_{\bar w}
\end{pmatrix}
\begin{pmatrix}F^*\partial_z\\F^*\partial_{\bar z}\end{pmatrix}.
\end{equation}

Applying the Cramer's rule of inverse matrices to the right hand matrix in \eqref{Eqn::Final::SharpddzRed::Exp2}, we have 
\begin{equation}\label{Eqn::Final::SharpddzRed::EqnF*ddz}
F^*\Coorvec z=\frac{\bar F'_{\bar w}}{(1-|a|^2)|F'_w|^2}(\partial_w+a\partial_{\bar w})=\frac{\partial_w+a\partial_{\bar w}}{(1-|a|^2)F'_w}\quad\text{in}\quad U_1\times V_1.
\end{equation}
Since $F^*\Coorvec z\in\Co^{\alpha,\beta}(U_1, V_1;\C^3)$ and $a\in\Co^{\alpha,\beta}(U_1,V_1;\C)$, we see that $F'_w\in\Co^{\alpha,\beta}_\loc(U_1,V_1;\C)$ and $F'_w\not\equiv0$ in $U_1$.

Take $U_2\times V_2\subseteq U_1\times V_1$ as a neighborhood of $(0,0)$ such that $\sup_{(w,\theta)\in U_2\times V_2}|\frac{F'_w(w,\theta)}{F'_w(0,\theta)}-1|<1$. 
We define 
$$f(w,\theta):=\log\frac{\bar F'_{\bar w}(w,\theta)}{\bar F'_{\bar w}(0,\theta)}=\sum_{n=1}^\infty\frac{(-1)^{n-1}}n\Big(\frac{\bar F'_{\bar w}(w,\theta)}{\bar F'_{\bar w}(0,\theta)}-1\Big)^n.$$

Since $\big[(w,\theta)\mapsto \frac{F'_w(w,\theta)}{F'_w(0,\theta)}\big]\in\Co^{\alpha,\beta}(U_2,V_2;\C)$, we see that $f\in\Co^{\alpha,\beta}(U_2,V_2;\C)$ as well. So \eqref{Eqn::Final::SharpddzRed::aPDE} follows:
\begin{equation}\label{Eqn::Final::SharpddzRed::PfTmp1}
    f_w+af_{\bar w}=\frac{\bar F'_{w\bar w}}{\bar F'_{\bar w}}+\frac{a\bar F'_{\bar w\bar w}}{\bar F'_{\bar w}}=\frac{(\bar F'_w+a\bar F'_{\bar w})_{\bar w}}{\bar F'_{\bar w}}-\frac{a_{\bar w}\bar F'_{\bar w}}{\bar F'_{\bar w}}=-a_{\bar w}.
\end{equation}


\noindent\ref{Item::Final::SharpddzRed::PDE} $\Rightarrow$ \ref{Item::Final::SharpddzRed::Chart}: This is the reverse of the above argument. By shrinking $U_2\times V_2$ we assume $U_2$ and $V_2$ are both convex neighborhoods.

Suppose $f\in\Co^{\alpha,\beta}(U_2,V_2;\C)$ is a solution to \eqref{Eqn::Final::SharpddzRed::aPDE}. Take $\mu:=e^{f}$, we know $\mu\in\Co^{\alpha,\beta}(U_2,V_2;\C)$ satisfies
\begin{equation}\label{Eqn::Final::SharpddzRed::PfTmp2}
    (a\mu)_{\bar w}=a_{\bar w}e^f+af_{\bar w}e^f=-f_we^f=-\mu_w,\quad\text{in}\quad U_2\times V_2.
\end{equation}
Therefore $\mu(\cdot,\theta)(d\bar w-a(\cdot,\theta)dw)$ is a closed 1-form in $U_2\subseteq\C^1$, for each $\theta\in\R$. We define $F':U_2\times V_2\to\C$ as 
\begin{equation}\label{Eqn::Final::SharpddzRed::DefF'}
    F'(w,\theta):=\int_0^1(w\cdot\bar \mu(tw,\theta)+\bar w\cdot\bar a(tw,\theta)\bar\mu(tw,\theta))dt.
\end{equation}

We see that $F'\in\Co^{\alpha,\beta}_\loc(U_1,V_1;\C)$ and satisfies $d\bar F'=\mu d\bar w-a\mu dw+\bar F'_\theta d\theta$. In other words
\begin{equation*}
    \nabla_w(F',\bar F')=\begin{pmatrix}F'_w&\bar F'_w
    \\ F'_{\bar w}&\bar F'_{\bar w}
    \end{pmatrix}=\begin{pmatrix}\bar \mu&-a\mu
    \\ -\bar a\bar\mu&\mu
    \end{pmatrix}\quad\Rightarrow\quad \det\big(\nabla_w(F',\bar F')\big)=|\mu|^2(1-|a|^2)>0.
\end{equation*}

Since $a,\mu\in \Co^{\alpha,\beta}$, we see that $F'\in\Co^{\alpha+1,\beta}$ and the matrix function $\nabla_w(F',\bar F')\in\Co^{\alpha,\beta}$ is invertible at $(0,0)$. Applying Lemma \ref{Lem::Hold::CompofMixHold} \ref{Item::Hold::CompofMixHold::InvFun} to $(\re F',\im F')\in\Co^{\alpha+1,\beta}(U_2,V_2;\R^2)$, we can find a neighborhood $U_1\times V_1\subseteq U_2\times V_2$ such that $F'(\cdot,\theta):U_1\to\C^1_z$ is injective for each $\theta\in V_1$ and its inverse map $\Phi'(z,s)=F'(\cdot,s)^\Inv(z)$ is $\Co^{\alpha+1,\beta}_{z,s}$ near $(F'(0,0),0)=(0,0)$.

Take $F(w,\theta):=(F'(w,\theta),\theta)$ and $\Phi:=F^\Inv$, we see that $\Phi(z,s)=(\Phi'(z,s),s)$. Therefore $F:U_1\times V_1\to\C^1\times\R^1$ is homeomorphic to its image (since $F,\Phi$ are both $\Co^{\alpha+1,\beta}\subset C^0$), and $\nabla_z(F^\Inv)=(\nabla_z\Phi',0)$ is a continuous map on $F(U_1\times V_1)$. This completes the proof of \ref{Item::Final::SharpddzRed::Chart::Reg}.

Now for each $\theta\in V_1$, $F'(\cdot,\theta):U_1\to\C_z$ is a $\Co^{\alpha+1}$-coordinate chart satisfying $d(\bar F'(\cdot,\theta))=\mu(\cdot,\theta)d\bar w-a\mu(\cdot,\theta)dw$ in $U_1$, so $(\partial_w+a\partial_{\bar w})F'=0$, which means $F'(\cdot,\theta)^*d\bar z=d\bar F'(\cdot,\theta)$ spans $\Se|_{U_1\times\{\theta\}}=\Span(\partial_w+a(\cdot,\theta)\partial_{\bar w})$. Taking a dual we get $\Se|_{U_1\times\{\theta\}}=\Span F'(\cdot,\theta)^*\Coorvec z$ for each $\theta\in V_1$.

Now $F(w,\theta)=(F'(w,\theta),\theta)$ and $\Phi(z,s)=(\Phi'(z,s),s)$, so $((\nabla_z\Phi)\circ F)(w,\theta)=((\nabla_z\Phi'(\cdot,\theta))\circ F'(\cdot,\theta))(w)$ hence $F^*\Coorvec z|_{(w,\theta)}=F'(\cdot,\theta)^*\Coorvec z|_w$ spans $\Se_{(w,\theta)}$ for each $w\in U_1$ and $\theta\in V_2$, finishing the proof of \ref{Item::Final::SharpddzRed::Chart::Span}.
\end{proof}
\begin{remark}We do not know how to relax the assumption $\alpha>\frac12$ to $\alpha>0$. The difficulty comes from \eqref{Eqn::Final::SharpddzRed::PfTmp1} and \eqref{Eqn::Final::SharpddzRed::PfTmp2} in the proof, where the intermediate calculations need a product of a $\Co^{\alpha-1}$-function and a $\Co^\alpha$-function. On the other hand, the PDE \eqref{Eqn::Final::SharpddzRed::aPDE} can still be defined when $0<\alpha\le\frac12$, by requiring $a$ to stay in the space, for example,  $\Co^{\alpha,\beta}(\C^1_w,\R^1_\theta;\C)\cap L^\infty(\R^1_\theta;W^{1,2}(\C^1_w;\C))$ (see Definition \ref{Defn::Final::HolSobSpace} and Proposition \ref{Prop::Final::Expa} \ref{Item::Final::Expa::Reg1}).
\end{remark}
For Step 2, the construction of $a(w,\theta)$ requires the following.


\begin{defn}\label{Defn::Final::HolSobSpace}
Let $k\ge0$ be an integer, $\beta>0$,  $1\le p<\infty$. Let $U\subseteq\R^n_x$ be an open set and let $(\psi_j)_{j=0}^\infty$ be a dyadic resolution of $\R^q_s$. We denote by $\Co^\beta_s(\R^q;W^{k,p}_x(U))$ the space of all distributions $f\in\D'(\R^q_s\times U_x)$ such that $\psi_j\ast_sf(\cdot,s)\in W^{k,p}(U)$ for every $s\in\R$ (recall Notation \ref{Note::Hold::ConvVar} for $\psi_j\ast_sf$), with
\begin{equation*}
    \|f\|_{\Co^\beta(\R^q;W^{k,p}(U))}=\|f\|_{\Co^\beta(W^{k,p})}:=\sup_{j\ge0,s\in\R^q}2^{j\beta}\|(\psi_j\ast_sf)(\cdot,s)\|_{W^{k,p}(U)}<\infty.
\end{equation*}

We denote by $L^\infty(\R^q;W^{k,p}(U))$ the space of all $L^1_\loc$-function $f:\R^q\times U\to\R$ such that $f(\cdot,s)\in W^{k,p}(U)$ for almost every $s\in\R^q$ with $\|f\|_{L^\infty(W^{k,p})}:=\essup_{s\in\R^q}\|f(\cdot,s)\|_{W^{k,p}(U)}<\infty$.
\end{defn}
One can check that $\Co^\beta(\R^q;W^{k,p}(U))$ is the same as putting $X=W^{k,p}(U)$ in Definition \ref{Defn::Intro::DefofHold} (see also Remark \ref{Rmk::Hold::RmkforBiHold}). In particular the space does not depend on the choice of $(\psi_j)_{j=0}^\infty$. We leave their proof to readers since we are allowed to fix a choice of $(\psi_j)_j$ in the later discussion.

When $U$ is a smooth domain, the space $L^\infty(\R^q;W^{k,p}(U))$ coincides with the classical vector-valued $L^\infty$-space (cf. Remark \ref{Rmk::Hold::RmkforBiHold} \ref{Item::Hold::RmkforBiHold::notVectMeas}) since $W^{k,p}(U)$ is a separable space.

\begin{lem}\label{Lem::Final::HolSobLem}
Let $\beta>0$ and $1<p<\infty$. Let $U\subset\R^n$ be a bounded smooth domain.
\begin{enumerate}[parsep=-0.3ex,label=(\roman*)]
    \item\label{Item::Final::HolSobLem::Prod1} The product map $[(f,g)\mapsto fg]:\Co^\beta(\R^q;L^p(U))\times L^\infty\Co^\beta(U,\R^q)\to\Co^\beta(\R^q;L^p(U))$ is bounded bilinear.
    \item\label{Item::Final::HolSobLem::Prod2} When $\beta<1$, the product map $[(f,g)\mapsto fg]:(L^\infty\cap W^{1,2})(U)\times\Co^{\beta-1}(U)\to\Co^{\beta-1}(U)$ is bounded bilinear.
\end{enumerate}
\end{lem}
\begin{proof}
    \ref{Item::Final::HolSobLem::Prod1} can follow from the boundedness of the product map $L^p(U)\times L^\infty(U)\to L^p(U)$ and that $\Co^\beta(\R^q)$ is closed under multiplications. More precisely, by Lemma \ref{Lem::Hold::LemParaProd} we have the following uniformly in $l\ge0$ and $s\in\R^q$:
    \begin{align*}
        &\|(\psi_l\ast_s(fg))(\cdot,s)\|_{L^p(U)}\le\|\psi_l\|_{L^1}\big(\sum_{j=l-2}^{l+2}\sum_{j'=0}^{j-3}+\sum_{j=l-2}^{l+2}\sum_{j'=0}^{j-3}+\sum_{\substack{j,j'\ge l-3\\|j-j'|\le2}}\big)\|(\psi_j\ast_sf)(\cdot,s)\|_{L^p(U)}\|(\psi_{j'}\ast_sg)(\cdot,s)\|_{L^\infty(U)}
        \\
        \lesssim&_{\psi}\|f\|_{\Co^\beta_s(L^p_x)}\|g\|_{L^\infty_x\Co^\beta_s}\big(\sum_{|j-l|\ge 2}\sum_{j'\ge j-3}+\sum_{|j'-l|\ge 2}\sum_{j\ge j'-3}+\sum_{\substack{j,j'\ge l-3\\|j-j'|\le2}}\big)2^{-j\beta}2^{-j'\beta}\lesssim_\beta\|f\|_{\Co^\beta_s(L^p_x)}\|g\|_{L^\infty_x\Co^\beta_s}2^{-l\beta}.
    \end{align*}
    
    Proof of \ref{Item::Final::HolSobLem::Prod2} follows from \cite[Chapter 2.8.1]{Chemin}, where in the reference we use Bony's decomposition $uv=T(u,v)+T(v,u)+R(u,v)$ for distributions $u,v$ on $\R^2$ (see \cite[(2.41)]{Chemin} and Remark \ref{Rmk::Hold::LemParaProd}). By \cite[Theorem 2.82]{Chemin} we have $T:L^\infty(\R^2)\times\Co^{\beta-1}(\R^2)\to\Co^{\beta-1}(\R^2)$ and $T:\Co^{\beta-1}(\R^2)\times L^\infty(\R^2)\to\Co^{\beta-1}(\R^2)$. By \cite[Theorem 2.85]{Chemin} we have $R:W^{1,2}(\R^2)\times\Co^{\beta-1}(\R^2)\to\Bs_{22}^\beta(\R^2)$. Here $\Bs_{22}^\beta(\R^2)=H^\beta(\R^2)$ is the fractional $L^2$-Sobolev space (see for example \cite[Definition 2.68]{Chemin} or \cite[Definition 2.3.1/2(i)]{Triebel1}).
    We have Sobolev embedding $\Bs_{22}^{\beta}(\R^2)\hookrightarrow\Co^{\beta-1}(\R^2)$, whose proof can be found in for example \cite[Chapter 2.7.1]{Triebel1}\footnote{We recall the correspondence of Besov space $\Co^\beta(\R^q)=\Bs_{\infty\infty}^\beta(\R^q)$ and Triebel-Lizorkin space $L^p(\R^n)=\mathscr F_{p2}^0(\R^n)$ from \cite[Theorems 2.5.7 and 2.5.12]{Triebel1}.}. Thus $R:W^{1,2}(\R^2)\times\Co^{\beta-1}(\R^2)\to\Co^{\beta-1}(\R^2)$ is bounded.
    
    Therefore $[(u,v)\mapsto uv]:(L^\infty\cap W^{1,2})(\R^2)\times\Co^{\beta-1}(\R^2)\to\Co^{\beta-1}(\R^2)$ is bilinearly bounded. Taking restrictions to $U$ we complete the proof.
\end{proof}

\begin{lem}\label{Lem::Final::InvDBar}
Let $\Pc_0$ and $\Pc$ be the zero Dirichlet boundary solution operators in Lemma \ref{Lem::Hold::LapInvBdd} with $U=\B^2_w\subset\C^1$ and $V=\R^1_\theta$ in the lemma. Define $\Tc_0:=4\partial_{\bar w}\Pc_0$ and $\Tc:=4\partial_{\bar w}\Pc$. Then 
\begin{enumerate}[parsep=-0.3ex,label=(\roman*)]
    \item\label{Item::Final::InvDBar::Bdd} $\partial_w\Tc f=f$ for all $f\in L^\infty(\B^2,\R^1;\C)$ and $\Tc$ has the following boundedness: for $\alpha>0$, $1<p<\infty$ and $\Xs\in\{L^\infty,\Co^\gamma:\gamma>0\}$,
\begin{gather}\label{Eqn::Final::InvDBar::Bdd}
    \Tc:\Co^{\alpha-1}\Xs(\B^2,\R^1;\C)\to\Co^\alpha\Xs(\B^2,\R^1;\C)\quad\text{and}\quad\Tc:\Co^\alpha(\R^1;L^p(\B^2))\to\Co^\alpha(\R^1;W^{1,p}(\B^2)).
\end{gather}
    \item\label{Item::Final::InvDBar::Holo} Moreover, if $f\in L^\infty\Co^\beta(\B^2,\R^1;\C)$ satisfies $f|_{\{|w|>\frac12\}}\equiv0$, then
\begin{equation}\label{Eqn::Final::InvDBar::Holo}
    f-\Tc\partial_w f\in\Co^\infty\Co^\beta(\B^2,\R^1;\C).
\end{equation}
\end{enumerate}

\end{lem}
\begin{proof}
That $\partial_w\Tc=\id$ and $\Tc:\Co^{\alpha-1}_w\Xs_\theta\to \Co^{\alpha}_w\Xs_\theta$ are the immediate consequences to Lemma \ref{Lem::Hold::LapInvBdd} since $\Delta_{\re w,\im w}=4\frac{\partial^2}{\partial w\partial \bar w}$.

We have boundedness $\Pc_0:L^p(\B^2)\to W^{2,p}(\B^2)$ by \cite[Theorem 9.13]{GilbargTrudinger}, thus $\Tc_0:L^p(\B^2;\C)\to W^{1,p}(\B^2;\C)$. We get $\Tc:\Co^\beta(L^p)\to\Co^\beta(W^{1,p})$ by the same reason to Lemma \ref{Lem::Hold::TOtimesId}. Thus we prove \ref{Item::Final::InvDBar::Bdd}.

\smallskip
To prove \ref{Item::Final::InvDBar::Holo}, we need to show that $\id-\Tc_0\partial_w:L^\infty(\frac12\B^2;\C)\to\Co^M(\B^2;\C)$ is bounded for all $M>0$. 

Let $g\in L^\infty(\frac12\B^2;\C)$ with zero extension outside $\frac12\B^2$. By formula of Green functions (see for example \cite[Theorem 2.12 and Section 2.2.c]{Evans}) we have for $w\in\B^2$,
\begin{equation*}
    g(w)-\Tc_0\partial_wg(w)=g(w)-\int_{\frac12\B^2}4g(z)\partial^2_{z\bar w}\big(\Ga(z-w)-\Ga\big(|z|w-\tfrac z{|z|}\big)\big)dxdy=\int_{\frac12\B^2}4g(z)\partial^2_{z\bar w}\Ga\big(|z|w-\tfrac z{|z|}\big)dxdy.
\end{equation*}
Here $z=x+iy\in\B^2$ and $\Ga(z)=\frac1{2\pi}\log|z|$ (which is the $n=2$ case in \eqref{Eqn::SecHolLap::NewtonPotent}). The first equality follows from integration by part, the second equality follows from the fact that $\partial_{z\bar w}\Ga(z-w)=\partial_{z\bar z}\Ga(z-w)=\frac14\Delta_{x,y}\Ga(z-w)=\frac{\delta_w(z)}4$.

Clearly $(z,w)\mapsto\Ga(|z|w-z/|z|)$ is a smooth function on $(z,w)\in r\B^2\times\B^2$ for any $r<1$. Thus for any integer $M\ge0$
\begin{equation*}
    \|\nabla^M(g-\Tc_0\partial_wg)\|_{L^\infty(\B^2;\C^{2^k})}\le\|g\|_{L^\infty(\frac12\B^2;\C)}\sup_{|w|<\frac12}\int_{\B^2}\big|\nabla^{M+2}_{z,w}\Ga\big(|z|w-\tfrac z{|z|}\big)\big|dxdy\lesssim_M\|g\|_{L^\infty(\frac12\B^2;\C)}.
\end{equation*}
We conclude that $\id-\Tc_0\partial_w:L^\infty(\frac12\B^2;\C)\to\Co^M(\B^2;\C)$ is bounded for all $M>0$. By Lemma \ref{Lem::Hold::TOtimesId}, we get the boundedness $\id-\Tc\partial_w:L^\infty\Co^\beta(\frac12\B^2,\R^1;\C)\to\Co^M\Co^\beta(\B^2,\R^1;\C)$. The proof is complete.
\end{proof}

\begin{lem}\label{Lem::Final::CalpNonlocal}
Let $\beta>0$ and let $(c_k)_{k=1}^\infty\subset\C$ be a sequence.
Define $g(\theta):=\sum_{k=1}^\infty c_ke^{2\pi i\cdot 2^k\theta}2^{-\beta\theta}$, then
\begin{enumerate}[parsep=-0.3ex,label=(\roman*)]
    \item\label{Item::Final::CalpNonlocal::Bdd} $\|g\|_{\Co^\beta(\R;\C)}\approx\sup_{k\ge0}|c_k|$ with the implied constant depending only on $\beta$.
    \item\label{Item::Final::CalpNonlocal::Main} If $\varlimsup_{k\to\infty}|c_k|=\infty$, then $g\notin\Co^\beta(I;\C)$ for every interval $I\subseteq\R$.
\end{enumerate}

\end{lem}
\begin{proof}
Let $\{\psi_k\}_{k=0}^\infty$ be a dyadic resolution for $\R$. Recall from \eqref{Eqn::Hold::RmkDyaSupp} $\supp\hat\psi_k\subset(-2^{k-1},-2^{k+1})\cup(2^{k-1},2^{k+1})$ holds for $k\ge1$. Since the Fourier support of $e^{2\pi i2^k\theta}$ is $\{2^k\}$, we have $\psi_k\ast g(\theta)=c_k2^{-\beta\theta}e^{2\pi i\cdot 2^k\theta}$ for all $k\ge1$. Therefore by Lemma \ref{Lem::Hold::HoldChar} \ref{Item::Hold::HoldChar::LPHoldChar}, we get the result \ref{Item::Final::CalpNonlocal::Bdd}:
$$\textstyle\|g\|_{\Co^\beta(\R)}\approx_{\psi,\beta}\sup_{k\ge0}2^{k\beta}\|\psi_k\ast g\|_{L^\infty}=\sup_{k\ge1}|c_k|.$$

Now suppose $\varlimsup_{k\to\infty}|c_k|=+\infty$, so $\|g\|_{\Co^\beta(\R)}\approx_{\psi,\beta}=\sup|c_k|=\infty$.

For a given interval $I\subseteq\R$, $I$ contains a dyadic subinterval, say $[ n2^{-k_0},(n+1)2^{-k_0}]$ where $n\in\Z$ and $k_0\in\Z_+$. Take $g^{k_0}(\theta):=\sum_{k=k_0+1}^\infty c_{k}2^{2\pi i2^k\theta}2^{-k\beta}$, so $$\textstyle g(\theta)-g^{k_0}(\theta)=\sum_{k=1}^{k_0}c_ke^{2\pi i2^k\theta}2^{-\beta\theta},$$
must have finite $\Co^\beta(\R)$-norm (though whose bound depends on $k_0$).

Since $g^{k_0}$ is $2^{-k_0}$-periodic, $\|g\|_{\Co^\beta(\R)}=\infty$ and $\|g-g^{k_0}\|_{\Co^\beta(\R)}<\infty$, we know $\|g^{k_0}\|_{\Co^\beta(n2^{-k_0},(n+1)2^{-k_0})}=\infty$. Therefore, we complete the proof of \ref{Item::Final::CalpNonlocal::Main} by:
$$\textstyle\|g\|_{\Co^\beta(I)}\ge\|g\|_{\Co^\beta([n2^{-k_0},(n+1)2^{-k_0}])}\ge \|g^{k_0}\|_{\Co^\beta([n2^{-k_0},(n+1)2^{-k_0}])}-\|g-g^{k_0}\|_{\Co^\beta(\R)}=\infty.\qedhere$$
\end{proof}

We now construct our $\Co^{\alpha,\beta}$-function $a(w,\theta)$ and finish the Step 2:
\begin{prop}\label{Prop::Final::Expa}
Let $\alpha,\beta>0$. There is a function $ A(w,\theta)$ defined in $\C^1_w\times\R^1_\theta$, that satisfies the following properties:
\begin{enumerate}[parsep=-0.3ex,label=(\roman*)]
    \item\label{Item::Final::Expa::Reg1} $ A\in\Co^{\alpha,\beta}(\C^1,\R^1;\C)\cap \Co^\beta(\R^1;W^{1,2}(\C^1;\C))$ and $[(w,\theta)\mapsto w\partial_{\bar w} A(w,\theta)]\in\Co^{\alpha,\beta}(\C^1,\R^1;\C)$. 
    \item\label{Item::Final::Expa::Reg2} $[(w,\theta)\mapsto w^{-1} A(w,\theta)]\in \Co^\frac\beta2(\R^1;L^p(\C^1;\C))$ for all  $1<p<\frac4{2-\min(\alpha,2)}$. In particular $w^{-1}A\in\Co^\frac\beta2(\R^1;L^\frac{\alpha+4}2)$.
    
    \item\label{Item::Final::Expa::Unbdd} $\Tc[\partial_{\bar w} A]\notin L^\infty_w\Co^\beta_\theta$ near $(w,\theta)=(0,0)$, where $\Tc$ be the inverse $\partial_w$-operator in Lemma \ref{Lem::Final::InvDBar}.

\end{enumerate}
\end{prop}

The construction is essentially the parameter version of \cite{Liding}. 
\begin{remark}
By taking some simple modifications from the construction below (see \eqref{Eqn::Final::Expa::Defofa}), one can also prove Proposition \ref{Prop::Final::Expa} for the case $\alpha=\infty$, $\beta<\infty$. We leave the details to readers.
\end{remark}

\begin{proof}[Proof of Proposition \ref{Prop::Final::Expa}]
First we find a function $b\in C_c^0(\B^2;\C)$ such that,
\begin{enumerate}[nolistsep,label=(b.\arabic*)]
    \item\label{Item::Final::Expa::BReg} $b\in  C^\infty_\loc(\B^2\backslash\{0\};\C)$, $b\in W^{1,2}(\B^2;\C)$ and $wb\in C_c^1(\B^2;\C)$. 
    \item\label{Item::Final::Expa::BUnbdd} $\Tc_0 \partial_{\bar w}b\notin L^\infty$ near $w=0$. Here $\Tc_0$ be the inverse $\partial_w$-operator in Lemma \ref{Lem::Final::InvDBar}.
\end{enumerate}

Indeed, we take a function $b$ such that
\begin{equation}\label{Eqn::Final::Expa::Defofb}
    b(w):=\partial_w\big(\bar w(-\log|w|)^{\frac13}\big)=-\tfrac16(-\log|w|)^{-\frac23}\tfrac{|w|^2}{w^2},\quad\text{when }|w|<\tfrac12.
\end{equation}
And $b$ is smoothly cutoffed outside $|w|<\frac12$.

Immediately we get $b\in C^\infty_\loc(\B^2\backslash\{0\};\C)$ and $wb(w)\in C^1_c$. And clearly we have $|\nabla b(w)|\lesssim|w|^{-1}(-\log|w|)^{-\frac23}$, so $b\in W^{1,2}$, which gives \ref{Item::Final::Expa::BReg}.

To prove \ref{Item::Final::Expa::BUnbdd}, since $\partial_w\Tc_0 \partial_{\bar w}b=\partial_{\bar w}b$ it suffices to show that if $U$ is a neighborhood of $0$ and $f:U\subseteq\frac12\B^2\to\C$ is a function satisfies $\partial_wf=\partial_{\bar w}b$, then $f\notin L^\infty$ near $0$.

Indeed by \eqref{Eqn::Final::Expa::Defofb} in a neighborhood of $0$ we have $f_w=(\bar w(-\log|w|)^\frac12)_{w\bar w}$, so $f-(\bar w(-\log|w|)^\frac12)_{\bar w}$ is annihilated by $\partial_w$ which is an anti-holmorphic function. Therefore $f-(\bar w(-\log|w|)^\frac12)_{\bar w}$ is a bounded function near $w=0$. On the other hand 
$$(\bar w(-\log|w|)^\frac13)_{\bar w}=(-\log|w|)^\frac13+O(1)\quad \text{as }w\to0,$$
is not bounded near $0$. We conclude that $f\notin L^\infty$ near $0$.

\medskip
Let $\rho_0\in\Sc(\C^1_w)$ be a Schwartz function whose Fourier transform satisfies $\hat\rho_0\in C_c^\infty(\B^2)$ and $\hat\rho_0|_{\frac12\B^2}\equiv1$. Define functions $(\rho_k,c_k)_{k=0}^\infty$ and $ A(w,\theta)$ as 
\begin{equation}\label{Eqn::Final::Expa::Defofa}
    \rho_k(w):=2^{n\frac\beta\alpha k}\rho_0(2^{\frac\beta\alpha k}w),\quad c_k(w):=\rho_k\ast b(w)-\rho_k\ast b(0),\quad k\ge0;\quad  A(w,\theta):=\sum_{j=0}^\infty c_j(w)e^{2\pi i2^j\theta}2^{-j\beta}.
\end{equation}

Thus, we have $wc_k=\int_\C(w-z)\rho_k(w-z)b(z)d\vol(z)=(w\rho_k)\ast b$ for $k\ge0$.
By scaling
\begin{equation}\label{Eqn::Final::Expa::EstRho}
    \|\rho_k\|_{L^1}=\|\rho_0\|_{L^1},\quad\|\nabla\rho_k\|_{L^1}=2^{\frac\beta\alpha k}\|\nabla\rho_0\|_{L^1},\quad\|w\nabla\rho_k\|_{L^1}=\|w\nabla\rho_0\|_{L^1}\quad\text{for each }k\ge1.
\end{equation}
Therefore,
\begin{equation}\label{Eqn::Final::Expa::EstC}
\begin{gathered}
    \|c_k\|_{L^\infty\cap W^{1,2}}\le2\|\rho_k\ast b\|_{L^\infty\cap W^{1,2}}\lesssim\|\rho_k\|_{L^1}\|b\|_{L^\infty\cap W^{1,2}}\lesssim1;\\
    \|w\partial_{\bar w} c_k\|_{L^\infty}=\|\partial_{\bar w}(w\rho_k)\ast b\|_{L^\infty}\le\|w\partial_{\bar w}\rho_k\|_{L^1}\|b\|_{L^\infty}\lesssim1.
\end{gathered}
\end{equation}

Let $(\phi_j)_{j=0}^\infty$ and $(\psi_k)_{k=0}^\infty$ be two dyadic resolutions for $\C^1_w$ and $\R^1_\theta$ respectively (recall Definition \ref{Defn::Hold::DyadicResolution}). Recall from \eqref{Eqn::Hold::RmkDyaSupp} $\hat\phi_0\in C_c^\infty(2\B^2)$, $\hat\psi_0\in C_c^\infty(2\B^1)$, $\supp\hat\phi_j\subset2^{j+1}\B^2\backslash2^{j-1}\B^2$ and $\supp\hat\psi_j\subset2^{j+1}\B^1\backslash2^{j-1}\B^1$ for $j\ge1$. Since $\supp\hat\rho_0\subset\B^2$, by scaling we have $\supp\hat\rho_k\subset2^{\frac\beta\alpha k}\B^2$. Thus,
\begin{equation}\label{Eqn::Final::Expa::Estconv}
    \phi_j\ast_w A(w,\theta)=\sum_{k\ge\max(0,\frac\alpha\beta(j-1))}(\phi_j\ast c_k)(w)e^{2\pi i2^k\theta}2^{-k\beta},\quad \psi_j\ast_\theta A(w,\theta)=c_j(w)e^{2\pi i2^j\theta}2^{-j\beta},\quad j\ge0.
\end{equation}
By \eqref{Eqn::Final::Expa::EstC} and \eqref{Eqn::Final::Expa::Estconv} we get $ A\in\Co^\alpha L^\infty\cap L^\infty\Co^\beta(\C^1,\R^1;\C)\cap \Co^\beta(\R^1;W^{1,2}(\C^1;\C))$, since for $j\ge0$,
\begin{equation*}
\begin{gathered}
    \|\phi_j\ast_w A\|_{L^\infty(\C^1\times\R^1;\C)}\le\sum_{k\ge\max(0,\frac\alpha\beta(j-1))}\|\phi_j\|_{L^1(\C^1)}\|c_k\|_{L^\infty(\C^1\times\R^1;\C)}2^{-k\beta}\lesssim_{\phi,c,\alpha,\beta}2^{-j\alpha};\\
    \sup_{\theta\in\R^1}\|\psi_j\ast_\theta A(\cdot,\theta)\|_{L^\infty\cap W^{1,2}(\C^1;\C)}\le\|\psi_j\|_{L^1(\R)}\|c_j\|_{L^\infty\cap W^{1,2}(\C^1;\C)}2^{-j\beta}\lesssim_{\psi,c,\beta}2^{-j\beta}.
\end{gathered}
\end{equation*}
Taking supremum over $j\ge0$ we get \ref{Item::Final::Expa::Reg1}.

\medskip Using \eqref{Eqn::Final::Expa::EstRho} and \eqref{Eqn::Final::Expa::EstC} we also have $\|c_k\|_{C^1}\lesssim 2^{\frac\beta\alpha k}$ and $\|w\nabla c_k\|_{C^0}\lesssim1$. Thus 
\begin{equation*}
    |\nabla c_k(w)|\lesssim\min(|w|^{-1},2^{\frac\beta\alpha k})\le|w|^{\min(\frac\alpha2,1)-1}2^{k\frac\beta2\cdot\min(\frac\alpha2,1)}\le|w|^{\min(\frac\alpha2,1)-1}2^{k\frac\beta2} ,\quad\text{for all}\quad k\ge0,\quad w\in\C^1.
\end{equation*}
By \eqref{Eqn::Final::Expa::Defofa} we have $c_k(0,\theta)\equiv0 $ for all $\theta\in\R$. Thus by \eqref{Eqn::Final::Expa::Estconv},
\begin{equation*}
    |\psi_k\ast_\theta (w^{-1}A)(w,\theta)|=|w^{-1}c_k(w)|2^{-k\beta}\le|\nabla c_k(w)|2^{-k\beta}\le|w|^{\min(\frac\alpha2,1)-1}2^{-k\frac\beta2},\quad k\ge0,\quad w\in\C^2,\quad\theta\in\R^1.
\end{equation*}

On the other hand $|w|^{-\gamma}\in L^{\frac2\gamma-}(\B^2)$ holds for all $0<\gamma<2$. Thus for $1<p<\frac2{1-\min(\alpha/2,1)}=\frac4{2-\min(\alpha,2)}$ we have $\sup_{k\ge0;\theta\in\R^1}2^{k\frac\beta2}\|\psi_k\ast_\theta (w^{-1}A)(\cdot,\theta)\|_{L^p(\C^1;\C)}<\infty$.

Clearly $2<\frac{\alpha+4}2<\frac4{2-\min(\alpha,2)}$, we finish the proof of \ref{Item::Final::Expa::Reg2}.

\medskip To prove \ref{Item::Final::Expa::Unbdd}, since $\Tc$ only acts on $w$-variable,  by \eqref{Eqn::Final::Expa::Defofa} we have 
$$\Tc  A_{\bar w}(w,\theta)=\sum_{k=1}^\infty2^{-k\alpha}e^{2\pi i2^k\theta}\Tc_0\partial_{\bar w} c_k(w)=\sum_{k=1}^\infty2^{-k\alpha}e^{2\pi i2^k\theta}\Tc_0(\rho_k\ast \partial_{\bar w}b)(w),\quad w\in\C^1,\ \theta\in\R.$$

Now $\lim\limits_{k\to\infty}\rho_k\ast(\partial_{\bar w}b)=\partial_{\bar w} b$ holds (as $L^2$-functions), and by \ref{Item::Final::Expa::BUnbdd} $\Tc_0 \partial_{\bar w}b\notin L^\infty$ near $w=0$. So for any neighborhood $U\subseteq\C^1$ of $w=0$, we have $\varlimsup_{k\to\infty}\|\Tc_0 \partial_wc_k\|_{L^\infty(U;\C)}=+\infty$. Therefore by Lemma \ref{Lem::Final::CalpNonlocal} \ref{Item::Final::CalpNonlocal::Main}, for any interval $I\subseteq\R$ containing $\theta=0$, we have $\|\Tc\partial_{\bar w} A\|_{L^\infty\Co^\alpha(U,I;\C)}\gtrsim \varlimsup_{k\to\infty}\|\Tc_0\partial_{\bar w}b\|_{L^\infty(U;\C)}=+\infty$.

So $\Tc \partial_{\bar w} A\notin L^\infty_w\Co^\beta_\theta$ near $(0,0)$, finishing the proof of \ref{Item::Final::Expa::Unbdd}.
\end{proof}

Now we prove Step 3, which is the non-existence of $f\in\Co^{\alpha,\beta}$ to \eqref{Eqn::Final::SharpddzRed::aPDE}. Cf. \cite[Proposition 7]{Liding}. Here we consider all $\alpha,\beta>0$.

\begin{thm}[$\bar\partial$-equation for sharpness of $F^*\Coorvec z\notin\Co^\alpha $]\label{Thm::Final::ProofExampleddz}
    Let $\alpha,\beta>0$, let $ A$ be as in Proposition \ref{Prop::Final::Expa}. Then there is a constant $\delta=\delta(\alpha,\beta, A)>0$, such that for any neighborhood $U_2\times V_2\subset\C^1_w\times\R^1_\theta$ of $(0,0)$ there is no $f\in\Co^{\alpha,\beta}(U_2,V_2;\C)$ solving \eqref{Eqn::Final::SharpddzRed::aPDE} with $a(w,\theta):=\delta\cdot A(w,\theta)$.
\end{thm}
This result contains the largest possible range for $\alpha$ and $\beta$. The proof can be simplified if $\alpha>1$.
\begin{proof}
Let $\Tc:\Co^{\alpha-1} L^\infty(\B^2,\R^1;\C)\to \Co^{\alpha} L^\infty(\B^2,\R^1;\C)$ be the inverse $\partial_w$-operator as in Lemma \ref{Lem::Final::InvDBar}. Recall that $\Tc:\Co^\beta(\R^1;L^p(\B^2;\C))\to \Co^\beta(\R^1;W^{1,p}(\B^2;\C))$ is also bounded for $1<p<\infty$. 

We define a linear operator $\Tf=\Tf_{ A}$ on functions in $\B^2\times\C^1$ as
\begin{equation}
    \Tf[\varphi]:=\Tc( A\cdot\partial_{\bar w}\varphi),\quad\varphi\in \Co^{\alpha} L^\infty(\B^2,\R^1;\C).
\end{equation}

By Lemma \ref{Lem::Hold::Product} \ref{Item::Hold::Product::Hold1}, Lemma  \ref{Lem::Final::HolSobLem} \ref{Item::Final::HolSobLem::Prod1} and \ref{Item::Final::HolSobLem::Prod2},
\begin{align*}
    \| A\partial_{\bar w}\varphi\|_{\Co^{\alpha-1} L^\infty(\B^2,\R^1;\C)}&\lesssim_\alpha\|A\|_{\Co^\alpha L^\infty}\|\partial_{\bar w}\varphi\|_{\Co^{\alpha-1} L^\infty}\lesssim_A\|\varphi\|_{\Co^{\alpha} L^\infty(\B^2,\R^1;\C)},&\text{provided }\alpha>\tfrac12;
    \\
    \| A\partial_{\bar w}\varphi\|_{\Co^{\alpha-1} L^\infty(\B^2,\R^1;\C)}&\lesssim_\alpha\|A\|_{L^\infty(\R^1;L^\infty\cap W^{1,2})}\|\partial_{\bar w}\varphi\|_{\Co^{\alpha-1} L^\infty}\lesssim_A\|\varphi\|_{\Co^{\alpha} L^\infty(\B^2,\R^1;\C)},&\text{provided }0<\alpha<1;
    \\
    \| A\partial_{\bar w}\varphi\|_{\Co^\beta(\R^1;L^p(\B^2;\C))}&\lesssim_{\beta}\| A\|_{L^\infty\Co^\beta}\|\partial_{\bar w}\varphi\|_{\Co^\beta(\R^1;L^p(\B^2;\C))}\lesssim_{ A,p}\|\varphi\|_{\Co^\beta(\R^1;W^{1,p}(\B^2;\C))},&\text{for all } 1<p<\infty.
\end{align*}

Taking $p=p_\alpha:=4+16/\alpha$, we conclude that $\Tf$ is a bounded endomorphism on both $\Co^\alpha L^\infty(\B^2,\R^1;\C)$ and $\Co^\alpha L^\infty(\B^2,\R^1;\C)\cap\Co^\beta(\R^1;W^{1,4+\frac{16}\alpha}(\B^2;\C))$.

We now take $\delta=\delta(\alpha,\beta, A)>0$ small such that
\begin{equation}\label{Eqn::Final::ProofExampleddz::DefDelta}
    \delta\big(\|\Tf\|_{\Co^{\alpha} L^\infty(\B^2,\R^1;\C)}+\|\Tf\|_{\Co^{\alpha} L^\infty(\B^2,\R^1;\C)\cap\Co^\beta(\R^1;W^{1,4+16/\alpha}(\B^2;\C))}\big)\le\tfrac12.
\end{equation}
Note that $\Tf$ is now contraction map in these two spaces.

We claim that $a(w,\theta)=\delta A(w,\theta)$ is as desired.

\medskip
Suppose by contradiction there is a $U_2\times V_2\ni (0,0)$ and a $f\in\Co^{\alpha,\beta}(U_2,V_2;\C)$ solving \eqref{Eqn::Final::SharpddzRed::aPDE}. Take $0<\eps_0<\frac12$ such that $\eps_0\B^2\subseteq U_2 $ and $(-\eps_0,\eps_0)\subseteq V_2$.  We define $\tilde\chi$ and $\chi$ such that
\begin{equation}\label{Eqn::Final::ProofExampleddz::Chi}
    \tilde\chi\in C_c^\infty(-\eps_0,\eps_0),\quad\tilde\chi|_{(-\frac12\eps_0,\frac12\eps_0)}\equiv1,\quad\chi\in C_c^\infty(\eps_0\B^2\times(-\eps_0,\eps_0)),\quad\chi(w,\theta):=\tilde\chi(|w|)\tilde\chi(\theta).
\end{equation}
Therefore $\chi\in C_c^\infty(U_2)$ equals to 1 in the neighborhood $\frac{\eps_0}2\B^2\times(-\frac{\eps_0}2,\frac{\eps_0}2)$ of $(0,0)$, and $\nabla_w\chi$ vanishes in the  neighborhood $\frac{\eps_0}2\B^2\times\R^1$ of $\{w=0\}=\{0\}_w\times\R^1_\theta$.

Define $g,h\in\Co^{\alpha,\beta}_c(U_2,V_2;\C)$ as $g:=\chi f$ and $h(w,\theta):=w g(w,\theta)$, so
\begin{equation}\label{Eqn::Final::ProofExampleddz::ghPDE}
    g_w+a\cdot g_{\bar w}=(\chi_w+a\chi_{\bar w})f-\chi a_{\bar w},\quad h_w+a\cdot h_{\bar w}=g+w(\chi_w+a\chi_{\bar w})f-\chi wa_{\bar w},\quad\text{in }\B^2_w\times\R^1_\theta.
\end{equation}

\noindent \textit{Claim 1:} $h\in\Co^\beta(\R^1;W^{1,4+\frac{16}\alpha}(\B^2;\C))$ and therefore $h_{\bar w}\in\Co^\beta(\R^1;L^{4+\frac{16}\alpha}(\B^2;\C))$.

Rewriting the second equation of \eqref{Eqn::Final::ProofExampleddz::ghPDE} we have
\begin{equation*}
    h=-\Tc(ah_{\bar w})+(h-\Tc h_w)+\Tc(g+w(\chi_w+a\chi_{\bar w})f-\chi wa_{\bar w}).
\end{equation*}

In other words $h=\varphi$ is a solution to the following affine linear equation:
\begin{equation}\label{Eqn::Final::ProofExampleddz::EqnVPhi}
    \varphi=-\delta\Tf[\varphi]+(h-\Tc h_w)+\Tc(g+w(\chi_w+a\chi_{\bar w})f-\chi wa_{\bar w}).
\end{equation}

By assumption $g\in\Co^{\alpha,\beta}(\B^2,\R^1;\C)$, $f\nabla_w\chi\in\Co^{\alpha,\beta} (\B^2,\R^1;\C^2) $. By Proposition \ref{Prop::Final::Expa} \ref{Item::Final::Expa::Reg1} $wa_{\bar w}\in \Co^{\alpha,\beta}(\B^2,\R^1;\C)$. Therefore by \eqref{Eqn::Final::InvDBar::Bdd} (using $L^\infty_w\Co^\beta_\theta\subset\Co^\beta_\theta(L^{4+\frac{16}\alpha}_w)$) we get 
\begin{equation}\label{Eqn::Final::ProofExampleddz::hTerm1}
    \Tc(g+w(\chi_w+a\chi_{\bar w})f-\chi wa_{\bar w})\in\Tc(\Co^{\alpha,\beta}(\B^2,\R^1;\C))\subset\Co^{\alpha}L^\infty(\B^2,\R^1;\C)\cap \Co^\beta(\R^1;W^{1,4+\frac{16}\alpha}(\B^2;\C)).
\end{equation}

By assumption $h\in\Co^{\alpha,\beta}(\B^2,\R^1;\C)\subset L^\infty\Co^\beta(\B^2,\R^1;\C)$ is supported in $\frac12\B^2_w\times\R^1_\theta$, so by \eqref{Eqn::Final::InvDBar::Holo} $h-\Tc h_w\in\Co^\infty\Co^\beta(\B^2,\R^1;\C)$. In particular 
\begin{equation}\label{Eqn::Final::ProofExampleddz::hTerm2}
    h-\Tc h_w\in\Co^{\alpha}L^\infty(\B^2,\R^1;\C)\cap \Co^\beta(\R^1;W^{1,4+\frac{16}\alpha}(\B^2;\C)).
\end{equation}

Using \eqref{Eqn::Final::ProofExampleddz::hTerm1}, \eqref{Eqn::Final::ProofExampleddz::hTerm2} and the assumption \eqref{Eqn::Final::ProofExampleddz::DefDelta}, $\big[\varphi\mapsto -\delta\Tf[\varphi]+(h-\Tc h_w)+\Tc(g+w(\chi_w+a\chi_{\bar w})f-\chi wa_{\bar w})\big]$ is a contraction map in both $\Co^{\alpha}L^\infty(\B^2,\R^1;\C)$ and $\Co^{\alpha} L^\infty(\B^2,\R^1;\C)\cap \Co^\beta(\R^1;W^{1,4+\frac{16}\alpha}(\B^2;\C))$. Since \eqref{Eqn::Final::ProofExampleddz::EqnVPhi} has a unique fixed point in both spaces, and  $h\in\Co^{\alpha} L^\infty(\B^2,\R^1;\C)$ is a fixed point to \eqref{Eqn::Final::ProofExampleddz::EqnVPhi}, we conclude that  $h\in\Co^\beta(\R^1;W^{1,4+\frac{16}\alpha}(\B^2;\C))$. Thus $h_{\bar w}\in\Co^\beta(\R^1;L^{4+\frac{16}\alpha}(\B^2;\C))$. This finishes the proof of Claim 1.

\medskip
\noindent\textit{Claim 2:} $a\cdot g_{\bar w}\in\Co^{\eps-1}_w\Co^\beta_\theta(\B^2,\R^1;\C)$ for some $\eps>0$, and therefore $\Tc(a\cdot g_{\bar w})\in L^\infty\Co^{\beta}(\B^2,\R^1;\C)$.

Let $(\psi_j)_{j=0}^\infty$ be a dyadic resolution for $\R^1_\theta$. By Lemma \ref{Lem::Hold::LemParaProd}, for $l\ge0$, on $\C^1\times\R^1$,
\begin{align*}
    \psi_l\ast_\theta (ag_{\bar w})&=\psi_l\ast_\theta\Big(\sum_{j=l-2}^{l+2}\sum_{j'=0}^{j-3}+\sum_{j=l-2}^{l+2}\sum_{j'=0}^{j-3}+\sum_{\substack{j,j'\ge l-3;|j-j'|\le2}}\Big)(\psi_j\ast_\theta a)\cdot(\psi_{j'}\ast_\theta g_{\bar w})=:P_l^1+P_l^2+P_l^3.
\end{align*}

Our goal is to show that for $\eps=\frac\alpha{8+2\alpha}>0$, $\sup_{l\ge0;\theta\in\R^1}2^{l\beta}\|P_l^\nu(\cdot,\theta)\|_{\Co^{\eps-1}(\C^1;\C)}<\infty$ for $\nu=1,2,3$.

By assumption $a\in \Co^\beta(\R^1;L^\infty\cap W^{1,2}(\C^1;\C))$ and $g_{\bar w}\in\Co^{\alpha-1}L^\infty\cap\Co^{-1}\Co^\beta(\C^1,\R^1;\C)\subset\Co^{\frac\alpha2-1}\Co^\frac\beta2(\C^1,\R^1;\C)$, so we have $\sup_{\theta\in\R^1}\|\psi_j\ast_\theta a(\cdot,\theta)\|_{L^\infty\cap W^{1,2}}\lesssim2^{-j\beta}$ and $\sup_{\theta\in\R^1}\|\psi_{j'}\ast_\theta g_{\bar w}(\cdot,\theta)\|_{\Co^{\frac\alpha2-1}}\lesssim2^{-j'\frac\beta2}$. Thus by Lemma \ref{Lem::Final::HolSobLem} \ref{Item::Final::HolSobLem::Prod2}, we have for each $\theta\in\R^1$,
\begin{align*}
    \|P^1_l(\cdot,\theta)\|_{\Co^{\frac\alpha2-1}}&\le\|\psi_l\|_{L^1}\sum_{j=l-2}^{l+2}\sum_{j'=0}^{j-3}\|\psi_j\ast_\theta a(\cdot,\theta)\|_{L^\infty\cap W^{1,2}}\|\psi_{j'}\ast_\theta g_{\bar w}(\cdot,\theta)\|_{\Co^{\frac\alpha2-1}}\lesssim_{a,g}\sum_{j=l-2}^{l+2}\sum_{j'=0}^{j-3}2^{-j\beta}2^{-j'\frac\beta2}\lesssim2^{-l\beta};
    \\
    \|P^3_l(\cdot,\theta)\|_{\Co^{\frac\alpha2-1}}&\le\|\psi_l\|_{L^1}\sum_{\substack{j,j'\ge l-3\\|j-j'|\le2}}\|\psi_j\ast_\theta a(\cdot,\theta)\|_{L^\infty\cap W^{1,2}}\|\psi_{j'}\ast_\theta g_{\bar w}(\cdot,\theta)\|_{\Co^{\frac\alpha2-1}}\lesssim_{a,g}\sum_{\substack{j,j'\ge l-3\\|j-j'|\le2}}2^{-j\beta}2^{-j'\frac\beta2}\lesssim2^{-l\frac32\beta}.
\end{align*}
Taking supremum over $\theta\in\R^1$ we get $\sup_{l\ge0;\theta\in\R^1}2^{l\beta}(\|P^1_l(\cdot,\theta)\|_{\Co^{\frac\alpha2-1}}+\|P^3_l(\cdot,\theta)\|_{\Co^{\frac\alpha2-1}})<\infty$, bounding the first and the third term.

For the second sum we use $(\psi_j\ast_\theta a)\cdot(\psi_{j'}\ast_\theta g_{\bar w})=(\psi_j\ast_\theta(w^{-1} a))\cdot(\psi_{j'}\ast_\theta wg_{\bar w})=(\psi_j\ast_\theta(w^{-1} a))\cdot(\psi_{j'}\ast_\theta h_{\bar w})$. By Proposition \ref{Prop::Final::Expa} \ref{Item::Final::Expa::Reg2} $\sup_{\theta\in\R^1}\|\psi_j\ast_\theta(w^{-1} a)(\cdot,\theta)\|_{L^\frac{\alpha+4}2}\lesssim2^{-j\frac\beta2}$, by Claim 1 $\sup_{\theta\in\R^1}\|\psi_{j'}\ast_\theta h_{\bar w}(\cdot,\theta)\|_{L^{4+\frac{16}\alpha}}\lesssim2^{-j'\beta}$. And since the product map $L^\frac{\alpha+4}2\times L^{4+\frac{16}\alpha}\to L^\frac{4(\alpha+4)}{\alpha+8}$ is bounded, we have
\begin{align*}
    \|P^2_l(\cdot,\theta)\|_{L^\frac{4(\alpha+4)}{\alpha+8}}&\le\|\psi_l\|_{L^1}\sum_{j'=l-2}^{l+2}\sum_{j=0}^{j'-3}\|\psi_j\ast_\theta (w^{-1}a)(\cdot,\theta)\|_{L^\frac{\alpha+4}2}\|\psi_{j'}\ast_\theta h_{\bar w}(\cdot,\theta)\|_{L^{4+\frac{16}\alpha}}\lesssim\sum_{j'=l-2}^{l+2}\sum_{j=0}^{j'-3}2^{-j\frac\beta2}2^{-j'\beta}\lesssim2^{-l\beta}.
\end{align*}

By Sobolev embedding $L^\frac{4(\alpha+4)}{\alpha+8}(\C^1)\hookrightarrow\Co^{\frac\alpha{8+2\alpha}-1}(\C^1)$ (also see \cite[Chapter 2.7.1]{Triebel1}), we conclude that $\sup_{l\ge0;\theta\in\R^1}2^{l\beta}\|P^2_l(\cdot,\theta)\|_{\Co^{\frac\alpha{8+2\alpha}-1}(\C^1;\C)}<\infty$.

Since $\frac\alpha{8+2\alpha}<\frac\alpha2$, we conclude that $ag_{\bar w}\in \Co^{\frac\alpha{8+2\alpha}-1}\Co^\beta(\C^1,\R^1;\C)$. By Lemma \ref{Lem::Final::InvDBar} \ref{Item::Final::InvDBar::Bdd} we get $\Tc(a\cdot g_{\bar w})\in\Co^\frac\alpha{8+2\alpha}\Co^\beta(\B^2,\R^1;\C)\subset L^\infty\Co^{\beta}(\B^2,\R^1;\C)$, finishing the proof of Claim 2.

\medskip\noindent\textit{Final Step:}
We use $\Tc(a\cdot g_{\bar w})\in L^\infty_w\Co^{\beta}_{\theta}$ to obtain a contradiction.

Now rewriting the first equation in \eqref{Eqn::Final::ProofExampleddz::ghPDE} we have
\begin{equation}\label{Eqn::Final::ProofExampleddz::PDEofg}
    g=-\Tc(ag_{\bar w})+(g-\Tc g_w)+\Tc((\chi_w+a\chi_{\bar w})f)-\Tc((\chi-1) a_{\bar w})-\Tc( a_{\bar w}).
\end{equation}

By Claim 2, we have $\Tc(ag_{\bar w})\in L^\infty\Co^{\beta}(\B^2,\R^1;\C)$. By assumption $f\in\Co^{\alpha,\beta}_{w,\theta}$, we have $(\chi_w+a\chi_{\bar w})f\in\Co^{\alpha,\beta}(\B^2,\R^1;\C)$, so $\Tc((\chi_w+a\chi_{\bar w})f)\in \Co^1\Co^\beta(\B^2,\R^1;\C)\subset L^\infty\Co^\beta(\B^2,\R^1;\C)$ as well.

Since $g\in\Co^{\alpha,\beta}_{w,\theta}\subset L^\infty_w\Co^\beta_\theta$ is supported in $\{|w|<\frac12\}\times\R^1$, by \eqref{Eqn::Final::InvDBar::Holo} $g-\Tc g_w\in L^\infty\Co^\beta(\B^2,\R^1;\C)$  as well.

Let $\mu(\theta):=\tilde\chi(2\theta)$, we have $\mu|_{(-\frac14\eps_0,\frac14\eps_0)}\equiv1$. Thus by \eqref{Eqn::Final::ProofExampleddz::Chi}, for every $(w,\theta)\in\frac{\eps_0}2\B^2\times(-\frac{\eps_0}4,\frac{\eps_0}4)$, using $\mu(\chi-1)|_{\frac{\eps_0}2\B^2\times(-\frac{\eps_0}4,\frac{\eps_0}4)}=0$ we have
\begin{equation*}
    \Tc((\chi-1) a_{\bar w})(w,\theta)=\mu(\theta)\Tc((\chi-1) a_{\bar w})(w,\theta)=\Tc(\mu(\chi-1) a_{\bar w})(w,\theta)=\big(\Tc(\mu(\chi-1) a_{\bar w})-\mu(\chi-1) a_{\bar w}\big)(w,\theta).
\end{equation*}

By assumption $\mu(\chi-1) a_{\bar w}$ is supported in $\{\frac{\eps_0}2<|w|<\eps_0\}\times(-\frac{\eps_0}4,\frac{\eps_0}4)$, so by Proposition \ref{Prop::Final::Expa} \ref{Item::Final::Expa::Reg1} we have $\mu(\chi-1) a_{\bar w}\in L^\infty\Co^\beta(\B^2,\R^1;\C)$. Using Lemma \ref{Lem::Final::InvDBar} \ref{Item::Final::InvDBar::Holo} we get $\Tc((\chi-1) a_{\bar w})\in L^\infty\Co^\beta(\B^2,\R^1;\C)$ as well.

However, by Proposition \ref{Prop::Final::Expa} \ref{Item::Final::Expa::Unbdd} we have $\Tc(a_{\bar w})\notin L^\infty_w\Co^\beta_{\theta}$ near $(0,0)$. Therefore near $(w,\theta)=(0,0)$, the left hand side of \eqref{Eqn::Final::ProofExampleddz::PDEofg} is in $L^\infty_w\Co^\beta_\theta$ while the right hand side is not, this contradicts to the assumption $f\in\Co^{\alpha,\beta}(U_2,V_2;\C)$ and finishes the proof.
\end{proof}

To summarize this subsection, what we obtained is the following:
\begin{thm}\label{Thm::Final::SharpSum}
    Let $\alpha>\frac12$ and $\beta>0$, let $a\in\Co^{\alpha,\beta}(\B^2_w,\R^1_\theta;\C)$ be as in Theorem \ref{Thm::Final::ProofExampleddz} (see \eqref{Eqn::Final::Expa::Defofa}). Let $\Se:=\Span(\partial_w+a(w,\theta)\partial_{\bar w})$ be a rank 1 complex Frobenius structure in $\C^1\times\R^1$. 
    
    Then for any neighborhood $U\times V\subseteq\C^1_w\times\R^1_\theta$ of $(0,0)$ there is no continuous map $F:U\times V\to\C^1_z\times\R^1_s$ such that $F$ is homeomorphic to its image, $\nabla_wF$ and $\nabla_z(F^\Inv)$ are both continuous, and $F^*\Coorvec z$ spans $\Se|_{U\times V}$.
    
    In particular we have the result for Theorem \ref{Thm::Intro::Sharpddz}: if $\alpha=\beta>1$, then there is no $C^1$-coordinate chart $F:U\times V\subseteq\C^1_z\times\R^1_s$ near $(0,0)$ such that $F^*\Coorvec z$ spans $\Se|_{U\times V}$.
\end{thm}
\begin{proof}
Suppose, by contrast, such continuous map $F$ exists. Then by Proposition \ref{Prop::Final::SharpddzRed} \ref{Item::Final::SharpddzRed::Chart} $\Rightarrow$ \ref{Item::Final::SharpddzRed::PDE} there are a neighborhood $U_2\times V_2\subseteq U\times V$ of $(0,0)$ and a function $f\in\Co^{\alpha,\beta}(U_2,V_2;\C)$ that solves the PDE $f_w+af_{\bar w}=-a_{\bar w}$. However by Theorem \ref{Thm::Final::ProofExampleddz} such $f$ would not exist, contradiction.

When $\alpha=\beta>1$, a $C^1$-coordinate chart $F$ implies $\nabla_wF\in C^0$ and $\nabla_z(F^\Inv)\in C^0$ automatically, so the assumptions of Proposition \ref{Prop::Final::SharpddzRed} \ref{Item::Final::SharpddzRed::Chart} are satisfied. By Proposition \ref{Prop::Final::SharpddzRed} \ref{Item::Final::SharpddzRed::Chart} $\Rightarrow$ \ref{Item::Final::SharpddzRed::PDE} and Theorem \ref{Thm::Final::ProofExampleddz} the proof follows.
\end{proof}
\appendix
\section{Some Results with non $C^1$ Parameters}\label{Section::QPIFT}

\begin{prop}\label{Prop::AppQPIFT}Let $\alpha>0$ and $0<\beta<\min(\frac{2\alpha+1}{\alpha+1},\frac32)$. Let $F\in\Co^{\alpha+1,\beta}(\B^n_x,\B^q_s;\R^n_y)$ be a map such that $F(\cdot,s)$ is a $\Co^{\alpha+1}$-diffeomorphism onto its image and $F(\B^n,s)\supseteq\B^n$ for each $s\in\B^q$. Then $\Phi(y,s):=F(\cdot,s)^\Inv(y)$ satisfies $\Phi\in\Co^{\alpha+1,\beta}(\B^n_x,\B^q_s;\R^n_y)$.

Moreover, for any $M>1$ there is a $K_4=K_4(n,q,\alpha,\beta,M)>0$ that does not depend on $F$, such that if $\|F\|_{\Co^{\alpha+1,\beta}(\B^n,\B^q;\R^n)}<M$ and $\inf_{x\in \B^n,s\in\B^q}|\det(\nabla_x F(x,s))^{-1}|>M^{-1}$ then $\|\Phi\|_{\Co^{\alpha+1,\beta}(\B^n,\B^q;\R^n)}<K_4$.

\end{prop}
The result is indeed true for all $0<\beta\le\alpha+1$. When $\beta>1$, the proof is simpler by applying the Inverse Function Theorem on the $\Co^\beta$-map $(x,s)\mapsto (F(x,s),s)$. In our case, we deal with the difference quotiences directly.
\begin{proof}It suffices to prove the quantitative result: by assumption, $\|F(\cdot,s)\|_{\Co^{\alpha+1}}<M$ and $\inf_{x\in \B^n}|\det(\nabla_x F(x,s))^{-1}|>M^{-1}$,  for each $s\in\B^q$. Applying \cite[Lemma 5.9]{CoordAdaptedII} we get $\|\Phi(\cdot,s)\|_{\Co^{\alpha+1}(\B^n;\R^n)}\lesssim_{n,\alpha,M}1$. 

Thus to prove the proposition, it suffices to prove $\|\Phi(y,\cdot)\|_{\Co^\beta(\B^q;\R^n)}\lesssim_{n,q,\alpha,\beta,M}1$ uniformly in $y\in\B^n$.
Without loss of generality, we assume $\alpha<1$ from now on, which means $\frac{2\alpha+1}{\alpha+1}<\frac32$.

Let $y\in\B^n$ and $s_0,s_1\in\B^q$. For $j=0,1$, applying mean value theorem to $F(\cdot,s_j)$,
\begin{equation}\label{Eqn::PfQPIFT::MeanValue}
    \Phi(y,s_0)-\Phi(y,s_1)=\nabla_xF\big((1-\theta_j)\Phi(y,s_0)+\theta_j\Phi(y,s_1),s_j\big)^{-1}\cdot\big(F(\Phi(y,s_0),s_j)-F(\Phi(y,s_1),s_j))\big),\text{ for some }\theta_j\in(0,1).
\end{equation}
Now $F\in\Co^{\alpha+1,\beta}(\B^n,\B^q;\R^n)$. By Lemma \ref{Lem::Hold::GradMixHold}, we have $\nabla_xF\in\Co^{\alpha,\frac{\alpha\beta}{2\alpha+1}}(\B^n,\B^q;\R^{n\times n})$ with $\|\nabla_xF\|_{\Co^{\alpha,\frac{\alpha\beta}{2\alpha+1}}_{x,s}}\lesssim_{n,q,\alpha,\beta}\|F\|_{\Co^{\alpha+1,\beta}_{x,s}}\le M$. By \cite[Lemma 5.7]{CoordAdaptedII} and the cofactor representation of $(\nabla_xF)^{-1}$, we can find a $C_1=C_1(n,q,\alpha,\beta,M)>0$ such that
\begin{equation}\label{Eqn::PfQPIFT::C1}
    \sup_{s\in\B^q}\|\nabla_xF(\cdot,s)^{-1}\|_{\Co^\alpha(\B^n;\R^{n\times n})}+\sup_{x\in\B^n}\|\nabla_xF(x,\cdot)^{-1}\|_{\Co^\frac{\alpha\beta}{2\alpha+1}(\B^q;\R^{n\times n})}<C_1,
\end{equation}
when the assumptions $\|F\|_{\Co^{\alpha+1,\beta}(\B^n,\B^q;\R^n)}<M$ and $\inf_{x\in \B^n}|\det(\nabla_x F(x,s))^{-1}|>M^{-1}$ are satisfied. 

We separate the discussions for $\beta<1$ and $\beta>1$.

\medskip
\noindent\textit{Case $\beta\in(0,1)$}: It suffices to show $\sup_{y,s_0,s_1}|\Phi(y,s_0)-\Phi(y,s_1)|\lesssim_{n,q,\alpha,\beta,M}|s_0-s_1|^\beta$.

We have $y=F(\Phi(y,s_0),s_0)=F(\Phi(y,s_1),s_1)$ since $\Phi(\cdot,s_j)=F(\cdot,s_j)^\Inv$. Therefore
\begin{align*}
    &\textstyle|\Phi(y,s_0)-\Phi(y,s_1)|\le\big(\sup_{\B^n\times\B^q}|(\nabla_xF)^{-1}|\big)\cdot|F(\Phi(y,s_0),s_0)-F(\Phi(y,s_1),s_0))|
    \\
    \le&\textstyle C_1|F(\Phi(y,s_1),s_1)-F(\Phi(y,s_1),s_0))|\le C_1\sup_{x\in\B^n}\|F(x,\cdot)\|_{\Co^\beta(\B^q)}|s_0-s_1|^\beta\le C_1M|s_0-s_1|^\beta.
\end{align*}
This complete the proof of $0<\beta<1$.

\medskip
\noindent\textit{Case $\beta\in(1,\frac{2\alpha+1}{\alpha+1}]$}: It suffices to show $\sup_{y,s_0,s_1}|\Phi(y,s_0)+\Phi(y,s_1)-2\Phi(y,\frac{s_0+s_1}2)|\lesssim_{n,q,\alpha,\beta,M}|s_0-s_1|^\beta$.

Fix a $\tilde y\in\B^n$. For $s_0,s_1\in\B^q$ we denote $s_\frac12:=\frac{s_0+s_1}2$, $x_i:=\Phi(\tilde y,s_i)$ for $i=0,\frac12,1$. 

Take $\eps=1-\frac{\alpha+1}{2\alpha+1}\beta$, so $0<\eps<1-\frac\beta{\alpha+1}$. Since $\|F\|_{\Co^{\alpha+1,1-\eps}_{x,s}}\lesssim_\eps\|F\|_{\Co^{\alpha+1,\beta}_{x,s}}\le M$ we have $\|F\|_{\Co^{\alpha+1,1-\eps}_{x,s}}\le\tilde M$ for some $\tilde M=\tilde M(n,q,\alpha,M)>0$. Let $C_2:=K_4(n,q,\alpha,1-\eps,\tilde M)$. By the proven case $\beta=1-\eps$ above we have $\sup_{x\in\B^n}\|\Phi(x,\cdot)\|_{\Co^{1-\eps}(\B^q;\R^n)}<C_2$. In particular $|x_i-x_j|=|\Phi(\tilde y,s_i)-\Phi(\tilde y,s_j)|\le C_2\cdot|s_i-s_j|^{1-\eps}$.


By the mean value theorem we can find $\xi_0$ and $\xi_1$ contained in the triangle with vertices $x_0,x_\frac12,x_1$ such that,
\begin{equation*}
    F(x_0,s_0)-F(x_\frac12,s_0)=\nabla_xF(\xi_0,s_0)\cdot(x_0-x_\frac12),\quad F(x_1,s_1)-F(x_\frac12,s_1)=\nabla_xF(\xi_1,s_1)\cdot(x_1-x_\frac12).
\end{equation*}
Clearly $|\xi_0-\xi_1|\le\max_{i,j=0,\frac12,1}|x_i-x_j|$. Therefore by \eqref{Eqn::PfQPIFT::C1} and the property $|x_i-x_j|\le C_2\cdot|s_i-s_j|^{\beta-1+\eps}$,
\begin{equation}\label{Eqn::PfQPIFT::CTmp}
    |\nabla_xF(\xi_0,s_0)^{-1}-\nabla_xF(\xi_1,s_1)^{-1}|\le C_1(|\xi_0-\xi_1|^\alpha+|s_0-s_1|^\frac{\alpha\beta}{2\alpha+1})\le C_1(C_2^\alpha+1)|s_0-s_1|^{\beta-1+\eps}.
\end{equation}
Here we use the fact that $|s_0-s_1|^\frac{\alpha\beta}{2\alpha+1}=|s_0-s_1|^{\beta-1+\eps}$ and $|\xi_0-\xi_1|^\alpha\le C_2^\alpha|s_0-s_1|^{\alpha(1-\eps)}\le C_2^\alpha|s_0-s_1|^{\beta-1+\eps}$.

Therefore using $\tilde y=F(x_j,s_j)$ for $j=0,\frac12,1$, we have
\begin{align*}
    x_0+x_1-2x_\frac12&=\nabla_xF(\xi_0,s_0)^{-1}\cdot(F(x_0,s_0)-F(x_\frac12,s_0))+\nabla_xF(\xi_1,s_1)^{-1}\cdot(F(x_1,s_1)-F(x_\frac12,s_1))
    \\
    &=\nabla_xF(\xi_0,s_0)^{-1}\cdot(2\tilde y-F(x_\frac12,s_0)-F(x_\frac12,s_1))+\big(\nabla_xF(\xi_0,s_0)^{-1}-\nabla_xF(\xi_1,s_1)^{-1}\big)\cdot(\tilde y-F(x_\frac12,s_1)).
\end{align*}

Applying \eqref{Eqn::PfQPIFT::C1} and \eqref{Eqn::PfQPIFT::CTmp} to the above equality:
\begin{align*}
    &|x_0+x_1-2x_\frac12|\le C_1|2F(x_\frac12,s_\frac12)-F(x_\frac12,s_0)-F(x_\frac12,s_1)|+C_1(C_2^\alpha+1)|s_0-s_1|^{\beta-1+\eps}\cdot|F(x_\frac12,s_\frac12)-F(x_\frac12,s_1)|
    \\
    \lesssim&_{q,\beta} C_1\|F\|_{L^\infty\Co^\beta}|s_0-s_1|^\beta+C_1(C_2^\alpha+1)\|F\|_{L^\infty\Co^{1-\eps}}|s_0-s_1|^{\beta-1+\eps+1-\eps}\le C_1(M+(C_2^\alpha+1)\tilde M)|s_0-s_1|^\beta.
\end{align*}
Thus $\sup_{y\in\B^n}\|\Phi(y,\cdot)\|_{\Co^\beta(\B^q)}\lesssim_{n,q,\alpha,\beta,M}1$ and we complete the proof.
\end{proof}

\begin{lem}\label{Lem::AppODEReg}
Let $\alpha>1$, $\beta\in(0,\alpha]$, and let $U_0\subseteq\R^n_x$, $V_0\subseteq\R^q_s$ be two open sets. Let $X\in\Co^{\alpha,\beta}_\loc(U_0,V_0;\R^n)$ be a vector field on $U_0$ with parameter on $V_0$.

Then for any $U_1\times V_1\Subset U_0\times V_0$ there is a $\eps_0>0$ and a unique map $\Phi(t,x,s):(-\eps_0,\eps_0)\times U_1\times V_1\to U_0\times V_0$ such that $\Phi(\cdot,\cdot,s)$ is $C^1$ for each $s\in V_1$; $\frac{\partial\Phi}{\partial t}(t,x,s)=X(\Phi(t,x,s),s)$ and $\Phi(0,x,s)=x$ in the domain. Moreover $\Phi\in\Co^{\alpha+1,\alpha,\beta}((-\eps_0,\eps_0),U_1,V_1;U_0)$.
\end{lem}
\begin{proof}
Take $\eps_0<\frac12\dist(U_1,\partial U_0)$. For each $s\in V_1$, $X(\cdot,s)$ is a $\Co^\alpha$-vector field. By Lemma \ref{Lem::ODE::ODEReg}, $\exp_{X(\cdot,s)}\in \Co^{\alpha+1,\alpha}((-\eps_0,\eps_0),U_1;U_0)$ is the unique $C^1$-map $\phi(\cdot,s)$ such that $\Coorvec t\phi(\cdot,s)=X(\cdot,s)\circ\phi(\cdot,s)$ and $\phi(0,x,s)=x$. Thus, our map $\Phi(t,x,s)$ is pointwisely defined and $C^1$ in $(t,x)$. By the same argument in the proof of Lemma \ref{Lem::ODE::ODEReg} we see that $\sup_s\|\Phi(\cdot,s)\|_{\Co^{\alpha+1,\alpha}}<\infty$.

When $\beta>1$, $(X(x,s),s)$ is a $\Co^\beta$-vector field on $U_0\times V_0$ and we see that $(\Phi(t,x,s),s)$ is its ODE flow, so by Lemma \ref{Lem::ODE::ODEReg} we have $\Phi\in\Co^{\beta+1,\beta}_{t,(x,s)}$. Since we already have $\sup_s\|\Phi(\cdot,s)\|_{\Co^{\alpha+1,\alpha}}<\infty$, we get $\Phi\in\Co^{\alpha+1,\alpha,\beta}_{t,x,s}$.

When $\beta<1$, by shrinking $U_0\times V_0$ if necessary, we can assume $\|X\|_{\Co^{\alpha,\beta}(U_0,V_0)}<\infty$. We fix $x\in U_1$ and $s_0,s_1\in V_1$. Define $\psi_j(t):=\Phi(t,x,s_j)$ for $j=0,1$ and $0\le t<\eps_0$, we see that
\begin{align*}
    &\textstyle|\psi_0(t)-\psi_1(t)|\le\int_0^t|X(\psi_0(\tau),s_0)-X(\psi_1(\tau),s_1)|d\tau
    \\
    \le&\textstyle t\sup_{x\in U_0}\|X(x,\cdot)\|_{\Co^\beta(V_1)}|s_0-s_1|^\beta+\sup_{s\in V_1}\|X(\cdot,s)\|_{C^{0,1}(U_0)}\int_0^t|\psi_0(\tau)-\psi_1(\tau)|d\tau.    
\end{align*}

By Gronwall's inequality we conclude that $|\psi_0(t)-\psi_1(t)|\lesssim_{\|X\|,\eps_0}|s_0-s_1|^\beta$ for all $0\le t<\eps_0$, where the implied constant does not depend on $x,s_0,s_1$. Replacing $t$ by $-t$, we see that $|\psi_0(t)-\psi_1(t)|\lesssim_{\|X\|,\eps_0}|s_0-s_1|^\beta$ is also true for $-\eps_0<t\le0$.
Therefore $\sup_{t,x}\|\Phi(t,x,\cdot)\|_{\Co^\beta}<\infty$ and we get $\Phi\in\Co^{\alpha+1,\alpha,\beta}_{t,x,s}$.

When $\beta=1$, without loss of generality we assume $\alpha<2$ since we only need to show $\sup_{t,x}\|\Phi(t,x,\cdot)\|_{\Co^1}<\infty$. We similarly fix $x\in U_1$, $s_0,s_1\in V_1$. We define additionally $s_\frac12:=\frac{s_0+s_1}2$, $\psi_\frac12(t):=\Phi(t,x,\frac{s_0+s_1}2)$ and $X_{ij}(t):=X(\psi_i(t),s_j)$ for $i,j=0,\frac12,1$ and $-\eps_0<t<\eps_0$. Therefore  for $t\ge0$,
\begin{align*}
    &\big|\psi_0(t)+\psi_1(t)-2\psi_\frac12(t)\big|\le\int_0^t|X_{00}(\tau)+X_{11}(\tau)-2X_{\frac12\frac12}(\tau)|d\tau
    \\
    \le&\int_0^t\big|\tfrac{X_{00}(\tau)+X_{01}(\tau)}2-X_{0\frac12}(\tau)\big|d\tau+\int_0^t\big|\tfrac{X_{10}(\tau)+X_{11}(\tau)}2-X_{1\frac12}(\tau)\big|d\tau
    \\&+\int_0^t|X_{0\frac12}(\tau)+X_{1\frac12}(\tau)-2X_{\frac12\frac12}(\tau)|d\tau+\int_0^t\big|\tfrac{X_{00}(\tau)-X_{01}(\tau)}2-\tfrac{X_{10}(\tau)-X_{11}(\tau)}2\big|d\tau
    \\
    \le&2t\sup_{x\in U_0}\|X(x,\cdot)\|_{\Co^1_s(V_1)}|s_0-s_1|+\sup_{s\in V_1}\|X(\cdot,s)\|_{\Co^1_x(U_0)}\int_0^t\big|\psi_0(\tau)+\psi_1(\tau)-2\psi_\frac12(\tau)\big|d\tau
    \\
    &+\tfrac12\|X\|_{\Co^\frac\alpha2\Co^\frac12(U_0,V_1)}|s_0-s_1|^\frac12\int_0^t|\psi_0(\tau)-\psi_1(\tau)|^\frac\alpha2d\tau.
\end{align*}

By Remark \ref{Rmk::Hold::RmkforBiHold} \ref{Item::Hold::RmkforBiHold::Interpo} we have $\|X\|_{\Co^\frac\alpha2\Co^\frac12(U_0,V_1)}\lesssim\|X\|_{\Co^\alpha \Co^0\cap\Co^0\Co^1}\lesssim\|X\|_{\Co^{\alpha,1}_{x,s}}$. By the proven case $\beta=\frac 1\alpha<1$ we have $|s_0-s_1|^\frac12\int_0^t|\psi_0(\tau)-\psi_1(\tau)|^\frac\alpha2d\tau\le|s_0-s_1|^\frac12t\sup_{t,x}\|\Phi(t,x,\cdot)\|_{\Co^\frac1\alpha}^\frac\alpha2|s_0-s_1|^{\frac\alpha2\cdot\frac1\alpha}\lesssim_{\alpha,X,\eps_0}|s_0-s_1|$. Therefore $\big|\psi_0(t)+\psi_1(t)-2\psi_\frac12(t)\big|\le C_{X,\alpha,\eps_0}(|s_0-s_1|+\int_0^t\big|\psi_0(\tau)+\psi_1(\tau)-2\psi_\frac12(\tau)\big|d\tau)$ for some $C_{X,\alpha,\eps_0}$ that does not depend on $x,s_0,s_1$. Thus by Gronwall's inequality again we get $\sup_{t,x}\|\Phi(t,x,\cdot)\|_{\Co^1}<\infty$, finishing the proof.
\end{proof}

\section{Holomorphic Extension of Laplacian}\label{Section::SecHolLap}

In this section, for a holomorphic function $f(z)$ we use $\nabla f=\partial_zf=(\partial_1f,\dots,\partial_nf)$ (cf. \eqref{Eqn::Intro::ColumnNote}).

We use $\Sp^{n-1}=\partial\B^{n-1}$ for unit sphere.
We recall $\Hb^n=\{x+iy\in\C^n:|x|<1,4|y|<1-|x|\}$ from \eqref{Eqn::HCone}. We use $\Oh(\Hb^n)=\Oh_\loc(\Hb^n)$ for the space of holomorphic functions in $\Hb^n$, endowed with  compact-open topology.

Our goal in this section is to prove Proposition \ref{Prop::HolLap}. The idea follows from Morrey \cite{Analyticity}, which only consider the cases for positive $\alpha$.
By taking direct modification to our construction (see \eqref{Eqn::SecHolLap::FormulaForP}) we can show that there exist an inverse Laplacian $\Pv$ and its extension $\tilde \Pv$ such that $\Pv:\Co^{\alpha-2}(\B^n)\to\Co^\alpha(\B^n)$ and $\tilde\Pv:\Co^{\alpha-2}_\Oh(\Hb^n)\to\Co^\alpha(\Hb^n)$ is bounded for all $-M<\alpha<M$, where $M$ is an arbitrarily fixed large number.

We illustrate some property of holomorphic H\"older space in Section \ref{Section::SecHLLem}. We construct the inverse Laplacian $\Pv$ in Definition \ref{Defn::SecHolLap::DefofP} and its holomorphic extension $\tilde\Pv$ in Definition \ref{Defn::SecHolLap::DefofTildeP}. We decompose the proof of Proposition \ref{Prop::HolLap} into Propositions \ref{Prop::HolLap::PfE}, \ref{Prop::SecHolLap::Pv}, \ref{Prop::SecHolLap::TildePisHolo} and \ref{Prop::SecHolLap::BddTildeP}. 

In this section we use $\Ga$ as the Newtonian potential in $\R^n$, namely
\begin{equation}\label{Eqn::SecHolLap::NewtonPotent}
    \Ga(x)=\begin{cases}\frac{|x|}2&n=1,\\\frac1{2\pi}\log|x|&n=2,\\
	\frac{|x|^{2-n}}{(2-n)|\Sp^{n-1}|}&n\ge3,
	\end{cases}\quad x\in\R^n\backslash\{0\}.
\end{equation}

Note that $\R^n\backslash\{0\}$ is contained in $\{z\in\C^n:\re(z\cdot z)>0\}=\{z\in\C^n:|\re z|>|\im z|\}$, which is a connect set in $\C^n$. So $x\in\R^n\backslash\{0\}\mapsto|x|$ has unique holomorphic extension to $\{|\re z|>|\im z|\}$ as 
\begin{equation}\label{Eqn::SecHolLap::Eqn|z|}
    z\in\big\{|\re z|>|\im z|\big\}\mapsto(z\cdot z)^\frac12,\quad\text{with principle argument }-\tfrac\pi2<2\arg z<\tfrac\pi 2.
\end{equation}

We still use $\Ga$ as the holomorphic extended fundamental solution on $\{|\re z|>|\im z|\}$. 

Therefore, for each $\theta\in\Sp^{n-1}=\partial\B^n$, the function $[x\in\B^n\mapsto\Ga(x-\theta)]$ can be holomorphic extended to $[z\in\Hb^n\mapsto\Ga(z-\theta)]$. Moreover, for every multindex $\nu\neq0$, since $|\partial^\nu\Ga(x)|\lesssim_\nu|x|^{2-n-|\nu|}$, we have the following
\begin{align}\label{Eqn::SecHolLap::DevGa1}
    |(\partial^\nu\Ga)(z-t)|\lesssim_\nu|\re z-t|^{2-n-|\nu|}&,\quad\forall t\in\R^n,z\in\C^n\text{ satisfying }|\re z-t|\ge2|\im z|.
\end{align}




\subsection{Holomorphic H\"older spaces and extension of harmonic functions}\label{Section::SecHLLem}
For holomorphic functions in a complex domain whose H\"older norms are finite, we can use the blow-up speed of its derivative near the boundary.

Let $\alpha<1$, recall in Definition \ref{Defn::PDE::HoloHoldSpace} that $\Co^\alpha_\Oh(\Hb^n)$ is the space of all holomorphic functions in $\Hb^n$ such that $\|f\|_{\Co^\alpha_\Oh}=|f(0)|+\sup_{z\in\Hb^n}\dist(z,\Hb)^{1-\alpha}|\nabla f(z)|<\infty$.

We have in fact $\Co^\alpha_\Oh(\Hb^n)=\Co^\alpha(\Hb^n;\C)\cap\Oh(\Hb^n)$, see \cite[Chapter 7]{ZhuSpaceBall} for some illustrations. For completeness we prove some partial results that are used later.

\begin{lem}[Hardy-Littlewood type characterizations]\label{Lem::SecHolLap::HLLem}
\begin{enumerate}[parsep=-0.3ex,label=(\roman*)]
    \item\label{Item::SecHolLap::HLLem::Grad} Let $\alpha<1$. The differentiation $[f\mapsto\nabla f]:\Co^\alpha_\Oh(\Hb^n)\to\Co^{\alpha-1}_\Oh(\Hb^n;\C^n)$ is bounded linear.
    \item\label{Item::SecHolLap::HLLem::0Char} Let $\alpha<0$. Then $\Co^\alpha_\Oh(\Hb^n)$ has equivalent norm 
    $f\mapsto\sup_{z\in\Hb^n}\dist(z,\partial\Hb^n)^{-\alpha}|f(z)|$.
    
    \item\label{Item::SecHolLap::HLLem::Res} The restriction map $[f\mapsto f|_{\B^n}]:\Co^\alpha_\Oh(\Hb^n)\to\Co^\alpha(\B^n;\C)$ and the inclusion map $\Co^\alpha_\Oh(\Hb^n)\hookrightarrow L^\infty(\Hb^n;\C)$ are both bounded linear for $0<\alpha<1$.
    \item\label{Item::SecHolLap::HLLem::Harm} Let $\alpha<1$. If $g\in\Co^\alpha(\B^n)$ satisfies $\Delta g=0$, then $\sup_{x\in\B^n}\dist(x,\Sp^{n-1})^{1-\alpha}|\nabla g(x)|<\infty$. Moreover there is a $C=C(n,\alpha)>0$ such that $$\sup_{x\in\B^n}\dist(x,\Sp^{n-1})^{1-\alpha}|\nabla g(x)|\le C_3\|g\|_{\Co^\alpha(\B^n)},\quad\forall g\in\Co^\alpha(\B^n)\cap\ker\Delta.$$
\end{enumerate}
\end{lem}

\begin{proof}
\ref{Item::SecHolLap::HLLem::Grad}:  By assumption $|\partial_jf(z)|\le\|f\|_{\Co^\alpha_\Oh}\dist(z,\partial\Hb^n)^{\alpha-1}$. By Cauchy's integral formula \begin{equation}\label{Eqn::SecHolLap::HLLem::CauInt}
    \partial_jg(z)=\frac1{2\pi  r}\int_0^{2\pi}e^{-i\theta}g(z+re^{i\theta}\mathbf e_j)d\theta,\quad 1\le j\le n,\quad 0<r<\dist(z,\partial\Hb^n),\quad g\in\Oh(\Hb^n).
\end{equation}

So take $r=\frac12\dist(z,\partial\Hb^n)$ and $g=\partial_kf$, we get $|\partial_{jk}f(z)|\le \frac1r\sup_{|w-z|=r}|\partial_k f(w)|\le2^{2-\alpha}\|f\|_{\Co^\alpha_\Oh}\dist(z,\partial\Hb^n)^{\alpha-2}$, for $1\le j,k\le n$, which means $\sup_{z\in\Hb^n}\dist(z,\partial\Hb^n)^{2-\alpha}|\nabla^2 f(z)|\lesssim\|f\|_{\Co^\alpha_\Oh}$. This proves \ref{Item::SecHolLap::HLLem::Grad}.
 
\medskip\noindent\ref{Item::SecHolLap::HLLem::0Char}: Now $\alpha<0$. Clearly $|f(0)|\le\dist(0,\partial\Hb^n)^\alpha\cdot \sup_z\dist(z,\partial\Hb^n)^{-\alpha}|f(z)|$. If $\sup_z\dist(z,\partial\Hb^n)^{-\alpha}|f(z)|<\infty$, then by \eqref{Eqn::SecHolLap::HLLem::CauInt} with $g=f$ and $r=\frac12\dist(z,\partial\Hb^n)$ we get $|\nabla f(z)|\le \frac1r\sup_{|w-z|=r}|f(w)|\lesssim\dist(z,\partial\Hb^n)^{\alpha}$. So $\|f\|_{\Co^\alpha_\Oh}\lesssim \sup_z\dist(z,\partial\Hb^n)^{-\alpha}|f(z)|$.

Conversely since $\Hb^n$ is a convex set, we have $f(z)=f(0)+\int_0^1z\cdot \nabla f(tz)dt$ for all $f\in\Oh(\Hb^n)$ and $z\in\Hb^n$. Note that $\dist(tz,\partial\Hb^n)\ge t\dist(z,\partial\Hb^n)+(1-t)\dist(0,\partial\Hb^n)$ for all $0\le t\le 1$. Therefore for $f\in\Co^\alpha_\Oh(\Hb^n)$,
\begin{align*}
    |f(z)|&\le|f(0)|+|z|\|f\|_{\Co^\alpha_\Oh}\int_0^1\big(t\dist(z,\partial\Hb^n)+(1-t)\dist(0,\partial\Hb^n)\big)^{\alpha-1}dt
    \\
    &\lesssim\|f\|_{\Co^\alpha_\Oh}+\|f\|_{\Co^\alpha_\Oh}\int_0^{1-\dist(z,\partial\Hb^n)}(1-t)^{\alpha-1}dt\lesssim_\alpha\dist(z,\partial\Hb^n)^\alpha\|f\|_{\Co^\alpha_\Oh}.
\end{align*}

So $\sup_z\dist(z,\partial\Hb^n)^{-\alpha}|f(z)|\lesssim\|f\|_{\Co^\alpha_\Oh} $, completing the proof of \ref{Item::SecHolLap::HLLem::0Char}.

\medskip\noindent\ref{Item::SecHolLap::HLLem::Res}: Now $0<\alpha<1$. It suffices to show $\Co^\alpha_\Oh(\Hb^n)\subset\Co^\alpha(\Hb^n;\C)$. The boundedness of the inclusion and restriction would then follow from the natural maps $\Co^\alpha(\Hb^n)\hookrightarrow L^\infty(\Hb^n)$ and $\Co^\alpha(\Hb^n)\twoheadrightarrow\Co^\alpha(\B^n)$ respectively.

Now for $z_0,z_1\in\Hb^n$, what we need is to show that $|f(z_0)-f(z_1)|\lesssim_{n,\alpha}|z_0-z_1|^\alpha\|f\|_{\Co^\alpha_\Oh(\Hb^n)}$.

Since $\Hb^n$ is convex, by its geometry we can find a point $w\in\Hb^n$ (depending on $z_0,z_1$) such that $|w-z_j|\le|z_0-z_1|$ and $\dist(tw+(1-t)z_j,\partial\Hb^n)\ge\frac1{100} t|z_0-z_1|$ for $j=0,1$ and $0\le t\le 1$. Therefore
\begin{align*}
    |f(z_0)-f(z_1)|&\le |f(z_0)-f(w)|+|f(z_1)-f(w)|\le \sum_{j=0}^1\int_0^1|z_j-w||\nabla f(tw+(1-t)z_j)|dt
    \\&\lesssim|z_0-z_1|\|f\|_{\Co^\alpha_\Oh}\int_0^1(|z_0-z_1|t)^{\alpha-1}dt\lesssim|z_0-z_1|^\alpha\|f\|_{\Co^\alpha_\Oh}.
\end{align*}

Thus $\Co^\alpha_\Oh(\Hb^n)\subset\Co^\alpha(\Hb^n;\C)$ and we finish the proof of \ref{Item::SecHolLap::HLLem::Res}.

\medskip\noindent\ref{Item::SecHolLap::HLLem::Harm}: Let $\rho\in C_c^\infty(\B^n)$ be a radial function such that $\int\rho=1$. Let $\rho_j(x):=2^{jn}\rho(2^jx)$ for $j\ge0$. Therefore, for a harmonic function $g$ in $\B^n$, by the ball average formula we have $g(x)=\rho_k\ast g(x)$ whenever $1-|x|>2^{-k}$.

Let $(\phi_j)_{j=0}^\infty$ be a dyadic resolution (see Definition \ref{Defn::Hold::DyadicResolution}) for $\R^n$. Since $\nabla\rho$ is Schwartz for every $M>0$ we have $\|\phi_j\ast \nabla\rho\|_{L^1}\lesssim_{\phi,\rho,M}2^{-Mj}$ for $j\ge0$, so by scaling, we get $\|\phi_j\ast\nabla \rho_k\|_{L^1}=2^k\|\phi_{j-k}\ast\nabla\rho\|_{L^1}\lesssim_{\phi,\rho,M}2^{k-M(j-k)}$ if $j\ge k$. Therefore $\|\phi_j\ast\nabla\rho_k\|_{L^1}\lesssim_{\phi,\rho,M}2^k\min(1,2^{-M(j-k)})$ for all $j,k\ge0$.

Let $\psi_j:=\sum_{k=j-1}^{j+1}\phi_j$ for $j\ge0$ (we use $\phi_{-1}:=0$). By \eqref{Eqn::Hold::RmkDyaSupp} we have $\psi_j\ast \phi_j=\phi_j$. Therefore
\begin{equation*}
    \|\psi_j\ast\nabla\rho_k\|_{L^1}\lesssim_M 2^k\min(1,2^{-M(j-k)}),\quad j,k\ge0.
\end{equation*}

Now for a harmonic function $g\in\Co^\alpha(\B^n)$, consider an extension $\tilde g:=E_xg\in\Co^\alpha(\R^n)$ where $E_x$ is in Lemma \ref{Lem::Hold::CommuteExt}. By Lemma \ref{Lem::Hold::HoldChar} \ref{Item::Hold::HoldChar::LPHoldChar} (if $0<\alpha<1$) and Definition \ref{Defn::Hold::NegHold} (if $\alpha\le0$) we have $\|\phi_k\ast\tilde g\|_{L^\infty}\lesssim_{\phi,E_x}\|g\|_{\Co^\alpha(\B^n)}2^{-k\alpha}$.
Thus, for any $x\in\B^n$, take $k\ge0$ to be the unique integer such that $1-2^{1-k}\le |x|<1-2^{-k}$, we get
\begin{align*}
    &|\nabla g(x)|=|\nabla\rho_k\ast g(x)|\le\|\nabla\rho_k\ast\tilde g\|_{L^\infty}\le\sum_{j=0}^\infty\|\phi_j\ast\psi_j\ast\nabla\rho_k\ast\tilde g\|_{L^\infty}=\sum_{j=0}^\infty\|\psi_j\ast\nabla\rho_k\ast\phi_j\ast\tilde g\|_{L^\infty}
    \\
    \le&\sum_{j=0}^\infty\|\psi_j\ast\nabla\rho_k\|_{L^1}\|\phi_j\ast\tilde g\|_{L^\infty(\R^n)}\lesssim_{\alpha,\phi}\sum_{j=0}^\infty2^k\min(1,2^{-(|\alpha|+1)(j-k)})2^{-k\alpha}\|\tilde g\|_{\Co^\alpha(\R^n)}\lesssim_{\alpha}2^{k(1-\alpha)}\|g\|_{\Co^\alpha(\B^n)}.
\end{align*}
We conclude that $|\nabla g(x)|\lesssim(1-|x|)^{\alpha-1}\|g\|_{\Co^\alpha(\B^n)}$ for $g\in\Co^\alpha(\B^n)$ and $x\in\B^n$, finishing the proof of \ref{Item::SecHolLap::HLLem::Harm}.
\end{proof}



We can now prove Proposition \ref{Prop::HolLap} \ref{Item::HolLap::E}. Recall that for a real analytic function $f$, we use $\Ex f$ to be the (unique) holomorphic extension to the complex domain where it is defined.

\begin{prop}\label{Prop::HolLap::PfE}The holomorphic extension
    $\Ex:\Co^\alpha(\B^n)\cap\ker\Delta\to\Co^\alpha_\Oh(\Hb^n)$ is defined and bounded for $-2<\alpha<1$.
\end{prop}
\begin{proof}Let $f\in\Co^\alpha(\B^n)$ be such that $\Delta f=0$. By Lemma \ref{Lem::SecHolLap::HLLem} \ref{Item::SecHolLap::HLLem::Harm} $f\in C^0_\loc(\B^n)$, and by Poisson's formula (See \cite[Chapter 2.2 (45)]{Evans} for example)
$$ f(x)=\frac{r^2-|x-x_0|^2}{|r\Sp^{n-1}|}\int_{\partial B(x_0,r)}\frac{f(\theta)d\sigma(\theta)}{|x-\theta|^n},\quad\text{for }x_0\in \B^n,\quad 0<r<1-|x_0|,\quad x\in B^n(x_0,r).$$
Here $d\sigma$ is the standard spherical measure. 

Using \eqref{Eqn::SecHolLap::Eqn|z|} we have holomorphic extension,
$$\Ex f(z)=\frac{r^2-(z-x_0)^2}{|r\Sp^{n-1}|}\int_{\partial B(x_0,r)}\frac{f(\theta)d\sigma(\theta)}{((z-\theta)^2)^\frac n2},\quad x_0\in\B^n,\quad r<1-|x_0|,\quad|\im z|<r-|\re z-x_0|.$$
For $z\in\Hb^n$, take $x_0=\re z$ and $r=\frac{1-|x_0|}2$, we can write
\begin{equation}\label{Eqn::SecHolLap::EqnEx}
    \Ex f(x+iy)=\frac{(\frac{1-|x|}2)^2+y^2}{|\frac{1-|x|}2\cdot\Sp^{n-1}|}\int_{\frac{1-|x|}2\Sp^{n-1}}\frac{f(x+\theta)d\sigma(\theta)}{((iy-\theta)^2)^\frac n2},\quad x+iy\in\Hb^n.
\end{equation}

Since $|y|<\frac14(1-|x|)=\frac12|\theta|$ for $\theta\in\frac{1-|x|}2\Sp^{n-1}$, we have $\re((iy-\theta)^2)\approx|\theta|^2\approx(1-|x|)^2$.

Note that $\nabla f$ is also harmonic, which gives $\nabla_z\Ex f=\Ex(\nabla_xf)$. Therefore by Lemma \ref{Lem::SecHolLap::HLLem} \ref{Item::SecHolLap::HLLem::Harm},
$$|\nabla\Ex f(x+iy)|\lesssim\frac{(1-|x|)^2}{1-|x|}\int_{\frac{1-|x|}2\Sp^{n-1}}\frac{|\nabla f(x+\theta)|d\sigma(\theta)}{(1-|x|)^n}\lesssim\sup\limits_{\partial B^n(x,\frac{1-|x|}2)}|\nabla f|\lesssim(1-|x|)^{\alpha-1}\|f\|_{\Co^\alpha(\B^n)}.$$

Since $\dist(z,\partial\Hb^n)\ge\dist(x,\partial\Hb^n)\approx 1-|x|$, we get $\sup_{z\in\Hb^n}\dist(z,\partial\Hb^n)^{1-\alpha}|\nabla\Ex f(z)|\lesssim\|f\|_{\Co^\alpha(\B^n)}$. 

Clearly $|\Ex f(0)|\lesssim\|f\|_{\Co^\alpha(\B^n)}$, so by definition of $\Co^\alpha_\Oh$-spaces we get $\|\Ex f\|_{\Co^\alpha_\Oh(\Hb^n)}\lesssim\|f\|_{\Co^\alpha(\B^n)}$, finishing the proof.
\end{proof}

\subsection{Construction and H\"older regularity for real inverse Laplacian $\Pv$}

We define $\Pv f=(\Ga\ast Ef)|_{\B^n}$ where $E$ is a concrete extension operator for $\B^n$ such that $E:\Co^\alpha(\B^n)\to\Co^\alpha(\R^n)$ is bounded linear for $-4<\alpha<1$. 
\begin{lem}[{\cite[Theorem 2.9.2]{Triebel1}}]\label{Lem::SecHolLap::ExtLem}
Let $M\in\Z_+$, let  $(a_j,b_j)_{j=1}^M\subset\R$ satisfy $b_j>0$, $1\le j\le M$ and $\sum_{j=1}^Ma_j(-b_j)^k=1$ for $-4\le k\le 4$. Let $(\tilde\chi_j)_{j=1}^M\subset C_c^\infty(\R)$ satisfies $\chi_j(t)\equiv1$ for $t$ closed to 0 and each $1\le j\le M$. Define
\begin{equation}\label{Eqn::HolLap::ExtHPOp}
    E^Hf(x',x_n):=\begin{cases}f(x',x_n)&x_n>0,\\\sum_{j=1}^M\tilde\chi_j(x_n)a_j\cdot f(x',-b_jx_n)&x_n<0.\end{cases}
\end{equation}
Then $E^H$ defines an extension operator on the half plane $\R^n_+=\{x_n>0\}$ such that $E^H:\Co^\alpha(\R^n_+)\to\Co^\alpha(\R^n)$ for all $-4<\alpha<1$.
\end{lem}

\begin{cor}
Let $(a_j,b_j)_{j=1}^9\subset\R$ satisfy $b_j>0$, $1\le j\le 9$ and $\sum_{j=1}^9a_j(-b_j)^{-k}=1$ for $-4\le k\le 4$. For $1\le j\le 9$, let $\tilde\chi_j\in C_c^\infty(\R)$ be such that $\tilde\chi_j\equiv1$ in an neighborhood of $0$.
Then 
\begin{equation}\label{Eqn::SecHolLap::ExtBallOp}
    Ef(e^{-\rho}\theta):=\begin{cases}f(e^{-\rho}\theta)&\rho>0,\\\sum_{j=1}^9a_j\tilde\chi_j(\rho)f(e^{\rho/b_j}\theta)&\rho<0,\end{cases}\quad\rho\in\R,\quad\theta\in\Sp^{n-1},
\end{equation}
defines an extension operator on $\B^n$ such that $E:\Co^\alpha(\B^n)\to\Co^\alpha(\R^n)$ is bounded linear for $-4<\alpha<1$.

Moreover its formal adjoint $E^*$ is given by
\begin{equation}\label{Eqn::SecHolLap::EqnE*}
    E^*g(x)=g(x)+\sum_{j=1}^9\frac{a_jb_j}{|x|^{n(b_j+1)}}\tilde\chi_j(b_j\log|x|)g\Big(\frac x{|x|^{b_j+1}}\Big),\quad x\in\B^n.
\end{equation}

\end{cor}
\begin{proof}
Clearly $(\rho,\theta)\mapsto e^{-\rho}\theta$ is a diffeomorphism from $\R\times\Sp^{n-1}$ to $\R^n\backslash\{0\}$ and maps $\{\rho>0\}$ onto $\B^n\backslash\{0\}$. By \cite[Theorem 2.10.2(i)]{Triebel1} we see that $\Co^\alpha=\Bs_{\infty\infty}^\alpha$ is preserved by bounded diffeomorphism. Therefore, by passing to coordinate covers of $\Sp^{n-1}$ if necessary (since we only need to worry about the place where $\rho$ is closed to $0\in\R$), by Lemma \ref{Lem::SecHolLap::ExtLem} with $b_j$ replacing by $\frac1{b_j}$, we know \eqref{Eqn::SecHolLap::ExtBallOp} is an extension operator for $\B^n$ and has $\Co^\alpha$-boundedness for $-4<\alpha<1$.

To see \eqref{Eqn::SecHolLap::EqnE*}, let $g:\R^n\to\R$, we have
\begin{align*}
    &\int_{\R^n}Ef(y)g(y)dt=\int_{\Sp^{n-1}} d\theta\int_\R Ef(e^{-\rho}\theta)g(e^{-\rho}\theta)e^{-n\rho}d\rho
   \\=&\int_{\Sp^{n-1}}\bigg(\int_{\R_+} f(e^{-\rho}\theta)g(e^{-\rho}\theta)e^{-n\rho}d\rho+\sum_{j=1}^9a_j\tilde\chi_j(\rho)\int_{\R_-}f(e^{\rho/b_j}\theta)g(e^{-\rho}\theta)e^{-n\rho}d\rho\bigg) d\theta
   \\=&\int_{\Sp^{n-1}} d\theta\int_{\R_+} f(e^{-r}\theta)g(e^{-r}\theta)e^{-nr}dr+\sum_{j=1}^9a_jb_j\int_{\Sp^{n-1}} d\theta\int_{\R_+}\tilde\chi_j(b_jr)f(e^{-r}\theta)g(e^{b_jr}\theta)e^{nr(b_j+1)}e^{-nr}dr
   \\=&\int_{\B^n} f(x)\bigg(g(x)+\sum_{j=1}^9\frac{a_jb_j\cdot\tilde\chi_j(b_j\log|x|)}{|x|^{n(b_j+1)}}g\Big(\frac x{|x|^{b_j+1}}\Big)\bigg)dx.
\end{align*}



This gives \eqref{Eqn::SecHolLap::EqnE*} and finishes the proof.
\end{proof}

We now give an concrete formula of $\Pv$ based on \eqref{Eqn::SecHolLap::EqnE*},

\begin{defn}\label{Defn::SecHolLap::DefofP} We define $\Pv$ be
\begin{equation}\label{Eqn::SecHolLap::FormulaForP}
    \Pv f(x):=\int_{\B^n} f(t)\bigg(\Ga(x-t)+\chi(|t|)\sum_{j=1}^9\frac{a_jb_j}{|t|^{n(b_j+1)}}\Ga\Big(x-\frac{t}{|t|^{b_j+1}}\Big)\bigg)dt.
    \end{equation}
Here $\chi\in C_c^\infty(\frac13,\frac53)$ satisfies $\chi|_{[\frac12,\frac32]}\equiv1$, and $a_j\in\R,b_j>0$ satisfy $\sum_{j=1}^9a_j(-b_j)^k=1$ for $-4\le k\le 4$.
\end{defn}

Thus we have Proposition \ref{Prop::HolLap} \ref{Item::HolLap::P} almost immediately.
\begin{prop}\label{Prop::SecHolLap::Pv}
    Let $\Pv$ defined as in \eqref{Eqn::SecHolLap::FormulaForP}, then $\Pv$ is an left inverse of Laplacian such that $\Pv:\Co^\alpha(\B^n)\to\Co^{\alpha+2}(\B^n)$ is bounded for $-4<\alpha<-1$. 
    
\end{prop}
\begin{proof}
By \eqref{Eqn::SecHolLap::EqnE*} we see that $\Pv f$ is of the form $\Pv f(x)=\int_{\B^n}f(t)E^*_t\Ga(x-t)dt$ for $f\in L^1(\B^n)$, where $E^*$ is the adjoint of some extension operator $E$ which is $\Co^\alpha$-bounded for $-4<\alpha<1$. Therefore $\Pv f(x)=(\Ga\ast Ef)(x)$ for $x\in\B^n$ and we get
$\Pv f=(\Ga\ast Ef)|_{\B^n}$. In particular $\Delta\Pv f=(\Delta\Ga\ast Ef)|_{\B^n}=f$ holds.

By construction, such $E$ has compact support, namely there is a bounded open set $U\subset\R^n$ (in fact we can take $U=B^n\big(0,e^{\max_j\frac1{3b_j}}\big)$) such that $\supp Ef\subset U$ has compact support for every $f\in\Co^{(-4)+}(\B^n)$. So $\Ga\ast Ef$ is a well-defined distribution.

Since $\Delta(\Ga\ast Ef)|_{\B^n}=Ef|_{\B^n}=f$, the classical interior regularity for elliptic equations (see \cite[Proposition 4.1]{TaylorPDE3} for example) shows that if $Ef\in\Co^\alpha(U)$, then $\Ga\ast Ef\in\Co^{\alpha+2}(\B^n)$. So $f\in\Co^\alpha(\B^n)$ implies $\Ga\ast Ef\in\Co^{\alpha+2}(\B^n)$. By the Closed Graph Theorem $\Pv=[f\mapsto\Ga\ast Ef]:\Co^\alpha(\B^n)\to\Co^{\alpha+2}(\B^n)$ is bounded linear.
\end{proof}



We decompose the integral in \eqref{Eqn::SecHolLap::FormulaForP} into the domains $\{t:|\re z-t|<\frac{1-|\re z|}2\}$ and $\{t:|\re z-t|>\frac{1-|\re z|}2\}$. We need the following estimate.
\begin{lem}\label{Lem::SecHolLap::AprioriBddforP}Let $a,b,\chi$ be given in Definition \ref{Defn::SecHolLap::DefofP}. There is a $C=C(n,a,b,\chi)>0$, such that for every $z\in\Hb^n$ and $t\in\B^n$ satisfying $|\re z-t|\ge\frac{1-|\re z|}2$,
\begin{equation}\label{Eqn::SecHolLap::AprioriBddforPEqn}
    \bigg|(\nabla\Ga)(z-t)+\chi(|t|)\sum_{j=1}^9\frac{a_jb_j}{|t|^{n(b_j+1)}}(\nabla\Ga)\Big(z-\frac{t}{|t|^{b_j+1}}\Big)\bigg|\le C|\re z-t|^{-n-3}(1-|t|)^4.
\end{equation}
\end{lem}
\begin{proof}
That $|\re z-t|\ge\frac{1-|\re z|}2$ implies $|\re z-t|\ge2|\im z|$, so by \eqref{Eqn::SecHolLap::DevGa1} the left hand side of \eqref{Eqn::SecHolLap::AprioriBddforPEqn} is finite.

When $1-|t|\ge\frac12|\re z-t|$, we have $|\re z-\frac t{|t|^{b_j+1}}|\ge \frac 1{b_j}|\re z-t|$. Thus by \eqref{Eqn::SecHolLap::DevGa1},
$$\textstyle\text{LHS of }\eqref{Eqn::SecHolLap::AprioriBddforPEqn}\lesssim|\re z-t|^{1-n}+|\re z-\frac t{|t|^{b_j+1}}|^{1-n}\lesssim_{n,b} |\re z-t|^{1-n}\lesssim|\re z-t|^{-3-n}(1-|t|)^4.$$

We then focus on the case $1-|t|<\frac12|\re z-t|$.

By assumption $\sum_{j=1}^9a_jb_j(-b_j)^k=-1$ for $0\le k\le3$. Therefore, by the Taylor's expansion of deg 3,  for any $\tilde \chi\in C_c^\infty(\R)$ that equals 1 in the neighborhood of 0, there is a $C'=C_{a,b,\tilde\chi}'>0$ such that
\begin{equation}\label{Eqn::SecHolLap::AprPEqn2}
    \textstyle\big|\psi(\rho)+\tilde\chi(\rho)\sum_{j=1}^9a_jb_j\psi(-b_j\rho)\big|\le C'\|\psi\|_{C^4}\cdot\rho^4,\quad\forall\ R>0,\  \psi\in C^4[-R\max_{1\le j\le 9}b_j,R],\  \rho\in[0,R].
\end{equation}

For fixed $t$ and $z$ in this case, we define $I_t:=\big[\log|t|\max\limits_{1\le j\le 9}b_j,-\log|t|\big]$ and $\psi_{z,t}(\rho):=e^{-n\rho}\cdot(\nabla\Ga)(z-e^{-\rho}\frac t{|t|})$ for $\rho\in I_t$. Since for $\rho\in I_t$ we have $|\re z-e^{-\rho}\frac t{|t|}|\ge|\re z-t|$ and by assumption $|\re z-t|\ge2|\im z|$ holds in this case, using \eqref{Eqn::SecHolLap::DevGa1} we have $\|\psi_{z,t}\|_{C^4(I_t)}\lesssim|\re z-t|^{-n-3}$.

Thus taking $\tilde\chi(\rho):=\chi(e^{-\rho})$ and $\rho=-\log|t|(=R)$ in \eqref{Eqn::SecHolLap::AprPEqn2} with $\psi=\psi_{z,t}$ we get
$$\bigg||t|^n(\nabla\Ga)(z-t)+\chi(|t|)\sum_{j=1}^9\frac{a_jb_j}{|t|^{nb_j}}(\nabla\Ga)\Big(z-\frac{t}{|t|^{b_j+1}}\Big)\bigg|\lesssim |\re z-t|^{-n-3}(1-|t|)^4.$$

Note that $1-|t|<\frac12|\re z-t|$ implies $\frac13<|t|<1$, so by multiplying $\frac1{|t|^n}$ we get \eqref{Eqn::SecHolLap::AprioriBddforPEqn} for the case $1-|t|<\frac12|\re z-t|$. This finishes the proof.
\end{proof}

\begin{cor}\label{Cor::SecHolLap::IntofP}
If $(1-|t|)^4f(t)\in  L^1(\B^n)$, then the following integral converges locally uniformly for  $z\in\Hb^n$: $$\displaystyle\int_{\{t\in\B^n:|\re z-t|\ge\frac{1-|\re z|}2\}}f(t)\Big(\Ga(z-t)-\chi(|t|)\sum_{j=1}^9\frac{a_jb_j}{|t|^{n(b_j+1)}}\Ga\big(z-\frac{t}{|t|^{b_j+1}}\big)\Big)dt.$$
\end{cor}
\begin{proof}By Lemma \ref{Lem::SecHolLap::AprioriBddforP} (taking an anti-derivative if necessary), for each $z\in\Hb^n$ there is a $C_z>0$ which is locally bounded in $z$, such that $\left|\Ga(z-t)-\sum_{j=1}^9\frac{a_jb_j}{|t|^{n(b_j+1)}}\Ga\big(z-\frac{t}{|t|^{b_j+1}}\big)\right|\le C_z(1-|t|)^4$ holds for $t\in\B^n\cap B^n(\re z,\frac{1-|\re z|}2)$. So the integral is bounded by $C_z\int_{\{t\in\B^n:|\re z-t|\ge\frac{1-|\re z|}2\}}(1-|t|)^4|f(t)|dt$ which converges since $[t\mapsto (1-|t|)^4f(t)]\in  L^1(\B^n)$.
\end{proof}

\subsection{Construction of the holomorphic extension $\tilde{\mathbf P}$}\label{Section::SecHolLap::DefTildeP}

We now extend $\Pv$ to the complex cone $\Hb^n$. 

In Sections \ref{Section::SecHolLap::DefTildeP} and \ref{Section::SecHolLap::BddTildeP}, we use a function\footnote{In \cite{Analyticity} he use $1-\lambda$ of our $\lambda$.} $\lambda:\B^n\times\B^n\to[0,1]$ as the following
\begin{equation}\label{Eqn::SecHolLap::DefLambda}
    \lambda(t,x):=\begin{cases}1-\frac{2|t-x|}{1-|x|},&\text{if }2|t-x|\le1-|x|\\0,&\text{if }2|t-x|\ge1-|x|\end{cases}.
\end{equation}
Immediately we have:
\begin{enumerate}[parsep=-0.3ex,label=($\Lambda$.\arabic*)]
    \item\label{Item::SecHolLap::LambdaPro1} $\lambda$ is locally Lipschitz, and the function $[(t,x+iy)\mapsto\lambda(t,x)\cdot y]:\B^n\times\Hb^n\to\R^n$ is bounded Lipschitz.
    \item\label{Item::SecHolLap::LambdaPro2} For every $t\in\B^n$ and $x+iy\in\Hb^n$, we have $4|\lambda(t,x)\cdot y|<1-|x|$ and $|(1-\lambda(t,x))y|\le\frac12|t-x|$.
\end{enumerate}

For each $z=x+iy\in\Hb^n$, let 
\begin{equation}\label{Eqn::SecHolLap::NotSzV}
    \textstyle S_z:=\{t+i\lambda(t,x) y:t\in\B^n\},\quad V(t,z):=1+iy\cdot\partial_t\lambda(t,x)=1+i\sum_{j=1}^ny^j(\partial_{t^j}\lambda)(t,x).
\end{equation}

By \ref{Item::SecHolLap::LambdaPro2} we see that $S_z\subset\Hb^n$, and $z-\zeta\in\{w\in\C^n:|\re w|>|\im w|\}$ for each $\zeta\in S_z$, so $\Ga(z-\zeta)$ is defined for $\zeta\in S_z$.

Fix $z\in\Hb^n$, by viewing $S_z$ as a set parameterized by $t$, we have a change of variable $\zeta=\zeta(t,z):=t+i\lambda(t,x)y$. Its Jacobian matrix is $\frac{\partial\zeta}{\partial t}(t,z)=1+iy\otimes\partial_t\lambda(t,x)$, so $\det\frac{\partial\zeta}{\partial t}(t,z)=1+iy\cdot\partial_t\lambda(t,x)$. We denote the ``volume form'' $V(t,z)$ as
\begin{equation*}
    V(t,x+iy):=1+iy\cdot\partial_t\lambda(t,x),\quad t\in\B^n,\quad x+iy\in\Hb^n.
\end{equation*}

In this way for an integrable function $g$ on $S_z$, we can formally write
$$\textstyle\int_{S_z}g(\zeta)d\zeta=\int_{\B^n}g(t+i\lambda(t,x)y)(1+iy\cdot \partial_t\lambda(t,x))dt.$$

\begin{defn}\label{Defn::SecHolLap::DefofTildeP}
Let $(a_j)_{j=1}^9,(b_j)_{j=1}^9,\chi$ to be as in Definition \ref{Defn::SecHolLap::DefofP}. Let $S_z$ be as in \eqref{Eqn::SecHolLap::NotSzV} for $z\in\Hb^n$. Define 
\begin{equation}\label{Eqn::SecHolLap::EqnHoloP0}
    \tilde\Pv f(z):=\int_{S_z}f(\zeta)\Ga(z-\zeta)d\zeta+\int_{\B^n}\chi(|t|)\sum_{j=1}^9\frac{a_jb_j}{|t|^{n(b_j+1)}}f(t)\Ga\Big(z-\frac{t}{|t|^{b_j+1}}\Big)dt,\quad\text{for }f\in\Co^{(-4)+}_\Oh(\Hb^n),
\end{equation}
in the sense that,
\begin{align}
    \tilde\Pv f(z)=\tilde\Pv f(x+iy)=
    &
    \label{Eqn::SecHolLap::EqnHoloP1}
    \int_{|t-x|<\frac{1-|x|}2}f(t+i\lambda(t,x)y)\Ga(x-t+i(1-\lambda(t,x))y)V(t,z)dt
    \\
    &\label{Eqn::SecHolLap::EqnHoloP2}
    +\int_{|t-x|<\frac{1-|x|}2}\sum_{j=1}^9\frac{a_jb_j\chi(|t|)}{|t|^{n(b_j+1)}}f(t)\Ga\Big(z-\frac{t}{|t|^{b_j+1}}\Big)dt\\
    &\label{Eqn::SecHolLap::EqnHoloP3}
    +\int_{t\in\B^n:|t-x|>\frac{1-|x|}2}f(t)\bigg(\Ga(z-t)-\sum_{j=1}^9\frac{a_jb_j\chi(|t|)}{|t|^{n(b_j+1)}}\Ga\Big(z-\frac{t}{|t|^{b_j+1}}\Big)\bigg)dt
    \\
    =:&\tilde\Pv_1 f(z)+\tilde\Pv_2 f(z)+\tilde\Pv_3 f(z)\label{Eqn::SecHolLap::EqnHoloPall}
\end{align}

\end{defn}

By \ref{Item::SecHolLap::LambdaPro2}, we see that \eqref{Eqn::SecHolLap::EqnHoloP1} is a classical Lebesgue integral pointwisely for $z\in\Hb^n$ and all $f\in\Oh(\Hb^n)$. The same holds for \eqref{Eqn::SecHolLap::EqnHoloP2}.

When $f\in\Co^{(-4)+}_\Oh(\Hb^n)$, by Lemma \ref{Lem::SecHolLap::HLLem} \ref{Item::SecHolLap::HLLem::0Char} we have $[t\mapsto (1-|t|)^4f(t)]\in L^\infty(\B^n;\C)$. By Corollary \ref{Cor::SecHolLap::IntofP},  \eqref{Eqn::SecHolLap::EqnHoloP3} is integrable as well.

\medskip
Note that $S_{x+i0}\equiv\B^n$. So comparing the expression \eqref{Eqn::SecHolLap::FormulaForP} and \eqref{Eqn::SecHolLap::EqnHoloP0} we see that $\tilde \Pv f(x+i0)=\Pv (f(\cdot+i0))(x)$. Therefore to prove $\tilde \Pv$ is the holomorphic extension of $\Pv$, what we need is the following:

\begin{prop}\label{Prop::SecHolLap::TildePisHolo}
Let $f\in\Co^{(-4)+}_\Oh(\Hb^n)$, then $\tilde \Pv f$ gives a holomorphic function defined on $\Hb^n$.

Moreover we have $(\tilde \Pv f)|_{\B^n}=\Pv[f|_{\B^n}]$ and $\Delta\tilde\Pv f=f$, where $\Delta=\sum_{j=1}^n\frac{\partial^2}{(\partial z^j)^2}$.
\end{prop}
We postpone its proof after Proposition \ref{Prop::SecHolLap::PfTildePHolo}. Essentially we can assume $f\in L^\infty(\Hb^n;\C)$ so that the two integrals in \eqref{Eqn::SecHolLap::EqnHoloP0} converge individually:
\begin{prop}\label{Prop::SecHolLap::PfTildePHolo}
    For $f\in L^\infty(\Hb^n;\C)\cap\Oh(\Hb^n)$, the function $z\in\Hb^n\mapsto\int_{S_z}f(\zeta)\Ga(z-\zeta)d\zeta$ is holomorphic. Moreover
    \begin{equation}
        \Coorvec{z^j}\int_{S_z}f(\zeta)\Ga(z-\zeta)d\zeta=\int_{\B^n}f(t+i\lambda(t,x)y)\Ga_{,j}(x-t+iy(1-\lambda(t,x)))V(t,z)dt,\quad 1\le j\le n,\quad z\in\Hb^n.
    \end{equation}
\end{prop}Here we use the comma notation $\Ga_{,j}(\zeta)=(\partial_{z^j}\Ga)(\zeta)$ for $\zeta$ in the domain.
\begin{proof} Using the comma notation we can write $(\partial_{z^j}f)(\zeta)=f_{,x^j}(\zeta)=-if_{,y^j}(\zeta)=f_{,j}(\zeta)$.

In the following computation, we use abbreviation $f=f(t+i\lambda(t,x)y)$, $\Ga=\Ga(x-t+i(1-\lambda(t,x))y)$ and $V=V(t,z)=1+iy\cdot\lambda_{,t}(t,x)$. First we claim to have the following (cf. \cite[(4.12)]{Analyticity}):
\begin{gather}\label{Eqn::SecHolLap::PfTildePHolo::Tmp1}
    \sum_{k=1}^n\Coorvec{t^k}(y^kf\Ga\nabla\lambda)=\sum_{k=1}^n(y^k(f_{,k}\Ga-f\Ga_{,k})V\nabla\lambda+y^kf\Ga(\nabla\lambda)_{,t^k}).
    \\\label{Eqn::SecHolLap::PfTildePHolo::Tmp2}
    \Coorvec{t^j}(\lambda f\Ga V)-i\sum_{k=1}^n\Coorvec{t^k}(y^k\lambda\lambda_{,t^j}f\Ga)=f\Ga\lambda_{,t^j}+\lambda(f_{,j}\Ga-f\Ga_{,j})V\quad\text{for }j=1,\dots,n.
\end{gather}
Here $\nabla\lambda=(\lambda_{,x^1},\dots,\lambda_{,x^n},\lambda_{,t^1},\dots,\lambda_{,t^n})$.

Indeed by chain rule, \eqref{Eqn::SecHolLap::PfTildePHolo::Tmp1} is done by the following,
    \begin{align*}
    &\textstyle\sum_{k=1}^n\Coorvec{t^k}(y^kf\Ga\nabla\lambda)
    \\
    =&\textstyle\sum_{k=1}^ny^kf_{,k}\Ga\nabla\lambda+\sum_{k,l=1}^niy^kf_{,l}y^l\lambda_{,t^k}\Ga\nabla\lambda-\sum_{k=1}^ny^kf\Ga_{,k}\nabla\lambda+\sum_{k,l=1}^niy^k\Ga_{,l}y^l\lambda_{,t^k}\nabla\lambda+\sum_{k=1}^ny^kf\Ga\nabla\lambda_{,t^k}
    \\
    =&\textstyle\sum_{k,l=1}^ny^k(1+iy^l\lambda_{,t^l})(f_{,k}\Ga-f\Ga_{,k})\nabla\lambda+\sum_{k=1}^ny^kf\Ga\nabla\lambda_{,t^k}
    \\
    =&\textstyle\sum_{k=1}^ny^k(f_{,k}\Ga-f\Ga_{,k})V\nabla\lambda+\sum_{k=1}^ny^kf\Ga\nabla\lambda_{,t^k}.
    \end{align*}
    
Since $\partial_t\lambda$ is the component of $\nabla\lambda$, using \eqref{Eqn::SecHolLap::PfTildePHolo::Tmp1} we get \eqref{Eqn::SecHolLap::PfTildePHolo::Tmp2} by the following
\begin{align*}
    &\textstyle\Coorvec{t^j}(\lambda f\Ga V)-i\sum_{k=1}^n\Coorvec{t^k}(y^k\lambda\lambda_{,t^j}f\Ga)
    \\
    =&\textstyle\lambda_{,t^j}f\Ga V+\lambda(f_{,j}\Ga-f\Ga_{,j})V+i\sum_{k=1}^ny^k\lambda\lambda_{,t^j}(f_{,k}\Ga-f\Ga_{,k})V+i\lambda f\Ga\sum_{k=1}^ny^k\lambda_{,t^jt^k}
    \\
    &\textstyle-i\sum_{k=1}^ny^k\lambda_{,t^k}\lambda_{,t^j}f\Ga-i\lambda\sum_{k=1}^n\Coorvec{t^k}(y^kf\Ga\lambda_{,t^j})
    \\
    =&\textstyle\lambda_{,t^j}f\Ga \big(1+i\sum_{k=1}^ny^k\lambda_{,t^k}\big)+\lambda(f_{,j}\Ga-f\Ga_{,j})V-i\sum_{k=1}^ny^k\lambda_{,t^k}\lambda_{,t^j}f\Ga
    \\
    =&f\Ga\lambda_{,t^j}+\lambda(f_{,j}\Ga-f\Ga_{,j})V.
\end{align*}

Next we compute $\Coorvec{x^j}\int_{\B^n}f\Ga Vdt$ and $\Coorvec{y^j}\int_{\B^n}f\Ga Vdt$. Note that $\lambda=0$ and $\nabla\lambda=0$ near $t\in\Sp^{n-1}$, so both \eqref{Eqn::SecHolLap::PfTildePHolo::Tmp1} and \eqref{Eqn::SecHolLap::PfTildePHolo::Tmp2} have integrals zero on $t\in\B^n$. Therefore,
\begin{align*}
    \textstyle\Coorvec{x^j}\int_{\B^n}f\Ga Vdt=&\textstyle\int_{\B^n}f\Ga_{,j}V+\sum_{k=1}^n\big(if_{,k}y^k\lambda_{,x^j}\Ga V-if\Ga_{,k}y^k\lambda_{,x^j}V+if\Ga y^k\lambda_{,t^kx^k}\big)dt
    \\
    =&\textstyle\int_{\B^n}f\Ga_{,j}Vdt+i\sum_{k=1}^n\int_{\B^n}\big(y^k(f_{,k}\Ga-f\Ga_{,k})V\lambda_{,x^j}+y^kf\Ga\lambda_{,x^jt^k}\big)dt
    \\
    =&\textstyle\int_{\B^n}f\Ga_{,j}Vdt+i\sum_{k=1}^n\int_{\B^n}\Coorvec{t^k}(y^kf\Ga\lambda_{,x^j})dt&(\text{by }\eqref{Eqn::SecHolLap::PfTildePHolo::Tmp1})
    \\
    =&\textstyle\int_{\B^n}f\Ga_{,j}Vdt+i\sum_{k=1}^n\int_{\Sp^{n-1}}y^kf\Ga\lambda_{,x^j}d\sigma_k(t)
    \\
    =&\textstyle\int_{\B^n}f\Ga_{,j}Vdt.&(\nabla\lambda|_{\Sp^{n-1}}=0)
\end{align*}
\begin{align*}
    \textstyle\Coorvec{y^j}\int_{\B^n}f\Ga Vdt=&\textstyle i\int_{\B^n}\big(f_{,j}\lambda\Ga V+f\Ga_{,j}(1-\lambda)V+f\Ga\lambda_{,t^j}\big)dt&(f_{,j}=-if_{,y^j})
    \\
    =&\textstyle i\int_{\B^n}f\Ga_{,j}Vdt+i\int_{\B^n}\big(f\Ga\lambda_{,t^j}+\lambda(f_{,j}\Ga-f\Ga_{,j})V\big)dt
    \\
    =&\textstyle i\int_{\B^n}f\Ga_{,j}Vdt+i\int_{\B^n}\Coorvec{t^j}(\lambda f\Ga V)dt+\sum_{k=1}^n\int_{\B^n}\Coorvec{t^k}(y^k\lambda\lambda_{,t^j}f\Ga)dt&(\text{by }\eqref{Eqn::SecHolLap::PfTildePHolo::Tmp2})
    \\
    =&\textstyle i\int_{\B^n}f\Ga_{,j}Vdt+i\int_{\Sp^{n-1}}\lambda f\Ga Vd\sigma_j(t)+\sum_{k=1}^n\int_{\Sp^{n-1}}y^k\lambda\lambda_{,t^j}f\Ga d\sigma_k(t)
    \\=&\textstyle i\int_{\B^n}f\Ga_{,j}Vdt.&(\lambda|_{\Sp^{n-1}}=0)
\end{align*}
Here we use $d\sigma_j$ as the $j$-th outer normal direction of $\Sp^{n-1}=\partial\B^n$.

Therefore $\Coorvec{x^j}\int_{\B^n}f\Ga Vdt=i\Coorvec{y^j}\int_{\B^n}f\Ga Vdt=\int_{\B^n}f\Ga_{,j}Vdt$, finishing the proof.
\end{proof}

\begin{proof}[Proof of Proposition \ref{Prop::SecHolLap::TildePisHolo}]
When $f\in L^\infty(\Hb^n;\C)$, two integrals in \eqref{Eqn::SecHolLap::EqnHoloP0} converge individually. 

By Proposition \ref{Prop::SecHolLap::PfTildePHolo} the first integral $z\mapsto\int_{S_z}f(\zeta)\Ga(z-\zeta)d\zeta$ is holomorphic in $\Hb^n$. Clearly the second integral in \eqref{Eqn::SecHolLap::EqnHoloP0} is holomorphic since $z\mapsto\Ga(z-t/|t|^{b_j+1})$ is a holomorphic in $\Hb^n$ as well. So $\tilde \Pv f\in\Oh(\Hb^n)$ when $f\in L^\infty(\Hb^n;\C)\cap\Oh(\Hb^n)$.

For general $f\in\Co^{(-4)+}_\Oh(\Hb^n)$, by Lemma \ref{Lem::SecHolLap::HLLem} \ref{Item::SecHolLap::HLLem::0Char} we have $f|_{\B^n}\in L^1(\B^n,(1-|x|)^4dx;\C)$. Take $f_\eps(z):=f((1-\eps)z)$ for $\eps>0$, we have $f_\eps\in L^\infty(\Hb^n;\C)$ for each $\eps$, and $f_\eps|_{\B^n}\xrightarrow{\eps\to0}f|_{\B^n}$ converges in $L^1(\B^n,(1-|x|)^4dx;\C)$. Therefore in the notation \eqref{Eqn::SecHolLap::EqnHoloPall} we have $\lim_{\eps\to0}\tilde\Pv_3f_\eps(z)=\tilde \Pv_3f(z)$ locally uniformly in $z$.

Since the convergence $f_\eps\xrightarrow{\eps\to0}f$ holds automatically in $\Oh(\Hb^n)$, we see that $\lim_{\eps\to0}\tilde\Pv_1f_\eps(z)=\tilde \Pv_1f(z)$ and $\lim_{\eps\to0}\tilde\Pv_2f_\eps(z)=\tilde \Pv_2f(z)$ locally uniformly in $z$ as well.
Therefore $\tilde\Pv f_\eps(z)\xrightarrow{\eps\to0}\tilde\Pv f(z)$ locally uniformly in $z\in\Hb^n$. By Proposition \ref{Prop::SecHolLap::PfTildePHolo} $\tilde\Pv f_\eps(z)$ is holomorphic, we conclude that $\tilde\Pv f(z)$ is also holomorphic.

For $f\in\Co^{(-4)+}_\Oh(\Hb^n)$, by Lemma \ref{Lem::SecHolLap::HLLem} \ref{Item::SecHolLap::HLLem::Res}  $f|_{\B^n}\in\Co^{(-4)+}(\B^n;\C)$ so $\Pv[f|_{\B^n}]\in\Co^{(-2)+}(\B^n;\C)$ is defined. 

By \eqref{Eqn::SecHolLap::NotSzV} we have $S_{x+i0}=\B^n$ for all $x\in\B^n$. By comparing \eqref{Eqn::SecHolLap::FormulaForP} with \eqref{Eqn::SecHolLap::EqnHoloP1}, \eqref{Eqn::SecHolLap::EqnHoloP2} and \eqref{Eqn::SecHolLap::EqnHoloP3}, we get $(\tilde \Pv f)|_{\B^n}=\Pv[f|_{\B^n}]$ for $f\in\Co^{(-4)+}_\Oh(\Hb^n)$.

And since $\B^n\subset\C^n$ is a totally real submanifold with full dimension $n$, by uniqueness of the holomorphic extension we have $\tilde \Pv f=\Ex\big[\Pv[f|_{\B^n}]\big]$ for all $f\in\Co^{(-4)+}_\Oh(\Hb^n)$. Therefore $\Delta_z\tilde \Pv f=\Ex\big[\Delta_x\Pv[f|_{\B^n}]\big]=\Ex[f|_{\B^n}]=f$ holds for $f\in\Co^{(-4)+}_\Oh(\Hb^n)$ as well.
\end{proof}

\subsection{H\"older-Zygmund regularities for $\tilde{\mathbf P}$}\label{Section::SecHolLap::BddTildeP}

Recall $\tilde\Pv$ in Definition \ref{Defn::SecHolLap::DefofTildeP}. In this section we are going to prove $\tilde \Pv:\Co^\alpha_\Oh(\Hb)\to\Co^{\alpha+2}_\Oh(\Hb)$ for $-4<\alpha<-1$. By Proposition \ref{Prop::SecHolLap::TildePisHolo} we already know $\tilde\Pv:\Co^{(-4)+}_\Oh(\Hb^n)\to\Oh(\Hb^n)$ is the holomorphic extension of $\Pv$. By Definition \ref{Defn::PDE::HoloHoldSpace} and the assumption $\alpha+2<1$, to prove Proposition \ref{Prop::HolLap} \ref{Item::HolLap::TildeP} it remains to show the following:
\begin{prop}\label{Prop::SecHolLap::BddTildeP}
    Let $-4<\alpha<-1$. There is a $C=C_{n,\alpha}>0$ such that $$|\nabla_z\tilde\Pv f(z)|\le C_{n,\alpha}\|f\|_{\Co^\alpha_\Oh(\Hb^n)}\dist(z,\partial\Hb^n)^{\alpha+1},\quad f\in\Co^\alpha_\Oh(\Hb^n).$$
\end{prop}
\begin{proof}
 By Proposition \ref{Prop::SecHolLap::PfTildePHolo} we have for $k=1,\dots,n$ and $z=x+iy\in\Hb^n$,
\begin{align}
    \partial_{z^k}\tilde\Pv f(z)=
    &\label{Eqn::SecHolLap::EstHoloP1}
    \int_{|t-x|<\frac{1-|x|}2}f(t+i\lambda(t,x)y)\Ga_{,k}(x-t+i(1-\lambda(t,x))y)(1+iy\cdot\partial_t\lambda(t,x))dt
    \\\label{Eqn::SecHolLap::EstHoloP2}
    &+\int_{|t-x|<\frac{1-|x|}2}f(t)\sum_{j=1}^9\frac{a_jb_j\chi(|t|)}{|t|^{n(b_j+1)}}\Ga_{,k}\Big(z-\frac t{|t|^{b_j+1}}\Big)dt
    \\
    &\label{Eqn::SecHolLap::EstHoloP3}
    +\int_{t\in\B^n:|t-x|>\frac{1-|x|}2}f(t)\bigg(\Ga_{,k}(z-t)-\sum_{j=1}^9\frac{a_jb_j\chi(|t|)}{|t|^{n(b_j+1)}}\Ga_{,k}\Big(z-\frac{t}{|t|^{b_j+1}}\Big)\bigg)dt.
\end{align}

We need to show the absolute value of \eqref{Eqn::SecHolLap::EstHoloP1}, \eqref{Eqn::SecHolLap::EstHoloP2} and \eqref{Eqn::SecHolLap::EstHoloP3} are all bounded by a constant times $\|f\|_{\Co^\alpha_\Oh}\dist(z,\partial\Hb^n)^{\alpha+1}$.
 
We first estimate \eqref{Eqn::SecHolLap::EstHoloP1}. By Lemma \ref{Lem::SecHolLap::HLLem} \ref{Item::SecHolLap::HLLem::0Char} we have $|f(z)|\lesssim\|f\|_{\Co^\alpha_\Oh}\dist(z,\partial\Hb^n)^\alpha$; by \eqref{Eqn::SecHolLap::DevGa1} and \eqref{Eqn::SecHolLap::DefLambda} we have $|F_{,k}(z-t+i(1-\lambda)y)|\lesssim|x-t|^{1-n}$; by property \ref{Item::SecHolLap::LambdaPro1} we have $|1+iy\cdot\partial_t\lambda|\lesssim1$. Therefore
\begin{equation}\label{Eqn::SecHolLap::PfEstHolo::Tmp1}
    \Big|\int_{|t-x|<\frac{1-|x|}2}f\Ga_{,k}Vdt\Big|\lesssim_{n,\alpha}\|f\|_{\Co^\alpha_\Oh}\int_{|t-x|<\frac{1-|x|}2}\dist(t+i\lambda y,\partial\Hb^n)^\alpha\cdot |x-t|^{1-n}dt 
\end{equation}


Since $t+i\lambda y$ lays in a line segment connecting $z$ and a point in $\partial B^n(x,\frac{1-|x|}2)$, by convexity of $\Hb^n$ we get $\dist(t+i\lambda(t,x)y,\partial\Hb^n)\ge\lambda(t,x)\dist(z,\partial\Hb^n)+(1-\lambda(t,x))\dist(\partial B^n(x,\frac{1-|x|}2),\partial\Hb^n)$. Since $\dist(z,\partial\Hb^n)\approx\frac{1-|x|}2-2|y|$ and $\dist\big(\partial B^n(x,\frac{1-|x|}2),\partial\Hb^n\big)\approx \frac{1-|x|}2$, by plugging in the expression of $\lambda(t,x)$ we get (recall $\alpha+1<0$)
\begin{align*}
    &\int_{|t-x|<\frac{1-|x|}2}\dist(t+i\lambda y,\partial\Hb^n)^\alpha\cdot |x-t|^{1-n}dt
    \\
    \le& \int_{|t-x|<\frac{1-|x|}2}\big(\lambda\dist(z,\partial\Hb^n)+(1-\lambda)\dist\big(\partial B^n(x,\tfrac{1-|x|}2),\partial\Hb^n\big)\big)^\alpha\cdot |x-t|^{1-n}dt 
    \\
    \lesssim&\int_{B^n(0,\frac{1-|x|}2)}\Big(\big(1-\tfrac{2|s|}{1-|x|}\big)\big(\tfrac{1-|x|}2-2|y|\big)+\tfrac{2|s|}{1-|x|}\tfrac{1-|x|}2\Big)^\alpha|s|^{1-n}ds
    \\
    \lesssim&\int_0^{\frac{1-|x|}2}\Big(\tfrac{1-|x|}2-2|y|+\tfrac{4|y|r}{1-|x|}\Big)^\alpha dr\lesssim_\alpha\Big(\tfrac{1-|x|}2-2|y|\Big)^{\alpha+1}\approx_\alpha\dist(z,\partial\Hb^n)^{\alpha+1}.
\end{align*}
Putting this in \eqref{Eqn::SecHolLap::PfEstHolo::Tmp1} we control \eqref{Eqn::SecHolLap::EstHoloP1}.

\medskip
For \eqref{Eqn::SecHolLap::EstHoloP2} and \eqref{Eqn::SecHolLap::EstHoloP3}, we use $\dist(t,\partial\Hb^n)\approx\dist(t,\partial\B^n)=1-|t|$ for $t\in\B^n$. By Lemma \ref{Lem::SecHolLap::HLLem} \ref{Item::SecHolLap::HLLem::0Char} $|f(t)|\lesssim\|f\|_{\Co^\alpha_\Oh}(1-|t|)^\alpha$.

By \eqref{Eqn::SecHolLap::DevGa1} we have $|\Ga_{,k}(z-\frac t{|t|^{b_j+1}})|\lesssim|x-\frac t{|t|^{b_j+1}}|^{1-n}\le (1-|x|)^{1-n}$ since $\frac t{|t|^{b_j+1}}\notin\B^n$ when $t\in\B^n$. So we bound \eqref{Eqn::SecHolLap::EstHoloP2} by the following
\begin{align*}
    &\Big|\int_{|t-x|<\frac{1-|x|}2}f(t)\sum_{j=1}^9\frac{a_jb_j\chi(|t|)}{|t|^{n(b_j+1)}}\Ga_{,k}\Big(z-\frac t{|t|^{b_j+1}}\Big)dt\Big|
    \lesssim\|f\|_{\Co^\alpha_\Oh}\int_{|t-x|<\frac{1-|x|}2}(1-|t|)^\alpha(1-|x|)^{1-n}dt
    \\\lesssim&\|f\|_{\Co^\alpha_\Oh}(1-|x|)^{\alpha+1}\lesssim\|f\|_{\Co^\alpha_\Oh}\dist(z,\partial\Hb^n)^{\alpha+1}.
\end{align*}

By Lemma \ref{Lem::SecHolLap::AprioriBddforP} we bound \eqref{Eqn::SecHolLap::EstHoloP3} by the following
\begin{align*}
    &\bigg|\int_{t\in\B^n:|t-x|>\frac{1-|x|}2}f(t)\bigg(\Ga_{,k}(z-t)-\sum_{j=1}^9\frac{a_jb_j\chi(|t|)}{|t|^{n(b_j+1)}}\Ga_{,k}\Big(z-\frac{t}{|t|^{b_j+1}}\Big)\bigg)dt\bigg|
    \\
    \lesssim&\|f\|_{\Co^\alpha_\Oh}\int_{t\in\B^n:|t-x|>\frac{1-|x|}2}(1-|t|)^\alpha|x-t|^{-n-3}(1-|t|)^4dt
    \\
    \lesssim&\|f\|_{\Co^\alpha_\Oh}\bigg(\int_{|t|<\frac{3|x|-1}2}+\int_{\substack{\frac{3|x|-1}2<|t|<1\\|t-x|>\frac{1-|x|}2}}\bigg)(1-|t|)^{4+\alpha}|x-t|^{-n-3}dt
    \\
    \lesssim&\|f\|_{\Co^\alpha_\Oh}\int_0^{\frac{3|x|-1}2}(1-r)^{\alpha+4}r^{n-1}dr\int_{|x|-r}^1\rho^{-n-3}\rho^{n-2}d\rho+\|f\|_{\Co^\alpha_\Oh}\int_{\frac{3|x|-1}2}^1(1-r)^{\alpha+4}dr\int_{\frac{1-|x|}2}^1\rho^{-n-3}\rho^{n-2}d\rho
    \\
    \lesssim&\|f\|_{\Co^\alpha_\Oh}\Big(\int_0^{|x|}(1-r)^{\alpha}dr+(1-|x|)^{\alpha+5}(1-|x|)^{-4}\Big)\approx\|f\|_{\Co^\alpha_\Oh}(1-|x|)^{\alpha+1}\lesssim\|f\|_{\Co^\alpha_\Oh}\dist(z,\partial\Hb)^{\alpha+1}.
\end{align*}
Here we use $\rho$ as the radial parameter of polar coordinates on $\Sp^{n-1}$ centered at $\frac x{|x|}$.

Thus \eqref{Eqn::SecHolLap::EstHoloP1}, \eqref{Eqn::SecHolLap::EstHoloP2} and \eqref{Eqn::SecHolLap::EstHoloP3} are all bounded by $\|f\|_{\Co^\alpha}\dist(z,\partial\Hb^n)^{\alpha+1}$. The proof is now complete.
\end{proof}

{\small\addcontentsline{toc}{section}{References}
\bibliographystyle{amsalpha}
\bibliography{Bibliography}}

\center{\textit{University of Wisconsin-Madison, Department of Mathematics, 480 Lincoln Dr., Madison, WI, 53706}}

\center{\textit{lyao26@wisc.edu}}

\center{MSC 2020: 58A30 (Primary), 32Q60, 53C15 and 53C12 (Secondary)}
\end{document}